\documentclass[12pt]{amsart}

\usepackage{amsmath}
\usepackage{amssymb}
\usepackage{latexsym}
\usepackage{mathdots}
\usepackage{xcolor}
\usepackage{float}

\usepackage[pagebackref,colorlinks=true,linkcolor=blue,citecolor=blue]{hyperref}

\UseRawInputEncoding

\newtheorem{thr}{{\sc Theorem}}[section]

\newtheorem{lem}[thr]{{\sc Lemma}}   
\newtheorem{cor}[thr]{{\sc Corollary}}
\newtheorem{prop}[thr]{{\sc Proposition}}
\newtheorem{note}[thr]{{\sc Note}}
\newtheorem{rem}[thr]{{\sc Remark}}
\newtheorem{rems}[thr]{{\sc Remarks}}
\newtheorem{defn}[thr]{{\sc Definition}}

\numberwithin{equation}{section}

\begin{document}

\title[The generic dual of p-adic groups]{The generic dual of p-adic groups and applications}

\vspace{.25in}

\author{Chris Jantzen}
\address{Department of Mathematics \\ East Carolina University \\ Greenville, NC 27858 U.S.A.}
\email{jantzenc@ecu.edu}

\author{Baiying Liu}
\address{Department of Mathematics \\ Purdue University \\ West Lafayette, IN 47907 U.S.A.}
\email{liu2053@purdue.edu}

\date{}

\subjclass[2000]{Primary: 22E50, 11S37}

\keywords{generic representations, local Langlands parameters}

\thanks{The research of the first-named author is partially supported by a summer research award from Thomas Harriot College of Arts and Sciences.
The research of the second-named author is partially supported by the NSF Grants DMS-1702218, DMS-1848058, and by start-up funds from the Department of Mathematics at Purdue University}

\begin{abstract}
In this paper, we give a uniform classification of the generic dual of quasi-split classical groups, their similitude counterparts, and general spin groups. As applications, for quasi-split classical groups, we show that the functorial lifting maps constructed by Cogdell, Kim, Piatetski-Shapiro and Shahidi are surjective. We also analyze structures of general local Langlands parameters and explicitly construct a distinguished element for each local $L$-packet.
\end{abstract}

\maketitle

\tableofcontents
\listoftables

\section{Introduction}\label{sect1}

Let $F$ be a $p$-adic field of characteristic 0.
In \cite{JL14}, the authors had two main results. The first was to classify the irreducible generic representations of $SO_{2n}(F)$. We then used this to show the surjectivity of the functorial lifting map constructed by Cogdell, Kim, Piatetski-Shapiro and Shahidi (\cite{CKPSS04}). In this paper, we extend these results. In particular, the surjectivity of functorial lifting map is extended to all quasi-split classical groups. Because the strategy used for the classification of generic representations does not depend on the structure of the groups in as delicate a manner, the same basic argument can be applied to the corresponding similitude groups, as well as general spin groups. Thus we classify generic representations for these groups as well. These classifications of generic representations are expected to be very useful, for examples, 
\begin{enumerate}
    \item for the problem of classifying the generic unitary dual, as considered in \cite{LMT04};
    \item towards a conjecture of Gross-Prasad and Rallis (see \cite[Conjecture 2.6]{GP92},
\cite{Kud94}, \cite{JS04, Liu11, JL14}, 
\cite[Appendix B]{GI16}) which states that a Langlands parameter is generic (i.e., its $L$-packet contains a generic representation) if and only if its adjoint $L$-function is holomorphic at $s=1$;
\item for the generic Arthur packet conjecture (see \cite{Sha11}, \cite{Liu11, JL14}, \cite{HK17}) which states that an $L$-packet of Arthur type has a generic member if and only if it is a tempered $L$-packet. 
\end{enumerate}
We discuss each of our main results in turn.

The classification of generic representations of $GL_n(F)$ was done in \cite{Jac77} (from which those for $SL_n(F)$ may be obtained--see \cite{Tad92}), those for $SO_{2n+1}(F)$ and $Sp_{2n}(F)$ were done in \cite{Mui98a}, and those for $SO_{2n}(F)$ in \cite{JL14}. To these groups, we add  $SO_{2n+2}^{\ast}(F)$, $U_{2n+1}(F)$, $U_{2n}(F)$, $GSp_{2n}(F)$, $GSO_{2n}(F)$,
$GU_{2n+1}(F)$, $GU_{2n}(F)$, $GSpin_{2n+1}(F)$, $GSpin_{2n}(F)$, $GSO^{\ast}_{2n+2}(F)$, and $GSpin_{2n+2}^{\ast}(F)$,
with the quasi-split groups defined by a quadratic extension $E=F(\sqrt{\varepsilon})$.

There are a few key ingredients in the classification of irreducible generic representations for these groups:

\begin{itemize}

\item results on induced representations and genericity for general linear groups,

\item Levi factors of (standard) parabolic subgroups having the form
\[
M=H_{n_1} \times \dots \times H_{n_k} \times G_{n_0},
\]
where 
\[
H_m=\left\{ \begin{array}{l} GL_m \mbox{ for }G_n \not=U_{2n+1},U_{2n},GU_{2n+1},GU_{2n}, \\
	\mbox{Res}_{E/F}GL_m \mbox{ for }G_n=U_{2n+1},U_{2n},GU_{2n+1},GU_{2n},
\end{array} \right.
\]

\item an explicit form for the Langlands classification and Casselman criterion,

\item a $\mu^{\ast}$ structure for calculating Jacquet module like that of Tadi\'c for classical groups,

\item cuspidal reducibility conditions,

\end{itemize}
and for similitude groups, formulas for

\begin{itemize}

\item twisting induced representations by characters,

\item central characters of induced representations.

\end{itemize}
Most of these are already known, though we do fill in a few gaps and make a correction or two. We also try to formulate things in a uniform way to facilitate treating the different groups together.

Let $\nu$ denote the absolute value. Consistent with \cite{BZ77}, we interpret this as a character of a general linear (resp., similitude) group via composition with the determinant (resp., the similitude character). A representation $\pi$ of a general linear (resp., similitude) group is essentially tempered if there is an $\varepsilon(\pi) \in {\mathbb R}$ such that $\nu^{-\varepsilon(\pi)}(\pi)$ is tempered, and similarly for square-integrable representations.
To have a uniform presentation of the result below, for an irreducible representation $\sigma$ of $G_n(F)$, we let
\[
\beta=\left\{ \begin{array}{l}
\frac{1}{2}\varepsilon(\sigma) \mbox{ if }G_n=GSpin_{2n+1}, GSpin_{2n} \mbox{ with }n=0, \\
\varepsilon(\sigma) \mbox{ if }G_n=GSpin_{2n+2}^{\ast}, \mbox{ or }G_n=GSpin_{2n+1},GSpin_{2n} \mbox{ with }n>0, \\
0 \mbox{ if $G_n$ is not a general spin group},
\end{array} \right.
\]
noting that by Lemma~\ref{beta}, $\beta$ depends only on the representation of (lower rank) $G_k(F)$ appearing in the supercuspidal support of $\sigma$.

The following theorem summarizes Propositions~\ref{gdsprop1} and \ref{gdsprop2} for square-integrable representations (adapted to essentially square integrable representations at the end of \S \ref{gdssection}) and the discussion of essentially tempered representations in \S \ref{gtsection}.

\begin{prop}\label{dsintro}

\ 

\begin{description}

\item[Generic Essentially Discrete Series]
Let $\Delta_i=[\nu^{-a_i}\tau_i, \nu^{b_i}\tau_i]=\{\nu^{-a_i}\tau_i,\nu^{-a_i+1}\tau_i, \ldots, \nu^{b_i}\tau_i\},$
a Zelevinsky segment, $1\leq i \leq k$, where $\tau_i$ is an irreducible unitary supercuspidal representation of a general linear group. Assume that if $i<j$ has $\tau_i \cong \tau_j$, then $a_i<b_i<a_j<b_j$. Let $\sigma^{(e0)}$ be an irreducible supercuspidal generic representation of $G_{n'}(F)$ and assume that for each $i$, one of the following holds (necessarily exclusive):
\begin{enumerate}
\item $\nu^{1+\beta}\tau_i \rtimes \sigma^{(e0)}$ is reducible, in which case $a_i \in \beta+\left({\mathbb Z} \setminus \{0\}\right)$ and $a_i \geq \beta-1$;
\item $\nu^{\frac{1}{2}+\beta}\tau_i \rtimes \sigma^{(e0)}$ is reducible, in which case $a_i \in \beta-\frac{1}{2}+{\mathbb Z}_{\geq 0}$;
\item $\nu^{\beta}\tau_i \rtimes \sigma^{(e0)}$ is reducible, in which case $a_i \in \beta+{\mathbb Z}_{\geq 0}$;
\item $\nu^x \tau_i \rtimes \sigma^{(e0)}$ is irreducible for all $x \in {\mathbb R}$ and (using $c$ to denote the outer automorphism described in \S \ref{sect3})
\begin{enumerate}
\item $\check{\tau}_i \cong \tau_i$ but $c^{d(\tau)} \cdot \sigma^{(e0)} \not \cong \sigma^{(e0)}$ for $SO_{2n}, SO_{2n+2}^{\ast},$

\item $\check{\tau}_i \cong \tau_i $  but $\omega_{\tau} \sigma^{(e0)} \not \cong \sigma^{(e0)}$ for $GSp_{2n}, GU_{2n+1}, GU_{2n},$

\item $\check{\tau}_i \cong \tau_i $  but $\omega_{\tau} (c^{d(\tau)} \cdot \sigma^{(e0)}) \not \cong \sigma^{(e0)}$ for $GSO_{2n}, GSO_{2n+2}^{\ast},$

\item
$\nu^{-2\beta}\omega_{\sigma^{(e0)}}\check{\tau}_i \cong \tau_i $  but $c^{d(\tau)} \cdot\sigma^{(e0)} \not \cong \sigma^{(e0)}$ for $GSpin_{2n}$, $GSpin_{2n+2}^{\ast}$,
\end{enumerate}
in which case $a_i \in\beta+ {\mathbb Z}_{\geq 0}$. 
Here, $\check{\ }$ denotes contragredient composed with Galois conjugation for unitary and general unitary groups and is just contragredient otherwise. This case does not occur for $SO_{2n+1}, Sp_{2n}, U_{2n+1}, U_{2n}, GSpin_{2n+1}$.
\end{enumerate}
Then, if $\pi$ is the generic subquotient of $\delta(\Delta_1) \times \dots \times \delta(\Delta_k) \rtimes \sigma^{(e0)}$, $\pi$ is essentially square-integrable.
Conversely, any generic  irreducible essentially square-integrable $\pi$ of a group $G_n(F)$ is of this form (with $\Delta_1, \dots, \Delta_k$ unique up to permutation), and further
\[
\pi \hookrightarrow \delta(\Delta_1) \times \dots \times \delta(\Delta_k) \rtimes \sigma^{(e0)}.
\]

\item[Generic Essentially Tempered]
 Let $\tau_1, \tau_2, \dots, \tau_c$ be   irreducible generic unitary supercuspidal representations of general linear groups and $\sigma^{(e2)}$ be an irreducible generic essentially square-integrable representation of $G_n(F)$. Let $\Psi_1, \dots, \Psi_c$ be segments of the form $\Psi_i=[\nu^{\beta+\frac{-k_i+1}{2}}\tau_i, \nu^{\beta+\frac{k_i-1}{2}}\tau_i]$. Then the generic component
\[
\sigma^{(et)} \leq \delta(\Psi_1) \times \dots \times
\delta(\Psi_c) \rtimes \sigma^{(e2)}
\]
is a generic essentially tempered representation. Furthermore, any generic essentially tempered representation may be realized this way (with inducing representation unique up to Weyl conjugation).
\end{description}

\end{prop}
Note that the essentially tempered claims above follow directly from a result of Harish-Chandra (cf.\, \cite[Proposition III.4.1]{Wal03}).  The classification of square-integrable and tempered representations are done in \S \ref{gdssection} and \S \ref{gtsection}, respectively; the results for essentially square-integrable and essentially tempered representations are obtained as consequences.

We need a bit of notation in order to present the next result in a uniform manner.
First, for a segment $\Sigma=[\nu^{-a} \xi, \nu^b \xi]$, we define $\check{\Sigma}=[\nu^{-b}\check{\xi}, \nu^{a}\check\xi]$, so that $\delta(\Sigma)^{\vee}=\delta(\check{\Sigma})$.
% where $\check{\ }$ denotes contragredient composed with Galois conjugation for unitary and general unitary groups and is just contragredient otherwise. 
Also, we set
\[
\omega_{\sigma^{(et)}}'=\left\{ \begin{array}{l} \omega_{\sigma^{(et)}} \mbox{ (central character) for general spin groups}, \\
	1 \mbox{ otherwise},
\end{array} \right.
\]
and similarly for $\omega_{\sigma^{(e2)}}'$ and $\omega_{\sigma^{(e0)}}'$.

The following theorem summarizes Theorems~\ref{gen1}, \ref{gen2}, \ref{gen3}, \ref{gen4}, and Note~\ref{interpretation}.

\begin{thr}\label{thm1intro}
Put $\delta(\Sigma_i)=\nu^{x_i}\delta_i$, $i
= 1, 2, \cdots, f$ as above (i.e., $x_1 \geq x_2 \geq \dots \geq x_f>\beta$). Then, the representation
$$
\delta(\Sigma_1)  \times \dots \times \delta(\Sigma_f) \rtimes
\sigma^{(et)}
$$
is irreducible if and only if 
$\{\Sigma_j\}_{j=1}^f$ and
$\sigma^{(et)}$ satisfy the following properties:

\begin{enumerate}

\item[(G1)] $\delta(\Sigma_i) \times \delta(\Sigma_j)$ and $\delta(\Sigma_i) \times \omega_{\sigma^{(et)}}'\delta(\check{\Sigma}_j)$ are irreducible  for all $1 \leq i\not= j \leq f$; and

\item[(G2)] $\delta(\Sigma_i) \rtimes \sigma^{(et)}$ is irreducible for
all $1 \leq i \leq f$.

\end{enumerate}
The reducibility for (1) is known from \cite{Zel80}; for (2) we write $\sigma^{(et)}$ as in Proposition~\ref{dsintro} as above. Then
$\delta(\Sigma) \rtimes \sigma^{(et)}$ is irreducible 
if and only if the following hold:

\begin{enumerate}
\item[(G3)] $\delta(\Sigma) \times \delta(\Psi_j)$ and $\omega_{\sigma^{(e2)}}'\delta(\check{\Sigma}) \times \delta(\Psi_j)$ are irreducible for all $1 \leq j \leq c$; and

\item[(G4)] $\delta(\Sigma) \rtimes \sigma^{(e2)}$ is irreducible.

\end{enumerate}
To understand when the second condition above holds, write $\sigma^{(e2)}$ as in Proposition~\ref{dsintro}. Then, $\delta(\Sigma) \rtimes \sigma^{(e2)}$ is irreducible if and only if the following hold:

\begin{enumerate}

\item[(G5)] $\delta(\Sigma) \times \delta(\Delta_i)$ and 
$\omega_{\sigma^{(e0)}}'\delta(\check{\Sigma}) \times \delta(\Delta_i)$ are irreducible for all $1 \leq i \leq k$; and

\item[(G6)] either (a) $\delta(\Sigma) \rtimes
\sigma^{(e0)}$ is irreducible, or (b) $\delta(\Sigma)=\delta([\nu^{1+\beta}\xi, \nu^{b+\beta} \xi])$,  with $\nu^{1+\beta}\xi \rtimes \sigma^{(e0)}$ reducible and there is some $i$ having $\delta(\Delta_i)=\delta([\nu^{1+\beta}\xi, \nu^{b_i+\beta}\xi])$ and $b_i \geq b$.

\end{enumerate}
Finally, for the second condition above, we have 
$\delta(\Sigma) \rtimes \sigma^{(e0)}$ is irreducible 
if and only if one of the following hold: for $\Sigma=[\nu^{-a}\xi, \nu^b \xi]$, we have

\begin{enumerate}

\item[(G7)] $\xi \not \cong \omega_{\sigma^{(0)}}'\check{\xi}$; or

\item[(G8)] $\xi \cong \omega_{\sigma^{(0)}}'\check{\xi}$ and the following:  (i) if $\nu^x \xi \rtimes \sigma^{(e0)}$ is reducible for some (necessarily unique) $x =\beta+\alpha$ with $\alpha \geq 0$, then $\pm \alpha \not \in \{-a-\beta, -a+1-\beta, \cdots, b-\beta\}$; (ii)
if $\nu^x \xi \rtimes \sigma^{(e0)}$ is irreducible for all $x \in {\mathbb R}$, then $0 \not \in \{-a-\beta, -a+1-\beta, \cdots, b-\beta\}$.

\end{enumerate}

\end{thr}

We take a moment to note a couple of misstatements in the introduction of \cite{JL14}.
First,  condition (2) for square-integrable generic representations was misstated in \cite[Proposition 1.1]{JL14} (with $\frac{1}{2}$ instead of $-\frac{1}{2}$). Similarly, in \cite[Theorem 1.2]{JL14}, condition (2) in the reducibility of $\delta(\Sigma) \rtimes \sigma^{(2)}$ should be $\nu\xi \rtimes \sigma^{(0)}$ reducible (rather than $\xi \rtimes \sigma^{(0)}$ as stated). Both are correct in the body of the paper.

The second part of this paper is to apply the above classification of generic representations to show the surjectivity of the functorial lifting maps for $G_n = SO_{2n+1}, Sp_{2n}, SO_{2n}, SO_{2n+2}^*, U_{2n+1}, U_{2n}$, quasi-split classical groups of $F$-rank $n$, constructed by Cogdell, Kim, Piatetski-Shapiro and Shahidi (\cite{CKPSS04}, \cite{CPSS11}). 

Let 
$N = 2n$ for $G_n=SO_{2n+1}, U_{2n}, SO_{2n}$, 
$N = 2n+2$ for $G_n=SO_{2n+2}^*$, 
 $N = 2n+1$ for $G_n=Sp_{2n}, U_{2n+1}$.
% Let $H_N = GL_N$ when $G_n=SO_{2n+1}, Sp_{2n}, SO_{2n}, SO_{2n+2}^*, U_{2n+1}, U_{2n}$, and let 
% $H_N=Res_{E/F} GL_N$ when $G_n=U_{2n+1}, U_{2n}$. 
By Langlands' principle of functoriality, the following table summarizes the cases of funtoriality we consider from $G_n$ to $H_N$:

\begin{table}[H]
\begin{tabular}{ |c|c|c| } 
 \hline
$G_n$ & $\iota: {}^L G_n \hookrightarrow {}^L H_N$ & $H_N$ \\ 
 \hline
$SO_{2n+1}$ & $Sp_{2n}(\mathbb{C}) \times W_F \hookrightarrow GL_{2n}(\mathbb{C}) \times W_F$ & $GL_{2n}$ \\ 
 $Sp_{2n}$ & $SO_{2n+1}(\mathbb{C}) \times W_F \hookrightarrow GL_{2n+1}(\mathbb{C}) \times W_F$ & $GL_{2n+1}$ \\ 
 $SO_{2n}$ & $SO_{2n}(\mathbb{C}) \times W_F \hookrightarrow GL_{2n}(\mathbb{C}) \times W_F$ & $GL_{2n}$\\
  $SO_{2n+2}^*$ & $SO_{2n+2}(\mathbb{C}) \rtimes W_F \hookrightarrow GL_{2n+2}(\mathbb{C}) \times W_F$ & $GL_{2n+2}$\\
  $U_{2n+1}$ & $GL_{2n+1}(\mathbb{C}) \rtimes W_F \hookrightarrow GL^{\times 2}_{2n+1}(\mathbb{C}) \rtimes W_F$ & $Res_{E/F} GL_{2n+1}$\\
    $U_{2n}$ & $GL_{2n}(\mathbb{C}) \rtimes W_F \hookrightarrow GL^{\times 2}_{2n}(\mathbb{C}) \rtimes W_F$ & $Res_{E/F} GL_{2n}$\\
 \hline
\end{tabular}\vspace{.2in}
\caption{Langlands functoriality}
\label{tab:Langlands functoriality}
\end{table}\vspace{-.2in}
\noindent
where $GL^{\times 2}_{k}(\mathbb{C}) = GL_{k}(\mathbb{C}) \times GL_{k}(\mathbb{C})$, and $GL_{k}(\mathbb{C}) \hookrightarrow GL^{\times 2}_{k}(\mathbb{C})$ is the diagonal embedding. 

In \cite{CKPSS04}, Cogdell, Kim, Piatetski-Shapiro and Shahidi, and \cite{CPSS11}, Cogdell, Piatetski-Shapiro and Shahidi, 
constructed a local functorial lifting $l$ from $\Pi^{(g)}(G_n)$ (generic representations of $G_n$) 
to a subset $\Pi^{(g)}_{\varepsilon}(H_N)$  of representations of $H_N$, satisfying the following conditions:
$$L(\sigma \times \pi, s) = L(l(\sigma) \times \pi, s),$$
$$\epsilon(\sigma \times \pi, s, \psi) = \epsilon(l(\sigma) \times \pi, s, \psi),$$
for any irreducible generic representation $\pi$ of $H_k(F)$, with $k \in \mathbb{Z}_{>0}$,
where $\psi$ is a fixed nontrivial character of $F$. The left hand sides are the local factors defined by Shahidi (\cite{Sha90a}), and the right hand sides are the local factors defined by Jacquet, Piatetski-Shapiro, Shalika (\cite{JPSS83}).

One of the main ingredients of this paper is that the local Langlands functorial lifting from irreducible unitary supercuspidal generic representations of $G_n(F)$ is surjective.
Arthur (\cite{Art13}) and Mok (\cite{Mok15}) proved this result using the trace formula method and the global descent result of Ginzburg, Rallis and Soudry (\cite{GRS11}). 
Jiang and Soudry (\cite{JS12}) (for $G_n=SO_{2n+1}, Sp_{2n}, SO_{2n}, SO_{2n+2}^*$), Soudry and Tanay (\cite{ST15}) (for $G_n=U_{2n}$),
constructed the local descent map from irreducible unitary supercuspidal representations of $H_{N}$ to irreducible supercuspidal representations of $G_n$.
The generalization of descent map from irreducible unitary supercuspidal representations of $H_N$ to their product is straightforward for $G_n=SO_{2n+1}, Sp_{2n}, SO_{2n}, SO_{2n+2}^*$, but for $G_n=U_{2n}, U_{2n+1}$, further work is needed. 

As an application of the classifcation of the generic dual $\Pi^{(g)}(G_n)$ of $G_n$, we show that the local functorial lifting
$l: \Pi^{(g)}(G_n) \rightarrow \Pi^{(g)}_{\varepsilon}(H_N)$ constructed
above by Cogdell, Kim, Piatetski-Shapiro and Shahidi is surjective.
Note that for $SO_{2n+1}$, in \cite{JS03, JS04}, Jiang and Soudry have already constructed the corresponding local Langlands functorial lifting, and proved that it is actually bijective.
In \cite{Liu11}, Liu proved the surjectivity for $Sp_{2n}$, and in \cite{JL14}, Jantzen and Liu proved the case of split $SO_{2n}$. 
Note that the details of the proofs in \cite{Liu11} and \cite{JL14} have been omitted. Here we gave uniform detailed proofs for all the quasis-plit classical groups. 
We remark that, for $Sp_{2n}$, $SO_{2n}$, and $SO_{2n+2}^*$, $l$ is expected
not to be injective by  \cite[Conjecture 3.7]{Jia06},
which is a refinement of the local converse theorem conjecture.

Let $\Phi(G_n)$ be the set of local Langlands parameters for $G_n$, which are
${}^LG_n$-conjugacy classes of admissible homomorphisms $W_F \times SL_2(\mathbb{C}) \rightarrow {}^LG_n$,
where $W_F \times SL_2(\mathbb{C})$ is the Weil-Deligne group. 
When $G_n=SO_{2n}, SO^*_{2n+2}$, given a local Langlands parameter $\phi \in \Phi(GL_{N})$, $\phi: W_F \times SL_2(\mathbb{C}) \rightarrow GL_{N}(\mathbb{C})$, assume that it factors through $G_{n}(\mathbb{C})$ and $\phi \ncong c \phi$ within $G_{n}(\mathbb{C})$, where $c \phi$
is the c-conjugate of $\phi$. Then $\phi$ produces two elements in 
$\Phi(G_{n})$ (see \cite[Chapter 1]{Art13}), which are denoted by $\phi$ and $c \phi$. To identify $\phi$ and $c \phi$, let $\widetilde{\Phi}(G_{n})$ be the set of $c$-conjugacy classes of $\phi \in \Phi(G_{n})$. For any $\phi \in \Phi(G_{n})$, denote its $c$-conjugacy class by $\widetilde{\phi}$.
Note that for any $\phi \in \Phi(G_{n})$, if $\phi \ncong c \phi$, then they automatically have the same twisted local factors since they come from the same local Langlands parameter $\phi \in \Phi(GL_{N})$. Define the twisted local factors of $\widetilde{\phi}$ to be those of $\phi$. When $G_n$ is not $SO_{2n}, SO^*_{2n+2}$, we simply put $\widetilde{\Phi}(G_{n})$ to be ${\Phi}(G_{n})$. 

The local functorial lifting $l$ enables us to assign
a parameter $\phi \in \widetilde{\Phi}(G_{n})$ to each $\sigma \in \Pi^{(g)}(G_{n})$,
which is exactly the parameter corresponding to $l(\sigma)$.
That is, there is a map 
$\iota: \Pi^{(g)}(G_{n}) \rightarrow \widetilde{\Phi}^{(g)}(G_{n})$,
where $\widetilde{\Phi}^{(g)}(G_{n})$ is the set of parameters corresponding to 
representations in $l(\Pi^{(g)}(G_n))$.
We show that the surjectivity of $l$ implies that of $\iota$. 

As another application of the classification of the generic dual of $G_n$, for any local Langlands 
parameter $\widetilde\phi \in \Phi(G_{n})$, by an explicit analysis of its structure, we construct a distinguished
irreducible representation $\sigma$ of $
G_{n}(F)$ such that
$\widetilde{\phi}$ and $\sigma$ have the same twisted local factors, as in \cite{JS04, Liu11, JL14}. We remark that Arthur (\cite{Art13}) and Mok (\cite{Mok15}) have already proved the local Langlands correspondence and constructed the local $L$-packets for $G_n$. However, this explicitly constructed member $\sigma$ in each local $L$-packet is very useful for certain problems, for example, it plays a crucial role in the work towards Jiang's conjecture on the wavefront sets of representations in local Arthur packets (see \cite{Jia14} and \cite{LS22}). 

We now discuss the contents by sections. The next two sections introduce notation and background material, and the groups to be considered in this paper, respectively.
\S \ref{sect4} contains the classifications of generic representations for our groups. This is broken into four parts, which classify generalized Steinberg, square-integrable generic, tempered generic, and generic representations, respectively. 
Then, we introduce Langlands philosophy of functoriality for quasi-split classical groups in \S \ref{sect5}, and prove the surjectivity of the local functorial lifting maps in \S \ref{sect6}. In \S \ref{sect7}, we analyze the structure of local
Langlands parameters and associate a particular representation to each local Langlands parameter. 

\subsection*{Acknowledgements} 
The authors would like to thank Mahdi Asgari, Sviatoslav Archava, Kwangho Choiy, Joseph Hundley, Muthu Krishnamurthy, and Marko Tadi\'c for helpful discussions on various aspects of this work.
The authors also would like to thank Dihua Jiang and Freydoon Shahidi for their interest and constant support.

\section{Notation and preliminaries}\label{sect2}

Let $F$ be a $p$-adic field of characteristic 0. We begin by defining the groups under consideration. To this end, let $J_m$ denote the $m \times m$ matrix having $j_{k,\ell}=\left\{ \begin{array}{l}1, \mbox{ if }k+\ell=m+1,  \\ 0, \mbox{ otherwise} \end{array} \right.$ (1's on the antidiagonal, 0's elsewhere).

We start with the split groups defined by forms. In the symplectic case. We take
\begin{align*}
 &\,GSp_{2n}(F)\\
 =&\,\left\{ X \in GL_{2n}(F) \,|\, {}^TX\left(\begin{array}{cc} 0 & -J_n \\ J_n & 0 \end{array} \right) X=\lambda \left(\begin{array}{cc} 0 & -J_n \\ J_n & 0 \end{array} \right), \lambda \in F^{\times} \right\}. 
\end{align*}
For the classical group $Sp_{2n}(F)$, one simply restricts to $\lambda=1$ above. We remark that for $X \in Sp_{2n}(F)$, it is automatic that $\det(X)=1$.

For the odd orthogonal case, we take
\[
SO_{2n+1}(F)=\left\{ X \in SL_{2n+1} \,|\, {}^TXJ_{2n+1}X=J_{2n+1} \right\}.
\]
Note that we do not consider the corresponding similitude group as it is just a direct product of $SO_{2n+1}(F)$ and $F^{\times}$.

We now turn to the even orthogonal case. Here, we first define
\[
GO_{2n}(F)=\left\{ X \in GL_{2n}(F) \,|\, {}^TXJ_{2n}X=\lambda J_{2n} \mbox{ for some }\lambda \in F^{\times} \right\}.
\]
Taking determinants, $\lambda^{2n}=(\det X)^2$; we set
\[
GSO_{2n}(F)=\{ X \in GO_{2n}(F) \,|\, \lambda^n=\det X \}.
\]
The classical group $SO_{2n}(F)$ then consists of those $X \in GSO_{2n}(F)$ having $\lambda=1$.

We now turn to the (non-split) quasi-split groups defined by forms. Here, we start with a quadratic extension $E=F(\sqrt{\varepsilon})$. As we do not have such groups in the symplectic or odd-orthogonal cases, we begin with the even orthogonal case. Set
\[
J^{(\varepsilon)}_{2n+2}=\left( \begin{array}{cccc}
	 & & & J_n \\
	 & 1 & & \\
	 & & -\varepsilon & \\
	J_n & & & 
\end{array} \right).
\]
We then define
\[
GO^{(\varepsilon)}_{2n+2}(F)=\left\{ X \in GL_{2n+2}(F) \,|\, {}^TXJ^{(\varepsilon)}_{2n+2}X=\lambda J^{(\varepsilon)}_{2n+2} \mbox{ for some }\lambda \in F^{\times} \right\}.
\]
and take
\[
GSO^{(\varepsilon)}_{2n+2}(F)=\{ X \in GO^{(\varepsilon)}_{2n+2}(F) \,|\, \lambda^{n+1}=\det X \}.
\]
The group $SO^{(\varepsilon)}_{2n+2}(F)$ then consists of those $X \in GSO^{(\varepsilon)}_{2n+2}(F)$ having $\lambda=1$.

For the unitary groups, we retain the quadratic extension above and set
\[
J'_N=\left\{ \begin{array}{l}
\left( \begin{array}{cccccc}
	 & & & & & 1 \\
	 & & & & -1 &\\
  & & & 1 & & \\
  & & \iddots & & & \\
  & -1 & & & & \\
  1 & & & & & 
\end{array} \right), \mbox{ if }N=2n+1, \\
\\
\left( \begin{array}{cc}
	 & J_n \\
	-J_n & 
\end{array} \right), \mbox{ if }N=2n.
\end{array}\right.
\]
Then,
\[
GU_N(F)=\{X \in \mbox{Res}_{E/F}GL_N(F) \,|\, {}^T\bar{X}J'_N X=\lambda J'_N \mbox{ for some }\lambda \in F^{\times} \},
\]
where $\bar{\ }$ denotes the Galois conjugation. Consistent with this, by definition, $GU_0(F) \cong F^{\times}$. Again, the unitary group $U_N(F)$ consists of those $X \in GU_N(F)$ having $\lambda=1$. We also remark that in the archimedean case ($F={\mathbb R}$ and $E={\mathbb C}$), one has $GU_{2n+1}({\mathbb R}) \cong U_{2n+1}({\mathbb R}) \times H$, where $H=\{zI \,|\, z>0\}$.

We now turn to the general spin groups. As we do not have convenient matrix realizations, we follow \cite{Asg02} and work from the root data in the split cases.
Write
\[
X={{\mathbb Z}e_0 \oplus } {\mathbb Z}e_1 \oplus \dots \oplus {\mathbb Z}e_n
\]
and
\[
\check{X}={ {\mathbb Z}\check{e}_0 \oplus }{\mathbb Z}\check{e}_1 \oplus \dots \oplus {\mathbb Z}\check{e}_n
\]
as the rational characters (resp., rational cocharacters) with the usual pairing. Then $GSpin_{2n+1}$ has roots and coroots
\[
\Pi=\{ e_1-e_2, e_2-e_3, \dots, e_{n-1}-e_n, e_n\}
\]
and
\[
\check{\Pi}=\{ \check{e}_1-\check{e}_2, \check{e}_2-\check{e}_3, \dots, \check{e}_{n-1}-\check{e}_n, 2\check{e}_n-\check{e}_0\}.
\]
Note that these are dual to the data for $GSp_{2n}$.

For $GSpin_{2n}$, we retain $X$ and $\check{X}$ from above. We then have roots and coroots
\[
\Pi=\{ e_1-e_2, e_2-e_3, \dots, e_{n-1}-e_n, e_{n-1}+e_n\}
\]
and
\[
\check{\Pi}=\{ \check{e}_1-\check{e}_2, \check{e}_2-\check{e}_3, \dots, \check{e}_{n-1}-\check{e}_n, \check{e}_{n-1}+\check{e}_n-\check{e}_0\}.
\]
Note that these are dual to the data for $GSO_{2n}$.

More generally, we have included the simple $F$-roots and co-roots for the groups under consideration in Appendix~\ref{F-roots}.

Finally, we turn to the (non-split) quasi-split general spin groups. Here, we follow \cite{HS16}. In particular, the quasi-split forms correspond to homomorphisms of $Gal(\bar{F}/F)$ into the automorphisms of the Dynkin diagram. In the odd case, such homomorphisms are trivial and one has only the split odd general spin groups. In the even case, one has a nontrivial homomorphism which may be parameterized by a quadratic extension $E=F(\sqrt{\varepsilon})$ as above, with the nontrivial element of $Gal(E/F)$ being mapped to the automorphism of the Dynkin diagram which interchanges the last two simple roots. This defines the group we denote as $GSpin_{2n+2}^{(\varepsilon)}$. A more detailed description can be found in \S \ref{sect3appendix}.

\medskip

We let $G_n(F)$ denote one of the following groups under consideration
$$Sp_{2n}(F),  SO_{2n+1}(F), SO_{2n}(F), SO_{2n+2}^{(\varepsilon)}(F), U_{2n+1}(F), U_{2n}(F),$$ $$GSp_{2n}(F), GSO_{2n}(F),  GSO_{2n+2}^{(\varepsilon)}(F),GU_{2n+1}(F), GU_{2n}(F),$$
$$GSpin_{2n+1}(F), GSpin_{2n}(F), GSpin_{2n+2}^{(\varepsilon)}(F),$$
and fix a Borel subgroup $B$. 
The standard parabolic subgroups containing $B$ may then be parameterized by subsets $\Phi \subset \Pi$. For $G_n(F)\not=SO_{2n}(F)$, $GSO_{2n}(F)$, $GSpin_{2n}(F)$, the standard parabolic subgroup associated to $\Phi=\Pi \setminus \{\alpha_{n_1}, \alpha_{n_1+n_2},\dots, \alpha_{n_1+ \dots+n_k}\}$ has the form $P=MU$ with
\[
M=H_{n_1}(F) \times \dots \times H_{n_k}(F) \times G_{n_0}(F),
\]
%\[
%M\cong \left\{ \begin{array}{l} GL_{n_1}(F) \times \dots \times GL_{n_k}(F) \times G_{n_0}(F) \\
%\mbox{ for }G_n \not=U_{2n+1},U_{2n},GU_{2n+1},GU_{2n}, \\
%GL_{n_1}(E) \times \dots \times GL_{n_k}(E) \times G_{n_0}(F) \\
%\mbox{ for }G_n=U_{2n+1},U_{2n},GU_{2n+1},GU_{2n},
%\end{array} \right.
%\]
where $n_1+\dots+n_k+n_0=n$ and
\[
H_m=\left\{ \begin{array}{l} GL_m \mbox{ for }G_n \not=U_{2n+1},U_{2n},GU_{2n+1},GU_{2n}, \\
	\mbox{Res}_{E/F}GL_m \mbox{ for }G_n=U_{2n+1},U_{2n},GU_{2n+1},GU_{2n}.
\end{array} \right.
\]
Note that we freely identify $\mbox{Res}_{E/F}GL_m(F)$ with $GL_m(E)$ in places.
We also note that in this context, we have
\[
G_0(F) \cong \left\{ \begin{array}{l}
	1 \mbox{ for } Sp_{2n}, SO_{2n+1}, SO_{2n}, U_{2n}, \\
	N_1(E/F) \mbox{ for } SO_{2n+2}^{(\varepsilon)}, U_{2n+1}, \\
	F^{\times} \mbox{ for } GSp_{2n}, GSO_{2n}, GSpin_{2n+1}, GSpin_{2n}, GU_{2n}, \\
	E^{\times} \mbox{ for }GSO_{2n+2}^{(\varepsilon)}, GSpin_{2n+2}^{(\varepsilon)}, GU_{2n+1}.
\end{array} \right.
\]

For $G_n(F)=SO_{2n}(F)$, $GSO_{2n}(F)$, $GSpin_{2n}(F)$, there is an outer autmorphism $c$ of the root data which interchanges the last two simple roots. Now, if $\alpha_{n-1},\alpha_n \in \Phi$ or $\alpha_n \not\in \Phi$, the situation is like that above for the other groups under consideration. If $\alpha_n \in \Phi$ but $\alpha_{n-1} \not\in \Phi$,  then $c \cdot \Phi$ contains $\alpha_{n-1}$ but not $\alpha_n$ (and is otherwise the same). In both cases $M \cong GL_{n_1}(F) \times \dots \times GL_{n_k}(F) \times G_0(F)$, but these are not conjugate in $G_n(F)$. However, through the use of an artifice--introduced in \cite{JL14} for $G_n=SO_{2n}$ and discussed in the next section for $G_n=GSO_{2n}$ and $GSpin_{2n}$--one may set things up so that for representations of $M$, the two situations are distinguished in the representation for $G_0(F)$. The ambuguity in writing $M \cong GL_{n_1}(F) \times \dots \times GL_{n_k}(F) \times G_0(F)$ is then eliminated, and we may consider standard Levi factors for these groups to also have the form $M\cong GL_{n_1}(F) \times \dots \times GL_{n_k}(F) \times G_{n_0}(F)$ with $n_1+\dots+n_k+n_0=n$, but with $n_0\not=1$ (noting, e.g., that $SO_2(F) \cong GL_1(F) \times SO_0(F)$ is better viewed as the latter).

We now take a moment to recall some notation from \cite{BZ77}, \cite{Tad94}. First, for $P=MN$ a standard parabolic subgroup of a $p$-adic group $G$, we let $i_{G,M}$ (resp., $r_{M,G}$) denote normalized induction  (resp., the normalized Jacquet module) with respect to $P$. Let $G=H_{k}(F)$ and  $P=MU$ the standard parabolic subgroup with $M=H_{k_1}(F) \times \dots \times H_{k_r}(F)$. If $\tau_1 \otimes \dots \otimes \tau_r$ is a representation of $M$, we let
\[
\tau_1 \times \dots \times \tau_r=i_{G,M}(\tau_1 \otimes \dots \otimes \tau_r).
\]
Similarly, suppose $P=MU$ is a standard parabolic subgroup of $G_n(F)$ with $M=H_{k_1}(F) \times \dots \times H_{k_r}(F) \times G_{k_0}(F)$. For $\tau_1 \otimes \dots \otimes \tau_r \otimes \sigma$ a representation of $M$, we let
\[
\tau_1 \times \dots \times \tau_r \rtimes \sigma=i_{G,M}(\tau_1 \otimes \dots \otimes \tau_r \otimes \sigma).
\]
Note that in the classical case, this allows $\sigma=1$, the trivial representation of $G_{0}(F)$--the trivial group.

We next discuss some structure theory from \cite{Zel80}. Let
\[
R=\bigoplus_{n \geq 0}{\mathcal R}(H_{n}(F))
\]
where ${\mathcal R}(G)$ denotes the Grothendieck group of the category of smooth finite-length representations of $G$.
We define multiplication on $R$ by extending the semisimplification of $\times$ to give the multiplication $\times: R\times R \longrightarrow R$. To describe the comultiplication on $R$, let $M_{(i)}$ denote the standard Levi factor for $H_{n}(F)$ having $M_{(i)}=H_{i}(F) \times H_{n-i}(F)$. For a representation $\tau$ of $H_{n}(F)$, we define
\begin{equation}\label{m-ast def}
m^{\ast}(\tau)=\sum_{i=0}^n r_{M_{(i)},G}(\tau),
\end{equation}
the sum of semisimplified Jacquet modules (lying in $R \otimes R$). This extends to a map $m^{\ast}:R \longrightarrow R \otimes R$. We note that with this multiplication and comultiplication--and antipode map given by the Zelevinsky involution (a special case of the general duality operator of \cite{Aub95}, \cite{SS97})--$R$ is a Hopf algebra.

Following \cite{Tad94} and \cite{Tad95}, we move to the classical and similitude groups under consideration. Set
\[
R[G]=\bigoplus_{n  \geq 0}{\mathcal R}(G_n(F)).
\]
We then extend the semisimplification of $\rtimes$ to a map $\rtimes :R \otimes R[G] \longrightarrow R[G]$. For the groups under consideration other than $SO_{2n}, GSO_{2n}$ and $GSpin_{2n}$, we define $\mu^{\ast}$ as in \cite{Tad95}: For $0 \leq i \leq n$, let $M_{(i)}$ be the standard Levi factor for $G_n(F)$ having $M_{(i)}=H_i(F) \times G_{n-i}(F)$. We then set
\begin{equation}\label{mu-ast def 1}
\mu^{\ast}=\sum_{i=0}^n r_{M_{(i)},G},
\end{equation}
a sum of semisimplified Jacquet modules. In the cases of $Sp_{2n},SO_{2n+1}$ addressed in \cite{Tad95} (and a number of other families from \cite{MT02}), this produces a twisted Hopf module structure. However, the general situation here is not quite as elegant--but just as useful calculationally. This is addressed in the next section; as are the cases of $SO_{2n}, GSO_{2n}$ and $GSpin_{2n}$ (based on the approach in \cite{JL14} for $SO_{2n}$).

Before discussing some specific representations of general linear groups, we need a bit of notation. As in \cite{BZ77}, we let $\nu=|\cdot|_F$ and interpret this as $\nu \circ \det$ on $GL_n(F)$ (with the $n$ determined by context); for $H_n(F) \cong GL_n(E)$, it is the corresponding character under the isomorphism. Similarly, on a similitude group, we let $\nu=\nu \circ \xi$, where $\xi$ denotes the similitude character of the group (again, determined by context; see Lemma~\ref{twisting}). This is always used with a representation (e.g., $\nu\pi$), with the underlying group that of the representation. Later, we apply this convention to other characters of $F^{\times}$ as well.

As in \cite{Zel80}, we consider segments of the form
\[
[\nu^a \tau, \nu^b\tau]=\{\nu^a\tau, \nu^{a+1}\tau, \dots, \nu^{b-1}\tau, \nu^b\tau\}
\]
for $\tau$ a unitary supercuspidal representation of a general linear group and $a \equiv b\,\mbox{mod}\,1$. The induced representation $\nu^a \tau \times \dots \times \nu^b\tau$ has a unique irreducible quotient (resp., subrepresentation) which we denote by $\delta([\nu^a \tau, \nu^b \tau])$ (resp., $\zeta([\nu^a \tau, \nu^b \tau])$). 
The representations $\delta([\nu^a \tau, \nu^b \tau])$ are essentially square-integrable (i.e., square-integrable after twisting by a character), and every irreducible essentially square-integrable representation has this form.
Given a segment $\Sigma=[\nu^a \tau, \nu^b\tau]$, let $\check\Sigma=[\nu^{-b} \check\tau, \nu^{-a}\check\tau]$, where $\check\tau$ is the contragredient of $\tau$ if $$G_n\neq U_{2n+1},U_{2n},GU_{2n+1},GU_{2n},$$
and is the Galois conjugate contragredient otherwise. 
We call $\tau$ self-dual (respectively, self-conjugate-dual in the case of unitary and general
unitary groups) if $\tau \cong \check\tau$.
We remark that square-integrable representations for general linear groups are generic (cf. \cite{Jac77}). The analogous representations for classical and similitude groups are discussed in \S \ref{Steinberg sect}.

\section{Groups}\label{sect3}

In this section, we give some background on the particular groups under consideration. This include material on the groups themselves as well as representation theoretic notation and results specialized to these groups (the Langlands classification/Casselman criterion, $\mu^{\ast}$ structure for Jacquet modules, etc.). While much of this material is known, we fill in a number of gaps.

We retain the assumption $char(F)=0$ in this section but note that for most applications here, $char(F) \not=2$ suffices. See Remark \ref{char} below.

For the split classical groups, the papers \cite{Mui98b}, \cite{JS04}, \cite{Liu11}, \cite{JL14} cover the problems considered in this paper. Thus our interest here is in the non-split quasi-split cases, as well as similitude groups (split or quasi-split). However, as the discussion of irreducible generic representations may be applied to the split classical groups as well, we include them here for the sake of completeness. In particular, we let $G_n$ be one of the following: $SO_{2n+1}$, $Sp_{2n}$, $SO_{2n}$, $SO_{2n+2}^{\ast}$, $U_{2n+1}$, $U_{2n}$, $GSpin_{2n+1}$, $GSpin_{2n}$, $GSpin_{2n+2}^{\ast}$, $GSp_{2n}$, $GSO_{2n}$, $GSO_{2n+2}^{\ast}$, $GU_{2n+1}$, and $GU_{2n}$. Note that the groups which are not split over $F$ require a quadratic extension (or equivalent) in their definition; denote this extension by $E=F(\sqrt{\varepsilon})$. For those non-split groups defined via bilinear forms, one may consult \cite{MVW87} or \cite{Bru63} for an explicit description of the anistropic part. For the general spin groups, consult \cite{Asg02} or \cite{HS16}; for the (non-split) quasi-split case, also see \S \ref{sect3appendix} on $GSpin_{2n+2}^{\ast}$. We also note that \cite{Arc} has identified an issue in $GSO_{2n+2}^{\ast}$ which also occurs in $SO_{2n+2}^{\ast}$ and seems to have gone unnoticed in the literature (including some work by the first-named author), namely, the need to account for the parity in the number of sign changes for the Weyl group action. This affects the characterization of when a unitary supercuspidal representation of a standard Levi factor is ramified, as well as the $\mu^{\ast}$ structure of Tadi\'c; we discuss these in more detail below.

One result needed in this paper is the standard module conjecture. This has been done in the generality needed for the groups at hand in \cite{HO13}.

%One result needed in this paper is the standard module conjecture. For the split groups $SO_{2n+1}(F)$, $Sp_{2n}(F)$, $SO_{2n}(F)$ and quasi-split groups, $SO_{2n+2}^{(\varepsilon)}(F)$, $U_{2n+1}(E/F)$, $U_{2n}(E/F)$, the result is in \cite{Mu2001}. For $GSpin_{2n+1}(F)$, $GSpin_{2n}(F)$, $GSpin_{2n+2}^{(\varepsilon)}(F)$ this is covered by Theorem 5.4 \cite{KimW}. The standard module conjecture for $GSp_{2n}(F)$, $GSO_{2n+2}(F)$, $GSO_{2n+2}^{(\varepsilon)}(F)$, $GU_{2n+1}(E/F)$, $GU_{2n}(E/F)$ follows from Corollary 3.5 \cite{A-S} and the standard module conjecture for the corresponding classical groups. 

We start by looking at the maximal split tori. If $G_n$ is one of the classical groups $SO_{2n+1}$, $Sp_{2n}$, $SO_{2n}$, $SO_{2n+2}^{\ast}$, $U_{2n+1}$, or $U_{2n}$, let $d(a_1, \dots, a_n)=\check{e}_1(a_1) \cdots \check{e}_n(a_n)$, with $a_i \in F^{\times}$. Note that with respect to the matrix realizations of these groups, we have
\[
d(a_1, \dots, a_n)=\left\{ \begin{array}{lr}
diag(a_1, \dots, a_n, 1, a_n^{-1}, \dots, a_1^{-1}) \\
\mbox{ for }G_n=SO_{2n+1}(F), U_{2n+1}(F), \\
diag(a_1, \dots, a_n, a_n^{-1}, \dots, a_1^{-1}) \\
\mbox{ for }G_n=Sp_{2n}(F), SO_{2n}(F), U_{2n}(F), \\
diag(a_1, \dots, a_n, 1, 1, a_n^{-1}, \dots, a_1^{-1}) \\
\mbox{ for }G_n=SO_{2n+2}^{\ast}(F). \\
\end{array} \right.
\]
The maximal split torus is then
\[
A=\{ d(a_1, \dots, a_n) \,|\, a_i \in F^{\times}\}.
\]

For the similitude groups 
$$GSp_{2n}, GSO_{2n}, GSpin_{2n+1}, GSpin_{2n}, GSO_{2n+2}^{\ast}, GU_{2n+1},GU_{2n},GSpin_{2n+2}^{\ast},$$
we set $d(a_1, \dots, a_n, a_0)=\check{e}_1(a_1)\check{e}_2(a_2) \dots \check{e}_n(a_n)\check{e}_0(a_0)$ for $a_i \in F^{\times}$. Again, for those given as matrix groups (i.e., those other than the general spin groups), these have matrix realizations
\[
d(a_1, \dots, a_n,a_0)=\left\{ \begin{array}{l}
diag(a_1, \dots, a_n, a_0a_n^{-1}, \dots, a_0a_1^{-1}) \\
\mbox{ for }G_n=GSp_{2n}(F), GSO_{2n}(F), GU_{2n}(F), \\
diag(a_1, \dots, a_n, a_0, a_0, a_0^2a_n^{-1}, \dots, a_0^2a_1^{-1})\\ \mbox{ for }G_n=GSO_{2n+2}^{\ast}(F), \\
diag(a_1, \dots, a_n, a_0, a_0^2a_n^{-1}, \dots, a_0^2a_1^{-1})\\ \mbox{ for }G_n=GU_{2n+1}(F). \\
\end{array} \right.
\]
The maximal split torus is then
\[
A=\{ d(a_1, \dots, a_n, a_0) \,|\, a_i \in F^{\times}\}.
\]

For the (non-split) quasi-split groups, the maximal quasi-split torus is also important in what follows. We start by looking at $G_n=U_{2n+1},U_{2n},GU_{2n+1},GU_{2n}, GSpin_{2n+2}^{(\varepsilon)}$. For $GU_{2n+1}(F)$, we have
\[
T=\{d(a_1, \dots, a_n, a_0)=diag(a_1, \dots, a_n, a_0, a_0\bar{a}_0\bar{a}_n^{-1}, \dots, a_0\bar{a}_0\bar{a}_1^{-1}) \,|\, a_i \in E^{\times} \}
\]
after identifying $H_1(F)$ with $E^{\times}$. Since the similitude factor is $a_0\bar{a_0}$, the quasi-split torus for $U_{2n+1}(F)$ is
\begin{align*}
T=\{d(a_1, \dots, a_n, a_0)=&diag(a_1, \dots, a_n, a_0, \bar{a}_n^{-1}, \dots, \bar{a}_1^{-1}) \, \\ &|\,a_1, \dots, a_n \in E^{\times} \text{ and } a_0\in N_1(E/F) \},
\end{align*}
where $N_1(E/F)$ denotes the norm one elements.
Similarly, for $GU_{2n}(F)$, we have
\begin{align*}
  T=\{d(a_1, \dots, a_n, a_0)=&diag(a_1, \dots, a_n, a_0\bar{a}_n^{-1}, \dots, a_0 \bar{a}_1^{-1}) \,\\
  &|\, a_1, \dots, a_n \in E^{\times} \mbox{ and }a_0 \in F^{\times} \}.  
\end{align*}
As the similitude factor is $a_0$, the quasi-split torus for $U_{2n}$ is
\[
  T=\{d(a_1, \dots, a_n)=diag(a_1, \dots, a_n, \bar{a}_n^{-1}, \dots, \bar{a}_1^{-1}) \,|\, a_1, \dots, a_n \in E^{\times}\}.  
\]
Note that these match the descriptions of \cite{Gol97} (though our form is slightly different in the odd case).
Viewing the quasi-split torus for $GSpin_{2n+2}^{(\varepsilon)}(F)$ as a subgroup of that for $GSpin_{2n+2}(E)$ (see \S \ref{sect3appendix} on $GSpin_{2n+2}^{\ast}$), we have elements of the quasi-split torus having the form (using $d_E$ for $GSpin_{2n+2}(E)$) $d_E(a_1, \dots, a_n, \bar{a}_0/a_0, a_0)$ with $a_0 \in E^{\times}$ and $a_i \in F^{\times}$ for $i>0$. To simplify the presentation below, we set
\[
d(a_1, \dots, a_n,a_0)=d_E(a_1, \dots, a_n, \bar{a}_0/a_0,a_0).
\]
The quasi-split torus is then
\[
T=\{d(a_1, \dots, a_n, a_0) \,|\, a_1, \dots, a_n \in F^{\times} \mbox{ and }a_0 \in E^{\times} \}.
\]

We now turn to the quasi-split orthogonal case. We first note that for the form used in defining $SO_{2n+2}^{(\varepsilon)}(F)$ and $GSO_{2n+2}^{(\varepsilon)}(F)$, the maximal quasi-split torus for $GSO_{2n+2}^{(\varepsilon)}(F)$ has matrix realization
\begin{align*}
  T=\{ & diag(a_1, \dots, a_n, X, (\det X) a_n^{-1}, \dots, (\det X) a_1^{-1}) \,\\
  & |\,
a_i \in F^{\times}, \, X \in GSO_2^{(\varepsilon)}(F) \},  
\end{align*}
where
\[
GSO_2^{(\varepsilon)}(F)=\left\{ X=\left( \begin{array}{cc} x & \varepsilon y \\ y & x \end{array}\right),  x, y \in F \mbox{ with } x^2-\varepsilon y^2 \not=0 \right\}.
\]
Note that $GSO_2^{(\varepsilon)}(F)\cong E^{\times}$ via $X= \left( \begin{array}{cc} x & \varepsilon y \\ y & x \end{array}\right) \longmapsto a_0=x+y\sqrt{\varepsilon}$. Under this isomorphism, $\det(X)$ corresponds to $N_{E/F}(a_0)$ (norm) and gives the similitude factor. We may abuse notation slightly and write this as
\[
T=\{ d(a_1, \dots, a_n, a_0) \,|\, a_1, \dots, a_n \in F^{\times} \mbox{ and }a_0 \in E^{\times}\}.
\]
Alternatively, realizing $GSO_{2n+2}^{(\varepsilon)}(F)$ by working inside $GSO_{2n+2}(E)$, in a manner similar to what is done for $GSpin_{2n+2}^{(\varepsilon)}(F)$ in \S \ref{sect3appendix}, one can arrange that the quasi-split torus genuinely has diagonal matrices. The maximal quasi-split torus for $SO_{2n+2}^{(\varepsilon)}(F)$ has the same form but with $\det(X)=1 \Leftrightarrow N_{E/F}(a_0)=1$.

We now look at the Weyl group action on the split tori. In the classical case, we have
\[
s_i \cdot d(a_1, \dots, a_{i-1}, a_i, a_{i+1}, a_{i+2}, \dots, a_n) =d(a_1, \dots, a_{i-1}, a_{i+1}, a_i, a_{i+2}, \dots, a_n)
\]
for $i<n$ and
\[
s_n \cdot d(a_1, \dots,a_{n-1},  a_n)=\left\{ \begin{array}{lr}
	d(a_1, \dots,a_{n-1},  a_n^{-1}) \\
	\mbox{ for }G_n=SO_{2n+1}, Sp_{2n}, SO_{2n+2}^{\ast}, U_{2n+1}, U_{2n},\\
	d(a_1, \dots, a_{n-2}, a_n^{-1},  a_{n-1}^{-1}) \\
	\mbox{ for }G_n=SO_{2n}.
\end{array} \right.
\]
If $G_n$ is one of the similitude groups,  the Weyl group acts as follows: for $i<n$, we have
\[
s_i \cdot d(a_1, \dots, a_{i-1}, a_i, a_{i+1}, a_{i+2}, a_n, a_0) =d(a_1, \dots, a_{i-1}, a_{i+1}, a_i, a_{i+2}, a_n, a_0)
\]
and for $i=n$,
\[
s_n \cdot d(a_1, \dots,a_{n-1},  a_n, a_0)=
\left\{ \begin{array}{l}
	d(a_1, \dots,a_{n-1},  a_n^{-1}, a_n a_0)\\ \mbox{ for }G_n=GSpin_{2n+1}, GSpin_{2n+2}^{\ast}, \\
	d(a_1, \dots,a_{n-1},  a_0a_n^{-1}, a_0)\\ \mbox{ for }G_n=GSp_{2n}, GU_{2n}, \\
	d(a_1, \dots, a_{n-2}, a_{n}^{-1},  a_{n-1}^{-1}, a_0 a_{n-1}a_n) \\
	\mbox{ for } G_n=GSpin_{2n}, \\
	d(a_1, \dots, a_{n-2}, a_0a_n^{-1},  a_0a_{n-1}^{-1}, a_0)\\ \mbox{ for } G_n=GSO_{2n}, \\
	d(a_1, \dots, a_{n-1}, a_0^2a_n^{-1}, a_0)\\
	\mbox{ for } G_n=GSO_{2n+2}^{\ast}, GU_{2n+1}.
\end{array} \right.
\]

For the (non-split) quasi-split groups, the action on the quasi-split torus is also important in what follows. The action of $s_i$, $i<n$ matches that above; for $i=n$, we have the following:
\[
s_n \cdot d(a_1, \dots, a_{n-1}, a_n, a_0)=\left\{ \begin{array}{l}
d(a_1, \dots, a_{n-1}, a_n^{-1}, \bar{a}_0)\\ \mbox{ if }G_n=SO_{2n+2}^{(\varepsilon)}, \\
d(a_1, \dots, a_{n-1}, a_0\bar{a}_0 a_n^{-1}, \bar{a}_0) \\
\mbox{ if }G_n=GSO_{2n+2}^{(\varepsilon)}, \\
d(a_1, \dots, a_{n-1}, \bar{a}_n^{-1}, a_0) \\
\mbox{ if }G_n=U_{2n+1}, \\
d(a_1, \dots, a_{n-1}, a_0\bar{a}_0\bar{a}_n^{-1}, a_0) \\
\mbox{ if }G_n=GU_{2n+1}, \\
d(a_1, \dots, a_{n-1}, a_0\bar{a}_n^{-1}, a_0)\\
\mbox{ if }G_n=GU_{2n},\\
d(a_1, \dots, a_{n-1}, a_n^{-1}, \bar{a}_0a_n)\\
\mbox{ if }G_n=GSpin_{2n+2}^{(\varepsilon)},
\end{array} \right.
\]
and
\[
s_n \cdot d(a_1, \dots, a_{n-1}, a_n)=d(a_1, \dots, a_{n-1}, \bar{a}_n^{-1})
\mbox{ if }G_n=U_{2n},\\
\]
recalling that
\[
a_1, \dots, a_n \in \left\{ \begin{array}{l} F^{\times} \mbox{ if }G_n=SO_{2n+2}^{(\varepsilon)}, GSO_{2n+2}^{(\varepsilon)}, GSpin_{2n+2}^{(\varepsilon)}, \\ E^{\times} \mbox{ if }G_n=U_{2n+1},U_{2n},GU_{2n+1}, GU_{2n}, \end{array} \right. 
\]
and
\[
a_0 \in \left\{ \begin{array}{l} N_1(E/F) \mbox{ if }G_n=SO_{2n+2}^{(\varepsilon)}, U_{2n+1} \\ E^{\times} \mbox{ if }G_n=GSO_{2n+2}^{(\varepsilon)}, GU_{2n+1}, GU_{2n}, GSpin_{2n+2}^{(\varepsilon)}. \end{array} \right.
\]
We remark that the restriction of the action of $s_n$ to the split tori matches the description above.

We now discuss centers and central characters for the similitude groups. In the split case, the center is just $\cap_{\alpha \in \Pi} \mbox{ker}\alpha$.
For an element $d(a_1, \dots, a_n, a_0)$ from the quasi-split torus to be in the center, it must be fixed under the action of the Weyl group. For $GSO_{2n+2}^{(\varepsilon)}(F)$, this implies the center lies in the split torus; the center then follows as in the split case. For $GSpin_{2n+2}^{(\varepsilon)}(F)$, the center contains $Z(GSpin_{2n+2}(E)) \cap GSpin_{2n+2}^{(\varepsilon)}(F)$ and is contained in the set of elements of the quasi-split torus fixed by the Weyl group; these match and give the center as described below. For $GU_{2n}(F)$, the center contains the scalar multiples of the identity which lie in $GU_{2n}(F)$ and is contained in the set of elements of the quasi-split torus fixed by the Weyl group; again, these match and give the center below. For $GU_{2n+1}(F)$, the center again contains the scalar multiples of the identity which lie in $GU_{2n+1}(F)$. To see that there is nothing else, consider the root group corresponding to the $F$-root $\alpha_n=e_n-e_0$; direct calculation shows that conjugation by $d(a_1, \dots, a_n,a_0)$ from $GU_{2n+1}$ commutes with this only if $a_na_0^{-1}=1$ (for $a_0,a_n \in E^{\times}$). Thus, to be in the center, we must have $a_0=a_n$; Weyl invariance then gives $a_1=\dots=a_n=a_0$ (so scalar). Thus we obtain the following conditions (see \cite[Proposition 4.3(v)]{Tad94} for $G_n=GSp_{2n}$, \cite[Proposition 2.3]{AS06} for $G_n=GSpin_{2n}$, and \cite[Proposition 2.10]{AS06} for $G_n=GSO_{2n}$):
\begin{table}[H]
\begin{tabular}{l|l}
$G_n$ & constraints for center  \\ \hline
$GSpin_{2n+1}$ & $a_1= \dots =a_n=1$  \\
$GSp_{2n}$ & $a_1= \dots =a_n=z$ and $a_0=z^2$    \\
$GSpin_{2n}$ & $a_1= \dots =a_n=\zeta$ with $\zeta^2=1$ \\
$GSO_{2n}$ & $a_1= \dots =a_n=z$ and $a_0=z^2$   \\
$GSO_{2n+2}^{\ast}$ & $a_1= \dots =a_n=a_0=z$ \\
$GSpin_{2n+2}^{\ast}$ & $a_1= \dots =a_n=\zeta$, $a_0=\zeta \bar{a}_0$ with $\zeta^2=1$  \\
$GU_{2n+1}$ & $a_1= \dots =a_n=a_0=z$  \\
$GU_{2n}$ & $a_1= \dots =a_n=z, \, a_0=z\bar{z}$
\end{tabular}
\caption{Centers}
\label{tab:center}
\end{table}

\begin{table}[H]
\begin{tabular}{l|l}
$G_n$ & $\omega_{\pi_1 \times \dots \times \pi_k \rtimes \pi_0}$ \\ \hline
$GSpin_{2n+1}$  & $\omega_{\pi_0}$ \\
$GSp_{2n}$ & $\omega_{\pi_1} \omega_{\pi_2} \dots \omega_{\pi_k}\omega_{\pi_0}$ if $n_0>0$ \\
& $\omega_{\pi_1} \omega_{\pi_2} \dots \omega_{\pi_k}\omega_{\pi_0}^2$ if $n_0=0$\\
$GSpin_{2n}$  & $\omega_{\pi_0}$\\
$GSO_{2n}$  & $\omega_{\pi_1} \omega_{\pi_2} \dots \omega_{\pi_k}\omega_{\pi_0}$ if $n_0>0$ \\
& $\omega_{\pi_1} \omega_{\pi_2} \dots \omega_{\pi_k}\omega_{\pi_0}^2$ if $n_0=0$ \\
$GSO_{2n+2}^{\ast}$ & $\omega_{\pi_1} \omega_{\pi_2} \dots \omega_{\pi_k} \omega_{\pi_0}$ \\
$GSpin_{2n+2}^{\ast}$ & $\omega_{\pi_0}$ \\
$GU_{2n+1}$  & $\omega_{\pi_1} \omega_{\pi_2} \dots \omega_{\pi_k} \omega_{\pi_0}$ \\
$GU_{2n}$ & $\omega_{\pi_1} \omega_{\pi_2} \dots \omega_{\pi_k} \omega_{\pi_0}$ if $n_0>0$ \\
&  $\omega_{\pi_1} \omega_{\pi_2} \dots \omega_{\pi_k} (\omega_{\pi_0} \circ N_{E/F})$ if $n_0=0$
\end{tabular}
\caption{Central characters}
\label{tab:center2}
\end{table}

\noindent
Note that for $GSpin_{2n}(F)$ with $n>1$, the center actally has $Z \cong \{\pm 1\} \times F^{\times}$; we abuse notation slightly and use $\omega_{\pi}$ to denote the central character of $\pi$ on the $F^{\times}$ part of $Z$ (which is what actually arises in (\ref{w_0action})); similarly for $GSpin_{2n+2}^{\ast}(F)$. We also note that $GSO_2^{\ast}(F) \cong GSpin_2^{\ast}(F) \cong E^{\times}$, so $\omega_{\pi_0}$ is technically a character of $E^{\times}$ in these cases. However, in practice it is the restriction to $F^{\times}$ that actually comes into play in Table~\ref{tab:center2} or (\ref{w_0action}).

%for Table \ref{tab:center2} above, it is the restriction to $F^{\times}$ which is needed (as that is all that appears in the center of a larger rank group).

\

For $G_n=SO_{2n}, GSO_{2n}$, or $GSpin_{2n}$--groups of type $D_n$--we would like to set things up so that the $\mu^{\ast}$ structure and Casselman criterion/Langlands classification have forms like that for the other similitude groups considered. This has been done in \cite{JL14} for $SO_{2n}(F)$ and \cite{Kim16} for $GSpin_{2n}$; we follow the same basic strategy for $GSO_{2n}$. For this, we must first take up the existence of an outer automorphism corresponding to the action of $c$ on $SO_{2n}(F)$ ($c$ defined on \cite[p.208]{JL14}, e.g.). Of course, for $GSO_{2n}(F)$, the same $c$ provides the outer automorphism. For $GSpin_{2n}(F)$, a corresponding $c$ of order two is constructed on \cite[p.622]{Kim09}.
The following lemma combines the discussion of $c$ for $GSpin_{2n}$ from \cite{Kim09} and \cite{HS16}, and Haar measure properties in \cite[Section 2]{BJ01}. The (based) root datum for $(\Pi, X, \check{\Pi}, \check{X})$ for $GSpin_{2n}$ is given in \cite{Asg02} and summarized in \S \ref{sect2} above; the action of $c$ on the data is described in \cite[Lemma 4.5]{HS16}. We include the proof for the sake of completeness.

\begin{lem}\label{lem c}
There exists an outer automorphism $c: GSpin_{2n}(F) \longrightarrow GSpin_{2n}(F)$ of order two with the following properties:

\begin{enumerate}

%\item $c \cdot \Pi=\Pi$ and $c \cdot \check{\Pi}=\check{\Pi}$.

\item $c \cdot (\Pi, X, \check{\Pi}, \check{X})=(\Pi, X, \check{\Pi}, \check{X})$.

\item $c$ is a homeomorphism in the $p$-adic topology.

\item $c$ preserves Haar measures on $GSpin_{2n}(F)$ and $GSpin_{2n}(F)/Z(F)$ (where $Z(F)$ is the center).

\end{enumerate}

\noindent
In particular, it then follows that if a representation $\pi$ of $GSpin_{2n}(F)$ is square-integrable (resp., tempered), then so is $c \cdot \pi$. Similar considerations apply to standard Levi factors and their representations--if $\theta$ is a square-integrable (resp., tempered) representation of a standard Levi factor $M$, then $c \cdot \theta$ is a square-integrable (resp., tempered) representation of the standard Levi factor $c(M)$.

\end{lem}

\noindent
\begin{proof}
First, note that for $GSO_{2n}$, the matrix $c$ (see \cite[p.208]{JL14}) has these properties. For $GSpin_{2n}$, there is an outer automorphism constructed in \cite[Section 2]{Kim09} which interchanges the last two simple roots and can be seen to be of order 2. This corresponds to the automorphism on the data in \cite[Section 4.3]{HS16} and has the properties there, which imply (1). We remark the the root data for $GSpin_{2n}$ is dual to that for $GSO_{2n}$ (cf. \cite{Asg02}); the actions of their respective $c$'s are dual.

%For $GSpin_{2n}$, observe that one has $c$ acting on $GSO_{2n}$, hence on its root data; as the data for $GSpin_{2n}$ is dual, we also have an action on the root data for $GSpin_{2n}$. This then gives rise to an automorphism on $GSpin_{2n}(F)$ (see Theorem 16.3.2 of \cite{Spr98}). This clearly satisfies (1).

 For (2), let $I_k$ denote the $k$th group in the Iwahori filtration, i.e., the subgroup generated by $\{u_{\alpha}(x) \,|\, x \in {\mathcal P}^k,\ \alpha \in \Delta^+\}$, $\{u_{\alpha}(y) \,|\, y \in {\mathcal P}^{k+1},\ \alpha \in \Delta^-\}$, and $\{\check{x}(z) \,|\, z \in 1+{\mathcal P}^k,\ \check{x} \in \check{X}\}$. This is a basis for the topology at the identity. From this description, one clearly has $c(I_k)=I_k$, from which the claim follows.

%For (3), let $\mu$ denote Haar measure on $G$. As $c$ is homeomorphic, it takes measurable sets to measurable sets, so $c \cdot \mu$ is defined. As in Lemma 2.2 \cite{B-J2001}, one then has $c \cdot \mu=\mu$.  Note that from the description of $Z$ above and the action of $c$ on the root data, it follows that $c(z)=z$ for all $z \in Z$. Thus one can pass to the quotient measure $\bar{\mu}$ on $G/Z$, and has $c \cdot \bar{\mu}=\bar{\mu}$. 

For (3), first let $\mu$ denote Haar measure on $G$. As $c$ is homeomorphic, it takes measurable sets to measurable sets, so $c \cdot \mu$ is defined. As in \cite[Lemma 2.2]{BJ01}, one then has $c \cdot \mu$ left-invariant hence a multiple of $\mu$; as $c \cdot \mu(I_k)=\mu(I_k)$, they are equal. Next, observe that since $c(Z)=Z$, we have a quotient action $c$ on $G/Z$. Further, since the action of $c$ on $G$ gives an automorphism of toplogical groups, the action of $c$ on $G/Z$ is also an automorphism of topological groups. As a consequence, if $\bar{\mu}$ denotes a Haar measure on $G/Z$, $c \cdot \bar{\mu}$ is also a Haar measure on $G/Z$, hence a multiple of $\bar{\mu}$. In fact, looking at the measure of a $c$-invariant subset of $G/Z$, it follows that $c \cdot \bar{\mu}=\bar{\mu}$. Therefore, if $\pi$ is a square-integrable (resp., tempered) representation of $G$, then $c \cdot \pi$ is also a square-integrable (resp., tempered) representation of $G/Z$.

The remaining claims now follow directly. \end{proof}

\begin{rem}\label{rmk c}
Similar considerations apply to $SO_{2n+2}^{\ast}$, $GSO_{2n+2}^{\ast}$ and $GSpin_{2n+2}^{\ast}$. For the form used for $SO_{2n+2}^{\ast}$, $GSO_{2n+2}^{\ast}$ (see \S \ref{sect2}), the matrix
\[
c=\left( \begin{array}{cccc} I_n & & & \\
     & 1 & & \\
     & & -1 & \\
     & & & I_n
     \end{array} \right)
\]
may be used. For $GSpin_{2n+2}^{\ast}$, we have $GSpin_{2n+2}^{\ast}(F) \subset GSpin_{2n+2}(E)$ and inherits the action of $c$.
Note that for $SO_{2n}^{\ast}$, $GSO_{2n+2}^{\ast}$, and $GSpin_{2n+2}^{\ast}$, $c$ acts trivially on the $F$-data.

%This also allows us to define a non-connected group $\tilde{G}=G \rtimes C$, where $C=\{e,c\}$ (as an algebraic group or on the $F$-points). We have not found this in the literature, but suppose it might reasonably be called $GPin_{2n}$. This is dual to the situation with $GSO_{2n}$ and $GO_{2n}$.
\end{rem}

We return to the task of setting things up in a more uniform way for the type $D_n$ similitude groups. Following \cite{JL14},
we let both $\chi \otimes e$ and $\chi \otimes c$ denote the character $\chi$ on $G_0(F) \cong F^{\times}$, but with different interpretations when used with parabolic induction. These play the role of $1 \otimes e$ and $1 \otimes c$ for $SO_{2n}$; the discussion below applies to $SO_{2n}$ if these are used instead. In particular, suppose $P=MU$ is a standard parabolic subgroup with $\alpha_n \not\in \Pi_M$. Then $M=GL_{m_1}(F) \times \dots \times GL_{m_k}(F) \times G_0(F)$. For representations $\pi_1, \dots, \pi_k$ of $GL_{m_1}(F), \dots, GL_{m_k}(F)$, we let $\pi_1 \otimes \dots \otimes \pi_k \otimes (\chi \otimes e)$ denote a representation of $M$, while $\pi_1 \otimes \dots \otimes \pi_k \otimes (\chi \otimes c)$ denotes a representation of $c(M)$ (the Levi factor of the standard parabolic subgroup $c(P)$). Thus, we write
$$
\pi_1 \times \dots \times \pi_k \rtimes (\chi \otimes e)=
i_{G,M}(\pi_1 \otimes \dots \otimes \pi_k \otimes \chi),
$$
and
$$
\pi_1 \times \dots \times \pi_k \rtimes (\chi \otimes c)=c(i_{G,M}(\pi_1 \otimes \dots \otimes \pi_k \otimes \chi)).
$$
In terms of the action of $c$, we take $c(\chi \otimes e)=\chi \otimes c$ and $c(\chi \otimes c)=\chi \otimes e$. Note that if $\chi'$ is a character of $F^{\times}$, then the representations $\chi' \rtimes (\chi \otimes e)$ and $\chi'^{-1} \rtimes (\chi'\chi \otimes c)$ constitute the same representation of $GSO_{2}(F) \cong F^{\times} \times F^{\times}$; similarly for $GSpin_{2}(F)$ with  $\chi' \rtimes (\chi \otimes e)$ and $\chi\chi'^{-1} \rtimes (\chi \otimes c)$.
(For either family of groups, if $M=GL_{m}(F) \times G_0(F)$ with $m>1$, we have $M$ and $c(M)$ nonconjugate and the corresponding induced representations are not in general equivalent.) Note that it is a straightforward consequence of induction in stages and $c \circ i_{G,M} \cong i_{G,c(M)} \circ c$ that
\begin{equation}\label{action of c}
c \cdot (\pi \rtimes \sigma) \cong \pi \rtimes c \cdot \sigma
\end{equation}
and
\begin{equation}\label{module}
\pi_1 \rtimes (\pi_2 \rtimes \sigma) \cong (\pi_1 \times \pi_2) \rtimes \sigma
\end{equation}
in this context.

%Note that if $\alpha_{n-1},\alpha_n \not\in \Pi$, then $c(M)=M$, $\tau_k$ is a representation of $F^{\times}$, and the representations $\tau_1 \otimes \dots \otimes \tau_{k-1} \otimes \tau_k \otimes (\chi \otimes e)$ and $\tau_1 \otimes \dots \otimes \tau_{k-1} \otimes \tau_k^{-1} \otimes (\tau_k\chi \otimes c)$ are the same.

We now take up how induced representations behave under twisting by characters. For concreteness, we consider the case of $GSpin_{2n+1}$ (see \cite{Kap17} for a slightly different approach to the same question).
If $\xi_n$ denotes a rational character of $GSpin_{2n+1}(F)$, $n>0$, then the restriction of $\xi_n$ is invariant under the action of the Weyl group. Checking the action of the simple reflections, $\xi_n|_A=c_1e_1+\dots +c_ne_n+c_0e_0$ has $c_1=\dots=c_n=\frac{1}{2}c_0$; we normalize so that $\xi_n|_A=e_1+\dots+e_n+2e_0$. If $n=0$, we take $\xi_0=e_0$. In either case, $\xi_n$ is a $\mathbb Z$-basis for $X(G)$ (rational characters on $G_n=GSpin_{2n+1}(F)$). Note that if $M=GL_k(F) \times GSpin_{2(n-k)+1}(F)$ is a standard Levi factor, it follows from this description that
\[
\xi_n|_M=\left\{ \begin{array}{l} 
	\det_k \otimes \xi_{n-k} \mbox{ if }k<n, \\
	\\
	\det_k \otimes \xi_0^2 \mbox{ if }k=n.
\end{array} \right.
\]

Let $\chi$ be a character of $F^{\times}$. We may identify $\chi$ with a character of $GSpin_{2n+1}(F)$ (resp., $GL_k(F)$) via $\chi \circ \xi_n$ (resp., $\chi \circ \det_k$). With this identification, we have
\[
\chi(\pi \rtimes \theta)\cong\left\{ \begin{array}{l}
	\chi\pi \rtimes \chi\theta \mbox{ if }k<n \\
	\\
	\chi\pi \rtimes \chi^2\theta \mbox{ if }k=n
\end{array}\right.
\]
The bifurcation in the formula is essentially dual to that in the central character formula for $GSp_{2n}(F)$ above (see \cite[Proposition 4.3(v)]{Tad94}). It is a straightforward matter to verify that these are indeed equivalent (with equivalence given by $f \in V_{\pi \rtimes \theta} \longmapsto \chi f \in V_{\chi\pi \rtimes \chi^{\varepsilon}\theta}$, $\varepsilon=1$ or $2$, as appropriate); see \cite[Proposition 1.9]{BZ77}.

We note that for $G_n=GSpin_{2n}$, we also have  $\xi_n|_A=e_1+\dots+e_n+2e_0$ for $n>1$ and $\xi_0=e_0$; for $G_n=GSpin_{2n+2}^{\ast}$, see \S \ref{sect3appendix}. Similar arguments then give the following
(also, see, e.g., \cite[(1.2)]{S-T} or \cite[Proposition 4.3]{Tad94} for $G_n=GSp_{2n}$ and \cite{CFK20} for $G_n=GSpin_{2n+1},GSpin_{2n}$).%\footnote{Immediately following (4.23) in \cite{ACS}, they refer to $e_0$ as the ``similitude character" for odd GSpin.}

\begin{lem}\label{twisting}
Let $\chi$ be a character of $F^{\times}$. We may identify $\chi$ with a character of $G_n(F)$ via $\chi \circ \xi$, where $\xi=\xi_n$ as above for $G_n=GSpin_{2n+1}$, $GSpin_{2n}$ and is the similitude character for $$G_n=GSp_{2n},GSO_{2n},GSO^{\ast}_{2n+2},GU_{2n+1},GU_{2n}$$
(Recall that for $GU_{2n+1}(F)$ and $GU_{2n}(F)$, the image of $\xi$ lies in $F^{\times}$, so $\chi$ a character of $F^{\times}$ is sufficient).  For $\pi$ and $\theta$ representations of $H_k(F)$ and $G_n(F)$, we then have the following:

\begin{enumerate}

\item For $G_n=GSpin_{2n+1}$,
$
\chi(\pi \rtimes \theta)\cong \left\{ \begin{array}{l}
	\chi\pi \rtimes \chi\theta \mbox{ if }n>0 \\
	\chi\pi \rtimes \chi^2\theta \mbox{ if }n=0.
\end{array}\right.
$

\item For $G_n=GSp_{2n}$, $GSO_{2n+2}^{\ast}$, $GU_{2n+1}, GU_{2n}$, $\chi(\pi \rtimes \theta)\cong \pi \rtimes \chi\theta$.

\item For $G_n=GSO_{2n}$, $\chi(\pi \rtimes \theta)\cong \pi \rtimes \chi\theta$ for $n \not=1$.

\item For $G_n=GSpin_{2n}$, $\chi(\pi \rtimes \theta)\cong \left\{\begin{array}{l}
	\chi\pi \rtimes \chi\theta \mbox{ if }n>1, \\
	\chi\pi \rtimes \chi^2\theta \mbox{ if }n=0.
	\end{array}\right.$

\item For $G_n=GSpin_{2n+2}^{\ast}$, $\chi(\pi \rtimes \theta)\cong \chi\pi \rtimes \chi\theta.$
\end{enumerate}
\end{lem}

\begin{rems}
\begin{enumerate}
\item
We have not discussed twisting by characters for $G_n=GSO_{2n},GSpin_{2n}$ when $n=1$. The issue is that $G_1(F) \cong F^{\times} \times F^{\times}$ for these groups, so $X(G_1)$ is actually 2-dimensional. One could reasonably define
\[
\begin{array}{rl}
\chi (\chi_1 \rtimes \chi_0) &=\chi_1 \rtimes \chi\chi_0 \mbox{ for } G_n=GSO_{2n}, \\
& =\chi\chi_1 \rtimes \chi^2\chi_0 \mbox{ for }G_n=GSpin_{2n}
\end{array}
\]
to make the formulas above work when $n=1$.

\item Note that the image of the similitude character need not be all of $F^{\times}$--e.g., for $GSO_{2n+2}^{\ast}$, \cite[Lemma 2.1]{Xu18} shows that it consists of only the norms of $E/F$, where $E$ is the associated quadratic extension.

\item In terms of the action of $c$, observe that $c \cdot \chi=\chi$. For this, it suffices that $c \cdot \xi=\xi$. For $G_n=GSO_{2n},GSO_{2n+2}^{\ast}$, this may be verified directly from ${}^T\!XJX=\lambda J$. For $G_n=GSpin_{2n}$, one has $(c \cdot \xi)|_A=\xi|_A$ by the definition of $\xi$ above and \cite[Lemma 4.5]{HS16}. Further, as $\xi|_U$ is trivial and $c \cdot U=U$, we have $(c \cdot \xi)|_U=\xi|_U$. Similarly, $(c \cdot \xi)|_{\bar{U}}=\xi|_{\bar{U}}$. As $A,U,\bar{U}$ suffice to generate $G$, we have $c\cdot \xi=\xi$. The similitude character for $GSpin_{2n+2}^{\ast}(F)$ may be obtained by restriction of that for $GSpin_{2n+2}(E)$ (see \S \ref{sect3appendix}) so also satisfies $c \cdot \xi=\xi$.

%If more detail is needed, let $\xi^{\ast}$ denote the restriction of $\xi$ (for $GSpin_{2n+2}(E)$ to $GSpin_{2n+2}^{\ast}(F)$. Direct calculation on the quasi-split torus shows that $c \cdot \xi^{\ast}=\xi^{\ast}$. Since $U^{\ast}(F) \subset U(E)$, it follows that $\xi^{\ast}$ and $c \cdot \xi^{\ast}$ are trivial on $U^{\ast}(F)$, hence equal. Similarly, one has $c \cdot \ni^{\ast}=\xi^{\ast}$  $\bar{U}^{\ast}(F)$.

\end{enumerate}
\end{rems}

\begin{note}
We note that for $G_n=GU_{2n+1}$, $GU_{2n}$ (or, $U_{2n+1}$, $U_{2n}$) one can also compose a character $\chi$ of $E^{\times}$ with $\det$ to produce a character of $G_n$. It is not difficult to see that this construction can produce characters which are not among those in Lemma ~\ref{twisting}, and vice-versa. Though we do not use these characters in what follows, we include the corresponding twisting formula for the sake of completeness:
\[
(\chi \circ \det ) (\pi \rtimes \theta)=(\chi \circ \tau \circ \det )\pi \rtimes [\chi \circ (\xi^k \cdot \det)] \theta,
\]
where $\tau(x)=x\bar{x}^{-1}$ on $E^{\times}$ and $M=H_k(F) \times G_{n-k}(F)$. We also remark that these characters of $GU_N(F)$ are related by $\det \cdot \,\overline{\det}=\xi^N$ (take $\det$ of ${}^T\!\bar{X} J' X=\lambda J'$).

%We make a brief technical remark. For the case of $G_n=U_n$, e.g., the image of $\mbox{det}$ lies in $N_1(E/F)$--the norm one elements. Thus, one could simply require $\chi$ to be a character of $N_1(E/F)$, rather than of $E^{\times}$. However, it follows from Hilbert's Theorem 90 and Pontriajin duality, that any such character may be realized as the restriction of a character of $E^{\times}$, so there is no loss in generality in starting with characters of $E^{\times}$.
\end{note}

The following definition and lemma allow for more concise versions of the Casselman criterion and Langlands classification.

\begin{defn}\label{beta def}
Let $\pi$ be an irreducible representation of $G_n$.
Write $\pi=\nu^{\varepsilon(\pi)}\pi_0$ with $\pi_0$ having a unitary central character. We set
\[
\beta(\pi)=\left\{ \begin{array}{l}
\varepsilon(\pi) \mbox{ if } G_n=GSpin_{2n+1},GSpin_{2n} \mbox{ and }n>0, \\
\frac{1}{2}\varepsilon(\pi) \mbox{ if } G_n=GSpin_{2n+1}, GSpin_{2n} \mbox{ and }n=0, \\
\varepsilon(\pi) \mbox{ if } G_n=GSpin_{2n+2}^{\ast}, \\
0 \mbox{ otherwise.}
\end{array} \right.
\]
\end{defn}

\begin{lem}\label{beta}
If $\pi \hookrightarrow \phi_1 \times \dots \phi_k \rtimes \sigma^{(e0)}$ is an embedding into a representation induced from supercuspidals, then $\beta(\pi)=\beta(\sigma^{(e0)})$.
\end{lem}

\noindent
\begin{proof}
We focus on the case $G_n=GSpin_{2n+1}, GSpin_{2n}$. The case $G_n=GSpin_{2n+2}^{\ast}$ is similar; the other cases are trivial.

If $\pi=\sigma^{(e0)}$ the result is immediate. So, suppose $\pi \not=\sigma^{(e0)}$. Note that we must then have $\beta(\pi)=\varepsilon(\pi)$. Recall that  $\pi=\nu^{\varepsilon(\pi)}\pi_0$ with $\pi_0$ having a unitary central character. Then
\[
\begin{array}{c}
\pi_0 \hookrightarrow \phi_1' \times \dots \phi_k' \rtimes \sigma_0 \Rightarrow
\pi \hookrightarrow \nu^{\varepsilon(\pi)}(\phi_1' \times \dots \phi_k' \rtimes \sigma_0).
\end{array}
\]
We then have
\[
\pi \hookrightarrow \nu^{\varepsilon(\pi)}\phi_1' \times \dots \nu^{\varepsilon(\pi)}\phi_k' \rtimes
\left\{ \begin{array}{l} \nu^{\varepsilon(\pi)}\sigma_0 \mbox{ if }n_\sigma>0, \\ \nu^{2\varepsilon(\pi)}\sigma_0 \mbox{ if }n_\sigma=0. \end{array} \right.
\]
Therefore,
\[
\begin{array}{c}
\sigma^{(e0)}=\left\{ \begin{array}{l} \nu^{\varepsilon(\pi)}\sigma_0 \mbox{ if }n_\sigma>0, \\ \nu^{2\varepsilon(\pi)}\sigma_0 \mbox{ if }n_\sigma=0 \end{array} \right. \\
\Downarrow \\
\varepsilon(\sigma^{(e0)})=\left\{ \begin{array}{l} \varepsilon(\pi) \mbox{ if }n_\sigma>0, \\ 2\varepsilon(\pi) \mbox{ if }n_\sigma=0 \end{array} \right. \\
\Downarrow \\
\beta(\pi)=\varepsilon(\pi)=\left\{ \begin{array}{l} \varepsilon(\sigma^{(e0)}) \mbox{ if }n_\sigma>0, \\ \frac{1}{2}\varepsilon(\sigma^{(e0)}) \mbox{ if }n_\sigma=0 \end{array} \right. =\beta(\sigma^{(e0)}),
\end{array}
\]
as needed. \end{proof}

\begin{note}\label{beta note}
In light of Lemma~\ref{beta}, we simply write $\beta$ below for $\beta(\pi)$, $\beta(\sigma^{(e0)})$, etc.
\end{note}

We now return to our general discussion of the results needed for the families of groups under consideration.
We start with the Casselman criterion. See \cite[Propositions III.1.1 and III.2.2]{Wal03} for the general result;  \cite[Section 6]{Tad94}, \cite[Proposition 4.2]{Asg02}--noting that it is missing a unitary central character hypothesis--and \cite[Proposition 3.2]{Kim09} for some of the specific groups under consideration. Suppose $\pi$ is irreducible.
Let $\phi_1 \otimes \dots \otimes \phi_k \otimes \sigma^{(e0)} \leq r_{M,G}(\pi)$, with $\phi_i$ an irreducible supercuspidal representation of $H_{n_i}(F)$ and $\sigma^{(e0)}$ an irreducible supercuspidal representation of $G_{n_0}(F)$. If $\pi$ is essentially square-integrable, then
\begin{equation}\label{Casselman criterion}
\begin{array}{rl}
n_1[\varepsilon(\phi_1)-\beta] & >0 \\
n_1[\varepsilon(\phi_1)-\beta]+n_2[\varepsilon(\phi_2)-\beta] &>0 \\
\vdots & \\
n_1[\varepsilon(\phi_1)-\beta]+n_2[\varepsilon(\phi_2)-\beta]+ \dots +n_k[\varepsilon(\phi_k)-\beta] & >0.
\end{array}
\end{equation}
Conversely, if the above inequalities hold for all such $\phi_1 \otimes \dots \otimes \phi_k \otimes \sigma^{(e0)}$, then $\pi$ is essentially square-integrable. The criterion for temperedness is similar, but with weak inequalities.  Note that for $G_n=SO_{2n}$,  if $\alpha_{n-1},\alpha_n \not\in \Pi_M$, both  $\phi_1 \otimes \dots \otimes \phi_{k-1} \otimes \phi_k \otimes (1 \otimes e)$ and the equivalent $\phi_1 \otimes \dots \otimes \phi_{k-1} \otimes \phi_k^{-1} \otimes (1 \otimes c)$ must be used in the Casselman criterion; for $GSO_{2n}$ (resp., $GSpin_{2n}$), both $\phi_1 \otimes \dots \otimes \phi_k \otimes \sigma^{(e0)}$ and $\phi_1 \otimes \dots \otimes \phi_{k-1} \otimes \phi_k^{-1} \otimes (\phi_k\chi \otimes c)$ (resp., $\phi_1 \otimes \dots \otimes \phi_{k-1} \otimes \chi\phi_k^{-1} \otimes (\chi \otimes c)$) must be used.

%Note that for similitude groups, an essentially square-integrable or essentially tempered representation is obtained by twisting by a character, which is discussed below.

We note that for $SO_2(F)$ (resp., $GSO_2(F)$, $GSpin_2(F)$), the representation $\chi \rtimes (1 \otimes e)=\chi^{-1} \rtimes (1 \otimes c)$ (resp., $\chi \rtimes (\chi_0 \otimes e)=\chi^{-1} \rtimes (\chi\chi_0 \otimes c)$, $\chi \rtimes (\chi_0 \otimes e)=\chi_0 \chi^{-1} \rtimes (\chi_0 \otimes c)$) with $\chi_0$ unitary is considered tempered but not square-integrable if $\chi$ is unitary. This interpretation is consistent with other cases of irreducible $\rho \rtimes \sigma^{(0)}$ and the inequalities above.

We now turn to the Langlands classification. For the general result, see \cite{Kon03}; more concretely, see \cite[Section 6]{Tad94} for $Sp_{2n}$ and $GSp_{2n}$; \cite[Section 1.3]{Jan93} for $SO_{2n+1}$ and $SO_{2n}$, with the latter interpreted in terms of the artifice above in \cite{Jan11}; \cite[Theorem 5.4 and Remark 5.5]{Kim09} for $GSpin_{2n+1}$, $GSpin_{2n}$, $GSpin_{2n}^{\ast}$, noting the assumption $T$ tempered there. Let $\delta_1, \dots, \delta_k$ be essentially square integrable representations of general linear groups and and $T$ an essentially tempered representation of $G_n(F)$ satisfying
\begin{equation}\label{Langlands classification}
\varepsilon(\delta_1) \geq \dots \geq \varepsilon(\delta_k)> \beta.
\end{equation}
Then the Langlands classification tells us that
\[
\delta_1 \times \dots \times \delta_k \rtimes T
\]
contains a unique irreducible quotient; denote it by $L(\delta_1 \otimes \dots \otimes \delta_k \otimes T)$. Further, any irreducible admissible representation may be written in this form, with the data unique up to permutations among representations of general linear groups having the same central exponents.

We also have occasion to use the Langlands classification in the subrepresentation setting. In this case, the inequalities in (\ref{Langlands classification}) are reversed and we use $L_{sub}(\delta_1 \otimes \dots \otimes \delta_k; T)$ for the corresponding (Langlands) subrepresentation.

For general linear groups, the Langlands classification is similar--if
\[
\varepsilon(\delta_1) \geq \dots \geq \varepsilon(\delta_k),
\]
then $\delta_1 \times \dots \times \delta_k$ has a unique irreducible quotient ${\mathcal L}(\delta_1 \otimes \dots \otimes \delta_k)$ and every irreducible representation may be written in this form, with $\delta_1 \otimes \dots \otimes \delta_k$ unqiue up to the order in which representations having the same central characters appear. The corresponding subrepresentation version has the inequalities reversed and ${\mathcal L}_{sub}(\delta_1 \otimes \dots \otimes \delta_k)$ the unique irreducible subrepresentation of $\delta_1 \times \dots \times \delta_k$. 

 Note that for $G_n=GSO_{2n}$ or $GSpin_{2n}$ and $\alpha_{n-1}, \alpha_n \in \Pi_M$, one can have $T=\chi \rtimes (\chi_0 \otimes d)$, $d=e$ or $c$, if it is essentially tempered, i.e., if $\chi$ unitary (for $GSO_{2n}$) or $e(\chi)=\frac{1}{2}e(\chi_0)$ (for $GSpin_{2n}$). The translation from the standard form of the Langlands classification to the description above using the artifice is then a straightforward matter.

\begin{rem}\label{action of c remark}
We also make a brief remark on the action of $c$ for $G_n=SO_{2n}$, $SO_{2n+2}^{\ast}$, $GSO_{2n}$, $GSO_{2n+2}^{\ast}$, $GSpin_{2n}$, or $GSpin_{2n+2}^{\ast}$. It is a straightforward matter to verify that with notation as above--and already noted in (\ref{action of c}) for the $D_n$ cases--that
\[
c \cdot(\tau \rtimes \sigma) \cong \tau \rtimes c \cdot \sigma.
\]
Further, if $\pi=L(\delta_1 \otimes \dots \otimes \delta_k \otimes T)$, then $c\cdot \pi=L(\delta_1 \otimes \dots \otimes \delta_k \otimes c\cdot T)$ (\cite[Proposition 4.5]{BJ01}).
\end{rem}

%A remark on orthogonal groups is in order. If $G_n(F)$ is a special orthogonal group, let $\tilde{G}_n(F)$ be the corresponding orthogonal group. In the odd case, we have $-I \in \tilde{G}_n(F)$ and $\tilde{G}_n(F) \cong G_n(F) \times \{\pm I\}$. In the even case, let $c$ be an element of $\tilde{G}_n(F) \setminus G_n(F)$ which preserves the positive roots and has $c^2=I$. Then $c$ acts nontrivially on $G_n(F)$; we have $\tilde{G}_n(F) \cong G_n(F) \rtimes C$ (where $C=\{I,c\}$). In the split case, this produces the possibility (CN) in \cite{JL14}, where $\rho \cong \tilde{\rho}$ but $\nu^x \rho \rtimes \sigma^{(0)}$ is irreducible for all $x \in {\mathbb R}$. For (CN) to occur, $\rho$ must be a representation of an odd general linear group; the condition to be ramified is then $\rho \otimes \sigma^{(0)} \cong \tilde{\rho} \otimes c \cdot \sigma^{(0)}$, which fails for self-contragredient $\rho$ if $c \cdot \sigma^{(0)} \not\cong \sigma^{(0)}$. In the non-split case, this is not an issue--the Weyl group contains both even and odd sign changes, so the condition to be ramified is simply $\rho \cong \tilde{\rho}$ regardless of the size of the underlying general linear group.

\

We now discuss cuspidal reducibility. In particular, suppose $\tau$ is an irreducible, unitary, supercuspidal representation of $H_m(F)$ and $\sigma^{(0)}$ an irreducible unitary supercuspidal representation of $G_n(F)$. As in \cite{Sha90a}, for a maximal parabolic subgroup $P=MU$ associated to $\Pi \setminus \{\alpha\}$, we set $\tilde{\alpha}=\langle \rho_U, \check{\alpha} \rangle^{-1} \rho_U$ with $\rho_U$ the half-sum of the $F$-roots from $U$. For $s \in {\mathbb R}$, we let $I(s\tilde{\alpha}, \tau \otimes \sigma^{(0)})$ denote the corresponding parabolically induced representation. It follows from work of Silberger and Harish-Chandra (\cite{Sil79}, \cite{Sil80}) that if $\tau \otimes \sigma^{(0)}$ is ramified (i.e., $w_0(\tau \otimes \sigma^{(0)}) \cong \tau \otimes \sigma^{(0)}$, where $w_0$ is the long element in the set of minimal-length double-coset representatives) that there is a unique $s_0 \geq 0$ such that $I(s_0\tilde{\alpha})$ reduces; if $\tau \otimes\sigma^{(0)}$ is not ramified, then $I(s\tilde{\alpha}, \tau \otimes \sigma^{(0)})$ is irreducible for all $s \geq 0$. Further, Theorem 8.1 of \cite{Sha90a} tells us that when $\tau \otimes \sigma^{(0)}$ is generic, $s_0 \in \{0, \frac{1}{2},1\}$. We now make this more explicit for the groups in question.

First, we describe the action of $w_0$ needed for the ramified condition. This is in \cite{MT02} (implicitly) for the cases of classical groups  $SO_{2n+1},Sp_{2n},U_{2n+1},U_{2n}$, \cite{Tad98a} for $GSp_{2n}$, \cite{Kim15} for $GSpin_{2n+1}$, and \cite{Kim16} for $GSpin_{2n}$; \cite[Lemma 4.17]{ACS16} covers a number of families considered below%\footnote{The $c$ seems to be missing in \cite{ACS}.}:
\begin{equation}\label{w_0action}
w_0 (\tau \otimes \sigma^{(0)})=\left\{ \begin{array}{l}
	\check{\tau} \otimes \sigma^{(0)} \mbox{ for }G_n=SO_{2n+1}, Sp_{2n}, U_{2n+1}, U_{2n}, \\
	\check{\tau} \otimes c^m \cdot \sigma^{(0)} \mbox{ for }G_n=SO_{2n}, SO_{2n+2}^{\ast}, \\
	\omega_{\sigma^{(0)}}\check{\tau} \otimes \sigma^{(0)} \mbox{ for }G_n=GSpin_{2n+1}, \\
	\check{\tau} \otimes \omega_{\tau}\sigma^{(0)} \mbox{ for }G_n=GSp_{2n}, GU_{2n+1}, GU_{2n}, \\
	\check{\tau} \otimes \omega_{\tau}(c^m \cdot \sigma^{(0)}) \mbox{ for }G_n=GSO_{2n}, GSO_{2n+2}^{\ast}, \\
	\omega_{\sigma^{(0)}}\check{\tau} \otimes c^m \cdot \sigma^{(0)} \mbox{ for }G_n=GSpin_{2n}, GSpin_{2n+2}^{\ast}.
\end{array} \right. 
\end{equation}
Note that we must have $n>1$ for $SO_{2n}, GSO_{2n}, GSpin_{2n}$. We also remark that the fact that $\sigma^{(0)}$ remains unchanged for most classical groups is what allows the notion of partial cuspidal support in \cite{MT02}.

Next, for $\Pi_M=\Pi \setminus \{\alpha_k\}$, the induced representation translates as follows:

\medskip
\noindent
\begin{table}[H]
\begin{tabular}{l|l}
$G_n$  & $I(s\tilde{\alpha}, \tau \otimes \sigma^{(0)})$ \\ \hline\hline
$SO_{2n+1}$, $SO_{2n+2}^{\ast}$   & $\nu^s \tau \rtimes \sigma^{(0)}$ for $k<n$ \\
$U_{2n+1}$, $GSpin_{2n+1}$ 	& $\nu^{\frac{s}{2}} \tau \rtimes \sigma^{(0)}$ for $k=n$\\ 
$GSpin_{2n+2}^{\ast}$ & \\
\hline
$Sp_{2n}$, $U_{2n}$ & $\nu^s \rho \rtimes \sigma^{(0)}$ \\ \hline
$SO_{2n}$ & $\nu^s \tau \rtimes \sigma^{(0)}$ for $k<n-1$ \\
	& $\nu^{\frac{s}{2}} \tau \rtimes \sigma^{(0)}$ for $k=n-1,n$ (so $\sigma^{(0)}=1 \otimes e$ or $1 \otimes c$) \\ \hline
$GSp_{2n}$, $GU_{2n}$ & $\nu^s \tau \rtimes \nu^{-\frac{ks}{2}}\sigma^{(0)}\cong \nu^{-\frac{ks}{2}}(\nu^s \tau \rtimes \sigma^{(0)})$ \\ \hline
$GSpin_{2n}$ & $\nu^s \tau \rtimes \sigma^{(0)}$ for $k<n-1$  \\
	& $\nu^{\frac{s}{2}} \tau \rtimes \sigma^{(0)}$ for $k=n-1,n$ (so $\sigma^{(0)}=\chi \otimes e$ or $\chi \otimes c$) \\ \hline
$GSO_{2n}$ & $\nu^s \tau \rtimes \nu^{-\frac{ks}{2}}\sigma^{(0)} \cong \nu^{-\frac{ks}{2}}(\nu^s \tau \rtimes \sigma^{(0)})$ for $k<n-1$ \\
	& $\nu^{\frac{s}{2}} \tau \rtimes \nu^{-\frac{ns}{4}}\sigma^{(0)} \cong \nu^{-\frac{ns}{4}}(\nu^{\frac{s}{2}} \tau \times \sigma^{(0)})$ for $k=n-1,n$ \\
&\hfill (so $\sigma^{(0)}=\chi \otimes e$ or $\chi \otimes c$) \\ \hline
$GSO_{2n+2}^{\ast}$, $GU_{2n+1}$ &  $\nu^s \tau \rtimes \nu^{-\frac{ks}{2}}\sigma^{(0)} \cong \nu^{-\frac{ks}{2}}(\nu^s \tau \rtimes \sigma^{(0)})$ for $k<n$ \\
	& $\nu^{\frac{s}{2}} \tau \rtimes \nu^{-\frac{ns}{4}}\sigma^{(0)} \cong \nu^{-\frac{ns}{4}}(\nu^{\frac{s}{2}} \rtimes \sigma^{(0)})$ for $k=n$ \\ \hline
\end{tabular}
\caption{$I(s\tilde{\alpha}, \tau \otimes \sigma^{(0)})$}
\label{tab:I(s)}
\end{table}

\noindent
This is a straightforward calculation; the example of $GSpin_{2n+2}^{\ast}$ is done as part of \S \ref{sect3appendix}; the example of $Sp_4(F)$ may be found in \cite{Sha91}. Many of these are also covered by  \cite[(4.20)]{ACS16}.

If $\tau \otimes \sigma^{(0)}$ is ramified and $\alpha \geq 0$ is the unique nonnegative value for which $\nu^{\alpha}\tau \rtimes \sigma^{(0)}$ is reducible, we say that $(\tau; \sigma^{(0)})$ satisfies (C$\alpha$). We claim $\alpha \in \{0, \frac{1}{2},1\}$. Table~\ref{tab:I(s)} coupled with \cite[Theorem 8.1]{Sha90a} tells us we have (C0), (C$\frac{1}{2}$), or (C1) except possibly for the Siegel parabolic(s) for $G_n=SO_{2n+1}$, $SO_{2n}$, $SO_{2n+2}^{\ast}$, $U_{2n+1}$, $GSpin_{2n+1}$, $GSpin_{2n}$, $GSO_{2n}$, $GSO_{2n+2}^{\ast}$, $GSpin_{2n+2}^{\ast}$, and $GU_{2n+1}$. For these, we must rule out the possibility of (C$\frac{1}{4}$) (corresponding to $s=\frac{1}{2}$ above). For $G_n=SO_{2n+1}$ and $SO_{2n}$, this possibility is eliminated by \cite{Sha92}; \cite[Th\'eor\`eme 3.1]{Moe14} further rules out this possibility for $G_n=SO_{2n+2}^{\ast}$, $GSpin_{2n+1}$, $GSpin_{2n}$, $GSpin_{2n+2}^{\ast}$, and $U_{2n+1}$. Thus it remains to deal with $G_n=GSO_{2n}$, $GSO_{2n+2}^{\ast}$, and $GU_{2n+1}$.

To address these cases, we have the following lemma:

\begin{lem}
Let $\tilde{G}_n=GSO_{2n}(F)$ (resp., $GSO_{2n+2}^{\ast}(F)$, $GU_{2n+1}(F)$) and $G_n=SO_{2n}(F)$ (resp., $SO_{2n+2}^{\ast}(F)$, $U_{2n+1}(F)$). Let $\tau$ is an irreducible supercuspidal representation of $H_n(F)$ and $\chi$ a character of
\[
\tilde{G}_0(F) \cong \left\{ \begin{array}{l}
F^{\times} \mbox{ for }\tilde{G}_0(F)=GSO_0(F), \\
E^{\times} \mbox{ for }\tilde{G}_0(F)=GSO_2^{\ast}(F) \mbox{ or }GU_1(F)
\end{array} \right.
\]
(so $\chi$ means $\chi \otimes e$ or $\chi \otimes c$ for $GSO_0(F)$). Let $\chi_0$ be the irreducible representation of
\[
G_0(F) \cong \left\{ \begin{array}{l}
1 \mbox{ for }G_0(F)=SO_0(F), \\
N_1(E/F) \mbox{ for }G_0(F)=SO_2^{\ast}(F) \mbox{ or }U_1(F).
\end{array} \right.
\]
given by restriction of $\chi$ (so $\chi_0=1 \otimes e$ or $1 \otimes c$ for $SO_0(F)$; trivial otherwise as similitude is 1). 
Then $\nu^s \tau \rtimes \chi$ reducible implies $\nu^s \tau \rtimes \chi_0$ is reducible.
\end{lem}

\begin{proof}
Let $\tilde{\mathcal A}_0=1 \times \tilde{G}_0(F) \subset H_n(F) \times \tilde{G}_0(F)$ in the Siegel parabolic for $\tilde{G}_n(F)$. We have $\tilde{G}_n=\tilde{\mathcal A}_0 G_n$. We remark that the resulting decomposition is unique for $GSO_{2n}(F)$.

Consider the map
\[
\begin{array}{rrl}
{\mathcal E}:&  V_{\nu^s \rho \rtimes \chi_0} \longrightarrow & V_{\nu^s \rho \rtimes \chi} \\
& f \longmapsto & \tilde{f},
\end{array}
\]
where $\tilde{f}$ is defined by extending $f$ as follows:
\[
f(\tilde{g})=f(\tilde{a}_0g)=\tilde{\delta}^{\frac{1}{2}}(\tilde{a}_0)\chi(\tilde{a}_0)f(g).
\]
It is a straightforward matter to show that ${\mathcal E}$ is well-defined. Further, it is also not difficult to show that ${\mathcal E}$ is bijective, with inverse
\[
{\mathcal E}^{-1}: \tilde{f} \longmapsto \left.\tilde{f}\right|_{G_n}.
\]
Noting that $G_n \subset \tilde{G}_n$, one can show that ${\mathcal E}$ is $G_n$-equivariant (or equivalently, that ${\mathcal E}^{-1}$ is $G_n$-equivariant). One then has $\nu^s \rho \rtimes \chi_0 \cong \left.\left(\nu^s \rho \rtimes \chi \right)\right|_{G_n}$. Therefore, if $\nu^s \rho \rtimes \chi$ is reducible, so is $\nu^s \rho \rtimes \chi_0$. \end{proof}

In particular, in light of Lemma~\ref{twisting}, this shows that if $(\tau, \chi)$ is (C$\alpha$) for $GSO_{2n+2}^{\ast}$ or $GU_{2n+1}$ (resp., $(\tau, \chi \otimes e)$ or $(\tau, \chi \otimes e)$ for $GSO_{2n}$), then we also have $(\tau, \chi_0)$ is (C$\alpha$) for $SO_{2n+2}^{\ast}$ or $U_{2n+1}$ (resp., $(\tau, 1 \otimes e)$ or $(\tau, 1 \otimes c)$ for $SO_{2n}$). Since it is known that we do not have (C$\frac{1}{4}$) for the classical cases, the corresponding result is immediate for their similitude counterparts. %Note: to have the generality needed for $U_{2n+1}$, will have to allow $\tau \rtimes \chi_0$ ($\chi_0$ a character of $N_1(E/F)$), noting the $\chi_0$ is really a special case of twisting by $\det$ rather than the similitude character. By comments in your similitude notes, there is a character $\chi$ of $E^{\times}$ with $\chi|_{N_1(E/F)}=\chi_0$; it seems likely that we would need any two such characters to differ by a twist of the similitude character.

There are also certain circumstances in which $\tau \otimes \sigma^{(0)}$ is not ramified but which can still contribute to square-integrable representations. These occur when $\tau$ is equivalent to the first factor of $w_0 (\tau \otimes \sigma^{(0)})$ in (\ref{w_0action}) but $\sigma^{(0)}$ is not equivalent to the second factor. In particular, we say that $(\tau; \sigma^{(0)})$ satisfies (CN) under the following conditions:

\ 

\noindent
\begin{table}[H]
\begin{tabular}{r|c}
$G_n$ \ \ \ & (CN) \\ \hline
$SO_{2n+1}$, $Sp_{2n}$, $U_{2n+1}, U_{2n}, GSpin_{2n+1}$ & none \\
$SO_{2n}$, $SO_{2n+2}^{\ast}$ & $\check{\tau} \cong \tau$ but $c^m \cdot \sigma^{(0)} \not\cong \sigma^{(0)}$ \\
$GSp_{2n}, GU_{2n+1}, GU_{2n}$ & $\check{\tau} \cong \tau$ but $\omega_{\tau}\sigma^{(0)} \not\cong \sigma^{(0)}$ \\
$GSO_{2n}, GSO_{2n+2}^{\ast}$ & $\check{\tau} \cong \tau$ but $\omega_{\tau}(c^m \cdot \sigma^{(0)}) \not\cong \sigma^{(0)}$ \\
$GSpin_{2n}$, $GSpin_{2n+2}^{\ast}$ & $\omega_{\sigma^{(0)}}\check{\tau} \cong \tau$ but $c^m \cdot \sigma^{(0)} \not\cong \sigma^{(0)}$ \\
\end{tabular}
\caption{(CN) conditions}
\label{tab:(CN)}
\end{table}

\noindent
%Note that $\check{\tau} \cong \tau$ implies $\omega_{\tau}\cdot (\omega_{\tau} \circ \eta)=1$ for $\eta$ the nontrivial element of $Gal(E/F)$; $\sigma^{(0)}$ unitary implies $|\omega_{\sigma^{(0)}}|=1$.
%Also, we explicitly note that for $SO_{2n}$ (resp., $GSO_{2n}$, $GSpin_{2n}$), a character $\chi \otimes (1 \otimes e)$ (resp. $\chi \otimes (\chi_0 \otimes e)$) satisfies (CN) if $\chi^2=1$ (resp., $\chi^2=1$ for $GSO_{2n}$, $\chi^2=\chi_0$ for $GSpin_{2n}$).  e.g., for $GSpin(2n)$, the usual irreducibility of $\rho \rtimes \sigma^{(0)}$--which would produce a Jacquet module $\rho \otimes \sigma^{(0)}+\omega_{\sigma^{(0)}}\tilde{\rho} \otimes c \cdot \sigma^{(0)}$--is now the irreducibility of $\chi \rtimes (\chi_0 \otimes e)=\chi \otimes \chi_0=\chi_0\chi^{-1} \rtimes (\chi_0 \otimes c)$. In particular, we can still view the situation as providing a connection from $\chi \rtimes (\chi_0 \otimes e)$ to $\chi_0\chi^{-1} \rtimes (\chi_0 \otimes c)$, but via equality rather than the usual irreducibility connection.

\begin{note}
Note that $\nu^x \tau \rtimes \sigma^{(0)}$ is reducible if and only if $\nu^{-x}\tau \rtimes \sigma^{(0)}$ is reducible--for $GSp_{2n}(F)$, this is \cite[Remark 2.2]{Tad98b}. The same argument applies to $GSO_{2n}(F)$, $GSO_{2n+2}^{(\varepsilon)}(F)$, $GU_{2n+1}(F)$, $GU_{2n}(F)$ (using (\ref{w_0action}) and Lemma~\ref{twisting}); for the remaining groups, it follows from (\ref{w_0action}). In particular, $-\alpha$ is the unique nonpositive value where reducibility occurs.
\end{note}

Our last task is to set up a $\mu^{\ast}$ structure like that for classical groups in \cite{Tad95} (also \cite{MT02}). In order to have a uniform presentation, we define a variation on Tadi\'c's $M^{\ast}$ below. This variation counts sign changes, as those are important for many of the groups under consideration.
Thus we formally define $N^{\ast}: R \longrightarrow R \otimes R \otimes R \otimes {\mathbb Z}(C)$ by
\[
N^{\ast}=(\check{\ } \otimes m^{\ast})_D \circ s \circ m^{\ast},
\]
where $C = \{e,c\}$, $s: \pi_1 \otimes \pi_2 \mapsto \pi_2 \otimes \pi_1$, and
\[
(\,\check{\ } \otimes m^{\ast})_D(\pi_1 \otimes \pi_2)=\left\{ \begin{array}{l}
	\check{\pi}_1 \otimes m^{\ast}(\pi_2) \otimes e\\
	\mbox{ if }\pi_1 \mbox{ is a representation of }H_{n_1}(F)
	\mbox{ with }n_1\mbox{ even,} \\
	\check{\pi}_1 \otimes m^{\ast}(\pi_2) \otimes c\\
	\mbox{ if }\pi_1 \mbox{ is a representation of }H_{n_1}(F)
	\mbox{ with }n_1\mbox{ odd.}
\end{array}\right.
\]

We first consider the groups $G_n=SO_{2n+1}$, $Sp_{2n}$, $U_{2n+1}$, $U_{2n}$, $GSpin_{2n+1}$, $GSp_{2n}$, $GU_{2n+1}$, and $GU_{2n}$. Here we define $\tilde{\rtimes}$ by
\[
(\rho_1 \otimes \rho_2 \otimes \rho_3 \otimes d) \tilde{\rtimes} (\rho \otimes \sigma) =\left\{ \begin{array}{l}
(\rho_1 \times \rho_2 \times \rho) \otimes (\rho_3 \rtimes \sigma)\\
\mbox{ for }G_n=SO_{2n+1}, Sp_{2n}, U_{2n+1}, U_{2n}, \\
(\omega_{\sigma}\rho_1 \times \rho_2 \times \rho) \otimes (\rho_3 \rtimes \sigma) \\
\mbox{ for }G_n=GSpin_{2n+1}, \\
(\rho_1 \times \rho_2 \times \rho) \otimes (\rho_3 \rtimes \omega_{\check{\rho}_1} \sigma) \\
\mbox{ for }G_n=GSp_{2n}, GU_{2n+1}, GU_{2n}. \\
%(\rho_1 \times \rho_2 \times \rho) \otimes (\rho_3 \rtimes \omega'_{\rho_1} \sigma) \mbox{ for }G_n=GU_{2n+1}. \\
\end{array} \right.
\]
Note that the action of $C$ is trivial in this case.
With $\mu^{\ast}$ as defined in (\ref{mu-ast def 1}), we then have
\[
\mu^{\ast}(\lambda \rtimes \pi)=N^{\ast}(\lambda) \tilde{\rtimes}\mu^{\ast}(\pi),
\]
an immediate consequence of \cite{Tad95} (for $G_n=SO_{2n+1}$, $Sp_{2n}$, and $GSp_{2n}$), \cite{MT02} (for $G_n=U_{2n+1}$ and $U_{2n}$), \cite{Kim15} (for $G_n=GSpin_{2n+1}$), and \cite{KM19} (for $G_n=GU_{2n}$); the case $G_n=GU_{2n+1}$ is similar.

For $G_n=SO_{2n+2}^{\ast}$, $GSO_{2n+2}^{\ast}$, or $GSpin_{2n+2}^{\ast}$, we define $\tilde{\rtimes}$ by%\footnote{Again, for $GSpin_{2n+2}^{\ast}$, this is still conjectural as with (\ref{w_0action}).}
\[
(\rho_1 \otimes \rho_2 \otimes \rho_3 \otimes d) \tilde{\rtimes} (\rho \otimes \sigma)=
\left\{ \begin{array}{l}
(\rho_1 \times \rho_2 \times \rho) \otimes d(\rho_3 \rtimes \sigma)\\
\mbox{ for }G_n=SO^{\ast}_{2n+2}, \\
(\rho_1 \times \rho_2 \times \rho) \otimes d(\rho_3 \rtimes \omega_{\tilde{\rho}_1}\sigma) \\
\mbox{ for }G_n=GSO^{\ast}_{2n+2}, \\
(\omega_{\sigma}\rho_1 \times \rho_2 \times \rho) \otimes d(\rho_3 \rtimes \sigma)\\
\mbox{ for }G_n=GSpin^{\ast}_{2n+2}.
\end{array} \right.
\]
Then, with $\mu^{\ast}$ as defined in (\ref{mu-ast def 1}),
\[
\mu^{\ast}(\lambda \rtimes \pi)=N^{\ast}(\lambda) \tilde{\rtimes} \mu^{\ast}(\pi).
\]
We outline the changes needed to the proof of \cite[Theorem 5.2]{Tad95} to deal with $GSpin_{2n+2}^{\ast}$ in \S \ref{sect3appendix}; the other cases are similar. Also note that for $SO_{2n+2}^{\ast}$, the inclusion of the action of $C$ represents a correction to \cite{MT02}.

\begin{rem}
We can also use this to produce $\mu^{\ast}$ structure for $SO_{2n+2}^{\ast}$ more like that of \cite{Tad95}. In particular, we define $M_D^{\ast}: R \longrightarrow R \otimes R \otimes {\mathbb Z}[C]$ by $M_D^{\ast}=(m \otimes 1)_D \circ N^{\ast}$, where
\[
(m \otimes 1)_D(\lambda_1 \otimes \lambda_2 \otimes \lambda_3 \otimes d)=(\lambda_1 \times \lambda_2) \otimes \lambda_3 \otimes d.
\]
Then,
\[
\mu^{\ast}(\tau \otimes \theta)=M_D^{\ast}(\tau) \rtimes \mu^{\ast}(\theta),
\]
where
\[
(\tau_1 \otimes \tau_2 \otimes d) \rtimes (\tau \otimes \theta)=(\tau_1 \times \tau) \otimes (\tau_2 \rtimes d \cdot \theta)
\]
(noting that $\tau_2 \rtimes d \cdot \theta \cong d \cdot(\tau_2 \rtimes \theta)$).

\end{rem}

The structures for $SO_{2n}$, $GSpin_{2n}$ are done in \cite{JL14} and \cite{Kim16}, respectively, and included in the summary below (though the presentation below is a minor variation on that in  \cite{Kim16}).
The structure for $GSO_{2n}$ (resp., $GSpin_{2n}$) is essentially a combination of that for $GSp_{2n}$ (resp., $GSpin_{2n+1}$) above and that for $SO_{2n}$.  Note that the Weyl groups here allow only even sign changes; again $N^{\ast}$ accounts for this.
Let $G_n$ be $GSO_{2n}$ or $GSpin_{2n}$. We can now construct a $\mu^{\ast}$ structure which closely resembles that for the other classical groups. To this end, we set
\[
\Omega_k=\left\{ \begin{array}{l}
	\Pi \setminus\{\alpha_k\} \mbox{ if }k \leq n-2, \\
	\Pi \setminus \{\alpha_{n-1},\alpha_n \} \mbox{ if }k=n-1, \\
	\Pi \setminus \{\alpha_n\} \mbox{ if }k=n.
\end{array} \right.
\]
Note $c\Omega_n=\Pi \setminus\{\alpha_{n-1}\}$. For $\pi$ an irreducible representation of $G_n(F)$ with $n \geq 2$, and $0 \leq k \leq n$, write $r_{M_{\Omega_k},G}(\pi)= \sum_{i \in I_k} \pi_{i,k} \otimes \theta_{i,k}$ and $r_{M_{c\Omega_n},G}(\pi)=\sum_{j \in J} \pi_j \otimes (\chi_j \otimes c)$. We then define
\begin{equation}\label{m-ast def2}
\mu^{\ast}(\pi)=\sum_{k=0}^n \sum_{i \in I_k} \pi_{i,k} \otimes \theta_{i,k}+\sum_{j \in J} \pi_j \otimes (\chi_j \otimes c)
\end{equation}
for $G_n=GSO_{2n}$ or $GSpin_{2n}$. This also applies to $G_n=SO_{2n}$ if one replaces $\chi_j$ by 1 in the second summand and in $r_{M_{c\Omega_n},G}(\pi)$ above.
For $n=0$, we have only $\chi \otimes e$ and $\chi \otimes c$ for $GSO_{2n}$ or $GSpin_{2n}$; for $d=e \mbox{ or }c$, we define
\[
\mu^{\ast}(\chi \otimes d)=1 \otimes (\chi \otimes d).
\]
The corresponding definition for $SO_{2n}$ with $n=0$ is $\mu^{\ast}(1 \otimes d)=1 \otimes (1 \otimes d)$.
For $n=1$, an irreducible representation of $GSO_2(F)$ has the form $\chi \rtimes (\chi' \otimes e)=\chi^{-1} \rtimes (\chi'\chi \otimes c)$ for $\chi$ a (quasi)character of $F^{\times}$ (noting that under $GSO_2(F) \cong F^{\times} \times F^{\times}$, this corresponds to the character $\chi \otimes \chi'$), and we set
\[
\mu^{\ast}(\chi \rtimes (\chi' \otimes e))=1 \otimes (\chi \rtimes (\chi' \otimes e))+ \chi \otimes (\chi' \otimes e)+\chi^{-1} \otimes (\chi\chi' \otimes c).
\]
The situation for $GSpin_2$ is similar except that $\chi \rtimes (\chi' \otimes e)=\chi'\chi^{-1} \rtimes (\chi' \otimes c)$. Thus, we take
\[
\mu^{\ast}(\chi \rtimes (\chi' \otimes e))=1 \otimes (\chi \rtimes (\chi' \otimes e))+\chi \otimes (\chi' \otimes e)+\chi'\chi^{-1} \otimes (\chi' \otimes c).
\]
For $SO_2$, we have $\chi \otimes (1 \otimes e)=\chi^{-1} \otimes (1 \otimes c)$ and take
\[
\mu^{\ast}(\chi \rtimes (1 \otimes e))=1 \otimes (\chi \rtimes (1 \otimes e))+\chi \otimes (1 \otimes e)+\chi^{-1} \otimes (1 \otimes c).
\]
In any of these cases, we linearly extend $\mu^{\ast}$ to a map $\mu^{\ast}: R[D] \longrightarrow R \otimes R[D]$.

We  take $\tilde{\rtimes}$ defined by
\[
(\rho_1 \otimes \rho_2 \otimes \rho_3 \otimes d) \tilde{\rtimes} (\rho \otimes \sigma)=
\left\{ \begin{array}{l}
(\rho_1 \times \rho_2 \times \rho) \otimes d(\rho_3 \rtimes \sigma)\\
\mbox{ for }G_n=SO_{2n}, \\
(\rho_1 \times \rho_2 \times \rho) \otimes d(\rho_3 \rtimes \omega_{\tilde{\rho}_1} \sigma)\\
\mbox{ for }G_n=GSO_{2n}, \\
(\omega_{\sigma}\rho_1 \times \rho_2 \times \rho) \otimes d(\rho_3 \rtimes \sigma)\\
\mbox{ for }G_n=GSpin_{2n}.
\end{array} \right.
\]
Then,
\[
\mu^{\ast}(\lambda \rtimes \pi)=N^{\ast}(\lambda) \tilde{\rtimes} \mu^{\ast}(\pi).
\]
We note that this follows from \cite{Kim16} for $G_n=GSpin_{2n}$; for $G_n=GSO_{2n}$, the argument is similar to that for $SO_{2n}$ in \cite{JL14}.

\subsection{Structure of  \texorpdfstring{$GSpin_{2n+2}^{\ast}(F)$}{}}\label{sect3appendix}

The purpose of this subsection is to establish certain key properties of $GSpin_{2n+2}^{\ast}$. To this end, let $E$ be a quadratic extension of $F$; write $E=F(\sqrt{\varepsilon})$. In particular, the goal here is to understand the $F$-data, maximal non-split torus, and Weyl group action and apply them to obtain the $\mu^{\ast}$ structure, description of characters, and calculation of $I(s\tilde{\alpha}, \tau \otimes \sigma^{(0)})$.

Given an algebraic group $G$ over $F$, the set of isomorphism classes of $F$-forms is in bijection with $H^1(Gal(\bar{F}/F), Aut(G))$ (see \cite[Chapter III.1]{Ser97}), where $Aut(G)$ is the group of $\bar{F}$-automorphisms of $G$. Then, the $F$-rational points of an $F$-form $H$ of $G$ can be obtained as follows:
$$H(F)= \{x \in G(\bar{F}) : a_s(s(x))=x \text{ for all }s \in Gal(\bar{F}/F) \},$$
where $a_s$ is a $1$-cocycle corresponding to $H$. 
% From the exact sequence
% $1 \rightarrow Inn(G) \rightarrow Aut(G) \rightarrow Out(G) \rightarrow 1$,
% depending on which values $a_s$ lands, we name $H$ as an inner form, quasi-split form.

Let $\sigma$ be the nontrivial element of $Gal(E/F)$ and $c$ be the outer conjugation corresponding to reflection of the Dynkin diagram of $GSpin_{2n+2}$. Let $a_\sigma=c$. Then, we have the following concrete realization of the  quasi-split $GSpin^{(\varepsilon)}_{2n+2}(F)$: 
\[
GSpin^{(\varepsilon)}_{2n+2}(F)=\{X \in GSpin_{2n+2}(E) \,|\, \sigma(X)=c \cdot X\}.
\]
We thank Kwangho Choiy for helpful discussions on this realization.

We start by looking at tori. In $GSpin_{2n+2}(E)$, set
\[
d(a_1, \dots, a_n, a_{n+1},a_0)=\check{e}_1(a_1) \dots \check{e}_n(a_n) \check{e}_{n+1}(a_{n+1}) \check{e}_0(a_0),
\]
as above. Then, the maximal split torus in $GSpin_{2n+2}(E)$ is
\[
\{ d(a_1, \dots, a_n, a_{n+1},a_0) \,|\, a_i \in E^{\times} \mbox{ for all }i=0,1, \dots, n+1\}
\]
To lie in the maximal (non-split) torus in $GSpin^{(\varepsilon)}_{2n+2}(F)$, we then require 
\begin{align*}
   d(\sigma(a_1), \dots, \sigma(a_n), \sigma(a_{n+1}), \sigma(a_0))&=c \cdot d(a_1, \dots, a_n, a_{n+1},a_0)\\
   &=d(a_1, \dots, a_n, a_{n+1}^{-1},a_{n+1}a_0), 
\end{align*}
noting the action of $c$ is given by \cite[Lemma 4.5]{HS16}. This then tells us
\[
\begin{array}{c}
\sigma(a_1)=a_1, \ \dots, \ 
\sigma(a_n)=a_n,\  \sigma(a_{n+1})=a_{n+1}^{-1}, \ \sigma(a_0)=a_{n+1}a_0 \\
\Downarrow \\
a_1, \dots, a_n \in F^{\times},\ a_{n+1}=\sigma(a_0)a_0^{-1}, \ a_0 \in E^{\times}.
\end{array}
\]
Thus, the maximal non-split torus in $GSpin^{(\varepsilon)}_{2n+2}(F)$ is
\[
T=\{ d(a_1, \dots, a_n, \sigma(a_0)/a_0,a_0) \,|\, a_i \in F^{\times} \mbox{ for }i=1, \dots n, \ a_0 \in E^{\times}\}
\]
For $a_0 \in F^{\times}$, we have $\sigma(a_0)/a_0=1$, so the maximal split torus in $GSpin^{(\varepsilon)}_{2n+2}(F)$ is
\[
A=\{ d(a_1, \dots, a_n, 1,a_0) \,|\, a_i \in F^{\times} \mbox{ for }i=0, \dots, n\}
\]
(elements written as $d(a_1, \dots, a_n, a_0)$ earlier).

For the $F$-roots for $GSpin^{(\varepsilon)}_{2n+2}(F)$, we restrict the roots for $GSpin_{2n+2}(E)$ to $A$. To make the results clearer, let
\[
E_i=e_i|_A \mbox{ for }i=0, \dots, n.
\]
Of course, a $\mathbb Z$-basis for the rational characters is
\[
X_F=\{E_1, \dots, E_n, E_0\}.
\]
Then,
\[
(e_1-e_2)|_A=E_1-E_2, \dots, (e_{n-1}-e_{n})|_A=E_{n-1}-E_{n}
\]
and
\[
(e_{n}-e_{n+1})|_A=(e_n+e_{n+1})|_A=E_n.
\]
Thus, the simple $F$-roots are
\[
\Pi_F=\{E_1-E_2, \dots, E_{n-1}-E_n, E_n\}.
\]

A $\mathbb Z$-basis for the $F$-cocharacters for $GSpin^{(\varepsilon)}_{2n+2}(F)$ is
\[
\check{X}_F=\{\check{E}_1, \dots, \check{E}_n, \check{E}_0\},
\]
where $\check{E}_i=\check{e}_i|_{F^{\times}}$ viewed as a cocharacter into $S^{\ast}$. This is easily seen to satisfy $\check{E}_j(E_i)=\delta_{i,j}$ (calculated via $t^{\check{E}_j(E_i)}=E_i(\check{E}_j(t)$). It remains to determine $\check{\Pi}_F$.

At this point, we note that philosophically, we can move between the relative Weyl group $W_F^{\ast}$ (see \cite[Section 15.3]{Spr98}), a $W_F^{\ast}$-invariant inner product, and the F-coroots. That is, given one, the other two may be determined. In the discussion below, we start with the relative Weyl group. From this, we then construct a $W_F^{\ast}$-invariant product and use that to determine the coroots.

We first observe that $W^{\ast}_F$ consists of the $c$-invariant elements of $W$. To be more precise, recall that $c \cdot \check{e}_k=\left\{ \begin{array}{l} \check{e}_k \mbox{ for }0 \leq k \leq n, \\
\check{e}_0-\check{e}_{n+1} \mbox{ for }k=n+1 \end{array} \right.$ (see \cite[Lemma 4.5]{HS16}). It is then a fairly straightforward matter to check that $w \in W$ is $c$-invariant if and only if $w \cdot \check{e}_{n+1}=\check{e}_{n+1}$ or $\check{e}_0-\check{e}_{n+1}$. This immediately gives a well-defined action of a $c$-invariant $w$ on $\check{E}_0, \dots, \check{E}_n$ (via $w \cdot \check{E}_i=(w \cdot \check{e}_i)|_{F^{\times}}$).
%Suppose $w \in W$ with $\bar{w}$ a representative in $GSpin(2n+2,F)$, noting that representatives may be chosen in $GSpin(2n+2,F)$ as it is split over $F$. Then $\bar{w} \in GSpin_{E/F}(2n+2,F)$ if
%\[
%\sigma(\bar{w})=c \cdot \bar{w} \Rightarrow \bar{w}= c \cdot \bar{w},
%\]
%giving $c$-invariance. Thus the $c$-invariant elements of $W$ lie in $W_F^{\ast}$.
On the other hand, if $w \in W_F^{\ast}$, we may lift it to a corresponding element $w' \in W$ by defining the action of $w'$ on $\check{e}_i$, $i=0, 1, \dots, n+1$. There is an obvious action of $w$ on $\check{e}_i$, $i=0,1, \dots, n$ (lifting from $\check{E}_i=\check{e}_i|_{F^{\times}}$); we take $w' \cdot \check{e}_i=w \cdot \check{e}_i$ for $i=0,1, \dots, n$. For $i=n+1$, we define
\[
w' \cdot \check{e}_{n+1}=\left\{ \begin{array}{l}
\check{e}_{n+1} \mbox{ if $w$ has an even number of sign changes}, \\
\check{e}_0-\check{e}_{n+1} \mbox{ if $w$ has an odd number of sign changes}.
\end{array} \right.
\]
So defined, $w'$ gives an element of $W$ whose action on $A$ matches that of $w$ and which is clearly $c$-invariant. We remark that if $s_i$ denotes the simple reflection corresponding to the simple root $\alpha_i$ for $GSpin_{2n+2}$ (usual ordering), then $s_1, \dots, s_{n-1}, s_ns_{n+1}=s_{n+1}s_n$ are $c$-invariant and generate $W_F^{\ast}$. Further, these give rise to the simple reflections corresponding to $\Pi_F$ as may be seen by their actions via restriction.

We now turn to $W$-invariant forms, starting with $GSpin_{2n+2}(E)$. We first note that it is a straightforward calculation to check that $\langle \cdot, \cdot \rangle$ defined by
\[
\langle e_i, e_j \rangle=\left\{ \begin{array}{l}
\hspace{.15in}\alpha \mbox{ if } j=i>0, \\
\hspace{.15in} 0 \mbox{ if } j\not=i \mbox{ with }i,j>0, \\
\hspace{-.09in}-\frac{1}{2}\alpha \mbox{ if }i=0 \mbox{ or }j=0 \mbox{ but not both}, \\
\hspace{.15in} \beta \mbox{ if }i=j=0,
\end{array} \right.
\]
is invariant under simple reflections, hence a $W$-invariant, symmetric bilinear form; positive definite if $\alpha, \beta>0$. It is also $c$-invariant. Further, any positive definite, symmetric, $W$-invariant bilinear form is $\langle \cdot, \cdot \rangle$ for some $\alpha, \beta>0$. In what follows, we take $\alpha=\beta=1$ for convenience. 

We now define $\langle \cdot, \cdot \rangle^{\ast}$ by
\[
\langle E_i , E_j \rangle^{\ast}=\langle e_i, e_j \rangle \mbox{ for } 0 \leq i,j \leq n,
\]
to get a symmetric, positive definite bilinear form.
To show the $W_F^{\ast}$-invariance of $\langle \cdot, \cdot \rangle^{\ast}$, it suffices to check it for $w=s_1, \dots, s_{n-1}, s_ns_{n+1}=s_{n+1}s_n$. For $k<n$, we note that $s_k$ permutes the $E_i$'s, so $s_k \cdot E_i= s_k \cdot E_j \Leftrightarrow i=j$ and
\[
\langle s_k \cdot E_i, s_k \cdot E_j \rangle^{\ast}=\delta_{i,j}=\langle E_i, E_j \rangle^{\ast}.
\]
For $s_ns_{n+1}$, we note that
\[
s_ns_{n+1} \cdot E_i=\left\{ \begin{array}{l}
	E_0+E_n \mbox{ if } i=0, \\
	E_i \mbox{ if } 1 \leq i \leq n-1, \\
	-E_n \mbox{ if }i=n,
\end{array} \right.
\]
from which it is a straightforward calculation to verify $\langle s_ns_{n+1}\cdot E_i, s_ns_{n+1} \cdot E_j \rangle^{\ast}=\langle E_i, E_j \rangle^{\ast}$.

We next determine the simple dual $F$-roots. For an $F$-root $\alpha$, we want
\[
\check{\alpha}(x_1E_1+ \dots +x_nE_n+x_0E_0)=2\frac{\langle \alpha, x_1E_1+ \dots +x_nE_n+x_0E_0 \rangle^{\ast}}{\langle \alpha, \alpha \rangle^{\ast}}.
\]
For $\alpha_i=E_i-E_{i+1}$, $1 \leq i \leq n-1$, we get
\[
\check{\alpha}_i(x_1E_1+ \dots +x_nE_n+x_0E_0)=x_i-x_{i+1};
\]
for $\alpha_n=E_n$,
\[
\check{\alpha}_n(x_1E_1+ \dots +x_nE_n+x_0E_0) =2x_n-x_0.
\]
We then have
\[
\check{\Pi}_F=\{ \check{E}_1-\check{E}_2, \dots, \check{E}_{n-1}-\check{E}_n, 2\check{E}_n-\check{E}_0 \}.
\]

While the action of $W_F^{\ast}$ on $A$ is described above, we also need the action on $T$. We now take up this question. For $1 \leq i \leq n-1$, we have
\begin{align*}
 &\,\, s_i \cdot d(a_1, \dots, a_{i-1}, a_i, a_{i+1}, a_{i+2}, \dots, a_n, \sigma(a_0)/a_0,a_0)\\
 &=d(a_1, \dots, a_{i-1}, a_{i+1}, a_i, a_{i+2}, \dots, a_n, \sigma(a_0)/a_0,a_0), 
\end{align*}
and
\begin{align*}
  &\,\, s_ns_{n+1} \cdot d(a_1, \dots, a_{n-1}, a_n, \sigma(a_0)/a_0,a_0)\\
  &=d(a_1, \dots, a_{n-1}, a_n^{-1}, a_0/\sigma(a_0),\sigma(a_0)a_n).  
\end{align*}
In particular, note that the action of $s_ns_{n+1}$ includes a Galois conjugation of $a_0$.

We now take a moment to discuss the proof of the formula $\mu^{\ast}(\lambda \rtimes \pi)=N^{\ast}(\lambda) \tilde{\rtimes}\mu^{\ast}(\pi)$ from above. The proof parallels that of Theorem 5.2 of \cite{Tad95} for $GSp_{2n}$, formally calculating $N^{\ast}(\lambda) \tilde{\rtimes} \mu^{\ast}(\pi)$ from the definitions above and comparing it to $\mu^{\ast}(\lambda \rtimes \pi)$ calculated via \cite[Lemma 2.12 (Geometrical Lemma)]{BZ77}. The calculation of $N^{\ast}(\lambda) \tilde{\rtimes} \mu^{\ast}(\pi)$ follows that of \cite[Theorem 5.2]{Tad95} very closely. Most of the calculation is, in fact, just a process of re-indexing summations, which does not depend on the underlying group. The difference in the definitions of $\tilde{\rtimes}$ is what results in the central character of the $GSpin$ representation being attached to the inverted $GL$ representation rather than the other way around. (Technical note: \cite {Tad95} does the contragredient for the inverted $GL$ representations as part of $\tilde{\rtimes}$ whereas here we include it in the definition of $N^{\ast}$, but this is not significant.)

The calculation of $\mu^{\ast}(\lambda \rtimes \pi)$ from \cite{BZ77} is done using the Weyl group double-coset representatives calculated in \cite[Section 4]{Tad95}. These calculations depend only on the Weyl group, not the underlying group, so we have the same representatives for $GSpin_{2n+2}^{\ast}$. The difference here arises in the action of these double-coset representatives, done in \cite[Lemma 5.1]{Tad95} and the discussion immediately preceding it. In particular, using the superscript to denote the rank of the underlying group and following the notation of \cite{Tad95},
\[
\begin{array}{l}
q_n(d,k)_{i_1,i_2}^{-1}\left(\pi_1^{(j)} \otimes \pi_2^{(i_1-d-k)} \otimes \pi_3^{(d)} \otimes \pi_4^{(i_2-d-k)} \otimes \sigma^{(n-i_1-i_2+d+k)}\right) \\
=\pi_1^{(j)} \otimes \pi_4^{(i_2-d-k)} \otimes\omega (\pi_3^{(d)})^{\vee} \otimes \pi_2^{(i_1-d-k)} \otimes \sigma^{(n-i_1-i_2+d+k)},
\end{array}
\]
where $\omega=\omega_{\sigma^{(n-i_1-i_2+d+k)}}$. With this, the calculation of $\mu^{\ast}(\lambda \rtimes \pi)$ then matches that of $N^{\ast} \tilde{\rtimes }\mu^{\ast}(\pi)$, as needed.

We next discuss characters of $GSpin_{2n+2}^{\ast}(F)$. The restriction of a rational character of $GSpin_{2n+2}^{\ast}(F)$ to $A$ produces a Weyl-invariant character of $A$. Write
\[
\lambda|_A=c_1E_1+ \dots +c_nE_n+c_0E_0, \hspace{,5in} c_1, \dots, c_n,c_0 \in {\mathbb Z}.
\]
Then, for $i<n$,
\[
\begin{array}{c}
s_i \cdot \lambda=c_1E_1+ \dots c_{i-1}E_{i-1} +c_{i+1}E_i+c_iE_{i+1}+c_{i+2}E_{i+2}+\dots+c_nE_n+c_0E_0 \\
\Downarrow \\
c_i=c_{i+1}.
\end{array}
\]
For $i=n$,
\[
\begin{array}{c}
s_n \cdot \lambda =c_1E_1+ \dots +c_{n-1}E_{n-1}+(c_0-c_n)E_n+c_0E_0 \\
\Downarrow \\
c_n=c_0-c_n.
\end{array}
\]
Combining these, we get $\lambda=cE_1+ \dots cE_{n-1}+2cE_n+cE_0$. We take $c=1$ for our basic character, so $\xi_n|_A=E_1+\dots+E_n+2E_0$. %which corresponds to the restriction of the similitude character from $GSpin_{2n+2}(E)$.

We now determine $I(s\tilde{\alpha}, \tau \otimes \sigma^{(0)})$ for $GSpin_{2n+2}^{\ast}$. By \cite[Section 1.2]{Sha10}, we may use the $F$-roots to calculate $\langle \rho_U, \alpha \rangle$. By \cite[(1.2.6)]{Sha10}, we have
\[
\langle \rho_U, \alpha \rangle=2\frac{(\rho_U, \alpha)}{(\alpha, \alpha)},
\]
where $(\cdot, \cdot)$ denotes a $W$-invariant inner product. Note that this is also $ \check{\alpha}(\rho_U)$, which can be used with the dual roots above.

We first look at $\rho_{U_{min}}$ for the Borel subgroup. The positive $F$-roots are $E_i-E_j$ for $i<j$ (1-dimensional root space), $E_i+E_j$ for $i<j$ (1-dimensional root space), $E_i$ (2-dimensional root space--corresponds to $u_{e_i-e_{n+1}}(x)u_{e_i+e_{n+1}}(\bar{x})$, $x \in E$). Then,
\[
\rho_{U_{min}}=\frac{1}{2}\left(
\sum_{i=1}^n\sum_{j=i+1}^n 1[(E_i-E_j)+(E_i+E_j)]+\sum_{i=1}^n 2 [E_i]\right)=\sum_{i=1}^n (n-i+1)E_i,
\]
noting the coefficients of 1 or 2 depend on the dimension of the root space. 
For the standard parabolic subgroup $P$ having $M=GL_k(F) \times GSpin_{2(n-k)+2}^{\ast}(F)$, we have
\[
\begin{array}{rl}
\rho_U
&= \rho_{U_{min}}-\rho_{U_{M,min}}\\ &=\displaystyle{\sum_{i=1}^n (n-i+1)E_i-\left[\frac{1}{2}\sum_{i=1}^k\sum_{j=i+1}^k(E_i-E_j)+\sum_{i=k+1}^n(n-i+1)E_i \right]}\\ 
&=\displaystyle{\sum_{i=1}^k \left(n+\frac{-k+1}{2}\right)E_i}
\end{array}
\]
If we observe that
\[
\check{\alpha}=\left\{ \begin{array}{l} \check{E}_k-\check{E}_{k+1} \mbox{ if }k<n, \\ 2\check{E}_n-\check{E}_0 \mbox{ if }k=n, \end{array} \right.
\]
we immediately obtain
\[
\langle \rho_P, \check{\alpha} \rangle=\left\{ \begin{array}{l}
	n+\frac{-k+1}{2} \mbox{ if }k<n, \\
	n+1 \mbox{ if }k=n.
\end{array} \right.
\]
Thus,
\[
\tilde{\alpha}=\left\{ \begin{array}{l}
E_1+ \dots +E_k \mbox{ for }k<n, \\
\frac{1}{2}(E_1+ \dots+E_n) \mbox{ for }k=n.
\end{array} \right.
\]
Noting that determinant on $GL_k$ corresponds to $E_1+ \dots +E_k$, we then see
\[
I(s\tilde{\alpha}, \tau \otimes \sigma^{(0)})=\left\{ \begin{array}{l}
\nu^s \tau \rtimes \sigma^{(0)} \mbox{ for }k<n, \\
\nu^{\frac{s}{2}}\tau \rtimes \sigma^{(0)} \mbox{ for }k=n.
\end{array} \right.
\]

%Be sure to include some comments on the role of $c$-invariant characters and cocharacters. The added description $S^{\ast}=T^{\ast} \cap T(F)$ may be a good starting point for the latter.

\begin{rem}\label{char}
While $char(F)\not=2$ is enough for most of the discussion in \S \ref{sect3}, there is one key result which requires $char(F)=0$ (at least at present): the characterization of generic cuspidal reducibility points (C$\alpha$) (based on \cite{Sha90a}).
\end{rem}

\section{Generic representations}\label{sect4}

In this section, we classify irreducible generic representations for the groups under consideration. Although not essential to the discussion which follows, in \S \ref{Whittaker models} we classify their Whittaker modules. \S \ref{Steinberg sect} constructs a basic family of representations needed later, which are essentially generalized Steinberg representations. In \S \ref{gdssection}, we start our classification by classifying square-integrable generic representations. Building on this, we classify irreducible tempered generic representations in \S \ref{gtsection} and irreducible generic (admissible) representations in \S \ref{generic}.

We note that the assumption $char(F)=0$ is not needed in \S \ref{Whittaker models}; $char(F) \not=2$ suffices. However, the Standard Module Conjecture (\cite{HO13}), generic cuspidal reducibility conditions (C$\alpha$) (using \cite{Sha90a}) and R-group results (following \cite{Gol94}, etc.) all use $char(F)=0$. So, while the requirement $char(F)=0$ is not directly needed in the combinatorial arguments presented in the rest of \S \ref{sect4}, the results do not hold without it.

\subsection{Whittaker models}\label{Whittaker models}

Recall that an irreducible representation is called generic if it admits a nontrivial Whittaker model. In this section, we review the properties of Whittaker models and classify the Whittaker models for the families of groups under consideration. We remark that this is more information than is actually needed in this paper but is included for the sake of completeness.

Let $G$ be a quasi-split group defined over $F$, $B=TU$ a fixed Borel subgroup, and $\alpha_1, \dots, \alpha_n$ the corresponding simple roots. To be precise, we view $G$ as in \cite{Spr98}--a based root datum $(X, \Pi, \check{X}, \check{\Pi})$ with an action of the Galois group (so simple root refers to elements of $\Pi$, not $F$-roots). Recall (see \cite[Appendix]{Sha74} or \cite[Section 3]{Sha88}) an $F$-morphism $f: U \longrightarrow \bar{F}$ is called non-degenerate if it satisfies the following: for  $u=u_{\alpha_1}(x_1) \dots u_{\alpha_n}(x_n)$ (with $u_{\alpha_i}$ root subgroup map from $\bar{F}$ to $U$ corresponding to $\alpha$),
\[
f(u)=k_1x_1+k_2x_2+ \dots +k_nx_n
\]
with $k_1, \dots, k_n \not=0$. If $G$ is quasi-split but not split, there are additional constraints described below. A generic character is then one of the form
\[
\Psi(u)=\Psi(u_{\alpha_1}(x_1)u_{\alpha_2}(x_2) \dots u_{\alpha_n}(x_n))=\psi(k_1x_1+ \dots +k_{n-1}x_{n-1}+k_nx_n),
\]
with $\psi$ a fixed nontrivial additive character of $F$. We remark that by \cite{Tat67}, any other nontrivial additive character has the form $\psi^{(b)}(x)=\psi(bx)$ for some $b \in F^{\times}$ hence produces the same family of possible $\Psi$. An irreducible representation $(\pi, G(F), V_{\pi})$ is $\Psi$-generic if it admits a Whittaker model with respect to $\Psi$, i.e., there is a nontrivial linear functional $\ell$ on $V_{\pi}$ satisfying
\[
\ell(\pi(u)v)=\Psi(u)\ell(v)
\]
for all $u \in U(F)$ and $v \in V_{\pi})$. Note that if $\pi$ admits a Whittaker model with respect to $\Psi$, it is unique (\cite{Rod73} in the split case and \cite{Sha74} more generally).

For $t \in T(F)$, let $t \cdot \Psi(u)=\Psi(t^{-1}ut)$. As in \cite{Jia06}, we note that if $\pi$ is $\Psi$-generic, then $\pi$ is also $t \cdot \Psi$ generic for all $t \in T(F)$. In particular, the $T(F)$-orbits of generic characters parameterize the distinct Whittaker models. In the remainder of this section, we parameterize those orits for the groups under consideration.

 Let $d=d(a_1, \dots, a_n)$ (for classical groups) and $d(a_1, \dots, a_n,a_0)$ (for similitude groups) be as before.

We begin with the split cases  (noting the odd and even special orthogonal groups are done in \cite{JS04} and \cite{JL14}, respectively, and are included here for the sake of completeness). For the classical groups, we have %(noting $d^{-1} \cdot \Psi(u)=\Psi(d \cdot u)$)
\begin{align*}
{}&d^{-1} \cdot \Psi[u_{\alpha_1}(x_1)\dots u_{\alpha_{n-1}}(x_{n-1}) u_{\alpha_n}(x_n)]\\
&=\Psi[u_{\alpha_1}(\alpha_1(d)x_1) \dots u_{\alpha_{n-1}}(\alpha_{n-1}(d)x_{n-1}) u_{\alpha_n}(\alpha_n(d)x_n)] \\
&=\left\{ \begin{array}{l}
\psi(k_1(a_1a_2^{-1})x_1+ \dots +k_{n-1}(a_{n-1}a_n^{-1})x_{n-1}+k_n(a_n^2)x_n) \\
\mbox{ for }G_n=Sp_{2n}, \\
\psi(k_1(a_1a_2^{-1})x_1+ \dots +k_{n-1}(a_{n-1}a_n^{-1})x_{n-1}+k_n(a_n)x_n) \\
\mbox{ for }G_n=SO_{2n+1}, \\
\psi(k_1(a_1a_2^{-1})x_1+ \dots +k_{n-1}(a_{n-1}a_n^{-1})x_{n-1}+k_{n-1}k_n(a_{n-1}a_n)x_n) \\
\mbox{ for }G_n=SO_{2n}.
\end{array}\right.
\end{align*}
If we take $a_1=a_n k_1^{-1}k_2^{-1} \dots k_{n-1}^{-1}$, $a_2=a_n k_2^{-1}k_3^{-1} \dots k_{n-1}^{-1}, \dots, a_{n-1}=a_n k_{n-1}^{-1}$, we get
\[
\begin{array}{rl}
&d^{-1} \cdot \Psi[u_{\alpha_1}(x_1)\dots u_{\alpha_{n-1}}(x_{n-1}) u_{\alpha_n}(x_n)]\\
&=\Psi[u_{\alpha_1}(\alpha_1(d)x_1) \dots u_{\alpha_{n-1}}(\alpha_{n-1}(d)x_{n-1}) u_{\alpha_n}(\alpha_n(d)x_n)] \\
&=\left\{ \begin{array}{l}
\psi(x_1+ \dots +x_{n-1}+k_n(a_n^2)x_n) \mbox{ for }G_n=Sp_{2n}, \\
\psi(x_1+ \dots +k_n(a_n)x_n) \mbox{ for }G_n=SO_{2n+1}, \\
\psi(x_1+ \dots +k_{n-1}^{-1}k_n(a_n^2)x_n) \mbox{ for }G_n=SO_{2n}.
\end{array}\right.
\end{array}
\]
From this,  we see that there is one orbit of generic characters for $SO_{2n+1}$ while the generic characters of $Sp_{2n}$ and $SO_{2n}$ are parameterized by $F^{\times}/(F^{\times})^2$.

For the split similitude groups, we have 
\[
\begin{array}{rl}
&d^{-1} \cdot \Psi[u_{\alpha_1}(x_1)\dots u_{\alpha_{n-1}}(x_{n-1}) , u_{\alpha_n}(x_n)] \\
&=\Psi[u_{\alpha_1}(\alpha_1(d)x_1) \dots u_{\alpha_{n-1}}(\alpha_{n-1}(d)x_{n-1}) u_{\alpha_n}(\alpha_n(d)x_n)] \\
&=\left\{ \begin{array}{l}
\psi(k_1(a_1a_2^{-1})x_1+ \dots +k_{n-1}(a_{n-1}a_n^{-1})x_{n-1}+k_n(a_n^2a_0^{-1})x_n) \\
\mbox{ for }G_n=GSp_{2n}, \\
\psi(k_1(a_1a_2^{-1})x_1+ \dots +k_{n-1}(a_{n-1}a_n^{-1})x_{n-1}+k_n(a_{n-1}a_na_0^{-1})x_n) \\
\mbox{ for }G_n=GSO_{2n}, \\
\psi(k_1(a_1a_2^{-1})x_1+ \dots +k_{n-1}(a_{n-1}a_n^{-1})x_{n-1}+k_n(a_n)x_n) \\
\mbox{ for }G_n=GSpin_{2n+1}, \\
\psi(k_1(a_1a_2^{-1})x_1+ \dots +k_{n-1}(a_{n-1}a_n^{-1})x_{n-1}+k_n(a_{n-1}a_n)x_n) \\
\mbox{ for }G_n=GSpin_{2n}.
\end{array}\right.
\end{array}
\]

For $G_n=GSp_{2n}$, if we choose $a_0=k_{n-1}^{-1}$, $a_n=1$, $a_{n-1}=k_{n-1}^{-1}$, $a_{n-2}=k_{n-1}^{-1}k_{n-2}^{-1}$, $\dots,$ $a_1=k_{n-1}^{-1}k_{n-2}^{-1} \dots k_1^{-1}$, we get
\[
d^{-1} \cdot \Psi(u_{\alpha_1}(x_1)u_{\alpha_2}(x_2) \dots u_{\alpha_n}(x_n))=\psi(x_1+ \dots +x_{n-1}+x_n);
\]
for $GSO_{2n}$, if we choose $a_0=k_{n-1}^{-1}k_n$, $a_n=1$, $a_{n-1}=k_{n-1}^{-1}$, $a_{n-2}=k_{n-1}^{-1}k_{n-2}^{-1}$, $\dots,$ $a_1=k_{n-1}^{-1}k_{n-2}^{-1} \dots k_1^{-1}$, we get
\[
d^{-1} \cdot \Psi(u_{\alpha_1}(x_1)u_{\alpha_2}(x_2) \dots u_{\alpha_n}(x_n))=\psi(x_1+ \dots +x_{n-1}+x_n).
\]
Thus there is one orbit in both cases.
For $GSpin_{2n+1}$ and $GSpin_{2n}$, as the roots match those for $SO_{2n+1}$ and $SO_{2n}$, respectively, one can take $a_0=1$ and the remaining values of $a_k$ as above to see there is one orbit in the odd case and orbits parameterized by $F^{\times}/(F^{\times})^2$ in the even case.

We now turn to the quasi-split groups. Fix $E=F(\sqrt{\varepsilon})$. First, $G_n=SO_{2n+2}^{\ast}$, $GSO_{2n+2}^{\ast}$, and $GSpin_{2n+2}^{\ast}$ have root data that of the corresponding split groups with Galois action given by $c$ as earlier. In particular, the $F$-points are those $X \in G_n(E)$ satisfying $\sigma(X)=c \cdot X$ ($\sigma \in Gal(E/F)$ nontrivial). Similarly, for $U_N$, the root data is that of $GL_N$ with Galois action -- also denoted $c$ for convenience -- given by $c: X \longmapsto J'_N{}^{-1} ({}^T\!X^{-1})J'_N$, with $J'_N$ as earlier. Again, the $F$-points are those $X \in GL_N(E)$ satisfying $\sigma(X)=c \cdot X$.

Let $d \in G_n(E)$ in the maximal torus be given by
\[
d=\left\{ \begin{array}{l}
	d(a_1, \dots, a_n, a_{n+1}) \mbox{ for }G_n=SO_{2n+2}^{\ast}, \\
	d(a_1, \dots, a_n, a_{n+1}, a_0) \mbox{ for }G_n=GSO_{2n+2}^{\ast}, GSpin_{2n+2}^{\ast}, \\
	d(a_1, \dots, a_n, a_{n+1}, a_{n+2}, \dots, a_{2n+1}) \mbox{ for }G_n=U_{2n+1}, \\
	d(a_1, \dots, a_n, a_{n+1}, \dots, a_{2n}) \mbox{ for }G_n=U_{2n}.
\end{array} \right.
\]
Then,
\[
c \cdot d=\left\{ \begin{array}{l}
	d(a_1, \dots, a_n, a_{n+1}^{-1}) \mbox{ for }G_n=SO_{2n+2}^{\ast}, \\
	d(a_1, \dots, a_n, a_0a_{n+1}^{-1}, a_0) \mbox{ for }G_n=GSO_{2n+2}^{\ast}, \\
	d(a_1, \dots, a_n, a_{n+1}^{-1}, a_0a_{n+1}) \mbox{ for }G_n=GSpin_{2n+2}^{\ast}, \\
	d(a_{2n+1}^{-1}, \dots, a_{n+2}^{-1}, a_{n+1}^{-1}, a_n^{-1}, \dots, a_1^{-1}) \mbox{ for }G_n=U_{2n+1}, \\
	d(a_{2n}^{-1}, \dots, a_{n+1}^{-1}, a_n^{-1}, \dots, a_1^{-1}) \mbox{ for }G_n=U_{2n}.
\end{array} \right.
\]
Thus, $d \in G_n(F)$ if $\sigma(d)=c \cdot d$, i.e.,
\begin{table}[H]
\begin{tabular}{r|l}
$G_n$ & condition \\ \hline
$SO_{2n+2}^{\ast}$ & $\bar{a}_1=a_1, \dots, \bar{a}_n=a_n, \bar{a}_{n+1}=a_{n+1}^{-1}$\\
$GSO_{2n+2}^{\ast}$ & $\bar{a}_1=a_1, \dots, \bar{a}_n=a_n, \bar{a}_{n+1}=a_0 a_{n+1}^{-1}$ \\
$GSpin_{2n+2}^{\ast}$ & $\bar{a}_1=a_1, \dots, \bar{a}_n=a_n, a_{n+1}=\bar{a}_0a_0^{-1}$\\
$U_{2n+1}$ & $a_{2n+1}=\bar{a}_1^{-1}, \dots, a_{n+2}=\bar{a}_n^{-1}, a_{n+1}=\bar{a}_{n+1}^{-1}$ \\
$U_{2n}$ & $a_{2n}=\bar{a}_1^{-1}, \dots, a_{n+1}=\bar{a}_n^{-1}$ \\
\end{tabular}
\caption{Quasi-split tori}
\label{tab:a conditions}
\end{table}
\noindent
(compare with the descriptions in \S \ref{sect3}).

Similarly, let $u \in G_n(E)$ in the unipotent radical be given by
\[
u=\left\{ \begin{array}{l}
u_{\alpha_1}(x_1) \dots u_{\alpha_{n-1}}(x_{n-1})u_{\alpha_n}(x_n)u_{\alpha_{n+1}}(x_{n+1}) \\
\mbox{ for }G_n=SO_{2n+2}^{\ast}, GSO_{2n+2}^{\ast}, GSpin_{2n+2}^{\ast}, \\
u_{\alpha_1}(x_1) \dots u_{\alpha_n}(x_n) u_{\alpha_{n+1}}(x_{n+1}) \dots u_{\alpha_{2n}}(x_{2n}) \\
\mbox{ for }G_n=U_{2n+1}, \\
u_{\alpha_1}(x_1) \dots u_{\alpha_{n-1}}(x_{n-1}) u_{\alpha_n}(x_n) u_{\alpha_{n+1}}(x_{n+1}) \dots u_{\alpha_{2n-1}}(x_{2n-1}) \\
\mbox{ for }G_n=U_{2n}.
\end{array} \right.
\]
Then
\[
c \cdot u=\left\{ \begin{array}{l}
u_{\alpha_1}(x_1) \dots u_{\alpha_{n-1}}(x_{n-1}) u_{\alpha_{n+1}}(x_n)u_{\alpha_n}(x_{n+1}) \\
\mbox{ for }G_n=SO_{2n+2}^{\ast}, GSO_{2n+2}^{\ast}, GSpin_{2n+2}^{\ast}, \\
u_{\alpha_{2n}}(x_1) \dots u_{\alpha_{n+1}}(x_n) u_{\alpha_n}(x_{n+1}) \dots u_{\alpha_1}(x_{2n}) \\
\mbox{ for }G_n=U_{2n+1}, \\
u_{\alpha_{2n-1}}(x_1) \dots u_{\alpha_{n+1}}(x_{n-1}) u_{\alpha_n}(x_n) u_{\alpha_{n-1}}(x_{n+1}) \dots u_{\alpha_1}(x_{2n-1}) \\
\mbox{ for }G_n=U_{2n}.
\end{array} \right.
\]
Thus $u \in G(F)$ if $\sigma(u)=c \cdot u$, i.e., 
\begin{table}[H]
\begin{tabular}{r|l}
$G_n$ & \mbox{condition} \\ \hline
$SO_{2n+2}^{\ast}$, $GSO_{2n+2}^{\ast}$, $GSpin_{2n+2}^{\ast}$ & $\bar{x}_1=x_1, \dots, \bar{x}_{n-1}=x_{n-1}, \bar{x}_{n+1}=x_n$\\
$U_{2n+1}$ & $x_{2n}=\bar{x}_1, \dots,  x_{n+1}=\bar{x}_n$ \\
$U_{2n}$ & $x_{2n-1}=\bar{x}_1, \dots, x_{n+1}=\bar{x}_{n-1}, \bar{x}_n=x_n$ \\
\end{tabular}
\caption{Unipotent radicals, quasi-split cases}
\label{tab:x conditions}
\end{table}

\noindent
Note that for $SO_{2n+2}^{\ast}, GSO_{2n+2}^{\ast}, GSpin_{2n+2}^{\ast}$, we have $x_1, \dots, x_{n-1} \in F$ and $x_n \in E$; for $G_n=U_{2n+1},U_{2n}$, $x_i \in E$ except for $x_n \in F$ in the case of $U_{2n}$.

Now, let
\begin{equation}\label{qs psi}
\Psi(u)=\psi(k_1x_1+ \dots +k_mx_m)
\end{equation}
with $m=n+1$ (resp., $2n, 2n-1$) for $G_n=SO_{2n+2}^{\ast}, GSO_{2n+2}^{\ast}, GSpin_{2n+2}^{\ast}$ (resp., $G_n=U_{2n+1}$, $G_n=U_{2n}$). Writing $k_i=k_{\alpha_i}$, the coefficients must satisfy $\sigma(k_{\alpha_i})=k_{c \cdot \alpha_i}$ (see \cite[Section 3]{Sha88} or \cite[Appendix]{Sha74}). This results in the same constraints as in Table~\ref{tab:x conditions} with $k_i$ in place of $x_i$. In particular,
\begin{equation}\label{qs psi 2}
\Psi(u)=\left\{ \begin{array}{l}
\psi(k_1x_1+ \dots +k_{n-1}x_{n-1} +k_n x_n+\bar{k}_n \bar{x}_n) \\
\mbox{ for }G_n=SO_{2n+2}^{\ast}, GSO_{2n+2}^{\ast}, GSpin_{2n+2}^{\ast}, \\
\psi(k_1x_1+ \dots +k_n x_n+\bar{k}_n \bar{x}_n+ \dots +\bar{k}_1 \bar{x}_1)\\
\mbox{ for }G_n=U_{2n+1}, \\
\psi(k_1x_1+ \dots +k_{n-1}x_{n-1} +k_n x_n+\bar{k}_{n-1} \bar{x}_{n-1}+ \dots \bar{k}_1 \bar{x}_1) \\
\mbox{ for }G_n=U_{2n}.
\end{array} \right.
\end{equation}
Next, in (\ref{qs psi}), we have
\[
d^{-1} \cdot \Psi(u)=\Psi(d \cdot u)=\psi\big(\alpha_1(d)k_1x_1+ \dots +\alpha_m(d)k_mx_m\big).
\]

For $G_n=SO_{2n+2}^{\ast}, GSO_{2n+2}^{\ast}, GSpin^{\ast}_{2n+2}$, using Table~\ref{tab:a conditions} above, we have
\[
\alpha_1(d)=a_1a_2^{-1}, \dots, \alpha_{n-1}(d)=a_{n-1}a_n^{-1},
\]
and
\[
\begin{array}{l}
\alpha_n(d)=a_na_{n+1}^{-1}=\left\{ \begin{array}{l}
a_n a_{n+1}^{-1} \mbox{ for }G_n=SO_{2n+2}^{\ast}, GSO_{2n+2}^{\ast}\\
a_n a_0 \bar{a}_0^{-1} \mbox{ for }G_n=GSpin_{2n+2}^{\ast}, \\
\end{array} \right. \\
\\
\alpha_{n+1}(d) =\left\{ \begin{array}{l}
a_na_{n+1}=\bar{a}_n \bar{a}_{n+1}^{-1} \mbox{ for }G_n=SO_{2n+2}^{\ast}, \\
a_n a_{n+1}a_0^{-1}=\bar{a}_n \bar{a}_{n+1}^{-1} \mbox{ for }G_n=GSO_{2n+2}^{\ast}, \\
a_n a_{n+1}=\bar{a}_n \bar{a}_0 a_0^{-1} \mbox{ for }G_n=GSpin_{2n+2}^{\ast}.
\end{array} \right.
\end{array}
\]
For $G_n=SO_{2n+2}^{\ast}$, taking $a_j=a_n k_n^{-1}k_{n-1}^{-1} \dots k_j^{-1}$ for $1 \leq j \leq n-1$ (noting that $a_1, \dots, a_n, k_1, \dots, k_{n-1} \in F^{\times}$) gives
\begin{equation}\label{orbit psi}
\Psi(d \cdot u)=\psi(x_1+\dots +x_{n-1}+k_n a_n a_{n+1}^{-1}x_n+\bar{k}_n \bar{a}_n \bar{a}_{n+1}^{-1} \bar{x}_n).
\end{equation}
Since $k_n \in E^{\times}$, $a_n \in F^{\times}$, and $a_{n+1} \in N_1(E/F)$, we have orbits parameterized by $E^{\times}/(F^{\times} \cdot N_1(E/F))$. The situation for $G_n=GSO_{2n+2}^{\ast}$ is similar except that instead of $a_{n+1}\bar{a}_{n+1}=1$ we have $a_{n+1}\bar{a}_{n+1}=a_0$ ($a_0 \in F^{\times}$), so may choose $a_n=1$, $a_{n+1}=k_n$ and $a_0=a_{n+1}\bar{a}_{n+1}$ to reduce (\ref{orbit psi}) to
\[
\Psi(d \cdot u)=\psi(x_1+\dots +x_{n-1}+x_n+\bar{x}_n),
\]
telling us there is only one orbit of generic characters in this case. For $G_n=GSpin_{2n+2}^{\ast}$, again taking $a_j=a_n k_n^{-1}k_{n-1}^{-1} \dots k_j^{-1}$ for $1 \leq j \leq n-1$ (noting that $a_1, \dots, a_n, k_1, \dots, k_{n-1} \in F^{\times}$) gives
\[
\Psi(d \cdot u)=\psi(x_1+\dots +x_{n-1}+k_n a_n a_0 \bar{a}_0^{-1} x_n+\bar{k}_n \bar{a}_n \bar{a}_0 a_0^{-1} \bar{x}_n).
\]
With $k_n \in E^{\times}$, $a_0\bar{a}_0^{-1} \in N_1(E/F)$, and $a_n \in F^{\times}$ (and noting every element of $N_1(E/F)$ may be written in the form $a_0 \bar{a}_0^{-1}$ by Hilbert's Theorem 90), we again have the orbits of generic characters parameterized by $E^{\times}/(F^{\times} \cdot N_1(E/F))$.

For $G_n=U_{2n+1}$, using Table~\ref{tab:a conditions} above, we have
\begin{align*}
  &\alpha_1(d)=a_1 a_2^{-1}, \dots, \alpha_n(d)=a_n a_{n+1}^{-1},\\
  &\alpha_{n+1}(d)=a_{n+1}a_{n+2}^{-1}=\bar{a}_n\bar{a}_{n+1}^{-1}, \dots,
\alpha_{2n}(d)=a_{2n}a_{2n+1}^{-1}=\bar{a}_1\bar{a}_2^{-1}.
\end{align*}
Noting that $a_j,k_j \in E^{\times}$ for $1 \leq j \leq n$ and $a_{n+1} \in N_1(E/F)$, we may take $a_{n+1}=1$ and $a_j=a_n k_n^{-1}k_{n-1}^{-1} \dots k_j^{-1}$ for $1 \leq j \leq n$. Then, (\ref{qs psi 2}) becomes
\[
\Psi(d \cdot u)=\psi(x_1+\dots +x_{n-1}+x_n+\bar{x}_n+\bar{x}_{n-1}+ \dots +\bar{x}_1),
\]
telling us there is only one orbit of generic characters. Since $G_n=GU_{2n+1}$ has the same unipotent radical but larger torus, it also has only one orbit of generic characters.

For $G_n=U_{2n}$, using Table~\ref{tab:a conditions} above, we have
\begin{align*}
    &\alpha_1(d)=a_1 a_2^{-1}, \dots, \alpha_{n-1}(d)=a_{n-1}a_n^{-1} \mbox{ and }\alpha_n(d)=a_n a_{n+1}^{-1}=a_n\bar{a}_n, \\
&\alpha_{n+1}(d)=a_{n+1}a_{n+2}^{-1}=\bar{a}_{n-1}\bar{a}_n^{-1}, \dots, \alpha_{2n}(d)=a_{2n}a_{2n+1}^{-1}=\bar{a}_1\bar{a}_2^{-1}.
\end{align*}
% \[
% \begin{array}{l}
% \alpha_1(d)=a_1 a_2^{-1}, \dots, \alpha_{n-1}(d)=a_{n-1}a_n^{-1} \mbox{ and }\alpha_n(d)=a_n a_{n+1}^{-1}=a_n\bar{a}_n, \\
% \hspace{0.5in}\alpha_{n+1}(d)=a_{n+1}a_{n+2}^{-1}=\bar{a}_{n-1}\bar{a}_n^{-1}, \dots, \alpha_{2n}(d)=a_{2n}a_{2n+1}^{-1}=\bar{a}_1\bar{a}_2^{-1}.
% \end{array}
% \]
We may take $a_j=a_n k_n^{-1}k_{n-1}^{-1} \dots k_j^{-1}$ for $1 \leq j \leq n-1$ (noting $a_j,k_j \in E^{\times}$ for $1 \leq j \leq n$ except that $k_n \in F^{\times}$). Then, (\ref{qs psi 2}) becomes
\[
\Psi(d \cdot u)=\psi(x_1+\dots +x_{n-1}+k_n a_n \bar{a}_nx_n+\bar{x}_{n-1}+ \dots +\bar{x}_1),
\]
telling us the orbits of generic characters are parameterized by $F^{\times}/N^{\times}(E/F)$. For $GU_{2n}$, the matrix realization tells us $d'=d(1, \dots, 1, a_0)$ (noting $a_0 \in F^{\times}$) has
\[
\Psi(d'd \cdot u)=\psi(x_1+\dots +x_{n-1}+k_n a_0^{-1}a_n \bar{a}_nx_n+\bar{x}_{n-1}+ \dots +\bar{x}_1),
\]
so we may choose $a_n=1$ and $a_0=k_n$ to reduce to $\psi(x_1+ \dots+x_{n-1}+x_n+\bar{x}_{n-1}+ \dots+\bar{x}_1)$ and see that there is only one orbit in this case (n.b. $x_1, \dots, x_{n-1} \in E$ and $x_n \in F$).

\subsection{Generalized Steinberg representations}\label{Steinberg sect}

In this section, we construct a family of representations which we call generalized Steinberg representations, as well as establishing some properties of them needed later. Except for $\tau \rtimes \sigma^{(0)}$ or its irreducible subquotients in the (C0) or (CN) cases, they are essentially square-integrable representations. For similitude groups, note that the assumption $\sigma^{(0)}$ unitary ensures unitary central character for $GSpin$ groups, giving square-integrable representations in those cases; for other similitude groups, a suitable twist is needed to make them square-integrable representations. For $\alpha>0$, these are strongly positive square-integrable representations (see \cite{MT02} in the classical case, as well as
\cite{Kim15}, \cite{Kim16}, \cite{KM19} for some of the other families considered; also \cite{Mui06}). We note that the results in this section do not involve genericity except for Corollary~\ref{Steinberggenericity}.

The following is known in some settings (\cite{Tad98a}):

\begin{lem}\label{GLtypeirr}
Let $\tau$ be an irreducible unitary supercuspidal representation of $H_{m}(F)$ and $\sigma^{(0)}$ an irreducible unitary supercuspidal representation of $G_{m_0}(F)$. Further, assume that $(\tau; \sigma^{(0)})$ satisfies (C$\alpha$) for some $\alpha \not\in\{0,1, N\}$. Then, $\delta([\nu^{-1}\tau, \nu\tau]) \rtimes \sigma^{(0)}$ is irreducible.
\end{lem}

\noindent
\begin{proof}
Note that the hypotheses imply $\tau \otimes \sigma^{(0)}$ ramified (see (\ref{w_0action})).  In particular, if $G_n=SO_{2n}, GSO_{2n}$, or $GSpin_{2n}$, the minimal (nonzero) Jacquet module has 8 terms, the same as for the other families considered (as either $\tau$ is from an even-dimensional general linear group or the leftover sign changes can be absorbed into $\sigma^{(0)}$).

By duality (\cite{Aub95}), it suffices to show that $\zeta([\nu^{-1}\tau, \nu\tau]) \rtimes \sigma^{(0)}$ is irreducible. We use an argument from the second example in \cite[Section 6]{Jan98}. Let $\pi$ be an irreducible subquotient of $\zeta([\nu^{-1}\tau, \nu\tau]) \rtimes \sigma^{(0)}$; by unitarity, necessarily a subrepresentation. Then, using the irreducibility of $\nu\tau \rtimes \sigma^{(0)}$, we have
\[
\pi \hookrightarrow \zeta([\nu^{-1}\tau, \nu\tau]) \rtimes \sigma^{(0)} \hookrightarrow\zeta([\nu^{-1}\tau, \tau]) \times \nu\tau \rtimes \sigma^{(0)} \cong \zeta([\nu^{-1}\tau, \tau]) \times \nu^{-1}\tau \rtimes \nu^{\kappa}\sigma^{(0)},
\]
where
\[
\kappa=\left\{ \begin{array}{l} m \mbox{ for $G_n=GSp_{2n}$, $GSO_{2n}$, $GSO_{2n+2}^{\ast}$, $GU_{2n+1}$, or $GU_{2n}$}, \\ 
	% 2 \mbox{ for }G_n=GSO_{2n}^{\ast} \\
	0 \mbox{ otherwise.}\end{array} \right.
\]
By the irreducibility of $\zeta([\nu^{-1}\tau, \tau]) \times \nu^{-1}\tau$, we have
\[
\pi \hookrightarrow \nu^{-1}\tau \times \zeta([\nu^{-1}\tau, \tau]) \rtimes \nu^{\kappa}\sigma^{(0)} \hookrightarrow \nu^{-1}\rho \times \nu^{-1}\tau \rtimes (\tau \rtimes \nu^{\kappa}\sigma^{(0)}).
\]
By the (subrepresentation version of the) Langlands classification, $\pi=L_{sub}(\nu^{-1}\tau \times \nu^{-1}\tau \otimes \tau \rtimes \nu^{\kappa}\sigma^{(0)})$, noting that $\tau \rtimes \nu^{\kappa}\sigma^{(0)}$ is irreducible by hypothesis and Lemma~\ref{twisting}. As this applies to any irreducible subquotient, and $L_{sub}(\nu^{-1}\tau \times \nu^{-1}\tau \otimes \tau \rtimes \nu^{\kappa}\sigma^{(0)})$ appears with multiplicity one in $\nu^{-1}\tau \times \nu^{-1}\tau \rtimes (\tau \rtimes \nu^{\kappa}\sigma^{(0)})$, we have the irreducibility claimed. \end{proof}

The following is based on \cite{Gol94} (also \cite{Gol95}, \cite{Gol97}) and requires characteristic zero.

\begin{cor}\label{Rgroupcor}
With hypotheses as in Lemma~\ref{GLtypeirr}, we have $\delta([\nu^{-1}\tau, \nu\tau]) \times \tau \rtimes \sigma^{(0)}$ irreducible.
\end{cor}

\noindent
\begin{proof}
For $G=SO_{2n+1}, Sp_{2n}$, or $SO_{2n}$ (resp., $G=U_{2n}$ or $U_{2n+1}$; $G=GSp_{2n}$, $GU_{2n}$, $GU_{2n+1}$; $G=SO_{2n+2}^{\ast}$; $G=GSpin_{2n+1}$ or $GSpin_{2n}$), the result follows from the irreducibility results above and \cite{Gol94} (resp., \cite{Gol95}; \cite{Gol97}; \cite{LMT04}; \cite{BG15}).

Note that  $\delta([\nu^{-1}\tau, \nu\tau]) \rtimes \sigma^{(0)}$ and $\tau \rtimes\sigma^{(0)}$ are irreducible (by Lemma~\ref{GLtypeirr} and the the assumption ($C\alpha$) with $\alpha \not\in\{0,1,N\}$). That this implies $\delta([\nu^{-1}\tau, \nu\tau]) \times \tau \rtimes \sigma^{(0)}$ is irreducible follows as in \cite[Section 4]{Gol94}

We sketch the argument, using the notation of \cite{Gol94}.
First, we observe that
\begin{align*}
  W(\delta([\nu^{-1}\tau, \nu\tau]) \otimes \tau \otimes \sigma^{(0)})&=\{w \in W\,| & \,w(\delta([\nu^{-1}\tau, \nu\tau]) \otimes \tau \otimes \sigma^{(0)}) \\
  & &\cong \delta([\nu^{-1}\tau, \nu\tau]) \otimes \tau \otimes \sigma^{(0)}\} \\
  &\cong {\mathbb Z}_2^2.&  
\end{align*}
% \[
% W(\delta([\nu^{-1}\tau, \nu\tau]) \otimes \tau \otimes \sigma^{(0)})=\{w \in W\,|\,w(\delta([\nu^{-1}\tau, \nu\tau]) \otimes \tau \otimes \sigma^{(0)}) \cong \delta([\nu^{-1}\tau, \nu\tau]) \otimes \tau \otimes \sigma^{(0)}\} \cong {\mathbb Z}_2^2.
% \]
Recall that the R-group is
\[
R=\{ w \in W(\delta([\nu^{-1}\tau, \nu\tau]) \otimes \tau \otimes \sigma^{(0)}) \,|\, w\beta>0 \mbox{ for all }\beta \in \Delta'\},
\]
where $\Delta'=\{ \alpha\in \Phi(P,A) \,|\, \mu_{\alpha}(\delta([\nu^{-1}\tau, \nu\tau]) \otimes \tau \otimes \sigma^{(0)})=0\}$ (with $\mu_{\alpha}(\delta([\nu^{-1}\tau, \nu\tau]) \otimes \tau \otimes \sigma^{(0)})$ the Plancherel measure as in \cite{Gol94})). 
Now, suppose $\tau$ is a representation of $H_{m_1}(F)$ and $\sigma^{(0)}$ a representation of $G_{m_0}(F)$ (defining $m_0,m_1$); by abuse of notation, we also let $e_i$ denote the restriction to $A$ of $e_i$. If we show $\Delta'=\Phi(P,A)=\{\alpha_{3m_1},\alpha_{4m_1}\}$, where $\alpha_k$ denotes the $k^{th}$ simple root as listed in the Appendix,
then $W(\delta([\nu^{-1}\tau, \nu\tau]) \otimes \tau \otimes \sigma^{(0)})$ is the Weyl group for $\Delta'$, so $R=1$ and we have irreducibility.
That $\Delta'=\Phi(P,A)$ follows as in Lemma 4.8 of \cite{Gol94}: both $\tau \otimes \sigma^{(0)}$ and $\delta([\nu^{-1}\tau, \nu\tau]) \otimes \sigma^{(0)}$ are ramified and both $\tau \rtimes \sigma^{(0)}$ and $\delta([\nu^{-1}\tau, \nu\tau]) \rtimes \sigma^{(0)}$ are irreducible. Consequently, $\mu(\delta([\nu^{-1}\tau, \nu\tau]) \otimes \sigma^{(0)})=0$ and $\mu(\tau \otimes \sigma^{(0)})=0$. As $\mu_{\alpha'_1}(\delta([\nu^{-1}\tau, \nu\tau]) \otimes \tau \otimes \sigma^{(0)})=\mu(\delta([\nu^{-1}\tau, \nu\tau]) \otimes \sigma^{(0)})$ and  $\mu_{\alpha'_2}(\delta([\nu^{-1}\tau, \nu\tau]) \otimes \tau \otimes \sigma^{(0)})=\mu(\tau \otimes \sigma^{(0)})$, we have $\Delta'$ as claimed. \end{proof}

Lemma~\ref{delicate} below is done in \cite[Section 6]{Tad98a} for $G=SO_{2n+1}$ and $Sp_{2n}$; the proof here uses the same basic approach. To facilitate the proof, we borrow some notation from \cite{Tad98a}, which is used frequently in the remainder of this section and \S \ref{gdssection}.

Let $\alpha=(m_1, \dots, m_k)$ be a tuple having $m_1+\dots+m_k \leq n$ and $M_{\alpha}$ the standard parabolic subgroup of $G_n$ having $M \cong H_{m_1}(F) \times \dots \times H_{m_k}(F) \times G_{n-(m_1+\dots+m_k)}(F)$. We then set
\begin{equation}\label{s notation}
s_{\alpha}=r_{M_{\alpha},G_n}.
\end{equation}
Similar notation is used for general linear groups but with $r_{\alpha}$ replacing $s_{\alpha}$.
If a representation of $G_n(F)$ is a subquotient of some $\tau_1 \times \dots \times\tau_k \rtimes \sigma$ with $\tau_i$ a supercuspidal representation of $H_{m_{\tau_i}}(F)$ and $\sigma$ a supercuspidal representation of $G_{m_0}(F)$,  we let $s_{GL}=s_{(m_1+\dots+m_k)}$.

\begin{lem}\label{delicate}
Suppose $(\tau; \sigma^{(0)})$ satisfies (C$\alpha$) with $\alpha \not\in\{0,1,N\}$. Then $\delta([\tau, \nu\tau]) \rtimes \sigma^{(0)}$ is irreducible.
\end{lem}

\noindent
\begin{proof}
Suppose $\delta([\tau, \nu\tau]) \rtimes \sigma^{(0)}$ were reducible. As
\[
s_{(m)}(\delta([\tau, \nu\tau]) \rtimes \sigma^{(0)})=\nu\tau \otimes (\tau \rtimes \sigma^{(0)})+\tau \otimes (\nu\tau \rtimes \sigma^{(0)}),
\]
(noting that the hypotheses imply $\tau \otimes \sigma^{(0)}$ ramified) with both terms on the right-hand side irreducible, we must have $\delta([\tau, \nu\tau]) \rtimes \sigma^{(0)}=\lambda_1 + \lambda_2$, where $s_{(m)}(\lambda_1)=\nu\tau \otimes (\tau \rtimes \sigma^{(0)})$ and $s_{(m)}(\lambda_2)=\tau \otimes (\nu\tau \rtimes \sigma^{(0)})$. We now consider $\delta([\nu^{-1}\tau, \tau]) \times \delta([\tau, \nu\tau]) \rtimes \sigma^{(0)}$. On the one hand,
\[
\delta([\nu^{-1}\tau, \tau]) \times \delta([\tau, \nu\tau]) \rtimes \sigma^{(0)}=\delta([\nu^{-1}\tau, \tau]) \rtimes \lambda_1+\delta([\nu^{-1}\tau, \tau]) \rtimes \lambda_2.
\]
On the other hand,
\begin{align*}
 &\delta([\nu^{-1}\tau, \tau]) \times \delta([\tau, \nu\tau]) \rtimes \sigma^{(0)}\\
 &=\delta([\nu^{-1}\tau,\nu\tau]) \times \tau \rtimes \sigma^{(0)}+{\mathcal L}_{sub}(\delta([\nu^{-1}\tau, \tau]) \otimes \delta([\tau, \nu\tau])) \rtimes \sigma^{(0)}.   
\end{align*}
% \[
% \delta([\nu^{-1}\tau, \tau]) \times \delta([\tau, \nu\tau]) \rtimes \sigma^{(0)}=\delta([\nu^{-1}\tau,\nu\tau]) \times \tau \rtimes \sigma^{(0)}+{\mathcal L}_{sub}(\delta([\nu^{-1}\tau, \tau]) \otimes \delta([\tau, \nu\tau])) \rtimes \sigma^{(0)}.
% \]
Therefore, noting the irreducibility in Corollary~\ref{Rgroupcor}, we must have $\delta([\nu^{-1}\tau,\nu\tau]) \times \tau \rtimes \sigma^{(0)} \leq \delta([\nu^{-1}\tau, \tau]) \rtimes \lambda_i$ for some $i$. However, 
$$|s_{(m,m,m,m)}(\delta([\nu^{-1}\tau, \tau]) \rtimes \lambda_i)|=48,$$
while $|s_{(m,m,m,m)}(\delta([\nu^{-1}\tau,\nu\tau]) \times \tau \rtimes \sigma^{(0)})|=64$, a contradiction. The lemma follows. \end{proof}

%There is a proof in Section 6 \cite{Tad98a} which handles the split classical case. In \cite{Tad98a}, this is deduced as a consequence of \cite{Gol1} and the known irreducibility of $\rho \rtimes \sigma^{(0)}$ and $\delta([\nu^{-1}\rho, \nu\rho]) \rtimes \sigma^{(0)}$. The irreducibility of $\delta([\nu^{-1}\rho, \nu\rho]) \rtimes \sigma^{(0)}$ is done here in Corollary~\ref{Rgroupcor} above, but with the same characteristic zero assumption as in \cite{Gol1}. The R-groups results of \cite{Gol1} have been extended in \cite{MT02} to remove the characteristic zero hypothesis and apply to the groups considered there. However, we need it more generally. Perhaps the argument in \cite{MT02} generalizes (the L-function results probably extend). Also, note that \cite{MT02} use their Basic Assumption. This smay now be a theorem based on \cite{Art} and \cite{Moe}, but not in sufficient generality.

%Note that for non-split classical groups,  if the representation were reducible, one of the subquotients would be square-integrable. As it does not appear in the M-T list, we must have irreducibility. \end{proof}

\begin{prop}\label{Steinberg}
Suppose $(\tau; \sigma^{(0)})$ satisfies either (1) (C$\alpha$) for $\alpha \in \frac{1}{2}{\mathbb Z}$ with $\alpha \geq 0$, or (2) (CN).

\begin{enumerate}

\item If $(\tau; \sigma^{(0)})$ satisfies (C$\alpha$), suppose $b \geq \alpha$ with $b \equiv \alpha \, \mbox{mod}\,1$.

\begin{enumerate}

\item If $\alpha=0$, write $\tau \rtimes \sigma^{(0)}=T_1(\tau; \sigma^{(0)}) \oplus T_{-1}(\tau; \sigma^{(0)})$ (with $T_1(\tau; \sigma^{(0)})$ the generic component if $\sigma^{(0)}$ is generic). Then, $\delta([\nu^{\alpha}\tau, \nu^b\tau]) \rtimes \sigma^{(0)}$ contains exactly two irreducible subrepresentations, which we denote $\delta_i([\tau, \nu^b \tau];\sigma^{(0)})$, $i=\pm 1$ (taking $(\delta_i([\tau, \nu^b \tau]; \sigma^{(0)}))=T_i(\tau; \sigma^{(0)})$ for $b=0$). Further, we have
\[
s_{(m)}\left(\delta_i([\tau, \nu^b \tau]; \sigma^{(0)})\right)=\nu^b \tau \otimes \delta_i([\tau, \nu^{b-1}\tau;\sigma^{(0)}).
\]

\item If $\alpha>0$, then $\delta([\nu^{\alpha}\tau, \nu^b\tau]) \rtimes \sigma^{(0)}$ contains a unique irreducible subrepresentation, which we denote $\delta([\nu^{\alpha}\tau, \nu^b \tau];\sigma^{(0)})$.  Further, we have
\[
s_{(m)}\left(\delta([\nu^{\alpha}\tau, \nu^b \tau]; \sigma^{(0)})\right)=\nu^b \tau \otimes \delta([\nu^{\alpha}\tau, \nu^{b-1}\tau]; \sigma^{(0)}).
\]

\end{enumerate}

\item  If $(\tau; \sigma^{(0)})$ satisfies (CN), suppose $b \geq 0$ with $b \equiv 0 \, \mbox{mod}\,1$. Then $\delta([\tau, \nu^b\tau]) \rtimes \sigma^{(0)}$ contains a unique irreducible subrepresentation $\delta([\tau, \nu^b \tau]; \sigma^{(0)})$. Further, we have
\[
s_{(m)}\left(\delta([\tau, \nu^b \tau]; \sigma^{(0)})\right)=
\nu^b \tau \otimes \delta([\tau, \nu^{b-1}\tau]; \sigma^{(0)}) \mbox{ if } b>0
\]
for $b>0$ (with $\tau \rtimes \sigma^{(0)}$ irreducible for $b=0$).

\end{enumerate}

\end{prop}

\noindent
\begin{proof} The proof is by induction on $b$. The case $b=\alpha$ (resp., $b=0$) is immediate from the definition of (C$\alpha$) (resp.,  (CN)).

 To uniformize the presentation, for $b>\alpha$ (resp., $b>0$ in the (CN) case), let
\begin{equation}\label{Steinbergshorthand}
\delta_i(\tau, b; \sigma^{(0)})=\left\{ \begin{array}{l}
\delta([\nu^{\alpha} \tau, \nu^b \tau];\sigma^{(0)}) \mbox{ for (C$\alpha$) with $\alpha>0$ ($i=1$ only)}, \\
\delta_i([\tau, \nu^b \tau]; \sigma^{(0)}) \mbox{ for (C0) ($i=\pm 1$)}, \\
\delta([\tau, \nu^b \tau]; \sigma^{(0)}) \mbox{ for (CN) ($i=1$ only)}.
\end{array} \right.
\end{equation}
 Note that it follows from the formula for $s_{min}(\delta_i(\tau, b; \sigma^{(0)}))$ and the inductive assumption that $s_{(m)}(\delta_i(\tau, b; \sigma^{(0)}))=\nu^b \tau \otimes \delta_i(\tau, b-1; \sigma^{(0)})$.

We now assume inductively that the result holds for $b$ with $b \geq \alpha$.
Consider the induced representations
\[
I=\left\{ \begin{array}{l} \delta([\nu^{\alpha}\tau, \nu^{b+1} \tau]) \rtimes \sigma^{(0)} \mbox{ for (C$\alpha$)}, \\
	\delta([\tau, \nu^{b+1} \tau]) \rtimes \sigma^{(0)} \mbox{ for (CN)},
\end{array}\right.
\]
\[
I_i'=\nu^{b+1} \tau \rtimes \delta_i(\tau, b;\sigma^{(0)}),
\]
and
\[
I''=\left\{ \begin{array}{l} \nu^{b+1}\tau \times \delta([\nu^{\alpha}\tau, \nu^b \tau]) \rtimes \sigma^{(0)} \mbox{ for (C$\alpha$)}, \\
	\nu^{b+1}\tau \times \delta([\tau, \nu^b \tau]) \rtimes \sigma^{(0)} \mbox{ for (CN)}.
\end{array}\right.
\]
Clearly, $I \hookrightarrow I''$. To see $I_i' \hookrightarrow I''$, note that it follows from induction and the Jacquet module formula that
\[
s_{GL}(\delta_i(\tau, b;\sigma^{(0)}))=\left\{ \begin{array}{l} \delta([\nu^{\alpha}\tau, \nu^b \tau]) \otimes \sigma^{(0)} \mbox{ for (C$\alpha$),} \\
	\delta([\tau, \nu^b \tau]) \otimes \sigma^{(0)}+\delta([\tau, \nu^b \tau]) \otimes \sigma'{}^{(0)} \mbox{ for (CN)}
\end{array} \right.
\]
(see (CN) conditions for the possible $\sigma'{}^{(0)}$); note that $\delta([\tau, \nu^b \tau]; \sigma^{(0)}) \cong \delta([\tau, \nu^b \tau]; \sigma'{}^{(0)})$. By Frobenius reciprocity--and replacing $\sigma^{(0)}$ by $\sigma'{}^{(0)}$ in the (CN) case if needed--it follows that $I_i'  \hookrightarrow I''$.

Observe that from the $\mu^{\ast}$ formula, if we interpret $\alpha=0$ in the (CN) case, we have
\[
s_{(m)}(I)=\nu^{b+1} \tau \otimes \left( \delta_i([\nu^{\alpha}\tau, \nu^b\tau]) \rtimes \sigma^{(0)} \right)+\nu^{-\alpha}\tau \otimes (\dots),
\]
noting that for $GSpin$ groups, (C$\alpha$) and (CN) require $\omega_{\sigma^{(0)}}\tilde{\tau} \cong \tau$ (and $\check{\tau} \cong \tau$ for the remaining groups). The particular representation appearing with $\nu^{-\alpha} \tau$ is not important for the argument which follows and is omitted to save space. Similarly,
\[
s_{(m)}(I_i')=\nu^{b+1} \tau \otimes  \delta_i(\tau, b;\sigma^{(0)})+\nu^{-b-1}\tau \otimes (\dots)+\nu^b\tau \otimes (\dots),
\]
and
\begin{align*}
    s_{(m)}(I'')&=\nu^{b+1} \tau \otimes \left( \delta_i([\nu^{\alpha}\tau, \nu^b\tau]) \rtimes \sigma^{(0)} \right)+\nu^{-b-1}\tau \otimes (\dots)\\
    &+\nu^b \tau \otimes ( \dots)+\nu^{-\alpha}\tau \otimes (\dots).
\end{align*}
% \[
% s_{(1)}(I'')=\nu^{b+1} \tau \otimes \left( \delta_i([\nu^{\alpha}\tau, \nu^b\tau]) \rtimes \sigma^{(0)} \right)+\nu^{-b-1}\tau \otimes (\dots)+\nu^b \tau \otimes ( \dots)+\nu^{-\alpha}\tau \otimes (\dots).
% \]
Note that $b+1>b,-b-1,-\alpha$. It is then easy to see that  $\nu^{b+1} \tau \otimes  \delta_i(\tau, b;\sigma^{(0)})$ appears with mutliplicity one in $s_{(m)}(I)$, $s_{(m)}(I_i')$, and $s_{(m)}(I'')$. Hence $I$ and $I_i'$ have a common irreducible subquotient $\pi$ satisfying $s_{(m)}(\pi) \geq \nu^{b+1} \tau \otimes  \delta_i(\tau, b;\sigma^{(0)})$. It remains to show that $s_{(m)}(\pi)=\nu^{b+1} \tau \otimes  \delta_i(\tau, b;\sigma^{(0)})$.
However, we immediately see that as long as $b \not= -\alpha$, the only term in common between $s_{(m)}(I)$ and $s_{(m)}(I_i')$ is $\nu^{b+1} \tau \otimes \delta_i(\tau, b; \sigma^{(0)})$. Then,
\[
s_{(m)}(\delta_i(\tau,b+1; \sigma^{(0)}))=\nu^{b+1} \tau \otimes \delta_i(\tau, b; \sigma^{(0)}),
\]from which the result is immediate.

If $b=-\alpha$, we must have $\alpha=0$ and $b=0$. In particular, $(\tau; \sigma^{(0)})$ is either (C0) or (CN). In the (CN) case, note that $\tau \rtimes \sigma^{(0)} \cong \tau \rtimes \sigma'{}^{(0)}$ with $\sigma'{}^{(0)}$ depending on the group, but $\sigma'{}^{(0)} \not\cong \sigma^{(0)}$. Observe that $\nu\tau \otimes \tau \otimes \sigma^{(0)}$ is regular. Further, we have
\[
\delta([\tau, \nu\tau]) \rtimes \sigma \leq\nu\tau \times \tau \rtimes \sigma^{(0)}
\]
and
\[
\delta([\tau, \nu\tau]) \rtimes \sigma'{}^{(0)} \leq\nu\tau \times \tau \rtimes \sigma'{}^{(0)}=\nu\tau \times \tau \rtimes \sigma^{(0)}.
\]
As
\[
s_{(m)}(\delta([\tau, \nu\tau]) \rtimes \sigma^{(0)})= \nu\tau \otimes (\tau \rtimes \sigma^{(0)})+\tau \otimes (\nu\tau \rtimes \sigma'{}^{(0)})
\]
 and
\[
s_{(m)}(\delta([\tau, \nu\tau]) \rtimes \sigma'{}^{(0)})= \nu\tau \otimes (\tau \rtimes \sigma^{(0)})+\tau \otimes (\nu\tau \rtimes \sigma^{(0)}),
\]
we see that $\delta([\tau, \nu\tau]) \rtimes \sigma^{(0)}$ and $\delta([\tau, \nu\tau]) \rtimes \sigma'{}^{(0)}$ must have an irreducible subquotient in common. Comparing Jacquet modules--and noting $\nu\tau \rtimes \sigma^{(0)} \not\cong \nu\tau \rtimes \sigma'^{(0)}$--we see that this irreducible subquotient must have Jacquet module consising of $\nu\tau \otimes (\tau \rtimes \sigma^{(0)})$. The (CN) case follows.

In the (C0) case, we adapt an argument from the proof of \cite[Proposition 3.11]{Jan96a}. We consider some induced representations which appear in $\nu\tau \times \tau \rtimes \sigma^{(0)}$. Observe that (noting the ramified conditions on $\tau \otimes \sigma^{(0)}$ required for (C0))
\[
s_{(m)}(\delta([\tau, \nu \tau]) \rtimes \sigma^{(0)})=\nu\tau \otimes T_1(\rho; \sigma^{(0)})+\nu\tau \otimes T_{-1}(\tau; \sigma^{(0)})+\tau \otimes L(\nu\tau \otimes \sigma^{(0)}),
\]
and for $i \in \{ \pm 1\}$,
\[
s_{(m)}(\nu\tau \rtimes T_i(\tau; \sigma^{(0)}))=\nu\tau \otimes T_i(\tau; \sigma^{(0)})+\nu^{-1}\tau \otimes T_i(\tau; \sigma^{(0)})+\tau \otimes L(\nu\tau \otimes \sigma^{(0)}).
\]
Let $\pi_i$, $i \in \{ \pm 1\}$, be the irreducible subquotient of $\nu\tau \times \tau \rtimes \sigma^{(0)}$ such that $s_{(m)}(\pi_i)$ contains (the unique copy of) $\nu \tau \otimes T_i(\tau; \sigma^{(0)})$. Comparing Jacquet modules above, we see that these are distinct and $s_{(m)}(\pi_i) \leq \nu\tau \otimes T_i(\tau; \sigma^{(0)})+\tau \otimes L(\nu\tau \otimes  \sigma)$. Further, by central character considerations, we must have $\pi_i \hookrightarrow \nu\tau \rtimes T_i(\tau; \sigma^{(0)})$. Now, since
\[
\delta([\tau, \nu \tau]) \rtimes \sigma^{(0)} \hookrightarrow
\nu\tau \times \tau \rtimes \sigma^{(0)} \cong \nu\tau \rtimes \left(T_1(\tau; \sigma^{(0)}) \oplus T_{-1}(\tau; \sigma^{(0)})\right),
\]
we see that both $\pi_1$ and $\pi_{-1}$ appear as subrepresentations of $\delta([\tau, \nu \tau]) \rtimes \sigma^{(0)}$. By the Langlands classification, we have $L(\delta([\tau, \nu \tau]) \otimes \sigma^{(0)})$ as unique irreducible quotient, making this distinct from $\pi_{\pm 1}$. Looking at the Jacquet modules above, it is then clear that $s_{(m)}(\pi_i)= \nu\tau \otimes T_i(\tau ; \sigma^{(0)})$ and $s_{(m)}(L(\delta([\tau, \nu \tau]) \otimes \sigma^{(0)})=\tau \otimes L(\nu\tau; \sigma^{(0)})$. The result follows. \hfill\end{proof}

\begin{cor}
The representations $\delta([\nu^{\alpha}\tau, \nu^b \tau]; \sigma^{(0)})$ (for (C$\alpha$) with $\alpha>0$), $\delta_i([\tau, \nu^b \tau]; \sigma^{(0)})$ (for (C0)), and $\delta([\tau, \nu^b \tau]; \sigma^{(0)})$ (for (CN)) of Proposition~\ref{Steinberg} are essentially square-integrable if $b>0$. If $b=0$--which can happen only in the cases of (C0) or (CN)--the representations are essentially tempered but not essentially square-integrable.
\end{cor}

\noindent
\begin{proof}
This follows from the Jacquet module characterization given in Proposition~\ref{Steinberg} and the Casselman criterion. \end{proof}

Recall that in the (C0) case, we have chosen $T_1(\tau; \sigma^{(0)})$ to be the $\psi_a$-generic component.

\begin{cor}\label{Steinberggenericity}
The representations $\delta([\nu^{\alpha}\tau, \nu^b \tau]; \sigma^{(0)})$ (for (C$\alpha$) with $\alpha>0$), $\delta_1([\tau, \nu^b \tau]; \sigma^{(0)})$ (for (C0)), and $\delta([\tau, \nu^b \tau]; \sigma^{(0)})$ (for (CN)) of Proposition~\ref{Steinberg} are $\psi_a$-generic.
\end{cor}

\noindent
\begin{proof}
We use the notation of (\ref{Steinbergshorthand}).

The proof is by induction on $b$. The base case $b=\alpha$ follows from the Standard Module Conjecture (for (C$\alpha$) with $\alpha>0$), choice of $T_1(\tau; \sigma^{(0)})$ (for (C0)), and irreducibility (for (CN)).
For $b>\alpha$, one observes inductively that both $I$ and $I_1'$ contain the generic subquotient;
the characterization of $\delta_1(\tau, b; \sigma^{(0)})$ as the unique common irreducible subquotient then finishes the proof. \end{proof}

\begin{rem}\label{dsrem}
A few words on square-integrability vs. essential square-integrability for similitude groups are in order at this point. Suppose $\nu^{x_1}\tau_1 \otimes \dots \otimes \nu^{x_k}\tau_k \otimes \sigma^{(0)}$ (with $\tau_1, \dots, \tau_k$ and $\sigma^{(0)}$ irreducible unitary supercuspidal representations) is in the Jacquet module for an essentially square-integrable representation $\pi$. For $G_n=GSpin_{2n+1}$, $GSpin_{2n}$, or $GSpin_{2n+2}^{\ast}$, the central character of $\pi$ is $\omega_{\sigma^{(0)}}$, so $\pi$ is square-integrable. On the other hand, for $G_n=GSp_{2n}$, $GSO_{2n}$, $GSO_{2n+2}^{\ast}$, $GU_{2n+1}$, or $GU_{2n}$,  the central character has the form $\omega_{\nu^{x_1}\tau_1} \dots \omega_{\nu^{x_k}\tau_k} \omega_{\sigma^{(0)}}'$ (where $\omega_{\sigma^{(0)}}'=\omega_{\sigma^{(0)}}$, $\omega_{\sigma^{(0)}}^2$, or $\omega_{\sigma^{(0)}} \circ N_{E/F}$). In particular, the central character is not unitary (by the requirement $n_1x_1+ \dots +n_kx_k>0$ in the Casselman criterion). To obtain a square-integrable representation, one can twist by a suitable unramified character $\chi_0$. This amounts to twisting $\sigma^{(0)}$ by $\chi_0$ (Lemma~\ref{twisting}). Thus, for $\pi$ square-integrable, we write
\[
\pi \hookrightarrow \nu^{x_1}\tau_1 \times \dots \nu^{x_k}\tau_k \rtimes \chi_0\sigma^{(0)},
\]
noting $\chi_0$ trivial for classical or general spin groups. For the generalized Steinberg representations of Proposition~\ref{Steinberg}, we write
\[
\delta_i([\nu^{\alpha}\tau, \nu^b \tau]; \chi_0\sigma^{(0)})=\chi_0 \delta_i([\nu^{\alpha}\tau, \nu^b \tau]; \sigma^{(0)}),
\]
noting they are generic for $i=1$ (by Corollary~\ref{Steinberggenericity}). We further note that for these groups, it follows from Proposition ~\ref{Steinberg} and Lemma~\ref{twisting} that
\[
s_{(\underbrace{m,\dots,m}_{b-\alpha+1})}\left(\delta_i([\nu^{\alpha}\tau, \nu^b \tau]; \chi_0\sigma^{(0)})\right)=\nu^b \tau \otimes \nu^{b-1}\tau \otimes \dots \otimes \nu^{\alpha}\tau \otimes \chi_0 \sigma^{(0)}.
\]
Similar notation is used for more general $\chi_0$ (not just that needed to ensure unitarity).
\end{rem}

In the classical case, the next lemma may be deduced from \cite{Mui04} or (via duality) \cite{Jan96a}.

\begin{lem}\label{dps}
Suppose $(\tau; \sigma^{(0)})$ satisfies (C1). Let $a,b \in {\mathbb N}$ with $a \leq b$. Then $\delta([\nu\tau, \nu^a \tau]) \rtimes \delta([\nu\tau, \nu^b \tau]; \sigma^{(0)})$ is irreducible.
\end{lem}

\noindent
\begin{proof}
We start with the case $a=b=1$, upon which the lemma is inductively based. The proof here follows that in \cite[Proposition 5.1]{Tad98a}. Note that to satisfy (C1), we must have $\tau \otimes \sigma^{(0)}$ ramified.

For $a=b=1$, we argue indirectly--suppose $\nu\tau \rtimes \delta(\nu\tau; \sigma^{(0)})$ were reducible. Letting
\[
\omega=\left\{ \begin{array}{l}
\omega_{\nu\tau} \mbox{ for }G_n=GSp_{2n}, GSO_{2n}, GSO_{2n+2}^{\ast}, GU_{2n+1}, GU_{2n},\\  
%\nu^2 \mbox{ for }G_n=GSO_{2n+2}^{\ast}, \\
1 \mbox{ otherwise}, \end{array} \right.
\]
we have
\[
s_{GL}(\nu\tau \rtimes \delta(\nu\tau; \sigma^{(0)}))=(\nu\tau \times \nu\tau) \otimes \sigma^{(0)}+(\nu\tau \times \nu^{-1}\tau) \otimes \omega\sigma^{(0)}.
\]
Therefore, were the induced representation reducible, one irreducible subquotient would satisfy $s_{GL}(\theta)=(\nu\tau \times \nu\tau) \otimes \sigma^{(0)}$. Note that by the Casselman criterion, $\theta$ is essentially square-integrable (and square-integrable for classical and general spin groups--see Table~\ref{tab:center2}).

Consider the induced representations $I_1=\tau \rtimes \theta$, $I_2=\delta([\nu^{-1}\tau, \nu\tau]) \rtimes \omega\sigma^{(0)}$, and $I_3=\nu\tau \times \nu\tau \times \tau \rtimes \sigma^{(0)}$. Note that $I_1,I_2 \leq I_3$. Observe that $\nu\tau \times \delta([\tau, \nu\tau]) \otimes \sigma^{(0)}$ appears with multiplicity two in $\mu^{\ast}(I_i)$ for $i=1,2,3$. If $\pi \leq I_1$ is an irreducible subquotient such that $\mu^{\ast}(\pi) \geq \nu\tau \times \delta([\tau, \nu\tau]) \otimes \sigma^{(0)}$, then $\pi \leq I_2$ as well. Observe that as $I_1$ is an essentially unitary representation (see Lemma~\ref{twisting} for similitude groups other than general spin groups), Frobenius reciprocity tells us $\mu^{\ast}(\pi) \geq \tau \otimes \theta$. However, as there are no terms of the form $\tau \otimes \dots$ in $\mu^{\ast}(I_2)$, we have a contradiction. Thus we must have had $\nu\tau \rtimes \delta(\nu\tau; \sigma^{(0)})$ irreducible.

%To address the $\omega=\nu$ for $GSp_{2n}$, etc.,  note that $\nu\tau \rtimes \delta(\nu\tau; \omega\sigma^{(0)})\cong \omega(\nu\tau \rtimes \delta(\nu\tau; \sigma^{(0)}))$, implying the irreducibility claimed.

We next show inductively that $\nu\tau \rtimes \delta([\nu\tau, \nu^b \tau]; \sigma^{(0)})$ is irreducible, with the base case $b=1$ done above. Observe that (using Proposition~\ref{Steinberg})
\begin{align*}
  &s_{(m)}(\nu\tau \rtimes \delta([\nu\tau, \nu^b \tau]; \sigma^{(0)}))\\
  &=\nu\tau \otimes \delta([\nu\tau, \nu^b \tau]; \sigma^{(0)})+\nu^{-1}\tau \otimes \delta([\nu\tau, \nu^b \tau]; \omega\sigma^{(0)}) \\
&+\nu^b \tau \otimes \nu\tau \rtimes \delta([\nu\tau, \nu^{b-1} \tau]; \sigma^{(0)}) \\
&=\theta+\theta'+\theta''  
\end{align*}
% \[
% \begin{array}{rl}
% s_{(1)}(\nu\tau \rtimes \delta([\nu\tau, \nu^b \tau]; \sigma^{(0)}))&=\nu\tau \otimes \delta([\nu\tau, \nu^b \tau]; \sigma^{(0)})+\nu^{-1}\tau \otimes \delta([\nu\tau, \nu^b \tau]; \omega\sigma^{(0)}) \\
% &+\nu^b \tau \otimes \nu\tau \rtimes \delta([\nu\tau, \nu^{b-1} \tau]; \sigma^{(0)}) \\
% &=\theta+\theta'+\theta''
% \end{array}
% \]
(defining $\theta,\theta',\theta''$). By inductive hypothesis, all three terms are irreducible. Let $\pi \leq \nu\tau \rtimes \delta([\nu\tau, \nu^b \tau]; \sigma^{(0)})$ be irreducible with $s_{(m)}(\pi) \geq \theta'$. Now,
\[
\begin{array}{c}
s_{(m,m)}(\pi) \geq s_{(m,m)}(\theta') \geq \nu^{-1}\tau \otimes \nu^b \tau \otimes \delta([\nu\tau, \nu^{b-1}\tau]; \sigma^{(0)}) \\
\Downarrow \\
s_{(2m)}(\pi) \geq (\nu^{-1}\tau \times \nu^b \tau) \otimes \delta([\nu\tau, \nu^{b-1}\tau]; \sigma^{(0)}) \\
\Downarrow \\
s_{(m)}\pi \geq \nu^b \tau \otimes ...
\end{array}
\]
noting the irreducibility of $\nu^{-1}\tau \times \nu^b \tau$. In particular, this forces $s_{(m)}(\pi) \geq \theta''$. If $b>2$, a similar argument starting with
 $s_{(m,m)}(\pi) \geq s_{(m,m)}(\theta'') \geq \nu^b \tau \otimes \nu \tau \otimes \delta([\nu\tau, \nu^{b-1}\tau]; \sigma^{(0)})$ tells us $s_{(m)}(\pi) \geq \theta$ as well. If $b=2$, we start with $s_{(m,m,m)}(\pi) \geq s_{(m,m,m)}(\theta'') \geq \nu^2\tau \otimes \nu \tau \otimes \nu\tau \otimes \sigma^{(0)}$. Then, $s_{(3m)}(\pi) \geq \delta([\nu\tau,\nu^2\tau]) \times \nu\tau \otimes \sigma^{(0)} \Rightarrow s_{(m)}(\pi) \geq \nu\tau \otimes \dots$, again giving $s_{(m)}(\pi) \geq \theta$. As we now have $s_{(m)}(\pi)$ accounting for the entire Jacquet module, irreducibility follows.

We now address $\delta([\nu\tau, \nu^a \tau]) \rtimes \delta([\nu\tau, \nu^b \tau]; \sigma^{(0)})$. Let $\pi \hookrightarrow \delta([\nu\tau, \nu^a \tau]) \rtimes \delta([\nu\tau, \nu^b \tau]; \sigma^{(0)})$ be irreducible.  Using the irreducibility already proved, we have
\[
\begin{array}{rl}
\pi &\hookrightarrow  \nu^a \tau \times \dots \times \nu^3\tau \times \nu^2\tau \times \nu\tau \rtimes \delta([\nu\tau, \nu^b \tau]; \sigma^{(0)}) \\
& \cong \nu^a \tau \times \dots \times \nu^3\tau \times \nu^2 \tau \times \nu^{-1}\tau \rtimes \delta([\nu\tau, \nu^b \tau]; \omega\sigma^{(0)}) \\
& \cong \nu^{-1} \tau \times \nu^a \tau \times \dots \times \nu^3\tau \times \nu^2 \tau \rtimes \delta([\nu\tau, \nu^b \tau]; \omega\sigma^{(0)}) \\
& \hookrightarrow \nu^{-1} \tau \times \nu^a \tau \times \dots \times \nu^3\tau \times \nu^2 \tau \times \delta([\nu\tau, \nu^b \tau]) \rtimes \omega\sigma^{(0)} \\
& \cong \nu^{-1} \tau \times \delta([\nu\tau, \nu^b \tau]) \times \nu^a \tau \times \dots \times \nu^3\tau \times \nu^2 \tau \rtimes \omega\sigma^{(0)} \\
& \cong \nu^{-1} \tau \times \delta([\nu\tau, \nu^b \tau]) \times \nu^a \tau \times \dots \times \nu^3\tau \times \nu^{-2} \tau \rtimes \omega^3\sigma^{(0)} \\
& \cong \nu^{-1} \tau \times \nu^{-2} \tau \times \delta([\nu\tau, \nu^b \tau]) \times \nu^a \tau \times \dots \times \nu^3 \tau \rtimes \omega^3\sigma^{(0)} \\
& \vdots \\
& \cong \nu^{-1}\tau \times \nu^{-2}\tau \times \dots \times \nu^{-a}\tau \rtimes \delta([\nu\tau, \nu^b \tau]) \rtimes  \omega^r\sigma^{(0)}
\end{array}
\]
($r=\frac{a(a+1)}{2}$). By \cite[Lemma 5.5]{Jan97} (see Lemma~\ref{Lemma 5.5} below), $\pi \hookrightarrow \lambda \times \delta([\nu\tau, \nu^b \tau]) \rtimes \omega^r\sigma^{(0)}$ for some irreducible $\lambda \leq \nu^{-1}\tau \times \nu^{-2}\tau \times \dots \times \nu^{-a}\tau$. Any $\lambda$ other than $\delta([\nu^{-a}\tau, \nu^{-1}\tau])$ would imply $r_{(m)}(\lambda) \geq \nu^{-x}\tau \dots$ for some $x \not=1$, hence $s_{(m)}(\pi) \geq \nu^{-x}\tau \otimes \dots$--a contradiction. Thus,
\[
\begin{array}{c}
\pi \hookrightarrow \delta([\nu^{-a}\tau, \nu^{-1}\tau]) \times \delta([\nu\tau, \nu^b \tau]) \rtimes \omega^r\sigma^{(0)} \\
\ \ \ \ \ \ \Downarrow \mbox{ (\cite[Lemma 5.5]{Jan97})} \\
\pi \hookrightarrow \delta([\nu^{-a}\tau, \nu^{-1}\tau]) \rtimes \phi
\end{array}
\]
for some irreducible $\phi \leq \delta([\nu\tau, \nu^b \tau]) \rtimes \omega^r\sigma^{(0)}$. Since the only term of the form $\delta([\nu^{-a}\tau, \nu^{-1}\tau]) \otimes \phi$ in $\mu^{\ast}(\delta([\nu^{-a}\tau, \nu^{-1}\tau]) \rtimes  \delta([\nu\tau, \nu^b \tau]; \sigma^{(0)}))$ is $\delta([\nu^{-a}\tau, \nu^{-1}\tau]) \otimes  \delta([\nu\tau, \nu^b \tau]; \omega^r\sigma^{(0)})$, Frobenius reciprocity tells us
\[
\pi \hookrightarrow \delta([\nu^{-a}\tau, \nu^{-1}\tau]) \times \delta([\nu\tau, \nu^b \tau]; \omega^r\sigma^{(0)}).
\]
By the Langlands classification,
\begin{align*}
 \,\,\pi
 &=L_{sub}( \delta([\nu^{-a}\tau, \nu^{-1}\tau]) \otimes \delta([\nu\tau, \nu^b \tau]; \omega_r\sigma^{(0)}))\\
& =L(\delta([\nu\tau, \nu^a \tau]) \otimes \delta([\nu\tau, \nu^b \tau]; \sigma^{(0)}))   
\end{align*}
% \[
% \pi=L_{sub}( \delta([\nu^{-a}\tau, \nu^{-1}\tau]) \otimes \delta([\nu\tau, \nu^b \tau]; \omega_r\sigma^{(0)}))
% =L(\delta([\nu\tau, \nu^a \tau]) \otimes \delta([\nu\tau, \nu^b \tau]; \sigma^{(0)}))
% \]
(\cite[Lemma 1.1]{Jan98}). Thus, as $\pi$ appears as both a subrepresentation and the unique irreducible quotient in $\delta([\nu\tau, \nu^a \tau]) \rtimes \delta([\nu\tau, \nu^b \tau]; \sigma^{(0)})$, we must have irreducibility, as needed.  \end{proof}

\begin{rem}
More generally, the same argument could be used to show that in the (C1) case, $\delta([\nu^c\tau, \nu^a \tau]) \rtimes \delta([\nu\tau, \nu^b \tau]; \sigma^{(0)})$ is irreducible when $1 \leq c \leq a \leq b$.
\end{rem}

\subsection{Square-integrable generic representations}\label{gdssection}

%Again, references to similar results in the literature. Proposition~\ref{gdsprop} below is analogous to \cite{T2002} but rephrased to accomodate the fact that the partial cuspidal support can change in the (CN) case.

The starting point for the classification of generic admissible representations is the classification of square-integrable generic representations, which we address in this section.

We begin with a general lemma, which is essentially a combination of \cite[Lemma 5.5]{Jan97} and a result of \cite{Rod73}, applied as in \cite[Lemma 1.1]{Mui98b}:

\begin{lem}\label{Lemma 5.5}
Suppose $\pi$ is an irreducible representation with $\pi \hookrightarrow i_{G,L}(\lambda)$. If $M>L$, we have the following:

\begin{enumerate}

\item There is an irreducible $\theta \leq i_{M,L}(\lambda)$ such that $\pi \hookrightarrow i_{G,M}(\theta)$.

\item If $\pi$ is generic (so that $\lambda$ generic), then $\theta$ must be the irreducible generic subquotient of $i_{M,L}(\lambda)$.

\end{enumerate}

\end{lem}

We record the following conditions on essentially square-integrable representations of general groups $\delta([\nu^{-a}\tau, \nu^b \tau])$ needed in the classification of square-integrable generic representations:

\begin{description}

\item[(DS1)]
If $(\tau; \sigma^{(0)})$ satisfies (C1), then $a \in {\mathbb N} \cup \{-1\}$.

\item[(DS2)]
If $(\tau; \sigma^{(0)})$ satisfies (C0), then $a \in {\mathbb Z}_{\geq 0}$.

\item[(DS3)]
If $(\tau; \sigma^{(0)})$ satisfies (C1/2), then $a \in -\frac{1}{2}+ {\mathbb Z}_{\geq 0}$.

\item[(DS4)]
If $(\tau; \sigma^{(0)})$ satisfies (CN), then $a \in {\mathbb Z}_{\geq 0}$.

\end{description}

\medskip

We retain the Jacquet module notation of (\ref{s notation}).

\begin{lem}\label{gdslemma}
Suppose $(\tau; \sigma^{(0)})$ satisfies (C$\alpha$) or (CN) with $\delta([\nu^{-a}\tau, \nu^b \tau])$ satisfying (the appropriate one of) (DS1)--(DS4) above. Let $\pi$ be the irreducible generic subquotient of $\delta([\nu^{-a}\tau, \nu^b \tau]) \rtimes \sigma^{(0)}$, with $b>a$. Then if $\eta \otimes \theta$ is an irreducible representation occurring in $s_{GL}
(\pi)$, we have
\begin{equation}\label{gdslemmaeq}
\eta \leq 2\sum_{i=-a}^{|a|} \delta([\nu^{-i+1}\tau, \nu^a \tau]) \times \delta([\nu^i \tau, \nu^b \tau]).
\end{equation}
\end{lem}

\noindent
\begin{proof} The only cases with $a<0$ have $a=-\frac{1}{2}$ (in the (C1/2) case) and $a=-1$ (in the (C1) case). In these situations, it follows from Corollary~\ref{Steinberggenericity} that $\pi$ is a generalized Steinberg representation. The result then follows from Proposition~\ref{Steinberg}. In the (C0) and (CN) cases, $a=0$ also corresponds to a generalized Steinberg; in the (C1/2) and (C1) cases, $a=0$ does not satisfy (DS1) or (DS3). Thus, we assume $a > 0$ below (so also have $|a|=a$).

As $b \geq a+1>1$, we have
\[
\begin{array}{c}
\pi \leq \delta([\nu^{-a}\tau, \nu^{\alpha-1}\tau]) \times \delta([\nu^{\alpha}\tau, \nu^b \tau]) \rtimes \sigma^{(0)} \\
\Downarrow \mbox{ (Lemma~\ref{Lemma 5.5} and Corollary~\ref{Steinberggenericity})} \\
\pi \leq \delta([\nu^{-a}\tau, \nu^{\alpha-1}\tau]) \times \delta_1([\nu^{\alpha}\tau, \nu^b \tau]; \sigma^{(0)}).
\end{array}
\]
By Proposition~\ref{Steinberg} and the $\mu^{\ast}$ structure, we immediately see that $\nu^{-a-1}\tau, \nu^{-a-2}\tau, \dots, \nu^{-b}\tau$ do not appear in the supercuspidal support of $\pi$. On the other hand, consider
\begin{align*}
   &\,s_{GL}(\pi) \leq s_{GL}(\delta([\nu^{-a}\tau, \nu^b \tau]) \rtimes \sigma^{(0)})\\
   &=\sum_{i=-a}^{b+1} \delta([\nu^{-i+1}\tau, \nu^a \tau]) \times \delta([\nu^i \tau, \nu^{b}\tau]) \otimes \omega_{i,a,\tau}(c^{m_{i,a,\tau}} \cdot \sigma^{(0)}), 
\end{align*}
% \[
% \mu^{\ast}_{GL}(\pi) \leq \mu^{\ast}_{GL}(\delta([\nu^{-a}\tau, \nu^b \tau]) \rtimes \sigma^{(0)})=\sum_{i=-a}^{b+1} \delta([\nu^{-i+1}\tau, \nu^a \tau]) \times \delta([\nu^i \tau, \nu^{b}\tau]) \otimes \omega_{i,a,\tau}(c^{m_{i,a,\tau}} \cdot \sigma^{(0)}),
% \]
with $\omega_{i,a,\tau}$  depending on the group (but nontrivial only for $G_n=GSp_{2n}$, $GSO_{2n}$, $GSO_{2n+2}^{\ast}$, $GU_{2n+1}$, and $GU_{2n}$) and $c^{m_{i,s,\tau}}$ relevant only for $SO_{2n+2}^{\ast}$, $GSO_{2n}$, $GSO_{2n+2}^{\ast}$, $GSpin_{2n}$, and $GSpin_{2n+2}^{\ast}$ in the (CN) case.
Here, terms from $\nu^{-a-1}\tau, \nu^{-a-2}\tau, \dots, \nu^{-b}\tau$ do appear when $i \geq a+2$. It follows that
\[
s_{GL}(\pi) \leq \sum_{i=-a}^{a+1} \delta([\nu^{-i+1}\tau, \nu^a \tau]) \times \delta([\nu^i \tau, \nu^{b}\tau]) \otimes \omega_{i,a,\tau}(c^{m_{i,a,\tau}} \cdot \sigma^{(0)}).
\]
Now, for $i=a+1$, we have
\begin{align*}
  &\,\delta([\nu^{-a}\tau, \nu^a \tau]) \times \delta([\nu^{a+1}\tau, \nu^b \tau])\\
&=\delta([\nu^{-a}\tau, \nu^b \tau])
+{\mathcal L}(\delta([\nu^{a+1}\tau, \nu^b \tau]) \otimes \delta([\nu^{-a}\tau, \nu^a \tau])).  
\end{align*}
% \[
% \delta([\nu^{-a}\tau, \nu^a \tau]) \times \delta([\nu^{a+1}\tau, \nu^b \tau])
% =\delta([\nu^{-a}\tau, \nu^b \tau])
% +{\mathcal L}(\delta([\nu^{a+1}\tau, \nu^b \tau]) \otimes \delta([\nu^{-a}\tau, \nu^a \tau])).
% \]
Observe that $\delta([\nu^{-a}\tau, \nu^b \tau])$ occurs in right-hand side of (\ref{gdslemmaeq}) as a second copy of the term corresponding to $i=-a$, so is not an issue. For ${\mathcal L}(\delta([\nu^{a+1}\tau, \nu^b \tau]) \otimes \delta([\nu^{-a}\tau, \nu^a \tau]))$ observe that any term in $r_{min}({\mathcal L}(\delta([\nu^{a+1}\tau, \nu^b \tau]) \otimes \delta([\nu^{-a}\tau, \nu^a \tau]))$ has exactly one $\nu^{-a}\tau,\nu^a \tau$, and $\nu^{a+1}\tau$ in its supercuspidal support, and the copy of $\nu^a\tau$ always appears to the left of the copy of $\nu^{a+1} \tau$. However, to be in $s_{GL}(\pi) \leq (\delta([\nu^{-a}\tau, \nu^b \tau]) \rtimes \sigma^{(0)})$, if there is only one copy of $\nu^a \tau$, it must appear to the right of $\nu^{a+1}\tau$. The lemma now follows. \end{proof}

\begin{defn}\label{def2}
Fix $\sigma^{(0)}$ and let $P'$ be a collection of  $\tau$'s such that $(\tau; \sigma^{(0)})$ satisfies (C$\alpha$) or (CN) and the appropriate one of (DS1)--(DS4) above. For each $\tau \in P'$, suppose we have a collection of segments $D_i(\tau)=[\nu^{-a_i(\tau)}\tau, \nu^{b_i(\tau)}]$, $i=1, 2, \ldots, e_{\tau}$, which satisfy
\[
a_1(\tau)<b_1(\tau)<a_2(\tau)<b_2(\tau)< \dots <a_{e_{\tau}}(\tau)<b_{e_{\tau}}(\tau).
\]
Let $X'=\{\tau \in P' | (\tau; \sigma^{(0)}) \text{ satisfies } (C1)\}$.
\end{defn}

\begin{prop}\label{gdsprop1}
With notation as above, the $\psi_a$-irreducible generic subquotient of
\[
\delta(\Delta_1) \times \dots \times \delta(\Delta_k)\rtimes \chi_0\sigma^{(0)}=
\left( \prod_{\tau \in P'}\prod_{i=1}^{e_{\tau}} \delta([\nu^{-a_i(\tau)}\tau, \nu^{b_i(\tau)}\tau]) \right) \rtimes \chi_0\sigma^{(0)}
\]
(defining $\Delta_1, \dots, \Delta_k$) is square-integrable, where $\chi_0$ ensures unitary central character as in Remark~\ref{dsrem}. We denote this representation by $\sigma^{(2)}=\delta(\Delta_1, \dots, \Delta_k; \chi_0\sigma^{(0)})_{\psi_a}$ and write $P'=P'(\sigma^{(2)})$, $X'=X'(\sigma^{(2)})$ for future reference.
\end{prop}

\noindent
\begin{proof}
Following \cite[Lemma 4.6]{Tad02}, we prove the following by induction on $k$:
\begin{equation}\label{Tadic}
\begin{array}{l}
\displaystyle{s_{GL}(\pi) \leq d\sum_{i_1=-a_1}^{|a_1|} \dots \sum_{i_k=-a_k}^{|a_k|}
\delta([\nu^{-i_1+1}\tau_1, \nu^{a_1}\tau_1]) \times \delta([\nu^{i_1}\tau_1, \nu^{b_1}\tau_1]) \times \dots} \\
\hfill\times
\displaystyle{\delta([\nu^{-i_k+1}\tau_k, \nu^{a_k}\tau_k]) \times \delta([\nu^{i_k}\tau_k, \nu^{b_k}\tau_k]) \otimes \omega_{i_1, \dots, i_k}(c^{m_{i_1, \dots, i_k}} \cdot \chi_0\sigma^{(0)})}
\end{array}
\end{equation}
for some $d$ depending on $\sigma^{(0)}$ and the segments, and $ \omega_{i_1, \dots, i_k}$, $c^{m_{i_1, \dots, i_k}}$ similar to their counterparts in the proof of Lemma~\ref{gdslemma} above. The case $k=1$ is covered by Lemma~\ref{gdslemma}.

Now, observe that for any $1 \leq j \leq k$,
\[
\begin{array}{c}
\pi \leq \delta(\Delta_j) \rtimes \left(\delta(\Delta_1) \times \dots \times \delta(\Delta_{j-1}) \times \delta(\Delta_{j+1}) \times\dots \times \delta(\Delta_k) \rtimes \chi_0\sigma^{(0)} \right) \\
\Downarrow \\
\pi \leq \delta(\Delta_j) \rtimes \lambda
\end{array}
\]
for some $\lambda \leq \delta(\Delta_1) \times \dots \times \delta(\Delta_{j-1}) \times \delta(\Delta_{j+1}) \times\dots \times \delta(\Delta_k) \rtimes \chi_0\sigma^{(0)}$. By genericity, it is the $\psi_a$-generic subquotient, $$\delta(\Delta_1, \dots \Delta_{j-1},\Delta_{j+1}, \dots, \Delta_k; \chi_0\sigma^{(0)})_{\psi_a},$$
essentially square-integrable by inductive hypothesis. Thus,
\[
s_{GL}(\pi)  \leq s_{GL}\left(\delta(\Delta_j) \rtimes \delta(\Delta_1, \dots \Delta_{j-1},\Delta_{j+1}, \dots, \Delta_k; \chi_0\sigma^{(0)})_{\psi_a}\right),
\]
where the right-hand side may be calculated by taking those terms in $N^{\ast}(\delta(\Delta_j)) \tilde{\rtimes} s_{GL}\left(\delta(\Delta_1, \dots \Delta_{j-1},\Delta_{j+1}, \dots, \Delta_k; \chi_0\sigma^{(0)})_{\psi_a}\right)$ having $N^{\ast}$ contribution of the form $\tau_1 \otimes \tau_2 \otimes 1$.
Then, for each such $j$, we get
\begin{align}\label{bound}
\begin{split}
&\,s_{GL}(\pi) \\
&\leq  d_j  \sum_{i_1=-a_1}^{|a_1|} \dots \sum_{i_{j-1}=-a_{j-1}}^{|a_{j-1}|}\sum_{i_{j}=-a_{j}}^{b_{j}+1}\sum_{i_{j+1}=-a_{j+1}}^{|a_{j+1}|}\dots\sum_{i_k=-a_k}^{|a_k|}\\
&  \hspace{.5cm} \delta([\nu^{-i_1+1}\tau_1, \nu^{a_1}\tau_1]) \times \delta([\nu^{i_1}\tau_1, \nu^{b_1}\tau_1]) \times \dots \times \delta([\nu^{-i_k+1}\tau_k, \nu^{a_k}\tau_k]) \\
& \hspace{.5cm} \times \delta([\nu^{i_k}\tau_k, \nu^{b_k}\tau_k]) \otimes \omega_{i_1,\dots,i_k}(c^{m_{i_1, \dots, i_k}} \cdot \chi_0\sigma^{(0)}).
\end{split}
\end{align}

Without loss of generality, we may assume that if $\tau_i=\tau_j$ for some $i<j$, then $a_i<b_i<a_j<b_j$. Looking at (\ref{bound}) for $j=1$ (noting $k>1$), we see that $\nu^{-a_k-1}\tau_k, \nu^{-a_k-2}\tau_k, \dots, \nu^{-b_k}\tau_k$ do not appear in the supercuspidal support of $\pi$. Therefore, looking at (\ref{bound}) for $j=k$, we may refine the bound by removing those terms which contain one of $\nu^{-a_k-1}\tau_k, \nu^{-a_k-2}\tau_k, \dots, \nu^{-b_k}\tau_k$, i.e., if $a_k>0$,  those terms having $i_k>a_k+1$. (If $a_k \leq 0$--which can happen if $\tau_i \not= \tau_k$ for any $i<k$--all but $i_k=|a_k|$ are removed and we immediately obtain the needed bound.) This gives
\begin{align}\label{bound2}
\begin{split}
&\,s_{GL}(\pi) \\
&\leq  d_k \sum_{i_1=-a_1}^{|a_1|} \dots \sum_{i_{k-1}=-a_{k-1}}^{|a_{k-1}|} \sum_{i_k=-a_k}^{a_k+1} \delta([\nu^{-i_1+1}\tau_1, \nu^{a_1}\tau_1]) \times \delta([\nu^{i_1}\tau_1, \nu^{b_1}\tau_1])\\
&\,  \times \cdots \times \delta([\nu^{-i_k+1}\tau_k, \nu^{a_k}\tau_k]) \times \delta([\nu^{i_k}\tau_k, \nu^{b_k}\tau_k]) \otimes \omega_{i_1, \dots, i_k}(c^{m_{i_1, \dots, i_k}} \cdot \chi_0\sigma^{(0)}).
\end{split}
\end{align}

The only terms in (\ref{bound2}) which are not part of (\ref{Tadic}) are those corresponding to $i_k=a_k+1$, i.e.,
\begin{align*}
d_k & \sum_{i_1=-a_1}^{|a_1|} \dots \sum_{i_{k-1}=-a_{k-1}}^{|a_{k-1}|}
\delta([\nu^{-i_1+1}\tau_1, \nu^{a_1}\tau_1]) \times \delta([\nu^{i_1}\tau_1, \nu^{b_1}\tau_1]) \\
& \hspace{.5cm}\times \cdots \times \delta([\nu^{-i_{k-1}+1}\tau_{k-1}, \nu^{a_{k-1}}\tau_{k-1}])  \\
& \hspace{.5cm} \times \delta([\nu^{i_{k-1}}\tau_{k-1}, \nu^{b_{k-1}}\tau_{k-1}])\times \delta([\nu^{-a_k}\tau_k, \nu^{a_k}\tau_k]) \\
&\hspace{.5cm} \times \delta([\nu^{a_k+1}\tau_k, \nu^{b_k}\tau_k]) \otimes \omega_{i_1, \dots, i_{k-1},a_k+1}(c^{m_{i_1, \dots, i_{k-1}, a_k+1}} \cdot\chi_0\sigma^{(0)}).
\end{align*}
Now, observe that
\begin{align*}
& \delta([\nu^{-a_k}\tau_k, \nu^{a_k}\tau_k]) \times \delta([\nu^{a_k+1}\tau_k, \nu^{b_k}\tau_k]) \\
& =\delta([\nu^{-a_k}\tau_k, \nu^{b_k}\tau_k])+{\mathcal L}_{sub}(\delta([\nu^{-a_k}\tau_k, \nu^{a_k}\tau_k]) \otimes \delta([\nu^{a_k+1}\tau_k, \nu^{b_k}\tau_k])).
\end{align*}
We note that in any term in the Jacquet module of ${\mathcal L}_{sub}(\delta([\nu^{-a_k}\tau_k, \nu^{a_k}\tau_k]) \otimes \delta([\nu^{a_k+1}\tau_k, \nu^{b_k}\tau_k]))$,
the copy of $\nu^{a_k}\tau_k$ always precedes the copy of $\nu^{a_k+1}\tau_k$; the opposite holds for $\delta([\nu^{-a_k}\tau_k, \nu^{b_k}\tau_k])$. Now, looking at (\ref{bound}) with $j=1$, we see the only terms there having $\nu^{-a_k}\tau_k$ in their supercuspidal support are those corresponding to $i_k=-a_k$, which all have the form $ \dots \times \delta([\nu^{-a_k}\tau_k, \nu^{b_k}\tau_k])$. In particular, the only copy of $\nu^{a_k+1}\tau_k$ in such a term always precedes the only copy of $\nu^{a_k}\tau_k$. This means that the terms above coming from ${\mathcal L}_{sub}(\delta([\nu^{-a_k}\tau_k, \nu^{a_k}\tau_k]) \otimes \delta([\nu^{a_k+1}\tau_k, \nu^{b_k}\tau_k])) \times \dots$ cannot contribute to $\mu^{\ast}_{GL}(\pi)$. Thus, removing those terms from (\ref{bound2}), we get
\begin{align*}
&\,s_{GL}(\pi)\\
&\leq  d_k \sum_{i_1=-a_1}^{|a_1|} \dots \sum_{i_{k-1}=-a_{k-1}}^{|a_{k-1}|} \sum_{i_k=-a_k}^{|a_k|}  \delta([\nu^{-i_1+1}\tau_1, \nu^{a_1}\tau_1]) \times \delta([\nu^{i_1}\tau_1, \nu^{b_1}\tau_1]) \\
& \mbox{\ \ \ \ \ \ }  \times \dots \times
\delta([\nu^{-i_k+1}\tau_k, \nu^{a_k}\tau_k]) \times \delta([\nu^{i_k}\tau_k, \nu^{b_k}\tau_k]) \\
& \mbox{\ \ \ \ \ \ }\otimes \omega_{i_1, \dots,i_k}(c^{m_{i_1, \dots, i_k}} \cdot \chi_0\sigma^{(0)}) \\
&+  d_k \displaystyle{\sum_{i_1=-a_1}^{|a_1|} \dots \sum_{i_{k-1}=-a_{k-1}}^{|a_{k-1}|}
\delta([\nu^{-i_1+1}\tau_1, \nu^{a_1}\tau_1]) \times \delta([\nu^{i_1}\tau_1, \nu^{b_1}\tau_1])} \\
& \mbox{\ \ \ \ \ \ } \times \dots \times
\delta([\nu^{-i_{k-1}+1}\tau_{k-1}, \nu^{a_{k-1}}\tau_{k-1}])
\times \delta([\nu^{i_{k-1}}\tau_{k-1}, \nu^{b_{k-1}}\tau_{k-1}])\\
& \mbox{\ \ \ \ \ \ } \times \delta([\nu^{-a_k}\tau_k, \nu^{b_k}\tau_k]) \otimes \omega_{i_1, \dots, i_{k-1},a_k+1}(c^{m_{i_1, \dots, i_{k-1},a_k+1}} \cdot \chi_0\sigma^{(0)}).
\end{align*}
If we take $i_k=-a_k$ in the first set of sums, we obtain the second set of sums. Therefore,
\begin{align*}
 &\,s_{GL}(\pi) \\
 &\leq  2d_k \sum_{i_1=-a_1}^{|a_1|} \dots \sum_{i_{k-1}=-a_{k-1}}^{|a_{k-1}|} \sum_{i_k=-a_k}^{|a_k|} \delta([\nu^{-i_1+1}\tau_1, \nu^{a_1}\tau_1]) \times \delta([\nu^{i_1}\tau_1, \nu^{b_1}\tau_1]) \\
& \mbox{\ \ \ \ } \times \dots \times
\delta([\nu^{-i_k+1}\tau_k, \nu^{a_k}\tau_k]) \times \delta([\nu^{i_k}\tau_k, \nu^{b_k}\tau_k]) \otimes \omega_{i_1, \dots, i_k}(c^{m_{i_1, \dots, i_k}} \cdot \chi_0\sigma^{(0)}),
\end{align*}
the needed inequality. Square-integrability is then immediate from the Casselman criterion. \end{proof}

\begin{rem}\label{Trem}
Observe that proof shows more: any irreducible subquotient appearing in both
\[
\delta(\Delta_1) \rtimes \delta(\Delta_2, \dots , \Delta_k; \sigma^{(0)})_{\psi_a}
\]
and
\[
\delta(\Delta_k) \rtimes \delta(\Delta_2, \dots, \Delta_{k-1}; \sigma^{(0)})_{\psi_a}
\]
is essentially square-integrable.
\end{rem}

We next turn to the task of showing that a square-integrable generic representation has the form $\delta(\Delta_1, \dots, \Delta_k; \chi_0\sigma^{(0)})_{\psi_a}$ as in Proposition~\ref{gdsprop1}. To start, we restrict the possible supercuspidal support in the next few lemmas. Note that Lemmas~\ref{dssupplem1}, \ref{dsirrlem}, and \ref{dssupplem2} apply to square-integrable representations in general, not just square-integrable generic representations. Thus we assume only that $\alpha \in \{0\} \cup \frac{1}{2}{\mathbb N}$ (known for many of the groups under consideration from \cite{Moe14} in characteristic zero and \cite{GL18} in positive characteristic).
 We also remark that the first condition in Lemma~\ref{dssupplem1} below is equivalent to $(\tau; \sigma^{(0)})$ does not satisfy (C$\alpha$) or (CN).

\begin{lem}\label{dssupplem1}
If
\[
\tau \not\cong \left\{ \begin{array}{l}
	\omega_{\sigma^{(0)}}\check{\tau} \mbox{ for }G_n=GSpin_{2n+1}, GSpin_{2n}, GSpin_{2n+2}^{\ast}, \\
	 \check{\tau} \mbox{ otherwise,}
\end{array}\right.
\] 
%$\tau \not\cong \check{\tau}$ (resp., $\tau \not\cong \omega_{\sigma^{(0)}}\check{\tau}$ for $G_n=GSpin_{2n+1}$ or $GSpin_{2n}$)
or $x \not\in \frac{1}{2}{\mathbb Z}$, then $\nu^x \tau$ does not appear in the supercuspidal support of a square-integrable representation.
\end{lem}

\noindent
\begin{proof} The proofs are essentially those in \cite{Tad98a} for classical groups and symplectic similitude groups or \cite{Asg02} for general spin groups (noting that the $\omega_{\sigma^{(0)}}$ is missing from \cite{Asg02} but the argument works the same way).

We first consider the case where $G_n \not=SO_{2n}$, $SO_{2n+2}^{\ast}$, $GSO_{2n}$, $GSO_{2n+2}^{\ast}$, $GSpin_{2n+1}$, $GSpin_{2n}$, or $GSpin_{2n+2}^{\ast}$. Now, suppose $\pi$ were a square-integrable representation with $\nu^x \tau$ in its supercuspidal support but either $\tau \not\cong \check{\tau}$ or $x \not\in \frac{1}{2}{\mathbb Z}$. Write
\[
\pi \hookrightarrow \nu^{x_1}\tau_1 \times \dots \times \nu^{x_n}\tau_n \rtimes \chi_0\sigma^{(0)}
\]
with $\tau_1, \dots, \tau_n,$ and $\sigma^{(0)}$ unitary supercuspidal; $\chi_0$ a character chosen to ensure $\pi$ has unitary central character (see Remark~\ref{dsrem}).
By commuting arguments, we may assume without loss of generality that
\[
\pi \hookrightarrow \nu^{x_1}\tau \times \dots \times \nu^{x_k}\tau \times \nu^{x_{k+1}}\check{\tau} \times \dots \times \nu^{x_m}\check{\tau} \times \Lambda \rtimes \chi_0\sigma^{(0)},
\]
where $x_i \equiv x \!\mod\! 1$ for $1 \leq i \leq k$, $x_i \equiv -x \! \mod \! 1$ for $k+1 \leq i \leq m$, and $\Lambda=\nu^{x_{m+1}}\tau_{m+1} \times \dots \times \nu^{x_n}\tau_n$ has $\nu^{x_i}\rho_i \not\in \{\nu^y \tau, \nu^{-y}\check\tau\}_{y \in x+{\mathbb Z}}$ for $m+1 \leq i \leq n$. Then, by the Casselman criterion (and Frobenius reciprocity), we have $x_1+ \dots +x_k>0$. Also, as (commuting)
\[
 \nu^{x_1}\tau \times \dots \times \nu^{x_k}\tau \times \nu^{x_{k+1}}\check{\tau} \times \dots \times \nu^{x_m}\check{\tau} \cong
\nu^{x_{k+1}}\check{\tau} \times \dots \nu^{x_m}\check{\tau} \times  \nu^{x_1}\tau \times \dots \times \nu^{x_k}\tau,
\]
we must have $x_{k+1}+ \dots +x_m>0$.

Now, suppose $x_1+ \dots +x_k  \geq x_{k+1}+ \dots +x_m$. Then,
\[
\begin{array}{rl}
\pi
& \hookrightarrow (\nu^{x_1}\tau \times \dots \times \nu^{x_k}\tau) \times (\nu^{x_{k+1}}\check{\tau} \times \dots \times \nu^{x_m}\check{\tau}) \times \Lambda \rtimes \sigma^{(0)} \\
& \\
& \cong \Lambda \times  (\nu^{x_{k+1}}\check{\tau} \times \dots \times \nu^{x_m}\check{\tau}) \times  (\nu^{x_1}\tau \times \dots \times \nu^{x_k}\tau) \rtimes \sigma^{(0)} \\
& \\
& \cong \Lambda \times  (\nu^{x_{k+1}}\check{\tau} \times \dots \times \nu^{x_m}\check{\tau}) \times  (\nu^{-x_k}\check{\tau} \times \dots \times \nu^{-x_1}\check{\tau}) \rtimes \omega\sigma^{(0)} \\
& \\
& \cong (\nu^{x_{k+1}}\check{\tau} \times \dots \times \nu^{x_m}\check{\tau}) \times  (\nu^{-x_k}\check{\tau} \times \dots \times \nu^{x_1}\check{\tau}) \times \Lambda \rtimes \omega\sigma^{(0)}
\end{array}
\]
for the appropriate $\omega$. Observe that $x_{k+1}+ \dots +x_m-x_k- \dots -x_1 \leq 0$, contradicting (via Frobenius reciprocity) the Casselman criterion. The argument if we originally had $x_1+ \dots +x_k \leq x_{k+1}+ \dots +x_m$ is similar but with $\nu^{x_{k+1}}\check{\tau}, \dots, \nu^{x_m}\check{\tau}$ getting inverted.

The cases where $G_n=SO_{2n}$, $SO_{2n+2}^{\ast}$, $GSO_{2n}, GSpin_{2n+1}$, $GSO_{2n+2}^{\ast}$, $GSpin_{2n}$, or $GSpin_{2n+2}^{\ast}$ are similar, but with (1) $\check{\tau}$ replaced by $\omega_{\sigma^{(0)}}\check{\tau}$ for $G_n=GSpin_{2n+1}$, $GSpin_{2n}$, or $GSpin_{2n+2}^{\ast}$ (noting that $\omega_{\sigma^{(0)}}\check{\tau}$ is still unitary supercuspidal), and (2) $\sigma^{(0)}$ possibly replaced by $c \sigma^{(0)}$ in some places for $G_n=SO_{2n}$, $SO_{2n+2}^{\ast}$, $GSO_{2n}$, $GSO_{2n+2}^{\ast}$, $GSpin_{2n}$, or $GSpin_{2n+2}^{\ast}$.
\end{proof}

For $G_n=Sp_{2n}$ or $SO_{2n+1}$, the following is a special case of  \cite[Theorem 9.1]{Tad98a}. 

\begin{lem}\label{dsirrlem}
\begin{enumerate}

\item If $(\tau; \sigma^{(0)})$ satisfies (C$\alpha$) with $\alpha\in -\frac{1}{2}+{\mathbb N}$, then $\delta([\nu^{-a}\tau, \nu^b \tau]) \rtimes \sigma^{(0)}$, $b>a$,  is irreducible when $a \in {\mathbb Z}$.

\item  If $(\tau; \sigma^{(0)})$ satisfies (CN) or (C$\alpha$) with $\alpha \in {\mathbb Z}_{\geq 0}$, then $\delta([\nu^{-a}\tau, \nu^b \tau]) \rtimes \sigma^{(0)}$, $b>a$,  is irreducible when $a \in \frac{1}{2}+{\mathbb Z}$.

\end{enumerate}
\end{lem}

\noindent
\begin{proof}
First, we address (1) with $a\leq -1$. Let $\pi=L_{sub}(\delta([\nu^{-b}\tau, \nu^a \tau]) \otimes \omega\sigma^{(0)})$, where $\omega$ is such that $\delta([\nu^{-b}\tau, \nu^a \tau]) \otimes \omega\sigma^{(0)}$ is conjugate to $\delta([\nu^{-a}\tau, \nu^b \tau]) \otimes \sigma^{(0)}$ (n.b.: $\tau \otimes \sigma^{(0)}$ ramified  ensures that such an $\omega$ exists; trivial for classical or general spin groups). Then,
\[
\begin{array}{rl}
\pi & \hookrightarrow \delta([\nu^{-b}\tau, \nu^{a}\tau]) \rtimes \omega\sigma^{(0)} \\
& \\
& \hookrightarrow \nu^{a}\tau \times \dots \times \nu^{-b+1}\tau  \times \nu^{-b}\tau \rtimes \omega\sigma^{(0)} \\
& \\
& \hookrightarrow \nu^{a}\tau \times \dots \times \nu^{-b+1}\tau  \times \nu^{b}\tau \rtimes \omega'\sigma^{(0)} \\
& \\
& \hookrightarrow \nu^{b}\tau \times \nu^{a}\tau \times \dots \times \nu^{-b+1}\tau \rtimes \omega'\sigma^{(0)},
\end{array}
\]
noting the irreducibility of $\nu^{b} \tau \rtimes \sigma^{(0)}$ and $\nu^b \tau \times \nu^x \tau$ for $x=a, a-1, \dots, -b+1$. Iterating this processs, we eventually arrive at
\[
\pi \hookrightarrow \nu^b \tau \times \nu^{b-1}\tau \dots \times \nu^{-a} \tau \rtimes \sigma^{(0)}.
\]
By Lemma~\ref{Lemma 5.5}, we have $\pi \hookrightarrow \lambda \rtimes \sigma^{(0)}$ for some irreducible $\lambda \leq \nu^b \tau \times \nu^{b-1} \tau \dots \times \nu^{-a}\tau \rtimes \sigma^{(0)}$. Any $\lambda$ other than $\delta([\nu^{-a}\tau, \nu^b \tau])$ would have $m^{\ast}(\lambda)$--hence $\mu^{\ast}(\pi)$--containing a term of the form $\nu^x \tau \otimes \dots$ with $x \in \{-a,-a+1,\dots, b-1\}$. As this is not possible, we have $\lambda=\delta([\nu^{-a}\tau, \nu^b \tau])$. This makes $\pi$ both a subrepresentation and unique irreducible quotient (Langlands quotient) of $\delta([\nu^{-a}\tau, \nu^b \tau]) \rtimes \sigma^{(0)}$, implying irreducibility. The argument for (2) with $a \leq-\frac{1}{2}$ is similar.

We now address (1) when $a=0$. In this case, there are terms in $\{ \tau, \nu\tau, \dots, \nu^b \tau\}$ and $\{\nu^{-b}\tau, \dots, \nu^{-1}\tau, \tau\}$ for which their product is reducible and the commuting argument above breaks down. We may repair the argument using Lemma~\ref{delicate}.
First, note that the result is immediate if $b=0$ and follows from Lemma~\ref{delicate} if $b=1$. Thus, assume $b \geq 2$. Then,
\[
\begin{array}{rl}
\pi & \hookrightarrow \delta([\nu^{-b}\tau, \tau]) \rtimes \omega \sigma^{(0)} \\
& \\
& \hookrightarrow  \delta([\nu^{-1}\tau, \tau]) \times \nu^{-2}\tau \times \dots \times \nu^{-b}\tau \rtimes \omega\sigma^{(0)} \\
& \\
& \cong \nu^b \tau \times \dots \times \nu^2 \tau \times \delta([\nu^{-1}\tau, \tau]) \rtimes \omega' \sigma^{(0)} \\
& \\
& \cong \nu^b \tau \times \dots \times \nu^2 \tau \times \delta([\tau, \nu\tau]) \rtimes \sigma^{(0)}
\end{array}
\]
using the commuting/inverting argument above in conjunction with Lemma~\ref{delicate}. The argument now concludes as above.

We now take up (1) when $a>0$. Using the case $a \leq 0$ above, we have
\[
\begin{array}{rl}
\pi & \hookrightarrow \delta([\nu^{-b}\tau, \nu^a \tau]) \rtimes \omega\sigma^{(0)} \\
& \\
& \hookrightarrow \delta([\nu\tau, \nu^a \tau]) \times \delta([\nu^{-b}\tau, \tau]) \rtimes \omega\sigma^{(0)} \\
& \\
& \cong \delta([\nu\tau, \nu^a \tau]) \times \delta([\tau, \nu^b \tau]) \rtimes \omega'\sigma^{(0)} \\
& \\
& \cong \delta([\tau, \nu^b \tau]) \times \delta([\nu\tau, \nu^a \tau]) \rtimes \omega'\sigma^{(0)} \\
& \\
& \cong \delta([\tau, \nu^b \tau]) \times \delta([\nu^{-a}\tau, \nu^{-1} \tau]) \rtimes \sigma^{(0)} \\
\end{array}
\]
for the appropriate $\omega'$. By Lemma~\ref{Lemma 5.5}, we have either
\[
\pi \hookrightarrow \delta([\nu^{-a}\tau, \nu^b \tau]) \rtimes \sigma^{(0)}
\]
or
\[
\pi \hookrightarrow {\mathcal L}_{sub}(\delta([\nu^{-a}\tau, \tau]) \otimes \delta([\nu\tau, \nu^b \tau])) \rtimes \sigma^{(0)}.
\]
As the latter would imply $\mu^{\ast}(\pi)$ contains terms of the form $\tau \otimes \dots$, which is not the case as $\mu^{\ast}(\pi) \leq N^{\ast}\left(\delta([\nu^{-a}\tau, \nu^b \tau])\right) \tilde{\rtimes} (1 \otimes \sigma^{(0)})$, it must be the former. Again, $\pi$ appears as both the unique irreducible quotient (Langlands quotient) and as a subrepresentation in $\delta([\nu^{-a}\tau, \nu^b \tau]) \rtimes \sigma^{(0)}$, implying irreducibility, as needed. \end{proof}

\begin{lem}\label{dssupplem2}
Suppose $(\tau; \sigma^{(0)})$ satisfies (C$\alpha$) or (CN). Then $\nu^x \tau$ can appear in the supercuspidal support of a square-integrable representation only if $x \in \alpha+{\mathbb Z}$ (resp., $x \in {\mathbb Z}$ in the (CN) case).
\end{lem}

\noindent
\begin{proof}
By Lemma~\ref{dssupplem1}, we need eliminate only the following possibilities: (1) $x \in {\mathbb Z}$ in the (C$\alpha$) case when $\alpha \in \frac{1}{2}+{\mathbb Z}_{\geq 0}$, and (2) $x \in \frac{1}{2}+{\mathbb Z}$ in the (CN) and (C$\alpha$), $\alpha \in {\mathbb Z}_{\geq 0}$, cases.

Suppose not and let $\pi$ be a square-integrable representation with such a $\nu^x \tau$ in its supercuspidal support. By a commuting argument (and Remark~\ref{dsrem}), we may write
\[
\pi\hookrightarrow \nu^{x_1}\tau \times \dots \times \nu^{x_k}\tau \times \nu^{x_{k+1}}\tau_{k+1} \times \dots \times \nu^{x_m}\tau_m \rtimes \chi_0\sigma^{(0)}
\]
where $x_1, \dots, x_k \in x+{\mathbb Z}$ and for each $k+1 \leq i \leq m$ we have either $\tau_i \not\cong \tau$ or $x_i \not\in x+{\mathbb Z}$. By Lemma~\ref{Lemma 5.5},
\[
\pi \hookrightarrow \Sigma \times \Lambda \rtimes \chi_0\sigma^{(0)}
\]
for some irreducible $\Sigma \leq \nu^{x_1}\tau \times \dots \times \nu^{x_k}\tau$ and $\Lambda \leq \nu^{x_{k+1}}\tau_1 \times \dots \times \nu^{x_{k+m}}\tau_m$. Further, by the Langlands classification for general linear groups (subrepresentation setting) and a commuting argument (or \cite[Section 2.2]{Jan00a}), we may write
\[
\Sigma \hookrightarrow \delta([\nu^{-a_1}\tau, \nu^{b_1}\tau]) \times  \delta([\nu^{-a_2}\tau, \nu^{b_2}\tau]) \times \dots \times  \delta([\nu^{-a_{\ell}}\tau, \nu^{b_{\ell}}\tau])
\]
with $b_1 \leq b_2 \leq \dots \leq b_{\ell}$.

We first claim $b_i>a_i$ for all $i$ with $1 \leq i \leq \ell$. Were this not the case, let $j$ be the smallest value such that $a_j>b_j$. Then, for $i<j$, we have $a_j>b_j\geq b_i \geq a_i$, hence $\delta([\nu^{-a_i}\tau, \nu^{b_i}\tau]) \times \delta([\nu^{-a_j}\tau, \nu^{b_j}\tau])$ irreducible. Commuting, we get
\[
\pi \hookrightarrow \delta([\nu^{-a_j}\tau, \nu^{b_j}\tau]) \times \delta([\nu^{-a_1}\tau, \nu^{b_1}\tau]) \times \dots
\]
contradicting the Casselman criterion. Thus $b_i>a_i$ for all $1 \leq i \leq \ell$.

We now apply Lemma~\ref{dsirrlem}:
\[
\begin{array}{rl}
\pi & \hookrightarrow \delta([\nu^{-a_1}\tau, \nu^{b_1}\tau]) \times \dots \times \delta([\nu^{-a_{\ell-1}}\tau, \nu^{b_{\ell-1}}\tau]) \\
& \mbox{\ \ \ }\times \delta([\nu^{-a_{\ell}}\tau, \nu^{b_{\ell}}\tau]) \times \Lambda \times \chi_0\sigma^{(0)} \\
& \hookrightarrow \delta([\nu^{-a_1}\tau, \nu^{b_1}\tau]) \times \dots \times \delta([\nu^{-a_{\ell-1}}\tau, \nu^{b_{\ell-1}}\tau]) \\
& \mbox{\ \ \ }\times \Lambda \times\delta([\nu^{-a_{\ell}}\tau, \nu^{b_{\ell}}\tau]) \times \chi_0\sigma^{(0)} \\
& \hookrightarrow \delta([\nu^{-a_1}\tau, \nu^{b_1}\tau]) \times \dots \times \delta([\nu^{-a_{\ell-1}}\tau, \nu^{b_{\ell-1}}\tau]) \\
& \mbox{\ \ \ }\times \Lambda \times\delta([\nu^{-b_{\ell}}\tau, \nu^{a_{\ell}}\tau]) \times \omega\chi_0 (c^r\sigma^{(0)}) \\
& \hookrightarrow \delta([\nu^{-a_1}\tau, \nu^{b_1}\tau]) \times \dots \times \delta([\nu^{-a_{\ell-1}}\tau, \nu^{b_{\ell-1}}\tau])\\
& \mbox{\ \ \ } \times\delta([\nu^{-b_{\ell}}\tau, \nu^{a_{\ell}}\tau]) \times \Lambda \times \omega \chi_0 (c^r \cdot \sigma^{(0)})
\end{array}
\]
for the appropriate $\omega$ (when $G_n=GSp_{2n},GSO_{2n},GSO_{2n+2}^{\ast}, GU_{2n+1}$ or $GU_{2n}$) and $c^r$ (for $G_n=SO_{2n}$, $SO_{2n+2}^{\ast}$, $GSO_{2n}$, $GSO_{2n+2}^{\ast}$, $GSpin_{2n}$, or $GSpin_{2n+2}^{\ast}$ in the (CN) case). If $\delta([\nu^{-a_{\ell-1}}\tau, \nu^{b_{\ell-1}}\tau]) \times\delta([\nu^{-b_{\ell}}\tau, \nu^{a_{\ell}}\tau])$ is irreducible, we may commute $\delta([\nu^{-b_{\ell}}\tau, \nu^{a_{\ell}}\tau])$ around $\delta([\nu^{-a_{\ell-1}}\tau, \nu^{b_{\ell-1}}\tau])$. If not, Lemma~\ref{Lemma 5.5} implies
\[
\begin{array}{r}
\pi \hookrightarrow \delta([\nu^{-a_1}\tau, \nu^{b_1}\tau]) \times \dots \times \delta([\nu^{-b_{\ell}}\tau, \nu^{b_{\ell-1}}\tau]) \times\delta([\nu^{-a_{\ell-1}}\tau, \nu^{a_{\ell}}\tau]) \\
\times \Lambda \times \omega\chi_0 (c^r \cdot \sigma^{(0)})
\end{array}
\]
or
\[
\begin{array}{r}
\pi \hookrightarrow \delta([\nu^{-a_1}\tau, \nu^{b_1}\tau]) \times \dots \times {\mathcal L}_{sub}(\delta([\nu^{-b_{\ell}}\tau, \nu^{a_{\ell}}\tau]) \otimes \delta([\nu^{-a_{\ell-1}}\tau, \nu^{b_{\ell-1}}\tau])) \\
\times \Lambda \times \omega\chi_0(c^r \cdot \sigma^{(0)}).
\end{array}
\]
In either case, we have
\[
\delta([\nu^{-a_1}\tau, \nu^{b_1}\tau]) \times \dots \times \delta([\nu^{-a_{\ell-2}}\tau, \nu^{b_{\ell-2}}\tau]) \times \delta([\nu^{-b_{\ell}}\tau, \nu^{b'_{\ell}}\tau]) \times \dots
\]
with $b'_{\ell-1} \leq b_{\ell}$. Iterating, we eventually arrive at
\[
\pi \hookrightarrow \delta([\nu^{-b_{\ell}}\tau, \nu^{b'_1}\tau]) \times \dots
\]
with $b'_1 \leq b_{\ell}$. However, this contradicts the Casselman criterion, finishing the proof. \end{proof}

\bigskip

For $\pi$ a square-integrable representation, write
\[
\pi \hookrightarrow \nu^{x_1}\rho_1 \times \dots \times \nu^{x_n}\rho_n \rtimes \chi_0\sigma^{(0)}
\]
with $x_1+ \dots +x_n$ as small as possible (as in \cite[Definition 4.1.1]{Jan00a}) and appropriate $\chi_0$ (see Remark~\ref{dsrem}). By Lemma~\ref{Lemma 5.5}, we have $\pi \hookrightarrow \phi \rtimes \chi_0\sigma^{(0)}$ for some irreducible $\phi \leq \nu^{x_1}\rho_1 \times \dots \times \nu^{x_n}\rho_n$. Write $\phi={\mathcal L}_{sub}(\delta([\nu^{-a_1}\tau_1, \nu^{b_1}\tau_1]) \otimes \dots \otimes \delta([\nu^{-a_k}\tau_k, \nu^{b_k}\tau_k]))$.
As the next results are for square-integrable generic representations only (not arbitrary square-integrable representations), we note that we may assume
\begin{align*}
 &\,\,{\mathcal L}_{sub}(\delta([\nu^{-a_1}\tau_1, \nu^{b_1}\tau_1]) \otimes \dots \otimes \delta([\nu^{-a_k}\tau_k, \nu^{b_k}\tau_k]))\\
&=\delta([\nu^{-a_1}\tau_1, \nu^{b_1}\tau_1]) \times \dots \times  \delta([\nu^{-a_k}\tau_k, \nu^{b_k}\tau_k])   
\end{align*}
is irreducible (by genericity). Then,
\begin{equation}\label{embedding}
\pi \hookrightarrow \delta([\nu^{-a_1}\tau_1, \nu^{b_1}\tau_1]) \times \dots \times  \delta([\nu^{-a_k}\tau_k, \nu^{b_k}\tau_k]) \rtimes \chi_0\sigma^{(0)}
\end{equation}
(with $(-a_1+ \dots+b_1)+\dots+(-a_k+\dots+b_k)$ minimal). Note that the argument in the proof of Lemma~\ref{dssupplem2} tells us $b_i>a_i$ for all $i$. Further, a commuting argument allows us to assume without loss of generality that $b_1 \leq b_2 \leq \dots \leq b_k$.

We also remark that for discrete series, one has $\beta=0$ (see Definition~\ref{beta def} and Lemma~\ref{beta}).

\begin{lem}\label{Dlemma}
Let $\pi$ be a square-integrable generic representation. With notation as above, $[\nu^{-a_i}\tau_i, \nu^{b_i}\tau_i]$ must satisfy one of (DS1)--(DS4).
\end{lem}

\noindent
\begin{proof}
Based on Lemmas~\ref{dssupplem1} and \ref{dssupplem2}, all that remains to be shown is the following:

\begin{enumerate}

\item $a_i \geq 0$ in the (C0) and (CN) cases.

\item $a_i \geq -\frac{1}{2}$ in the (C1/2) case, and

\item $a_i \geq -1$ and $a_i \not=0$ in the (C1) case.

\end{enumerate}

Fix a $\delta([\nu^{-a_i}\tau_i, \nu^{b_i}\tau_i])$ and let $\tau=\tau_i$.

For (2), we argue indirectly. Let $j$ be maximal with $\tau_j \cong \tau$ and $a_j<-\frac{1}{2}$. Then, $\nu^{-a_j}\tau_j \times \delta([\nu^{-a_\ell}\tau_\ell, \nu^{b_\ell}\tau_\ell])$ is irreducible for all $\ell>j$. Thus, writing $\Lambda=\delta([\nu^{-a_1}\tau_1, \nu^{b_1}\tau_1]) \times \dots \times \delta([\nu^{-a_{j-1}}\tau_{j-1}, \nu^{b_{j-1}}\tau_{j-1}])$, we have
\[
\begin{array}{rl}
\pi & \hookrightarrow \Lambda  \times \delta([\nu^{-a_j+1}\tau, \nu^{b_j}\tau]) \times \nu^{-a_j}\tau \times \delta([\nu^{-a_{j+1}}\tau_{j+1}, \nu^{b_{j+1}}\tau_{j+1}])\\
& \, \hspace{0.35in} \times \dots \times \delta([\nu^{-a_k}\tau_k, \nu^{b_k}\tau_k]) \rtimes \chi_0\sigma^{(0)} \\
& \cong \Lambda \times \delta([\nu^{-a_j+1}\tau, \nu^{b_j}\tau]) \times \delta([\nu^{-a_{j+1}}\tau_{j+1}, \nu^{b_{j+1}}\tau_{j+1}]) \\
& \, \hspace{0.3in} \times \dots \times \delta([\nu^{-a_k}\tau_k, \nu^{b_k}\tau_k])  \times \nu^{-a_j}\tau \rtimes \chi_0\sigma^{(0)} \\
& \cong \Lambda \times \delta([\nu^{-a_j+1}\tau, \nu^{b_j}\tau]) \times \delta([\nu^{-a_{j+1}}\tau_{j+1}, \nu^{b_{j+1}}\tau_{j+1}]) \\
& \, \hspace{0.3in}
\times \dots \times \delta([\nu^{-a_k}\tau_k, \nu^{b_k}\tau_k])  \times \nu^{a_j}\tau \rtimes \omega_j\chi_0\sigma^{(0)},
\end{array}
\]
noting the irreducibility of $\nu^{-a_j}\tau \rtimes \chi_0\sigma^{(0)} \cong \nu^{a_j}\tau \rtimes \omega_j \chi_0\sigma^{(0)}$. As $a_j<-a_j$, this contradicts the minimality of $x_1+\dots +x_n$, finishing this case. The same argument shows we cannot have $a_i<-1$ in the (C1) case or $a_i<0$ in the (C0),(CN) cases.

It remains to show that $a_i \not=0$ in the (C1) case. We first consider the case where $a_j \not=-1$ for all $j$ having $\tau_j \cong \tau$. Then, arguing indirectly, let $j$ be the largest such value for which $a_j=0$. Again,  writing $\Lambda=\delta([\nu^{-a_1}\tau_1, \nu^{b_1}\tau_1]) \times \dots \times \delta([\nu^{-a_{j-1}}\tau_{j-1}, \nu^{b_{j-1}}\tau_{j-1}])$, we have
\[
\begin{array}{rl}
\pi & \hookrightarrow \Lambda \times \delta([\tau, \nu^{b_j}\tau]) \times \delta([\nu^{-a_{j+1}}\tau_{j+1}, \nu^{b_{j+1}}\tau_{j+1}]) \times \dots  \\
& \hfill \times \delta([\nu^{-a_k}\tau_k, \nu^{b_k}\tau_k]) \rtimes \chi_0\sigma^{(0)} \\
& \\
\pi & \hookrightarrow \Lambda \times \delta([\nu^{-a_{j+1}}\tau_{j+1}, \nu^{b_{j+1}}\tau_{j+1}]) \times \dots \times \delta([\nu^{-a_k}\tau_k, \nu^{b_k}\tau_k]) \\
& \hfill \times \delta([\tau, \nu^{b_j}\tau]) \rtimes \chi_0\sigma^{(0)}
\end{array}
\]
noting that for all $\ell>j$ having $\tau_\ell \cong \tau$, we have $b_\ell \geq b_j$ and $a_\ell>a_j=0$ implying the irreducibility of $\delta([\tau, \nu^{b_j}\tau]) \times \delta([\nu^{-a_\ell}\tau, \nu^{b_\ell}\tau])$. By Lemma~\ref{Lemma 5.5},
\[
\pi \hookrightarrow \Lambda \times \delta([\nu^{-a_{j+1}}\tau_{j+1}, \nu^{b_{j+1}}\tau_{j+1}]) \times \dots \times \delta([\nu^{-a_k}\tau_k, \nu^{b_k}\tau_k]) \rtimes \theta
\]
for some irreducible $\theta \leq \delta([\tau, \nu^{b_j}\tau]) \rtimes \chi_0\sigma^{(0)}$. Now, noting that $$\delta([\nu\tau, \nu^{b_j}\tau]; \chi_0\sigma^{(0)})$$
is the irreducible generic subquotient of $\delta([\nu\tau, \nu^{b_j}\tau]) \rtimes \chi_0\sigma^{(0)}$ (Remark~\ref{dsrem}), we have
\[
\begin{array}{c}
\theta \leq \delta([\nu\tau, \nu^{b_j}\tau]) \times \tau \rtimes \chi_0\sigma^{(0)}=\tau \times \delta([\nu\tau, \nu^{b_j}\tau]) \rtimes \chi_0\sigma^{(0)} \\
\mbox{\ \ \ \ \ \ \ \ \ \ \ }\Downarrow \mbox{ (Lemma~\ref{Lemma 5.5})} \\
\theta \leq \tau \rtimes \delta([\nu\tau, \nu^{b_j}\tau]; \chi_0\sigma^{(0)}).
\end{array}
\]
As the induced representation is essentially unitary, $\tau \rtimes \delta([\nu\tau, \nu^{b_j}\tau];\chi_0 \sigma^{(0)})$ decomposes as a direct sum, so
\[
\begin{array}{c}
\theta \hookrightarrow \tau \rtimes \delta([\nu\tau, \nu^{b_j}\tau; \chi_0\sigma^{(0)}) \\
\Downarrow \\
\begin{array}{rl}
\pi & \hookrightarrow \Lambda \times \delta([\nu^{-a_{j+1}}\tau_{j+1}, \nu^{b_{j+1}}\tau_{j+1}]) \times \dots \times \delta([\nu^{-a_k}\tau_k, \nu^{b_k}\tau_k]) \times  \tau \\
& \hfill \rtimes \delta([\nu\tau, \nu^{b_j}\tau]; \chi_0\sigma^{(0)}) \\
& \cong \tau \times \Lambda \times \delta([\nu^{-a_{j+1}}\tau_{j+1}, \nu^{b_{j+1}}\tau_{j+1}]) \times \dots \times \delta([\nu^{-a_k}\tau_k, \nu^{b_k}\tau_k]) \\
& \hfill \rtimes \delta([\nu\tau, \nu^{b_j}\tau]; \chi_0\sigma^{(0)})
\end{array}
\end{array}
\]
(noting the irreducibility of $\tau \times \delta([\nu^{-a_\ell}\tau_\ell, \nu^{b_\ell}\tau_\ell])$ for $\ell<j$), contradicting the Casselman criterion.

%Now, suppose $a_j=-1$ for some $j$ with $\tau_j \cong \tau$.
To finish, we first establish that among the $j$ having $\tau_j \cong \tau$, there at most one index such $a_j=-1$. Suppose not and let $j_1<j_2$ be the two largest values such that $a_{j_1}=a_{j_2}=-1$ with $\tau_{j_1} \cong \tau_{j_2} \cong \tau$. Writing $\Lambda_1=\delta([\nu^{-a_1}\tau_1, \nu^{b_1}\tau_1]) \times \dots \times \delta([\nu^{-a_{j_1-1}}\tau_{j_1-1}, \nu^{b_{j_1-1}}\tau_{j_1-1}])$, $\Lambda_2=\delta([\nu^{-a_{j_1+1}}\tau_{j_1+1}, \nu^{b_{j_1+1}}\tau_{j_1+1}]) \times \dots \times \delta([\nu^{-a_{j_2-1}}\tau_{j_2-1}, \nu^{b_{j_2-1}}\tau_{j_2-1}])$, and $\Lambda_3=\delta([\nu^{-a_{j_2+1}}\tau_{j_2+1}, \nu^{b_{j_2+1}}\tau_{j_2+1}]) \times \dots \times \delta([\nu^{-a_k}\tau_k, \nu^{b_k}\tau_k])$, the usual commuting argument gives
\[
\begin{array}{rl}
\pi & \hookrightarrow \Lambda_1 \times \delta([\nu\tau, \nu^{b_{j_1}}\tau]) \times \Lambda_2 \times \delta([\nu\tau, \nu^{b_{j_2}}\tau]) \times \Lambda_3 \rtimes \chi_0\sigma^{(0)} \\
& \\
& \cong \Lambda_1 \times \Lambda_2 \times \Lambda_3 \times \delta([\nu\tau, \nu^{b_{j_1}}\tau]) \times \delta([\nu\tau, \nu^{b_{j_2}}\tau]) \rtimes \chi_0\sigma^{(0)} \\
& \hfill \Downarrow \mbox{ (Lemma~\ref{Lemma 5.5} and Corollary~\ref{Steinberggenericity})\ \ \ \ \ \ \ \ \ \ \ \ } \\
\pi & \hookrightarrow \Lambda_1 \times \Lambda_2 \times \Lambda_3 \times \delta([\nu\tau, \nu^{b_{j_1}}\tau]) \rtimes \delta([\nu\tau, \nu^{b_{j_2}}\tau]; \chi_0\sigma^{(0)}) \\
& \\
 & \hookrightarrow \Lambda_1 \times \Lambda_2 \times \Lambda_3 \times \delta([\nu^2\tau, \nu^{b_{j_1}}\tau]) \times \nu\tau \rtimes \delta([\nu\tau, \nu^{b_{j_2}}\tau]; \chi_0\sigma^{(0)}) \\
& \\
& \cong \Lambda_1 \times \Lambda_2 \times \Lambda_3 \times \delta([\nu^2\tau, \nu^{b_{j_1}}\tau]) \times \nu^{-1}\tau \rtimes \delta([\nu\tau, \nu^{b_{j_2}}\tau]; \omega\chi_0\sigma^{(0)}) \\
\end{array}
\]
by Lemma~\ref{dps} (twisted by $\chi_0$), for the appropriate $\omega$. However, this contradicts the minimality of $x_1+ \dots +x_n$. Thus there is at most one $j$ having $\tau_j \cong \tau$ and $a_j=-1$.

Finally, suppose we had $j_1,j_2$  with $\tau_{j_1}=\tau_{j_2}=\tau$, $a_{j_1}=0$ and $a_{j_2}=-1$; without loss of generality, assume $j_1$ maximal. As above, we may commute to get
\[
\pi \hookrightarrow  \Lambda_1 \times \Lambda_2 \times \Lambda_3 \times \delta([\tau, \nu^{b_{j_1}}\tau]) \times \delta([\nu\tau, \nu^{b_{j_2}}\tau]) \rtimes \chi_0\sigma^{(0)},
\]
noting that if $j_1>j_2$, then $\delta([\nu\tau, \nu^{b_{j_2}}\tau]) \times \delta([\tau, \nu^{b_{j_1}}\tau])$ is irreducible so may be commuted into the order above (though that is not crucial to the argument below). Again, by Lemma~\ref{Lemma 5.5} and Corollary~\ref{Steinberggenericity}, we have
\[
\pi \hookrightarrow \Lambda_1 \times \Lambda_2 \times \Lambda_3 \times \theta
\]
for some irreducible generic $\theta \leq \delta([\tau, \nu^{b_{j_1}}\tau]) \times \delta([\nu\tau, \nu^{b_{j_2}}\tau]) \rtimes \chi_0\sigma^{(0)}$. If $b_{j_1}>b_{j_2}$, the unique irreducible generic representation with this supercuspidal support is $\delta([\nu^{-b_{j_2}}\tau, \nu^{b_{j_1}}\tau]; \omega\chi_0\sigma^{(0)})_{\psi_a}$ for the appropriate $\omega$, so $\theta=\delta([\nu^{-b_{j_2}}\tau, \nu^{b_{j_1}}\tau]; \omega\chi_0\sigma^{(0)})_{\psi_a}$. Then,
\[
\begin{array}{rl}
\pi & \hookrightarrow \Lambda_1 \times \Lambda_2 \times \Lambda_3 \rtimes \delta([\nu^{-b_{j_2}}\tau, \nu^{b_{j_1}}\tau]; \omega\chi_0\sigma^{(0)})_{\psi_a} \\
& \\
& \hookrightarrow \Lambda_1 \times \Lambda_2 \times \Lambda_3 \times \delta([\nu^{-b_{j_2}}\tau, \nu^{b_{j_1}}\tau]) \rtimes \omega\chi_0\sigma^{(0)},
\end{array}
\]
which contradicts the minimality of $x_1+ \dots +x_n$. A similar argument applies if $b_{j_1}<b_{j_2}$. If $b_{j_1}=b_{j_2}=b$, the irreducible generic subquotient of $\delta([\tau, \nu^b\tau]) \times \delta([\nu\tau, \nu^b \tau]) \rtimes  \chi_0\sigma^{(0)}$ is also (by supercuspidal support considerations) the irreducible generic component of $\delta([\nu^{-b}\tau, \nu^b\tau]) \rtimes \omega \chi_0\sigma^{(0)}$ for the appropriate $\omega$. Noting the essential unitarity of this representation, we then have
\[
\pi \hookrightarrow \Lambda_1 \times \Lambda_2 \times \Lambda_3 \times \delta([\nu^{-b}\tau, \nu^b\tau]) \rtimes \omega \chi_0\sigma^{(0)},
\]
giving the same contradiction as above. This finishes the case where $a_j=-1$ for some $j$ having $\tau_j \cong \tau$, and the lemma. \end{proof}

%\begin{note}
%Be sure to discuss partial cuspidal support for $GSp_{2n}$. See Lemma~\ref{dssupplem2}.
%\end{note}

\begin{lem}\label{one segment}
Let $\pi=\delta([\nu^{-a}\tau, \nu^b \tau]; \chi_0\sigma^{(0)})_{\psi_a}$ as in Proposition~\ref{gdsprop1}. Then,
\[
\pi \hookrightarrow \delta([\nu^{-a}\tau, \nu^b \tau]) \rtimes \chi_0\sigma^{(0)}.
\]
Further, we have
\[
\pi \hookrightarrow \left\{
\begin{array}{l}
\delta([ \nu\tau, \nu^a \tau]) \times \delta([\nu\tau, \nu^b \tau]) \rtimes T_1(\tau; \omega_1\chi_0\sigma^{(0)}) \\
\mbox{ in the (C0) case,} \\
\delta([\nu^{\frac{1}{2}}\tau, \nu^a \tau]) \times \delta([\nu^{\frac{1}{2}}\tau, \nu^b \tau]) \rtimes \omega_2\chi_0\sigma^{(0)} \\
\mbox{ in the (C1/2) case,} \\
\delta([ \nu\tau, \nu^a \tau]) \times \delta([\nu\tau, \nu^b \tau]) \rtimes (\tau \rtimes \omega_1\chi_0c_1 \cdot \sigma^{(0)})\\
\mbox{ in the (C1) and (CN) cases,}
\end{array} \right.
\]
where $c_1,\omega_1, \omega_2$ are such that the inducing representation appears in the Jacquet module of $\delta([\nu^{-a}\tau, \nu^b \tau]) \rtimes \chi_0\sigma^{(0)}$. %In particular, letting
%\[
%\delta_i=\left\{ \begin{array}{l} \delta([\nu^{-a}\tau, \nu^{-1}\tau]) \mbox{ for (C0),C(1),(CN),} \\
%	\delta([\nu^{-a}\tau, \nu^{-\frac{1}{2}}\tau]) \mbox{ for (C1/2)},
%\end{array} \right.
%\]
%we have
%\[
%\omega_i=\left\{ \begin{array}{l}
%	1 \mbox{ for classical and general spin groups}, \\
%	\omega_{\delta_i} \mbox{ for }G_n=GSp_{2n}, GSO_{2n}, GSO_{2n+2}^{\ast}, GU_{2n+1}, GU_{2n}. \\
%\end{array} \right.
%\]
\end{lem}

\noindent
\begin{proof}
As in (\ref{embedding}), write
\[
\pi \hookrightarrow \delta([\nu^{-x_1}\tau, \nu^{y_1}\tau]) \times \dots \times \delta([\nu^{-x_s}\tau, \nu^{y_s}\tau]) \rtimes \chi'_0\sigma^{(0)}
\]
(with $(-x_1+ \dots +y_1)+ \dots +(-x_s+ \dots +y_s)$ minimal).
Recall that the irreducibility of $ \delta([\nu^{-x_1}\tau, \nu^{y_1}\tau]) \times \dots \times \delta([\nu^{-x_s}\tau, \nu^{y_s}\tau])$ and Casselman criterion tell us $y_i>x_i$ for all $i$.

We first argue that $s=1$. Let $m(t)$ be the number of times $\nu^{\pm t}\tau$ appears in  $\nu^{t_1}\tau \otimes \dots \otimes \nu^{t_r}\tau \otimes \omega'  \chi_0' (c' \cdot \sigma^{(0)}) \leq r_{M,G}(\pi)$, noting that this is well-defined and depends only on the supercuspidal support of $\pi$ (by the requirements for (C$\alpha$) and (CN), $\nu^x \tau$ becomes $\nu^{-x}\tau$ under block sign change). In the (C0) and (CN) cases, we have $m(0)=1$, from which we see that $s=1$ (noting Lemma~\ref{Dlemma}). In the (C1) case, we also have $m(0)=1$. This means we have $x_i \geq 0$ for at least one $i$; by Lemma~\ref{Dlemma}, we must then have $x_i>0$ for this $i$. Noting that $m(1)=2$, $\delta([\nu^{-x_i}\tau, \nu^{y_i} \tau])$ then accounts for both copies of $\nu^{\pm 1}\tau$ in the supercuspidal support. As any additional $\delta([\nu^{-x_j}\tau, \nu^{y_j}\tau])$ would increase $m(1)$, we see that we also have $s=1$ here. Finally, in the (C1/2) case, we have $m(\frac{1}{2})=2$. There are then two possibilities: (1) $s=1$ and $x_1>0$, or (2) $s=2$ and $x_i=-\frac{1}{2}$ for both $i$. To eliminate (2), observe that
\[
\begin{array}{rl}
\theta 
& \hookrightarrow \delta([\nu^{\frac{1}{2}}\tau, \nu^{y_1}\tau]) \times  \delta([\nu^{\frac{1}{2}}\tau, \nu^{y_2}\tau]) \rtimes \chi_0'\sigma^{(0)} \\
& \\
&\hookrightarrow \delta([\nu^{\frac{3}{2}}\tau, \nu^{y_1}\tau]) \times  \delta([\nu^{\frac{3}{2}}\tau, \nu^{y_2}\tau]) \times \nu^{\frac{1}{2}}\tau \times \nu^{\frac{1}{2}}\tau \rtimes \chi_0'\sigma^{(0)} \\
& \mbox{\ \ \ \ \ }\Downarrow \mbox{ (Lemma~\ref{Lemma 5.5})} \\
\theta & \hookrightarrow \delta([\nu^{\frac{3}{2}}\tau, \nu^{y_1}\tau]) \times  \delta([\nu^{\frac{3}{2}}\tau, \nu^{y_2}\tau]) \rtimes \xi
\end{array}
\]
for some irreducible generic $\xi \leq \nu^{\frac{1}{2}}\tau \times \nu^{\frac{1}{2}}\tau \rtimes \chi_0'\sigma^{(0)}$. As the generic component of $\nu^{\frac{1}{2}}\tau \times \nu^{\frac{1}{2}}\tau \rtimes \chi_0'\sigma^{(0)}$ is a subrepresentation of
$\delta([\nu^{-\frac{1}{2}}\tau, \nu^{\frac{1}{2}}\tau]) \rtimes \omega\chi_0'\sigma^{(0)}$ ($\omega$ trivial for classical or general spin groups, $\omega=\omega_{\nu^{-\frac{1}{2}}\tau}$ or $\omega_{\nu^{-\frac{1}{2}}\tau}^2$ for the other similitude groups), we have  $\xi \hookrightarrow \delta([\nu^{-\frac{1}{2}}\tau, \nu^{\frac{1}{2}}\tau]) \rtimes \omega\chi_0'\sigma^{(0)}$. Then,
\[
\theta \hookrightarrow  \delta([\nu^{\frac{3}{2}}\tau, \nu^{y_1}\tau]) \times  \delta([\nu^{\frac{3}{2}}\tau, \nu^{y_2}\tau]) \times \delta([\nu^{-\frac{1}{2}}\tau, \nu^{\frac{1}{2}}\tau]) \rtimes \omega\chi_0'\sigma^{(0)},
\]
which contradicts the minimality of $x_1+ \dots +x_n$. Thus $s=1$.

By supercuspidal support considerations, the only possibilities with $s=1$ are $\delta([\nu^{-x_1}\tau, \nu^{y_1} \tau])=\delta([\nu^{-a}\tau, \nu^{b} \tau])$ or $\delta([\nu^{-b}\tau, \nu^a \tau])$, and as it is clearly not the latter (e.g., by Proposition~\ref{gdsprop1} and the Casselman criterion), we have 
\[
\pi \hookrightarrow \delta([\nu^{-a}\tau, \nu^b\tau]) \rtimes \chi_0 \sigma^{(0)},
\]
as needed.

For the second embedding, we do the (C1) case; the remaining cases are similar. Observe that
\[
\begin{array}{c}
\pi \hookrightarrow \delta([\nu^{-a}\tau, \nu^b \tau]) \rtimes \chi_0\sigma^{(0)}
\hookrightarrow \delta([\nu\tau, \nu^b \tau]) \times \delta([\nu^{-a}\tau, \tau]) \rtimes \chi_0\sigma^{(0)} \\
\Downarrow \mbox{ (Lemma~\ref{Lemma 5.5})}\\
\pi \hookrightarrow \delta([\nu\tau, \nu^b \tau]) \times {\mathcal T}
\end{array}
\]
where ${\mathcal T}\leq \delta([\nu^{-a}\tau, \tau]) \rtimes \chi_0\sigma^{(0)}$ is the irreducible generic subquotient. Note that by supercuspidal support considerations, ${\mathcal T}$ is also the generic component of $\tau \rtimes \delta([\nu\tau, \nu^a \tau]; \omega_1\chi_0\sigma^{(0)})_{\psi_a}$. As the inducing representation is essentially unitary, we have ${\mathcal T} \hookrightarrow \tau \rtimes \delta([\nu\tau, \nu^a \tau];\omega_1\chi_0\sigma^{(0)})_{\psi_a}$. Thus,
\[
\begin{array}{c}
\pi \hookrightarrow
\delta([\nu\tau, \nu^b \tau]) \times \tau \times \delta([\nu\tau, \nu^a \tau]) \rtimes \omega_1\chi_0\sigma^{(0)} \\
\Downarrow \mbox{ (Lemma~\ref{Lemma 5.5})} \\
\begin{array}{rl}
\pi  \hookrightarrow &
\delta([\tau, \nu^b \tau]) \times \delta([\nu\tau, \nu^a \tau]) \rtimes \omega_1\chi_0\sigma^{(0)} \\
 \cong & \delta([\nu\tau, \nu^a \tau]) \times \delta([\tau, \nu^b \tau]) \rtimes \omega_1\chi_0\sigma^{(0)} \\
\hookrightarrow & \delta([\nu\tau, \nu^a \tau]) \times \delta([\nu\tau, \nu^b \tau]) \rtimes (\tau \rtimes \omega_1\chi_0 \sigma^{(0)})
\end{array} \\
\end{array}
\]
as claimed.
 \end{proof}

\begin{prop}\label{gdsprop2}
Let $\pi$ be an irreducible square-integrable  generic representation. Then $\pi$ is of the form $\delta(\Delta_1, \dots,  \Delta_k; \chi_0\sigma^{(0)})_{\psi_a}$ as in Proposition~\ref{gdsprop1}. Further,
\[
\pi \hookrightarrow \delta(\Delta_1) \times \dots\times \delta(\Delta_k) \rtimes \chi_0\sigma^{(0)}
\]
and the segments which appear are unique up to permutations of the $\Delta_i$.
\end{prop}

\noindent
\begin{proof}
Consider the embedding in (\ref{embedding}) (with the assumptions that $x_1+ \dots x_n$ is minimal and the $d_i$'s are nondecreasing) and recall that $a_i<b_i$ for all $i$. Fix a particular $\tau$ and via a commuting argument, write
\[
\pi \hookrightarrow \delta([\nu^{-c_1}\tau, \nu^{d_1}\tau]) \times \dots \times  \delta([\nu^{-c_{\ell}}\tau, \nu^{d_{\ell}}\tau]) \times \Lambda \rtimes \chi_0\sigma^{(0)},
\]
where $\Lambda$ contains all the $\delta([\nu^{-a_i}\tau_i, \nu^{b_i}\tau_i])$ having $\tau_i \not\cong \tau$ and $c_i,d_i$ are the $a_i,b_i$ for those $\tau_i$ having $\tau_i \cong \tau$.

As a result of  Lemma~\ref{Dlemma}, it suffices to show that $c_{i+1}>d_i$ for all $1\leq i \leq \ell-1$ (noting that if $\ell=1$ there is nothing to show). Suppose this were not the case and let $i$ be the largest index such that $d_i \geq c_{i+1}$. Write
\[
\pi \hookrightarrow \Lambda_1 \times \delta([\nu^{-c_i}\tau, \nu^{d_i}\tau]) \times \delta([\nu^{-c_{i+1}}\tau, \nu^{d_{i+1}}\tau]) \times \Lambda_2 \rtimes \chi_0\sigma^{(0)},
\]
where $\Lambda_1=\delta([\nu^{-c_1}\tau, \nu^{d_1}\tau]) \times \dots \times \delta([\nu^{-c_{j-1}}\tau, \nu^{d_{j-1}}\tau])$ and $$\Lambda_2=\delta([\nu^{-c_{i+2}}\tau, \nu^{d_{i+2}}\tau]) \times \dots \times \delta([\nu^{-c_{\ell}}\tau, \nu^{d_{\ell}}\tau]) \times \Lambda.$$
For $j>i+1$, we have $d_j>c_j > d_{i+1} > c_{i+1}$, so $\delta([\nu^{-c_{i+1}}\tau, \nu^{d_{i+1}}\tau]) \times \delta([\nu^{-c_j}\tau, \nu^{d_j}\tau])$ is irreducible. A similar argument applies to $$\delta([\nu^{-c_i}\tau, \nu^{d_i}\tau]) \times \delta([\nu^{-c_j}\tau, \nu^{d_j}\tau]).$$
Therefore, we may commute $\delta([\nu^{-c_i}\tau, \nu^{d_i}\tau])$ and $\delta([\nu^{-c_{i+1}}\tau, \nu^{d_{i+1}}\tau])$ to the right to get
\[
\begin{array}{c}
\pi \hookrightarrow \Lambda_1 \times \Lambda_2 \times \delta([\nu^{-c_i}\tau, \nu^{d_i}\tau]) \times \delta([\nu^{-c_{i+1}}\tau, \nu^{d_{i+1}}\tau])  \rtimes \chi_0\sigma^{(0)} \\
\mbox{\ \ \ \ \ \ }\Downarrow \mbox{ (Lemma~\ref{Lemma 5.5})} \\
\pi \hookrightarrow \Lambda_1 \times \Lambda_2 \rtimes \theta
\end{array}
\]
for some irreducible generic $\theta \leq \delta([\nu^{-c_i}\tau, \nu^{d_i}\tau]) \times \delta([\nu^{-c_{i+1}}\tau, \nu^{d_{i+1}}\tau])  \rtimes \chi_0\sigma^{(0)}$.

If $d_i=c_{i+1}$, write $d_i$ for both. Then,
\[
\begin{array}{rl}
\theta & \leq \delta([\nu^{-c_i}\tau, \nu^{d_i}\tau]) \times \delta([\nu^{-d_i}\tau, \nu^{d_{i+1}}\tau])  \rtimes \chi_0\sigma^{(0)} \\
& \\
& \leq \delta([\nu^{-d_i}\tau, \nu^{d_i}\tau]) \times \delta([\nu^{-c_i}\tau, \nu^{d_{i+1}}\tau]) \rtimes \chi_0\sigma^{(0)}.
\end{array}
\]
Now, $\delta([\nu^{-c_i}\tau, \nu^{d_{i+1}}\tau]; \chi_0\sigma^{(0)})_{\psi_a}$ is the irreducible generic subquotient of $\delta([\nu^{-c_i}\tau, \nu^{d_{i+1}}\tau]) \rtimes \chi_0\sigma^{(0)}$ (definition in Proposition~\ref{gdsprop1}). Thus,
\[
\theta \leq \delta([\nu^{-d_i}\tau, \nu^{d_i}\tau]) \times \delta([\nu^{-c_i}\tau, \nu^{d_{i+1}}\tau]; \chi_0\sigma^{(0)})_{\psi_a}.
\]
As the inducing representation is essentially unitary, we have
\[
\theta \hookrightarrow \delta([\nu^{-d_i}\tau, \nu^{d_i}\tau]) \times \delta([\nu^{-c_i}\tau, \nu^{d_{i+1}}\tau]; \chi_0\sigma^{(0)})_{\psi_a}.
\]
It then follows that
\[
\pi \hookrightarrow \Lambda_1 \times \Lambda_2 \times \delta([\nu^{-d_i}\tau, \nu^{d_i}\tau]) \times \delta([\nu^{-c_i}\tau, \nu^{d_{i+1}}\tau]; \chi_0\sigma^{(0)})_{\psi_a}.
\]
As $\delta([\nu^{-c_j}\tau, \nu^{d_j}\tau]) \times \delta([\nu^{-d_i}\tau, \nu^{d_i}\tau])$ is irreducible for all $j<i$ (as $c_j <d_j \leq d_i$), a commuting argument gives
\[
\pi \hookrightarrow \delta([\nu^{-d_i}\tau, \nu^{d_i}\tau]) \times \Lambda_1 \times \Lambda_2 \rtimes \delta([\nu^{-c_i}\tau, \nu^{d_{i+1}}\tau]; \chi_0\sigma^{(0)})_{\psi_a},
\]
which (by Frobenius reciprocity) contradicts the Casselman criterion for the square-integrability of $\pi$. Thus we cannot have $d_i=c_{i+1}$.

If $d_i>c_{i+1} \geq 0$, first suppose $c_{i+1}>c_i$. Then, we have
\[
\begin{array}{rl}
\theta & \leq \delta([\nu^{-c_i}\tau, \nu^{d_i}\tau]) \times \delta([\nu^{-c_{i+1}}\tau, \nu^{d_{i+1}}\tau])  \rtimes \chi_0\sigma^{(0)} \\
& \\
& \leq  \delta([\nu^{c_{i+1}+1}\tau, \nu^{d_i}\tau]) \times \delta([\nu^{-c_i}\tau, \nu^{c_{i+1}}\tau]) \times \delta([\nu^{-c_{i+1}}\tau, \nu^{d_{i+1}}\tau])  \rtimes \chi_0\sigma^{(0)} \\
& \\
& \leq  \delta([\nu^{-d_i}\tau, \nu^{-c_{i+1}-1}\tau]) \times \delta([\nu^{-c_i}\tau, \nu^{c_{i+1}}\tau]) \times \delta([\nu^{-c_{i+1}}\tau, \nu^{d_{i+1}}\tau]) \\
& \hfill \rtimes \omega'\chi_0 c'\cdot\sigma^{(0)} \\
& \mbox{\ \ \ \ \ \ \ \ }\Downarrow \mbox{ (Lemma~\ref{Lemma 5.5} and definition in Proposition~\ref{gdsprop1})}\\
\theta & =  \delta([\nu^{-c_i}\tau, \nu^{c_{i+1}}\tau], [\nu^{-d_i}\tau, \nu^{d_{i+1}}\tau];  \omega'\chi_0c' \cdot\sigma^{(0)})_{\psi_a},
\end{array}
\]
for the appropriate $\omega',c'$.
From Lemma~\ref{k=2} below, we then have
\[
\begin{array}{c}
\theta \hookrightarrow \delta([\nu^{-c_i}\tau, \nu^{c_{i+1}}\tau]) \times \delta([\nu^{-d_i}\tau, \nu^{d_{i+1}}\tau]) \rtimes \omega'\chi_0 c'\cdot\sigma^{(0)} \\
\Downarrow \\
\pi \hookrightarrow \Lambda_1 \times \Lambda_2 \times  \delta([\nu^{-c_i}\tau, \nu^{c_{i+1}}\tau]) \times \delta([\nu^{-d_i}\tau, \nu^{d_{i+1}}\tau]) \rtimes \omega'\chi_0 c'\cdot\sigma^{(0)}.
\end{array}
\]
However, this contradicts the minimality of $x_1+ \dots +x_n$. The argument if $c_i>c_{i+1}$ is similar, but with $\theta= \delta([\nu^{-c_{i+1}}\tau, \nu^{c_i}\tau], [\nu^{-d_i}\tau, \nu^{d_{i+1}}\tau];  \omega''\chi_0 c''\cdot\sigma^{(0)})_{\psi_a}$. The argument is also similar if $c_i=c_{i+1}$, but in this case, $\theta$ is the generic component of $\delta([\nu^{-c_i}\tau, \nu^{c_i}\tau]) \rtimes \delta_1([\nu^{-d_i}\tau, \nu^{d_{i+1}}\tau]; \chi_0 \sigma^{(0)})$. This no longer has $\theta$ essentially square-integrable, but one still has $\theta \hookrightarrow \delta([\nu^{-c_i}\tau, \nu^{c_i}\tau]) \rtimes \delta([\nu^{-d_i}\tau, \nu^{d_{i+1}}\tau]; \sigma^{(0)})_{\psi_a}$ by essential unitarity. Thus we have eliminated the possibility $d_i>c_{i+1} \geq 0$.

It now remains to eliminate the possibility $d_i>c_{i+1}$ with $c_{i+1}<0$. By Lemma~\ref{Dlemma}, the only such possibilities are $c_{i+1}=-\frac{1}{2}$ in the (C1/2) case and $c_{i+1}=-1$ in the (C1) case. In either case, we must have $c_i \geq c_{i+1}$. We first argue that we cannot have $c_i=c_{i+1}$. For the (C1) case, this is done in the proof of (3) in Lemma~\ref{Dlemma}. In the (C1/2) case, we have
\[
\pi \hookrightarrow \Lambda_1 \times \Lambda_2 \times \delta([\nu^{\frac{1}{2}}\tau, \nu^{d_i}\tau]) \times \delta([\nu^{\frac{1}{2}}\tau, \nu^{d_{i+1}}\tau]) \rtimes  \chi_0\sigma^{(0)}.
\]
By the definition in Proposition~\ref{gdsprop1}, the irreducible generic subquotient of $ \delta([\nu^{\frac{1}{2}}\tau, \nu^{d_i}\tau]) \rtimes \delta([\nu^{\frac{1}{2}}\tau, \nu^{d_{i+1}}\tau]) \rtimes  \chi_0\sigma^{(0)}$ is $ \delta([\nu^{-d_i}\tau, \nu^{d_{i+1}}\tau]; \chi_0\sigma^{(0)})_{\psi_a}$. Therefore,
\[
\begin{array}{rl}
 \pi & \hookrightarrow \Lambda_1 \times \Lambda_2 \rtimes \delta([\nu^{-d_i}\tau, \nu^{d_{i+1}}\tau]; \chi_0\sigma^{(0)})_{\psi_a} \\
& \\
& \hookrightarrow  \Lambda_1 \times \Lambda_2 \rtimes \delta([\nu^{-d_i}\tau, \nu^{d_{i+1}}\tau]) \rtimes \chi_0\sigma^{(0)}
\end{array} \]
by Lemma~\ref{one segment}, contradicting the minimality of $x_1+ \dots +x_n$. Thus, $c_i>c_{i+1}$. Now, observe that $d_i<d_{i+1}$ implies $\delta([\nu^{-c_i}\tau, \nu^{d_i}\tau]) \times \delta([\nu^{-c_{i+1}}\tau, \nu^{d_{i+1}}\tau])$ reducible, also contradicting the conditions in (\ref{embedding}). Thus we must have $d_i=d_{i+1}$. However, were this the case we would have (letting $\alpha=\frac{1}{2}$ or $1$, as appropriate)
\[
\pi \hookrightarrow \Lambda_1 \times \Lambda_2 \rtimes \theta
\]
for $\theta$ the irreducible generic subquotient of $\delta([\nu^{-c_i}\tau, \nu^{d_i}\tau]) \times \delta([\nu^{\alpha}\tau, \nu^{d_i}\tau]) \rtimes  \chi_0\sigma^{(0)}$. Then $\theta$ is also the generic component of the (essentially unitary) representation $\delta([\nu^{-d_i}\tau, \nu^{d_i}\tau]) \rtimes \delta([\nu^{\alpha}\tau, \nu^{c_i}\tau]);  \chi_0\sigma^{(0)})$. Thus,
\[
\pi \hookrightarrow \Lambda_1 \times \Lambda_2 \times
\delta([\nu^{-d_i}\tau, \nu^{d_i}\tau]) \rtimes \delta([\nu^{\alpha}\tau, \nu^{c_i}\tau]);  \chi_0\sigma^{(0)}),
\]
contradicting the minimality of $x_1+ \dots +x_n$ and finishing this case.

%Thus, $c_i=c_{i+1}$. In the (C1) case, we then have
%\[
%\pi \hookrightarrow \Lambda_1 \times \Lambda_2 \times %\delta([\nu\tau, \nu^{d_i}\tau]) \times \delta([\nu\tau, %\nu^{d_{i+1}}\tau]) \rtimes \chi_0\sigma^{(0)}.
%\]
%As the irreducible generic subquotient of $\delta([\nu\tau, %\nu^{d_{i+1}}\tau]) \rtimes \sigma^{(0)}$ is $\delta([\nu\tau, %\nu^{d_{i+1}}\tau]; \chi_0\sigma^{(0)})$ %(Corollary~\ref{Steinberggenericity}), we have
%\[
%\begin{array}{c}
%\pi \hookrightarrow \Lambda_1 \times \Lambda_2 \times %\delta([\nu\tau, \nu^{d_i}\tau]) \rtimes \delta([\nu\tau, %\nu^{d_{i+1}}\tau]; \chi_0\sigma^{(0)}) \\
%\Downarrow \mbox{ (Lemma~\ref{dps})} \\
%\pi \hookrightarrow \Lambda_1 \times \Lambda_2 \times %\delta([\nu^{-d_i}\tau, \nu^{-1}\tau]) \rtimes \delta([\nu\tau, %\nu^{d_{i+1}}\tau]; \omega\chi_0\sigma^{(0)})
%\end{array}
%\]
%for suitable $\omega$ (nontrivial only for similitude groups other than general spin groups). This contradicts the minimality of $x_1+ \dots +x_n$. 

We have now shown that $\pi$ embeds in an induced representation of the form in Proposition~\ref{gdsprop1}; that $\pi$ is the corresponding representation from Proposition \ref{gdsprop1} is immediate from genericity. It remains to show the uniqueness of the data claimed.

The argument for the uniqueness of the data follows the approach of \cite[Lemma 2.1.1]{Jan18} (based on \cite[Lemma 3.1]{Jan00b}), essentially the observation that the supercuspidal support determines the segment ends.
Suppose there were two such sets of data for $\pi$ and fix $\tau \in P$. Suppose one set has corresponding segments $[\nu^{-a_1(\tau)}, \nu^{b_1(\tau)}]$, $\ldots, [\nu^{-a_k(\tau)}, \nu^{b_k(\tau)}]$ and the other $[\nu^{-a_1'(\tau)}, \nu^{b_1'(\tau)}], \ldots, [\nu^{a_{\ell}'(\tau)}, \nu^{b_{\ell}'(\tau)}]$. As in Lemma~\ref{one segment}, let $m_{\tau}(x)$ be the number of times $\nu^{\pm x}\tau$ appears in a term in the minimal Jacquet module of $\pi$. It is not difficult, based on the conditions $a_1(\tau)<b_1(\tau)< \dots <a_k(\tau)<b_k(\tau) $ and similarly for $a_i'(\tau),b_i'(\tau)$--that the largest value of $x$ having $m_{\tau}(x)=1$ is $x=b_k(\tau)=b'_{\ell}(\tau)$. The largest $x$ having $m_{\tau}(x)=2$ is $x=a_k(\tau)=a'_{\ell}(\tau)$; if no such $x$ exists, $a_k(\tau)=a'_{\ell}(\tau)=\left\{\begin{array}{l} -\alpha \mbox{ for (C$\alpha$),} \\ 0 \mbox{ for (CN)}\end{array}\right.$ (and $k=\ell=1$). Iterating this argument gives the uniqueness of the segments claimed.
\end{proof}

\begin{rem} Observe that if 
$(\tau, \sigma^{(0)})$ satisfies (CN), then $$\delta([\tau, \nu\tau]; \sigma^{(0)})_{\psi_a} \cong \delta([\tau, \nu\tau]; c\sigma^{(0)})_{\psi_a},$$
so one may not have uniqueness of $\sigma^{(0)}$.
\end{rem}

\begin{lem}\label{k=2}
Suppose $a <b<c<d$ with $$\theta=\delta([\nu^{-a}\tau, \nu^b \tau], [\nu^{-c}\tau, \nu^d \tau]; \chi_0\sigma^{(0)})_{\psi_a}$$ as in Proposition~\ref{gdsprop1} (so (DS1)--(DS4) satisfied). Then,
\[
\theta \hookrightarrow \delta([\nu^{-a}\tau, \nu^b \tau]) \times \delta([\nu^{-c}\tau, \nu^d \tau]) \rtimes \chi_0\sigma^{(0)}.
\]
\end{lem}

\noindent
\begin{proof}
As in (\ref{embedding}), write
\[
\theta \hookrightarrow \delta([\nu^{-x_1}\tau, \nu^{y_1}\tau]) \times \dots \times \delta([\nu^{-x_s}\tau, \nu^{y_s}\tau]) \rtimes \chi_0'(c' \cdot \sigma^{(0)})
\]
with $\delta([\nu^{-x_1}\tau, \nu^{y_1}\tau]) \times \dots \times \delta([\nu^{-x_s}\tau, \nu^{y_s}\tau])$ irreducible (and $(-x_1+ \dots +y_1)+ \dots +(-x_s+ \dots +y_s)$ minimal). Further, by Lemma~\ref{Dlemma}, we have $x_i \geq -\alpha$ for (C$\alpha$) or $x_i \geq 0$ for (CN).

We next show that $s=2$. As in Lemma~\ref{one segment}, let $m(t)$ be the number of times $\nu^{\pm t}\tau$ appears in  $\nu^{t_1}\tau \otimes \dots \otimes \nu^{t_r}\tau \otimes \chi_0 \omega (d \cdot \sigma^{(0)}) \leq r_{M,G}(\theta)$. In the (C0) and (CN) cases, we have $m(0)=2$, from which we see that $s=2$. In the (C1) case, we have $m(0)=2$ if $a\not=-1$ (resp., $m(0)=1$ if $a=-1$). This means we have $x_i \geq 0$ for at least two $i$ when $a \not=-1$ (resp., at least one $i$ when $a=-1$); by Lemma~\ref{Dlemma}, we must then have $x_i>0$ for these $i$. Noting that $m(1)=4$ when $a \not=-1$ (resp., $m(1)=3$ when $a=-1$), these then account for all 4 (resp., all 3) copies of $\nu^{\pm 1}\tau$ in the supercuspidal support. As any additional $\delta([\nu^{-x_j}\tau, \nu^{y_j}\tau])$ would increase that value, we see that we also have $s=2$ here. Finally, in the (C1/2) case, first suppose $a \not=-\frac{1}{2}$. In this case, $m(\frac{1}{2})=4$. There are then three possibilities: (1) $s=2$ and $x_i>0$ for both $i$, (2) $s=3$ with $x_3>0$ and $x_1=x_2=-\frac{1}{2}$, or (3) $s=4$ and $x_i=-\frac{1}{2}$ for all $i$. To eliminate (2), observe that
\[
\begin{array}{rl}
\theta 
& \hookrightarrow \delta([\nu^{\frac{1}{2}}\tau, \nu^{y_1}\tau]) \times  \delta([\nu^{\frac{1}{2}}\tau, \nu^{y_2}\tau]) \times \delta([\nu^{-x_3}\tau, \nu^{y_3}\tau]) \rtimes \chi_0' (c' \cdot \sigma^{(0)}) \\
& \\
&\cong \delta([\nu^{-x_3}\tau, \nu^{y_3}\tau]) \times \delta([\nu^{\frac{1}{2}}\tau, \nu^{y_1}\tau]) \times  \delta([\nu^{\frac{1}{2}}\tau, \nu^{y_2}\tau]) \rtimes \chi_0' (c' \cdot \sigma^{(0)}) \\
& \Downarrow \mbox{ (Lemma~\ref{Lemma 5.5} and definition in Proposition~\ref{gdsprop1})} \\
\theta & \hookrightarrow \delta([\nu^{-x_3}\tau, \nu^{y_3}\tau]) \rtimes \delta([\nu^{-y_1}\tau, \nu^{y_2}\tau]; \omega'\chi_0' (c'' \cdot \sigma^{(0)}))_{\psi_a} \\
& \Downarrow \mbox{ (Lemma~\ref{one segment})} \\
\theta & \hookrightarrow \delta([\nu^{-x_3}\tau, \nu^{y_3}\tau]) \times \delta([\nu^{-y_1}\tau, \nu^{y_2}\tau]) \rtimes \omega' \chi_0' (c'' \cdot \sigma^{(0)})
\end{array}
\]
which contradicts the minimality of $x_1+ \dots +x_n$. We may eliminate (3) similarly, leaving $s=2$ as the only possibility. The argument when $a=-\frac{1}{2}$ is similar but somewhat easier as there are fewer cases to consider; we omit the details.

Next, observe that $x_1,x_2,y_1,y_2$ must be $a,b,c,d$ in some order by supercuspidal support considerations (as the values of $m$ change at the segment ends; see \cite[Lemma 3.1]{Jan00b} for a more general version of this observation). To satisfy $\delta([\nu^{-x_1}\tau, \nu^{y_1}\tau]) \times \delta([\nu^{-x_1}\tau, \nu^{y_1}\tau])$ irreducible and the Casselman criterion, there are only two possibilities: $\delta([\nu^{-x_1}\tau, \nu^{y_1}\tau]) \otimes \delta([\nu^{-x_2}\tau, \nu^{y_2}\tau])=\delta([\nu^{-a}\tau, \nu^{b}\tau]) \otimes \delta([\nu^{-c}\tau, \nu^{d}\tau])$ or $\delta([\nu^{-a}\tau, \nu^{c}\tau]) \otimes \delta([\nu^{-b}\tau, \nu^{d}\tau])$. To see that the latter does not hold, suppose it did. In the (C1) case, one then has
\[
\begin{array}{rl}
\theta & \hookrightarrow \delta([\nu^{-a}\tau, \nu^{c}\tau]) \times \delta([\nu^{-b}\tau, \nu^{d}\tau]) \rtimes \chi_1 \sigma^{(0)} \\
& \\
& \hookrightarrow \delta([\nu\tau, \nu^{c}\tau]) \times \delta([\nu^{-a}\tau, \tau]) \times \delta([\nu^{-b}\tau, \nu^{d}\tau]) \rtimes \chi_1 \sigma^{(0)} \\
& \\
& \cong \delta([\nu\tau, \nu^{c}\tau]) \times \delta([\nu^{-b}\tau, \nu^{d}\tau]) \times \delta([\nu^{-a}\tau, \tau]) \rtimes \chi_1 \sigma^{(0)} \\
& \\
&  \hookrightarrow  \delta([\nu\tau, \nu^{c}\tau]) \times  \delta([\tau, \nu^{d}\tau]) \times  \delta([\nu^{-b}\tau,\nu^{-1}\tau]) \times \delta([\nu^{-a}\tau, \tau]) \rtimes \chi_1 \sigma^{(0)} \\
& \mbox{\ \ \ \ \ \ \ \ }\Downarrow \mbox{(Lemma~\ref{Lemma 5.5})} \\
&  \hookrightarrow  \delta([\nu\tau, \nu^{c}\tau]) \times  \delta([\tau, \nu^{d}\tau]) \rtimes  \delta([\nu^{-a}\tau, \nu^{b}\tau]; \omega_1\chi_1 \sigma^{(0)})_{\psi_a} \\
& \mbox{\ \ \ \ \ \ \ \ } \Downarrow \mbox{(Lemma~\ref{one segment})} \\
&  \hookrightarrow  \delta([\nu\tau, \nu^{c}\tau]) \times  \delta([\tau, \nu^{d}\tau]) \times \delta([\nu\tau, \nu^{a}\tau]) \times  \delta([\nu\tau, \nu^{b}\tau])  \rtimes (\tau \rtimes \omega_2\chi_1 \sigma^{(0)}) \\
& \\
&  \cong  \delta([\nu\tau, \nu^{a}\tau]) \times  \delta([\nu\tau, \nu^{b}\tau]) \times \delta([\nu\tau, \nu^{c}\tau]) \times  \delta([\tau, \nu^{d}\tau])  \rtimes (\tau \rtimes \omega_2\chi_1 \sigma^{(0)})\\
& \mbox{\ \ \ \ \ \ \ \ }\Downarrow \mbox{(Lemma~\ref{Lemma 5.5})} \\
&  \hookrightarrow  \delta([\nu\tau, \nu^{a}\tau]) \times  \delta([\nu\tau, \nu^{b}\tau]) \rtimes  T_1([\nu^{-c}\tau, \nu^{d}\tau]; \tau \rtimes \omega_3\chi_1 \sigma^{(0)})
\end{array}
\]
where $T_1([\nu^{-c}\tau, \nu^d \tau]; \tau \rtimes \omega_3 \chi_1 \sigma^{(0)})$ is the generic component of 
$$\delta([\nu\tau, \nu^{c}\tau]) \times  \delta([\tau, \nu^{d}\tau])  \rtimes (\tau \rtimes \omega_3 \chi_1 \sigma^{(0)}).$$
By supercuspidal support considerations, this is also the generic component of $\tau \rtimes \delta([\nu^{-c}\tau, \nu^d \tau]; \omega_3 \chi_1 \sigma^{(0)})_{\psi_a}$. By essential unitarity,
\[
T_1([\nu^{-c}\tau, \nu^d \tau]; \tau \rtimes \omega_3 \chi_1 \sigma^{(0)}) \hookrightarrow \tau \rtimes \delta([\nu^{-c}\tau, \nu^d \tau]; \omega_3 \chi_1 \sigma^{(0)})_{\psi_a}.
\]
Thus,
\[
\begin{array}{rl}
\theta & \hookrightarrow \delta([\nu\tau, \nu^{a}\tau]) \times  \delta([\nu\tau, \nu^{b}\tau]) \times \tau  \rtimes \delta([\nu^{-c}\tau, \nu^d \tau]; \omega_3 \chi_1 \sigma^{(0)})_{\psi_a} \\
& \mbox{\ \ \ \ \ \ \ \ } \Downarrow \mbox{ (Lemma~\ref{one segment})} \\
&  \hookrightarrow  \delta([\nu\tau, \nu^{a}\tau]) \times  \delta([\nu\tau, \nu^{b}\tau]) \times  \tau \times \delta([\nu^{-c}\tau, \nu^{d}\tau]) \rtimes \omega_3 \chi_1 \sigma^{(0)} \\
& \\
& \cong \delta([\nu^{-c}\tau, \nu^{d}\tau])  \times  \delta([\nu\tau, \nu^{a}\tau]) \times  \delta([\nu\tau, \nu^{b}\tau]) \times \tau \rtimes \omega_3 \chi_1  \sigma^{(0)} \\
& \mbox{\ \ \ \ \ \ \ \ }\Downarrow \mbox{(Lemma~\ref{Lemma 5.5})} \\
&  \hookrightarrow  \delta([\nu^{-c}\tau, \nu^{d}\tau]) \rtimes  \delta_1([\nu^{-a}\tau, \nu^{b}\tau]; \chi_0\sigma^{(0)}) \\
&\mbox{\ \ \ \ \ \ \ \ } \Downarrow \mbox{ (Lemma~\ref{one segment})} \\
\theta & \hookrightarrow \delta([\nu^{-c}\tau, \nu^{d}\tau]) \times \delta([\nu^{-a}\tau, \nu^{b}\tau]) \rtimes \chi_0 \sigma^{(0)},
\end{array}
\]
implying the needed embedding. The (C0), (C1/2), and  (CN) cases are argued similarly. \end{proof}

%\begin{note}
%It may be worth mentioning that $\delta([-a,b]; \sigma)=\delta([-a,b]; \omega_{\tau}\sigma)$.

%Also, we had originally argued as follows in Lemma~\ref{k=2}:
%As in  Equation (\ref{embedding}), we have
%\[
%\theta \hookrightarrow \delta([\nu^{-x_1}\tau, \nu^{y_1}\tau]) \times \dots \times \delta([\nu^{-x_s}\tau, \nu^{y_s}\tau]) \rtimes \sigma^{(0)}
%\]
%with $\delta([\nu^{-x_1}\tau, \nu^{y_1}\tau]) \times \dots \times \delta([\nu^{-x_s}\tau, \nu^{y_s}\tau]) \leq \nu^{t_1}\tau \times \dots \times %\nu^{t_r}\tau$ and irreducible. Note that by the Casselman criterion and a commutativity argument, we must have $y_i>x_i$ for all $i$.
%Elsewhere, we used a longer argument to obtain the same result--could it be shortened?

%It probably needs to be clearer that the minimality of $x_1+ \dots +x_n$ is part of Equation (\ref{embedding}).
%\end{note}

We close this section by discussing essentially square-integrable representations for the similitude cases. Let $\sigma^{(e2)}$ be an irreducible, generic, essentially square-integrable representation and write $\sigma^{(e2)}=\chi\sigma^{(2)}$ for some character $\chi$; without loss of generality, we may assume $\chi=|\cdot|^s$, $s \in {\mathbb R}$. Now,
\[
\sigma^{(2)}=\delta(\Delta'_1, \dots, \Delta'_k; \chi_0\sigma^{(0)})_{\psi_a},
\]
for some $\Delta'_1, \dots, \Delta'_k, \chi_0\sigma^{(0)}$ as above. By Lemma~\ref{twisting} and Proposition~\ref{gdsprop1}, if $\sigma^{(0)}$ is a representation of $G_{n_0}(F)$ (and recalling $n_0 \not=1$ for $GSO_{2n}$ and $GSpin_{2n}$),
\begin{align*}
\chi\sigma^{(2)} & \hookrightarrow \chi(\delta(\Delta'_1) \times \dots \times \delta(\Delta'_k) \rtimes \chi_0 \sigma^{(0)}) \\
& \cong
	\left\{ \begin{array}{l}
	\delta(\Delta'_1) \times \dots \times \delta(\Delta'_k) \rtimes \chi\chi_0\sigma^{(0)} \\
	\mbox{ if } G_n=GSp_{2n}, GSO_{2n}, GSO_{2n+2}^{\ast}, GU_{2n+1}, GU_{2n},\\
	\chi\delta(\Delta'_1) \times \dots \times  \chi\delta(\Delta'_k) \rtimes \chi\chi_0\sigma^{(0)} \\
	\mbox{ if } G_n=GSpin_{2n+1}, GSpin_{2n} \mbox{ with }n_0>0, \\
	\chi\delta(\Delta'_1) \times \dots \times \chi\delta(\Delta'_k) \rtimes \chi^2\chi_0\sigma^{(0)}  \\
	\mbox{ if } G_n=GSpin_{2n+1}, GSpin_{2n} \mbox{ with }n_0=0, \\
	\chi\delta(\Delta'_1) \times \dots \times  \chi\delta(\Delta'_k) \rtimes \chi\chi_0\sigma^{(0)} \\
	\mbox{ if } G_n=GSpin_{2n+2}^{\ast} \mbox{ (any $n_0$)}. \\
\end{array} \right.
\end{align*}
Letting $\delta(\Delta_i)$ be $\delta(\Delta'_i)$ or $\chi\delta(\Delta'_i)$ as needed, and $\sigma^{(e0)}$ be $\chi\chi_0\sigma^{(0)}$ or $\chi^2\chi_0\sigma^{(0)}$ as appropriate, we write
\begin{equation}\label{essentially sq int}
\sigma^{(e2)}=\delta(\Delta_1, \dots, \Delta_k, \sigma^{(e0)})_{\psi_a}.
\end{equation}

For groups other than $GSpin_{2n+1}$, $GSpin_{2n}$, and $GSpin_{2n+2}^{\ast}$, we still have $\Delta_1, \dots, \Delta_k$ satisfying {\bf (DS1)--(DS4)}. For $GSpin_{2n+1}$, $GSpin_{2n}$, and $GSpin_{2n+2}^{\ast}$, we note that $\chi_0$ is trivial. Writing $\chi=|\cdot|^s$ with $s \in {\mathbb R}$ as above, we have $\omega_{\sigma^{(e0)}}=|\cdot|^s\omega_{\sigma^{(0)}}$ unless $n_0=0$ (where $\sigma^{(0)}$ is a representation of $G_{n_0}$) and the general spin group is split. This implies the
conditions {\bf (DS1)--(DS4)} are shifted up by the exponent of $\sigma^{(e0)}$. When $n_0=0$ and the group is split, the shift is half the exponent.
%\[
%a_0=\left\{  \begin{array}{l}
%	a \mbox{ if }G_n \mbox{ not a general spin group}, \\
%	a-\varepsilon(\sigma^{(e0)}) \mbox{ if }G_n=GSpin_{2n+1} \mbox{ or }GSpin_{2n} \mbox{ with }n_0>0, \\
%	a-\frac{1}{2}\varepsilon(\sigma^{(e0)}) \mbox{ if }G_n=GSpin_{2n+1} \mbox{ or }GSpin_{2n} \mbox{ with }n_0=0, \\
%	a-\varepsilon(\sigma^{(e0)}) \mbox{ if }G_n=GSpin_{2n+2}^{\ast}. \\
%\end{array} \right.
%\]
This then gives the following conditions on essentially square-integrable representations of general groups $\delta([\nu^{-a}\tau, \nu^b \tau])$ which occur in the classification of generic essential square-integrable representations, with $\beta=\beta(\sigma^{(e0)})$ as in Definition~\ref{beta def}:

\begin{description}

\item[(EDS1)]
If $(\tau; \sigma^{(0)})$ satisfies (C1), then $a \in \beta+\left({\mathbb N} \cup \{-1\}\right)$.

\item[(EDS2)]
If $(\tau; \sigma^{(0)})$ satisfies (C0), then $a \in\beta+{\mathbb Z}_{\geq 0}$.

\item[(EDS3)]
If $(\tau; \sigma^{(0)})$ satisfies (C1/2), then $a \in \beta-\frac{1}{2}+{\mathbb Z}_{\geq 0}$.

\item[(EDS4)]
If $(\tau; \sigma^{(0)})$ satisfies (CN), then $a \in \beta+{\mathbb Z}_{\geq 0}$.

\end{description}

\subsection{Tempered generic representations}\label{gtsection}

Let $\sigma^{(2)}$ be a $\psi_a$-square-integrable generic representation of $G_{k_0}(F)$ and $\beta_1, \dots, \beta_c$ (repetition possible) irreducible unitary supercuspidal representations of $H_{k_1}(F)$, $\dots$, $H_{k_c}(F)$, resp. For $i=1, \dots, c$, let
$\delta(\Psi'_i)=\delta([\nu^{-e_i}\beta_i, \nu^{e_i}\beta_i])$, $2e_i \in {\mathbb Z}_{\geq 0}$ be a sequence of (unitary) square-integrable representations (of $H_{k_i(2e_i+1)}(F)$, $i=1, \dots, c$, resp.). Then the unique $\psi_a$-generic component $\sigma^{(t)}$ 
\begin{equation}\label{tempered generic}
    \sigma^{(t)} \hookrightarrow \delta(\Psi'_1) \times \dots \times \delta(\Psi'_c) \rtimes \sigma^{(2)}
\end{equation}
is a tempered representation of $G_n(F)$, where $n=k_0+2\sum_{i=1}^c(2e_i+1)k_i$.

It follows from a result of Harish-Chandra (see \cite[Proposition 4.1]{Wal03}) that all irreducible tempered generic representations of $G$ are obtained this way, and the inducing data are unique up to conjugation. In particular, the data $\{\delta(\Psi'_i)\}_{i=1}^c$ and $\sigma^{(2)}$ are determined up to replacements of the form
\[
\delta(\Psi'_i) \leftrightarrow \delta(\check{\Psi}'_i)
\]
for $SO_{2n+1}, Sp_{2n}, U_{2n+1},U_{2n}$. The replacement $\delta(\Psi'_i) \leftrightarrow \delta(\check{\Psi}'_i)$ also requires $\sigma^{(2)} \leftrightarrow c^{(2e_i+1)k_i} \cdot \sigma^{(2)}$ for $SO_{2n}$ and $SO_{2n+2}^{\ast}$; $\sigma^{(2)} \leftrightarrow \omega_{\delta(\Psi'_i)}\sigma^{(2)}$ for $GSp_{2n}$, $GU_{2n+1}$, and $GU_{2n}$; $\sigma^{(2)} \leftrightarrow \omega_{\delta(\Psi'_i)}(c^{(2e_i+1)k_i} \cdot \sigma^{(2)})$ for $GSO_{2n}$ and $GSO_{2n+2}^{\ast}$.
For general spin groups, the replacement is
\[
\delta(\Psi'_i) \leftrightarrow \omega_{\sigma^{(2)}}\delta(\check{\Psi}'_i)
\]
(see (\ref{w_0action})); for $GSpin_{2n}$ and $GSpin_{2n+2}^{\ast}$, we also have $\sigma^{(2)} \leftrightarrow c^{(2e_i+1)k_i} \cdot \sigma^{(2)}$.
Note that in the cases of $SO_{2n}$, $SO_{2n+2}^{\ast}$, $GSp_{2n}$, $GSO_{2n}$, $GSO_{2n+2}^{\ast}$, $GU_{2n+1}$, $GU_{2n}$, $GSpin_{2n}$, and $GSpin_{2n+2}^{\ast}$, the effects on $\sigma^{(2)}$ accumulate. Also note that the replacements indicated need not be nontrivial.

We close this section by discussing essentially tempered representations in the similitude cases. If $\sigma^{(et)}$ is an irreducible (generic) essentially tempered representation, we may write $\sigma^{(et)}=\nu^{\varepsilon}\sigma^{(t)}$ for some irreducible (generic) tempered represetation $\sigma^{(t)}$ and some $\varepsilon =\varepsilon(\sigma^{(et)})\in {\mathbb R}$. Letting $\beta=\beta(\sigma^{(et)})$ as in Definition~\ref{beta def} and Note~\ref{beta note}, it follows from Lemma~\ref{twisting} that
\begin{align*}
   \sigma^{(et)}&=\nu^{\varepsilon}\sigma^{(t)} \hookrightarrow \nu^{\varepsilon}\left(\delta(\Psi'_1) \times \dots \times \delta(\Psi'_c) \rtimes \sigma^{(2)}\right) \\
   &\cong \nu^{\beta}\delta(\Psi'_1) \times \dots \times \nu^{\beta}\delta(\Psi'_c) \rtimes \sigma^{(e2)}, 
\end{align*}
where
\[
\sigma^{(e2)}=\left\{ \begin{array}{l} \nu^{2\varepsilon}\sigma^{(2)} \mbox{ if }G_n=GSpin_{2n+1}, GSpin_{2n} \mbox{ with }k_0=0, \\
	\nu^{\varepsilon}\sigma^{(2)} \mbox{ otherwise,}
\end{array}\right.
\]
%Note that by Lemma~\ref{beta}, $\beta=\beta(\sigma^{(e2)})$.
Letting
\begin{equation}\label{Psi}
\delta(\Psi_i)=\nu^{\beta}\delta(\Psi'_i)
\end{equation}
for $i=1, \dots, c$, we may then write
\begin{equation}\label{essentially tempered}
\sigma^{(et)} \hookrightarrow \delta(\Psi_1) \times \dots \times \delta(\Psi_c) \rtimes \sigma^{(e2)}.
\end{equation}
%we may write $\sigma^{(et)}=\chi\sigma^{(t)}$ for some irreducible (generic) tempered represetation $\sigma^{(t)}$. We have
%\[
%\begin{array}{rl}
%\chi\sigma^{(t)} & \hookrightarrow \chi(\delta(\Psi'_1) \times \dots \times \delta(\Psi'_c) \rtimes \sigma^{(2)}) \\
%& \cong
%	\left\{ \begin{array}{l}
%	\delta(\Psi'_1) \times \dots \times \delta(\Psi'_c) \rtimes \chi\sigma^{(2)} \mbox{ if } G_n=GSp_{2n}, GSO_{2n}, GSO_{2n+2}^{\ast}, GU_{2n+1}, GU_{2n},\\
%	\chi\delta(\Psi'_1) \times \dots \times  \chi\delta(\Psi'_c) \rtimes \chi\sigma^{(2)}  \mbox{ if } G_n=GSpin_{2n+1}, GSpin_{2n} \mbox{ with }k_0>0, \\
%	\chi\delta(\Psi'_1) \times \dots \times \chi\delta(\Psi'_c) \rtimes \chi^2\sigma^{(2)}  \mbox{ if } G_n=GSpin_{2n+1}, GSpin_{2n} \mbox{ with }k_0=0, \\
%\chi\delta(\Psi'_1) \times \dots \times  \chi\delta(\Psi'_c) \rtimes \chi\sigma^{(2)}  \mbox{ if } G_n=GSpin_{2n+2}^{\ast} \\
%\end{array} \right.
%\end{array}
%\] 
%(noting $k_0 \not=1$ for $G_n=GSO_{2n}$ and $GSpin_{2n}$).
%Letting $\Delta(\Psi_i)=\delta(\Psi_i')$ or $\chi\delta(\Psi_i')$ as needed and $\sigma^{(e2)}=\sigma^{(2)}$, $\chi\sigma^{(2)}$, or $\chi^2\sigma^{(2)}$ as appropriate, we have
%\begin{equation}\label{essentially tempered}
%\sigma^{(et)}\hookrightarrow \delta(\Psi_1) \times \dots \times \delta(\Psi_c) \rtimes \sigma^{(e2)}.
%\end{equation}

\subsection{Generic representations}\label{generic}

We consider the representations
$\delta(\Sigma_1)$, $\ldots, \delta(\Sigma_f),$ where
\begin{equation} \label{equg1}
\Sigma_1 = [v^{-q_1}\xi_1, v^{-q_1 + w_1}\xi_1], 
\end{equation}
$$\Sigma_2 = [v^{-q_2}\xi_2, v^{-q_2 + w_2}\xi_2],$$
$$\vdots$$
$$\Sigma_f = [v^{-q_f}\xi_f, v^{-q_f + w_f}\xi_f],$$
and $\xi_1, \xi_2, \cdots, \xi_f$ are irreducible unitary and
supercuspidal, with possible repetitions, $q_i \in \mathbb{R}$, $w_i
\in \mathbb{Z}_{\geq0}$. Let $\sigma^{(et)}$ be a generic essentially tempered representation of $G_n(F)$ as in \eqref{essentially tempered}, $\beta=\beta(\sigma^{(et)})$ (see Definition~\ref{beta def} and Note~\ref{beta note}), and suppose
\[
q_i \not=\frac{w_i}{2}-\beta
%\left\{ \begin{array}{l}
%\frac{w_i}{2}-\frac{e(\sigma^{(et)})}{2} \mbox{ if }G_n=GSpin_{2n+1},GSpin_{2n} %\mbox{ with }n>0, \\
%\frac{w_i}{2}-e(\sigma^{(et)}) \mbox{ if }G_n=GSpin_{2n+1},GSpin_{2n} \mbox{ with }n=0, \\
%\frac{w_i}{2}-\frac{e(\sigma^{(et)})}{2} \mbox{ if }G_n=GSpin_{2n+2}^{\ast}, \\
%\frac{w_i}{2} \mbox{ otherwise}.
%\end{array} \right.
\]
for all $i$.

We are interested in the induced representation $\delta(\Sigma_1) \times \dots \times \delta(\Sigma_f) \rtimes \sigma^{(et)}$. In the Grothendieck group, we have $\delta(\Sigma_i) \times \delta(\Sigma_j)=\delta(\Sigma_j) \times \delta(\Sigma_i)$ and
\begin{equation}\label{equalities}
\delta(\Sigma_i) \rtimes \sigma^{(et)}=\left\{ \begin{array}{l}
	\delta(\check{\Sigma}_i) \rtimes \sigma^{(et)} \mbox{ if } G_n=SO_{2n+1}, Sp_{2n},U_{2n+1}, U_{2n}, \\
	\delta(\check{\Sigma}_i) \rtimes c^{n_i}\cdot\sigma^{(et)} \mbox{ if } G_n=SO_{2n},SO_{2n+2}^{\ast}, \\
	\omega_{\sigma^{(et)}}\delta(\check{\Sigma}_i) \rtimes \sigma^{(et)} \mbox{ if } G_n=GSpin_{2n+1}, \\
	\delta(\check{\Sigma}_i) \rtimes \omega_{\delta(\Sigma_i)}\sigma^{(et)} \mbox{ if } G_n=GSp(2n,F), GU_{2n+1}, GU_{2n}, \\
	\delta(\check{\Sigma}_i) \rtimes \omega_{\delta(\Sigma_i)}(c^{n_i} \cdot\sigma^{(et)}) \mbox{ if } G_n=GSO_{2n}, GSO_{2n+2}^{\ast},\\
	\omega_{\sigma^{(et)}}\delta(\check{\Sigma}_i) \rtimes c^{n_i}\cdot\sigma^{(et)} \mbox{ if } G_n=GSpin_{2n}, GSpin_{2n+2}^{\ast}
\end{array}\right.
\end{equation}
(where $\delta(\Sigma_i)$ is a representation of $H_{n_i}(F)$).
Therefore, replacing $\delta(\Sigma_i)$, $\sigma^{(et)}$ by their counterparts on the right-hand side of (\ref{equalities}) and commuting as needed,  
we may--and do--assume that the exponents of 
$$\delta(\Sigma_1),
\delta(\Sigma_2), \cdots, \delta(\Sigma_f), \sigma^{(et)}$$
are in the order needed for Langlands data (see (\ref{Langlands classification})):
\begin{equation}\label{Langlands classification II}
\frac{w_1}{2}-q_1 \geq \dots \geq \frac{w_f}{2}-q_f> \beta.
\end{equation}
Recall that (Standard Module Conjecture -- see \cite{HO13}) the Langlands quotient 
$$L(\delta(\Sigma_1) \otimes \dots \otimes \delta(\Sigma_f) \otimes \sigma^{(et)})$$ 
is generic if and only if $\delta(\Sigma_1) \times \dots \times \delta(\Sigma_f) \rtimes \sigma^{(et)}$ is irreducible. In the remainder of this section, we determine the conditions under which this happens.

First, we have Theorem~\ref{gen1} below. The proof is essentially the same as in \cite{Jan96b} (see \cite[Theorem 3.3 and Remark 3.4]{Jan96b}), which in turn is based on that in \cite{Tad94}. To allow for a more uniform argument, we make the following notational conventions: for $\delta(\Sigma)$ a representation of $H_{m}(F)$, let
\begin{equation}\label{conventions}
\begin{array}{c}
\omega'_{\Sigma}=\left\{ \begin{array}{l}
	\omega_{\delta(\Sigma)} \mbox{ for } G_n=GSp_{2n},GSO_{2n},GSO_{2n+2}^{\ast},GU_{2n+1},GU_{2n}, \\
	1 \mbox{ otherwise},
\end{array} \right. \\
\omega'_{\sigma^{(et)}}=\left\{ \begin{array}{l}
	\omega_{\sigma^{(et)}} \mbox{ for }G_n=GSpin_{2n+1}, GSpin_{2n}, GSpin_{2n+2}^{\ast},\\
	1 \mbox{ otherwise};
\end{array} \right. \\
c_{\Sigma}=\left\{ \begin{array}{l}
	c^{m} \mbox{ for }G_n=SO_{2n}, SO_{2n+2}^{\ast}, GSO_{2n}, GSO_{2n+2}^{\ast}, GSpin_{2n}, GSpin_{2n+2}^{\ast}, \\
	1 \mbox{ otherwise}.
\end{array} \right.
\end{array}
\end{equation}
With these conventions, we have the following consequence of (\ref{equalities}):
\begin{equation}\label{w_0actionII}
\delta(\Sigma_i) \rtimes \sigma^{(et)}=\omega'_{\sigma^{(et)}} \delta(\check{\Sigma}_i) \rtimes \omega'_{\Sigma_i}(c_{\Sigma_i} \cdot \sigma^{(et)})
\end{equation}
for all the families of groups under consideration.

%Note that for $GSp(2n,F)$, the irreducibility of $\delta(\Sigma_i) \rtimes \sigma^{(et)}$ in $(G2)$ below implies the irreducibility of $\delta(\Sigma_i) \rtimes \omega\sigma^{(et)} \cong \omega(\delta(\Sigma_i) \rtimes \sigma^{(et)})$ needed in the proof.

\begin{thr}\label{gen1}With notation as above (including (\ref{Langlands classification II})), the representation $\sigma$ of $G$
defined by
$$
\sigma:=\delta(\Sigma_1) \times \delta(\Sigma_2) \times \dots \times \delta(\Sigma_f) \rtimes \sigma^{(et)}
$$
is irreducible if and only if 
$\{\Sigma_j\}_{j=1}^f$ and
$\sigma^{(et)}$ satisfy the following properties (notation as in (\ref{conventions})):\\

\noindent
($G1$) $\delta(\Sigma_i) \times \delta(\Sigma_j)$ and $\delta(\Sigma_i) \times \omega'_{\sigma^{(et)}}\delta(\check{\Sigma}_j)$ are irreducible  for all $1 \leq i\not= j \leq f$;

$$\delta(\Sigma_i) \rtimes \sigma^{(et)} \text{ is irreducible for
all }1 \leq i \leq f. \leqno(G2)$$
\end{thr}

\noindent
\begin{proof}
For the first part of ($G1$), if $\delta(\Sigma_i) \times \delta(\Sigma_j)$ is reducible for some $i \not= j$, then $\sigma$ is clearly reducible. Also, for ($G2)$, if $\delta(\Sigma_i) \rtimes \sigma^{(et)}$ is reducible, then $\sigma$ is reducible. For the second part of ($G1$), (\ref{w_0actionII}) tells us that
\[
\sigma=\delta(\Sigma_1) \times \dots \times \delta(\Sigma_{j-1}) \times \omega'_{\sigma^{(et)}}\delta(\check{\Sigma}_j) \times \delta(\Sigma_{j+1}) \times \dots \times \delta(\Sigma_f)  \rtimes \omega'_{\Sigma_j}(c_{\Sigma_i} \cdot \sigma^{(et)})
\]
in the Grothendieck group. It now follows that the reducibility of $\delta(\Sigma_i) \times \omega'_{\sigma^{(et)}}\delta(\check{\Sigma}_j)$ for $i \not= j$ implies the reducibility of $\sigma$, as needed. Thus conditions ($G1$) and ($G2$)  are necessary for the irreducibility of $\sigma$.

To see that ($G1$) and ($G2$) are sufficient, let $\pi=L(\delta(\Sigma_1) \otimes \dots \otimes \delta(\Sigma_f) \otimes \sigma^{(et)})$. Then, $\pi=L_{sub}(\omega'_{\sigma^{(et)}}\delta(\check{\Sigma}_1) \otimes \omega'_{\sigma^{(et)}}\delta(\check{\Sigma}_2) \otimes \dots \otimes \omega'_{\sigma^{(et)}}\delta(\check{\Sigma}_f) \otimes  \omega'_{\Sigma_1}\omega'_{\Sigma_2} \dots \omega'_{\Sigma_f}(c_{\Sigma_1} c_{\Sigma_2} \dots c_{\Sigma_j} \cdot \sigma^{(et)}))$ (e.g., \cite[Lemma 1.1]{Jan98}). Now, using ($G1$) and ($G2$), and noting that $\delta(\Sigma_i) \times \delta(\Sigma_j)$ (resp., $\delta(\Sigma_i) \rtimes \sigma^{(et)}$) is irreducible if and only if $\omega'_{\sigma^{(et)}} \delta(\check{\Sigma}_i) \times \omega'_{\sigma^{(et)}}\delta(\check{\Sigma}_j)$ (resp., $\omega'_{\sigma^{(et)}}\delta(\check{\Sigma}_i) \rtimes \omega'_{\Sigma_{i+1}}\dots \omega'_{\Sigma_f}(c_{\Sigma_{i+1}} \dots c_{\Sigma_f} \cdot\sigma^{(et)})$) is irreducible (see Lemma~\ref{twisting} and Remark~\ref{action of c remark}), we get
\[
\begin{array}{rl}
\pi
\hookrightarrow & \omega'_{\sigma^{(et)}}\delta(\check{\Sigma}_1) \times \omega'_{\sigma^{(et)}}\delta(\check{\Sigma}_2) \times \dots \times \omega'_{\sigma^{(et)}}\delta(\check{\Sigma}_f) \\
&\rtimes  \omega'_{\Sigma_1}\omega'_{\Sigma_2} \dots \omega'_{\Sigma_f} (c_{\Sigma_1}c_{\Sigma_2} \dots c_{\Sigma_f} \cdot \sigma^{(et)}) \\
& \\
\cong &
\omega'_{\sigma^{(et)}}\delta(\check{\Sigma}_2) \times \dots \times \omega'_{\sigma^{(et)}}\delta(\check{\Sigma}_f) \times  \omega'_{\sigma^{(et)}}\delta(\check{\Sigma}_1) \\
&\rtimes \omega'_{\Sigma_1}\omega'_{\Sigma_2} \dots \omega'_{\Sigma_f} (c_{\Sigma_1}c_{\Sigma_2} \dots c_{\Sigma_f} \cdot \sigma^{(et)}) \\
& \\
\cong & \omega'_{\sigma^{(et)}}\delta(\check{\Sigma}_2) \times \dots \times \omega'_{\sigma^{(et)}}\delta(\check{\Sigma}_f) \times  \delta(\Sigma_1) \\
&\rtimes \omega'_{\Sigma_2} \dots \omega'_{\Sigma_f} (c_{\Sigma_2} \dots c_{\Sigma_f} \cdot \sigma^{(et)}) \\
& \\
\cong & \delta(\Sigma_1) \times \omega'_{\sigma^{(et)}}\delta(\check{\Sigma}_2) \times \dots \times \omega'_{\sigma^{(et)}}\delta(\check{\Sigma}_f) \\
& \rtimes \omega'_{\Sigma_2} \dots \omega'_{\Sigma_f} (c_{\Sigma_2} \dots c_{\Sigma_f} \cdot \sigma^{(et)}) \\
& \vdots \mbox{ (continuing with $\delta(\Sigma_2)$, $\delta(\Sigma_3), \dots$ in succession)} \\
& \\
\cong & \delta(\Sigma_1) \times \delta(\Sigma_2) \times \dots \times \delta(\Sigma_f) \rtimes \sigma^{(et)}.
\end{array}
\]
This induced representation is the standard module admitting $\pi$ as its Langlands quotient. As $\pi$ appears with multiplicity one in the the standard module and is both a subrepresentation and the unique irreducible quotient, we must have irreducibility, as needed. Thus ($G1$) and ($G2$) are also sufficient. \end{proof}

We now take up the question of when $\delta(\Sigma) \rtimes \sigma^{(et)}$ is irreducible. As above, to uniformize the presentation, we retain the notational conventions of Note~\ref{beta note},  (\ref{Psi}), (\ref{essentially tempered}) and (\ref{conventions}).
We note that for a character $\chi$ of $F^{\times}$, the corresponding character $\chi \circ \xi_n$ of $GSpin_{2n+1}(F)$, $GSpin_{2n}(F)$, or $GSpin_{2n+2}^{\ast}(F)$
 (see discussion preceding Lemma~\ref{twisting}) satisfies
\[
\omega_{\chi \circ \xi_n}=\left\{ \begin{array}{l} \chi^2 \mbox{ for }G_n=GSpin_{2n+1} \mbox{ or }GSpin_{2n}\mbox{ if }n>0, \\ \chi  \mbox{ for }G_n=GSpin_{2n+1} \mbox{ or }GSpin_{2n}\mbox{ if }n=0, \\
\chi^2 \mbox{ for }G_n=GSpin_{2n+2}^{\ast} \end{array} \right.
\]
 (e.g., for $n>0$, one calculates $\chi \circ (e_1+\dots+e_n+2e_0)(d(1, \dots, 1, a_0))$--see \S \ref{sect3}). In particular, if $\varepsilon=\varepsilon(\sigma^{(e0)})$, $\beta=\beta(\sigma^{(e0)})$, and $\sigma^{(e0)}$ a representation of $G_{n_0}(F)$,
\begin{equation} \label{beta exponent}
\begin{array}{rl}
\omega_{\sigma^{(e0)}}
& =\omega_{\nu^{\varepsilon}\sigma^{(0)}} \\
&=\left\{ \begin{array}{l} \nu^{2 \varepsilon} \omega_{\sigma^{(0)}} \mbox{ if }G_{n_0}=GSpin_{2n_0+1},GSpin_{2n_0} \mbox{ with }n_0>0 \\ \nu^{\varepsilon} \omega_{\sigma^{(0)}} \mbox { if } G_{n_0}= GSpin_{2n_0+1}, GSpin_{2n_0} \mbox{ with } n_0=0, \\ 
\nu^{2 \varepsilon} \omega_{\sigma^{(0)}} \mbox{ if }G_{n_0}=GSpin^{\ast}_{2n_0+2} \\
\end{array} \right. \\
& =\nu^{2\beta}\omega_{\sigma^{(0)}}.
\end{array}
\end{equation}

\begin{thr}\label{gen2}
For $\Sigma=[\nu^{-q}\xi, \nu^{-q+w}\xi]$ and $\sigma^{(et)}$ as above (satisfying (\ref{Langlands classification II})),
$\delta(\Sigma) \rtimes \sigma^{(et)}$ is irreducible 
if and only if the following hold:\\

\noindent
$(G3)$ $\delta(\Sigma) \times \delta(\Psi_j)$ and $\omega'_{\sigma^{(e2)}}\delta(\check{\Sigma}) \times \delta(\Psi_j)$
are irreducible for all $1 \leq j \leq c$ (where $\Psi_j$ is given in (\ref{essentially tempered}) and $\omega'_{\sigma^{(e2)}}$ in (\ref{conventions})),
and
$$\delta(\Sigma) \rtimes \sigma^{(e2)} \text{ is irreducible.}
 \leqno(G4)$$
\end{thr}

\noindent
\begin{proof} Let $\delta_i=\delta(\Psi_i)$.

First, suppose ($G3$) and ($G4$) hold. We must show $\delta(\Sigma) \rtimes  \sigma^{(et)}$ is irreducible.
Let (see (\ref{conventions}) for notation)
\[
\pi=L(\delta(\Sigma) \otimes \sigma^{(et)})=L_{sub}(\omega'_{\sigma^{(e2)}}\delta(\check{\Sigma}); \omega'_{\Sigma}(c_{\Sigma} \cdot \sigma^{(et)})).
\]
noting $\omega'_{\sigma^{(et)}}=\omega'_{\sigma^{(e2)}}$ by Table \ref{tab:center} and (\ref{essentially tempered}). Then,
\[
\begin{array}{rl}
\pi & \hookrightarrow \omega'_{\sigma^{(e2)}}\delta(\check{\Sigma}) \times \delta_1 \times \dots \times \delta_k \rtimes  \omega'_{\Sigma}(c_{\Sigma} \cdot \sigma^{(e2)}) \\
& \mbox{ (using ($G3$))} \\
& \cong \delta_1 \times \dots \times \delta_k \times \omega_{\sigma^{(e2)}}\delta(\check{\Sigma}) \rtimes  \omega'_{\Sigma}(c_{\Sigma} \cdot \sigma^{(e2)}) \\
& \mbox{ (using ($G4$))} \\
& \cong \delta_1 \times \dots \times \delta_k \times \delta(\Sigma) \rtimes  \sigma^{(e2)} \\
& \mbox{ (using ($G3$))} \\
&\cong \delta(\Sigma) \times \delta_1 \times \dots \times \delta_k \rtimes  \sigma^{(e2)}.
\end{array}
\]
Therefore, $ \pi \hookrightarrow \delta(\Sigma) \rtimes T$ for some irreducible $T \leq \delta_1 \times \dots \times \delta_k \rtimes \sigma^{(e2)}$. We claim $T= \sigma^{(et)}$. To see this, observe that
\[
\mu^{\ast}(\pi) \leq \mu^{\ast}(\delta(\Sigma) \rtimes T)=\mu^{\ast}(\omega'_T \delta(\check{\Sigma}) \rtimes \omega'_{\Sigma}(c_{\Sigma} \cdot T) )
\]
By properties of the Langlands classification (\cite{BJ08}), $\omega'_T \delta(\check{\Sigma}) \rtimes \omega'_{\Sigma}(c_{\Sigma} \cdot T)$ is the only term in $\mu^{\ast}(\delta(\Sigma) \rtimes T)$ of its central character. By uniqueness of the Langlands subrepresentation data for $\pi$, we must then have $T=\sigma^{(et)}$. However, we now have $\pi$ appearing as both a subrepresentation and the Langlands quotient in $\delta(\Sigma) \rtimes \sigma^{(et)}$. This contradicts multiplicity one in the Langlands classification unless $\delta(\Sigma) \rtimes \sigma^{(et)}$ is irreducible, as needed.

Now, suppose (at least) one of ($G3$) or ($G4$)  fails. We must show $\delta(\Sigma) \rtimes \sigma^{(et)}$ is reducible.

First, suppose ($G3$) fails with some $\delta(\Sigma) \times \delta_i$ reducible (the slightly harder case); without loss of generality, say $\delta(\Sigma) \times \delta_1$. Note that this implies $\delta(\check{\Sigma}) \times \check{\delta}_1$ is reducible. Write $\sigma^{(et)} \hookrightarrow \delta_1 \rtimes T$. Now, suppose $\delta(\Sigma) \rtimes \sigma^{(et)}$ were irreducible (hence generic) so $\pi=\delta(\Sigma) \rtimes \sigma^{(et)}$. We consider $\check{\pi}=\delta(\check{\Sigma}) \rtimes \check{\sigma}^{(et)}$. By contragredience, we have
\[
\begin{array}{c}
\check{\pi} \cong \delta(\check{\Sigma}) \rtimes \check{\sigma}^{(et)} \hookrightarrow \delta(\check{\Sigma}) \times \delta(\check{\Psi}_1) \rtimes \check{T} \\
\Downarrow \mbox{ (Lemma~\ref{Lemma 5.5})} \\
\check{\pi} \hookrightarrow \delta(\check{\Sigma} \cap \check{\Psi}_1) \times \delta(\check{\Sigma} \cup\check{\Psi}_1) \rtimes \check{T} \\
\mbox{or} \\
\check{\pi} \hookrightarrow {\mathcal L}_{sub}(\delta(\check{\Sigma}) \otimes \delta(\check{\Psi}_1)) \rtimes \check{T}.
\end{array}
\]
Noting that $\check{\pi}$ is generic (with respect to $\psi_a \circ \sigma$ in the unitary or general unitary case, we must have the former. On the other hand, we claim that properties of the Langlands classification imply it must be the latter. In particular, we claim that
$\mu^{\ast}( \delta(\check{\Sigma} \cap \check{\Psi}_1) \times \delta(\check{\Sigma} \cup\check{\Psi}_1) \rtimes \check{T})$ contains no terms of the form $\delta(\check{\Sigma}) \otimes \dots$. To this end, we make things more explicit: noting that $\Psi=[\nu^{-e_1+\beta}\beta_1, \nu^{e_1+\beta}\beta_1]$ (see (\ref{Psi})), we have
\[
\check{\Sigma}=[\nu^{-w+q}\check{\xi}, \nu^q \check{\xi}] \mbox{ and }
\check{\Psi}_1=[\nu^{-\beta-e_1}\check{\beta}_1, \nu^{-\beta+e_1}\check{\beta}_1].
\]
Reducibility requires $\beta_1=\xi$ and $-w+q<-\beta-e_1\leq q+1<-\beta+e_1+1$. Then,
\[
\check{\Sigma} \cap \check{\Psi}_1=[\nu^{-\beta-e_1}\check{\xi}, \nu^q \check{\xi}] \mbox{ and }\check{\Sigma} \cup \check{\Psi}_1=[\nu^{-w+q}\check{\xi}, \nu^{-\beta+e_1}\check{\xi}].
\]
If $\mu^{\ast}(\check{T})=\sum_i \kappa_i \otimes \theta_i$, then by the $\mu^{\ast}$ structures discussed in \S \ref{sect3}, we must have
\begin{align}\label{G3equation}
\begin{split}
\delta([\nu^{-w+q}\check{\xi}, \nu^{q}\check{\xi}]) \leq & \,\, \delta([\nu^{x_1}\check{\xi}, \nu^q \check{\xi}]) \times \delta([\omega'_{\check{T}}\nu^{y_1}\xi, \omega'_{\check{T}}\nu^{e_1+\beta} \xi])
\times \delta([\nu^{x_2}\check{\xi}, \nu^{-\beta+e_1} \check{\xi}])\\
& \times \delta([\omega'_{\check{T}}\nu^{y_2}\xi, \omega'_{\check{T}}\nu^{w-q} \xi]) \times \kappa_i.
\end{split}
\end{align}
First, observe that since $e_1-\beta>q$, we must have $x_2=e_1-\beta+1$. Using (\ref{beta exponent}), we have
\[
\varepsilon(\omega'_{\check{T}})=-2\beta
\]
for general spin groups.
It then follows from (\ref{Langlands classification II}) that $\varepsilon(\omega_{\check{T}}'\nu^{q-w}\xi)=q-w-2\beta>q$, also giving $y_2=w-q+1$. Further, $\varepsilon(\omega_{\check{T}}'\nu^{e_1+\beta}\xi)=e_1-\beta>q$, giving $y_1=e_1+\beta+1$.
Thus we are reduced to  
\begin{equation}\label{G3equation II}
\delta([\nu^{-w+q}\check{\xi}, \nu^{q}\check{\xi}]) \leq \delta([\nu^{x_1}\check{\xi}, \nu^q \check{\xi}])
\times \kappa_i.
\end{equation}
We now argue that $\nu^{-w+q}\check{\xi}$ cannot appear in the right-hand side of (\ref{G3equation II}). First, $\nu^{-w+q}\check{\xi}$ cannot appear in $\delta([\nu^{x_1}\check{\xi}, \nu^{q}\check{\xi}])$ since $x_1 \geq -e_1-\beta>-w+q$. It also cannot appear in $\kappa_i$ as we would have
\[
\varepsilon(\kappa_i)=\frac{-w+q+x_1-1}{2} \leq -\frac{w}{2}+q<-\beta,
\]
contradicting the Casselman criterion for the essential temperedness of $\check{T}$ %(cf.(\ref{Casselman criterion}))
and finishing the argument for the case where the first possibility in (G3) fails.

Now suppose it is $\omega'_{\sigma^{(e2)}}\delta(\check{\Sigma}) \times \delta_1$ which is reducible. In this case, we use
\begin{align*}
   \pi \hookrightarrow \omega'_{\sigma^{(e2)}} \delta(\check{\Sigma}) \rtimes \omega'_{\Sigma}(c_{\Sigma} \cdot \sigma^{(et)}) \hookrightarrow &\,\, \omega'_{\sigma^{(e2)}} \delta(\Sigma) \rtimes \omega'_{\Sigma}(c_{\Sigma} \cdot(\delta_1 \rtimes T))\\
  & \cong  \omega'_{\sigma^{(e2)}} \delta(\Sigma) \times \delta_1 \rtimes \omega'_{\Sigma} (c_{\Sigma} \cdot T) 
\end{align*}
(see Lemma~\ref{twisting}, noting $\omega'_{\Sigma}$ trivial for the $GSpin$ cases) and argue as above.

The proof for ($G4$)  is similar. Here, since $(G3$) has been addressed above, we are free to assume ($G3$) holds but ($G4$) fails. Then,
\[
\check{\pi} \hookrightarrow \delta(\check{\Sigma}) \times \check{\delta}_1 \times \dots \times \check{\delta}_k \rtimes  \check{\sigma}^{(e2)}
\cong \check{\delta}_1 \times \dots \times \check{\delta}_k \times \delta(\check{\Sigma}) \rtimes  \check{\sigma}^{(e2)}.
\]
Write  $\delta(\check{\Sigma}) \rtimes  \check{\sigma}^{(e2)}=L(\delta(\check{\Sigma}) \otimes  \check{\sigma}^{(e2)})+\sum_j Q_j$. Note that by the Standard Module Conjecture, $L(\delta(\check{\Sigma}) \otimes  \check{\sigma}^{(e2)})$ is not generic. Again, by Lemma~\ref{Lemma 5.5},
\[
\begin{array}{c}
\check{\pi} \hookrightarrow \check{\delta}_1 \times \dots \times \check{\delta}_k \times \delta(\check{\Sigma}) \rtimes  \check{\sigma}^{(e2)} \\
\Downarrow \\
\check{\pi} \hookrightarrow \check{\delta}_1 \times \dots \times \check{\delta}_k \rtimes L(\delta(\check{\Sigma}) \otimes  \check{\sigma}^{(e2)}) \\
\mbox{or} \\
\check{\pi} \hookrightarrow \check{\delta}_1 \times \dots \times \check{\delta}_k \rtimes Q_i
\end{array}
\]
for some $i$. As above, genericity implies it must be the second, but properties of the Langlands classification require it be the first. We again have $\check{\pi}$ appearing with multiplicity two in $\check{\delta}_1 \times \dots \times \check{\delta}_k \times \delta(\check{\Sigma}) \rtimes  \check{\sigma}^{(e2)}$, giving a contradiction and finishing the proof.
\end{proof}

We now take up the question of when $\delta(\Sigma) \rtimes \sigma^{(e2)}$ is irreducible. Write
\[
\sigma^{(e2)}=\delta(\Delta_1, \dots, \Delta_k; \sigma^{(e0)})_{\psi_a}
\]
as in (\ref{essentially sq int}).
Again, to uniformize the presentation, we retain the notational conventions (\ref{conventions}). We also point out that it follows from Lemma~\ref{twisting} that for similitude groups,
\begin{equation}\label{irreducibility equation}
\lambda \rtimes \sigma^{(e0)} \mbox{ is irreducible } \Leftrightarrow \nu^{-\beta}\lambda \rtimes \sigma^{(0)} \mbox{ is irreducible},
\end{equation}
with $\beta=\beta(\sigma^{(e0)})$ as in Definition~\ref{beta def}.

\begin{thr}\label{gen3}
For $\Sigma$ as above, 
$\delta(\Sigma) \rtimes \sigma^{(e2)}$ is irreducible 
if and only if the following hold:\\

\noindent
$(G5)$ $\delta(\Sigma) \times \delta(\Delta_i)$ and 
$\omega'_{\sigma^{(e0)}}\delta(\check{\Sigma}) \times \delta(\Delta_i)$ are irreducible for all $i=1, \dots, k$, and

\noindent
$(G6)$ either (a) $\delta(\Sigma) \rtimes
\sigma^{(e0)}$ is irreducible, or (b) $q=-1+\beta$ (so $(\xi, \sigma^{(0)})$ satisfies $(C1)$) and there is some $i$ having $\delta(\Delta_i)=\delta([\nu^{1+\beta}\xi, \nu^{b_i+\beta}\xi])$ with $b_i \geq -q+w$.

%(b) $q=-1+\beta$ (so $(\xi, \sigma^{(0)})$ satisfies $(C1)$) with
%$a_j(\xi)=-1+\beta$ and $b_j(\xi) \geq -q+w$ for some $1 \leq j %\leq e_{\xi}$.

\end{thr}

\noindent
\begin{proof}
We first address $(\Leftarrow)$. That is, we assume ($G5$) and ($G6$) both hold and show $\delta(\Sigma) \rtimes \sigma^{(e2)}$ is irreducible. We assume it is ($G6$)(a) which holds, and comment on the changes needed for ($G6$)(b) afterwards.

Let $\pi \hookrightarrow \delta(\Sigma) \rtimes \sigma^{(e2)}$ be an irreducible subrepresentation. Then,
\begin{equation}\label{inv}
\begin{array}{rl}
\pi & \hookrightarrow \delta(\Sigma) \times \delta([\nu^{-a_1}\tau_1, \nu^{b_1}\tau_1]) \times \dots \times \delta([\nu^{-a_k}\tau_k, \nu^{b_k}\tau_k]) \rtimes  \sigma^{(e0)} \\
& \mbox{ (using ($G5$))} \\
& \cong  \delta([\nu^{-a_1}\tau_1, \nu^{b_1}\tau_1]) \times \dots \times \delta([\nu^{-a_k}\tau_k, \nu^{b_k}\tau_k]) \times \delta(\Sigma) \rtimes  \sigma^{(e0)} \\
& \mbox{ (using ($G6$)(a))} \\
& \cong \delta([\nu^{-a_1}\tau_1, \nu^{b_1}\tau_1]) \times \dots \times \delta([\nu^{-a_k}\tau_k, \nu^{b_k}\tau_k]) \times \omega'_{\sigma^{(e0)}} \delta(\check{\Sigma}) \rtimes \omega'_{\Sigma}  (c_{\Sigma} \cdot \sigma^{(e0)}) \\
& \mbox{ (using ($G5$))} \\
&\cong \omega'_{\sigma^{(e0)}} \delta(\check{\Sigma}) \times  \delta([\nu^{-a_1}\tau_1, \nu^{b_1}\tau_1]) \times \dots \times \delta([\nu^{-a_k}\tau_k, \nu^{b_k}\tau_k]) \rtimes \omega'_{\Sigma} (c_{\Sigma} \cdot \sigma^{(e0)}).
\end{array}
\end{equation}
In particular, $\pi$ has a term of the form $\omega'_{\sigma^{(e0)}} \delta(\check{\Sigma}) \otimes \theta$ in its Jacquet module. Note that $\delta(\Sigma) \rtimes \sigma^{(e2)}=\omega'_{\sigma^{(e0)}} \delta(\check{\Sigma})\rtimes \omega'_{\Sigma} (c_{\Sigma} \cdot \sigma^{(e2)})$ (as $\omega'_{\sigma^{(e2)}}=\omega'_{\sigma^{(e0)}}$ -- see Table \ref{tab:center}). From properties of the Langlands classification (\cite{BJ08}), the only term of the form $\omega'_{\sigma^{(e0)}} \delta(\check{\Sigma}) \otimes \theta$ in the Jacquet module of $\omega'_{\sigma^{(e0)}} \delta(\check{\Sigma}) \rtimes \omega'_{\Sigma}(c_{\Sigma} \cdot \sigma^{(e2)})$ is $\omega'_{\sigma^{(e0)}} \delta(\check{\Sigma}) \otimes \omega'_{\Sigma}(c_{\Sigma} \cdot  \sigma^{(e2)})$. Thus $\theta=\omega'_{\Sigma}(c_{\Sigma} \cdot \sigma^{(e2)})$. Now,
\[
\begin{array}{c}
\pi \hookrightarrow \omega'_{\sigma^{(e0)}} \delta(\check{\Sigma})\rtimes \omega'_{\Sigma} (c_{\Sigma} \cdot \sigma^{(e2)}) \\
\Downarrow \\
\pi \cong L(\delta(\Sigma) \otimes \sigma^{(e2)})
\end{array}
\]
as $L(\delta(\Sigma) \otimes \sigma^{(e2)})$ is the Langlands subrepresentation of $\omega'_{\sigma^{(e0)}} \delta(\check{\Sigma})\rtimes \omega'_{\Sigma} (c_{\Sigma} \cdot \sigma^{(e2)})$. Thus, $\pi$ appears as both irreducible subrepresentation (from above) and unique irreducible quotient in $\delta(\Sigma) \rtimes \sigma^{(e2)}$, contradicting multiplicity one in the Langlands classification unless we have irreducibility.

We now discuss the changes needed if it is ($G6$)(b) which holds. Since $ \delta([\nu^{-a_i}\tau_i, \nu^{b_i}\tau_i]) \times  \delta([\nu^{-a_j}\tau_j, \nu^{b_j}\tau_j])$ is irreducible for all $i \not= j$--hence may be commuted while preserving equivalences--we may without loss of generality assume $\tau_k \cong \xi$ with $a_k=-1+\beta$. Then, to produce the inversion of $\delta(\Sigma)$ of (\ref{inv}), first observe that\footnote{This represents a minor correction to the proof of \cite[Theorem 3.11]{JL14}.}
\[
\begin{array}{c}
\begin{array}{rl}
\sigma^{(e0)} & \hookrightarrow \delta([\nu^{-a_1}\tau_1, \nu^{b_1}\tau_1]) \times \dots \times \delta([\nu^{-a_{k-1}}\tau_{k-1}, \nu^{b_{k-1}}\tau_{k-1}]) \\
&\hfill \times \delta([\nu^{1-\beta}\xi, \nu^{b_k}\xi]) \rtimes  \sigma^{(e0)}
\end{array} \\
\Downarrow \mbox{ (Lemma~\ref{Lemma 5.5})} \\
\sigma^{(0)} \hookrightarrow \delta([\nu^{-a_1}\tau_1, \nu^{b_1}\tau_1]) \times \dots \times \delta([\nu^{-a_{k-1}}\tau_{k-1}, \nu^{b_{k-1}}\tau_{k-1}]) \rtimes \theta
\end{array}
\]
for some irreducible $\theta \leq \delta([\nu^{1-\beta}\xi, \nu^{b_k}\xi]) \rtimes  \sigma^{(e0)}$. By genericity, 
$$
\theta=\delta([\nu^{1-\beta}\xi, \nu^{b_k}\xi];  \sigma^{(e0)})_{\psi_a}.
$$
Then,
\[
\begin{array}{rl}
\pi  \hookrightarrow & \delta(\Sigma) \times \delta([\nu^{-a_1}\tau_1, \nu^{b_1}\tau_1]) \times \dots \times \delta([\nu^{-a_{k-1}}\tau_{k-1}, \nu^{b_{k-1}}\tau_{k-1}]) \\
& \rtimes \delta([\nu^{1-\beta}\xi, \nu^{b_k}\xi];  \sigma^{(e0)})_{\psi_a} \\
& \mbox{ (using ($G5$))} \\
\cong & \delta([\nu^{-a_1}\tau_1, \nu^{b_1}\tau_1]) \times \dots \times \delta([\nu^{-a_{k-1}}\tau_{k-1}, \nu^{b_{k-1}}\tau_{k-1}]) \\
& \rtimes \left( \delta(\Sigma) \rtimes \delta([\nu^{1-\beta}\xi, \nu^{b_k}\xi];  \sigma^{(e0)})_{\psi_a}\right) \\
\cong& 
\delta([\nu^{-a_1}\tau_1, \nu^{b_1}\tau_1]) \times \dots \times \delta([\nu^{-a_{k-1}}\tau_{k-1}, \nu^{b_{k-1}}\tau_{k-1}]) \\
& \rtimes
\left(\omega'_{\sigma^{(e0)}}\delta(\check{\Sigma}) \rtimes \omega'_{\Sigma} c_{\Sigma} \cdot \delta([\nu^{1-\beta}\xi, \nu^{b_k}\xi];  \sigma^{(e0)})_{\psi_a}\right),
\end{array}
\]
by irreducibility (Lemma~\ref{dps} and Lemma~\ref{twisting}), noting that we have $\omega'_{\delta([\nu^{1-\beta}\xi, \nu^{b_k}\xi];  \sigma^{(e0)})_{\psi_a}}=\omega'_{\sigma^{(e0)}}$ from Table~\ref{tab:center2}. The rest of the argument proceeds the same way.

%we do the following:
%\[
%\begin{array}{c}
%\begin{array}{rl}
%\pi & \hookrightarrow \delta(\Sigma) \times \delta([\nu^{-a_1}\tau_1, \nu^{b_1}\tau_1]) \times \dots \times \delta([\nu^{-a_{k-1}}\tau_{k-1}, \nu^{b_{k-1}}\tau_{k-1}]) \\
%&\hfill \times \delta([\nu^{1-\beta_0}\xi, \nu^{b_k}\xi]) \rtimes  \sigma^{(e0)} \\
%& \mbox{ (using ($G5$))} \\
%& \cong  \delta([\nu^{-a_1}\tau_1, \nu^{b_1}\tau_1]) \times \dots \times \delta([\nu^{-a_{k-1}}\tau_{k-1}, \nu^{b_{k-1}}\tau_{k-1}]) \times \delta(\Sigma) %\\
%&\hfill \times \delta([\nu^{1-\beta_0}\xi, \nu^{b_k}\xi]) \rtimes  \sigma^{(e0)}
%\end{array} \\
%\Downarrow \mbox{ (Lemma~\ref{Lemma 5.5})} \\
%\pi \hookrightarrow \delta([\nu^{-a_1}\tau_1, \nu^{b_1}\tau_1]) \times \dots \times \delta([\nu^{-a_{k-1}}\tau_{k-1}, \nu^{b_{k-1}}\tau_{k-1}]) \times %\delta(\Sigma) \rtimes \theta
%\end{array}
%\]
%for some irreducible $\theta \leq \delta([\nu^{1-\beta_0}\xi, \nu^{b_k}\xi]) \rtimes  \sigma^{(e0)}$. By genericity, 
%$$
%\theta=\delta([\nu^{1-\beta_0}\xi, \nu^{b_k}\xi];  \sigma^{(e0)})_{\psi_a}.
%$$
%We now observe that
%\[
%\delta(\Sigma) \rtimes \delta([\nu^{1-\beta_0}\xi, \nu^{b_k}\xi];  \sigma^{(e0)})_{\psi_a} \cong
%\delta(\check{\Sigma}) \rtimes \omega_1 c_1 \cdot \delta([\nu^{1-\beta_0}\xi, \nu^{b_k}\xi];  \sigma^{(e0)})_{\psi_a}
%\]
%by irreducibility (Lemma~\ref{dps} and Lemma~\ref{twisting}). (Note that as we have $(\xi; \sigma^{(0)})$ satisfying (C1), $\omega_1 \delta(\check{\Sigma}) \cong\delta(\check{\Sigma})$). The rest of the argument proceeds the same way.

We now turn to $(\Rightarrow)$. First, suppose it is (G5) which fails and it is $\delta(\Sigma) \rtimes \delta(\Delta_i)$ which reduces; without loss of generality, we take $i=1$. Now, observe that by Lemma~\ref{Lemma 5.5},
\[
\sigma^{(e2)} \hookrightarrow \delta(\Delta_1) \rtimes \delta(\Delta_2, \dots, \Delta_k; \sigma^{(e0)})_{\psi_a}
\]
Thus,
\[
\delta(\Sigma) \rtimes \sigma^{(e2)} \hookrightarrow \delta(\Sigma) \times \delta(\Delta_1) \rtimes \delta(\Delta_2, \dots, \Delta_k; \sigma^{(e0)})_{\psi_a};
\]
the argument now proceeds as in the proof of Theorem~\ref{gen2}; in particular, the proof that reducibility of $\delta(\Sigma) \times \delta(\Psi_j)$ implies the reducbility of $\delta(\Sigma) \rtimes \sigma^{(e0)}$. The case where $\omega'_{\sigma^{(e0)}}\delta(\check{\Sigma}) \rtimes \delta(\Delta_i)$ is reducible is similar.
\end{proof}

\begin{rem}
We remark that the results in Theorems \ref{gen2} and \ref{gen3} do not hold if $\sigma^{(et)}$ and $\sigma^{(e2)}$ are not assumed to be the generic subquotients of (\ref{essentially tempered}) and the induced representation in Proposition~\ref{gdsprop1}. To show this, consider the case of classical groups and suppose $(\rho, \sigma)$ satisfies (C1/2). Looking at \cite[Example 14.1.4]{MT02} and using the notation there, we note that there are 4 square-integrable subrepresentations of
\[
\delta([\nu^{\frac{-2k_1+1}{2}}\rho, \nu^{\frac{2k_2-1}{2}}\rho]) \times \delta([\nu^{\frac{-2k_3+1}{2}}\rho, \nu^{\frac{2k_4-1}{2}}\rho]) \rtimes \sigma,
\]
corresponding to $\varepsilon_1, \varepsilon_4, \varepsilon_{13},\varepsilon_{16}$. By \cite{Han10}, the square-integrable generic representation is the one corresponding to $\varepsilon_1$. Suppose $k_3=k_2+1$. By \cite[Lemma 3.2 (1)]{Jan18}, $\nu^{\frac{2k_3+1}{2}}\rho \rtimes \delta$ is reducible if and only if $\varepsilon(\rho, 2k_3)=\varepsilon(\rho, 2k_2)$, i.e., if $\varepsilon=\varepsilon_1$ or $\varepsilon_{16}$ (but not $\varepsilon_4$ or $\varepsilon_{13}$). One can construct a similar example for tempered representations using \cite[Theorem 4.7]{Jan18}.
\end{rem}

It remains to address the question of when $\delta(\Sigma) \rtimes \sigma^{(e0)}$ is reducible (see \cite[Theorem 9.1]{Tad98a} and \cite[Theorem 3.13]{JL14} for the split classical groups). Recall that Lemma~\ref{beta} implies $\beta=\beta(\sigma^{(et)})=\beta(\sigma^{(e0)})$, so $w,q$ satisfy the following (from (\ref{Langlands classification II})):
\[
\frac{w}{2}-q> \beta.
\]

\begin{thr}\label{gen4}
Write $\sigma^{(e0)}=\nu^{\varepsilon_0}\sigma^{(0)}$ and set
$$\omega_{\sigma^{(0)}}'=\left\{ \begin{array}{l} \omega_{\sigma^{(0)}} \mbox{ if }G_n=GSpin_{2n+1},GSpin_{2n},GSpin_{2n+2}^{\ast}, \\
	1 \mbox{ if not},
\end{array} \right. $$
as in (\ref{conventions}).
For $\Sigma$ as above (in particular, requiring $\frac{w}{2}-q>\beta$), we have $\delta(\Sigma) \rtimes \sigma^{(e0)}$ irreducible if and only if the following hold:

\noindent
$(G7)$ $\omega_{\sigma^{(0)}}'\check{\xi} \not\cong \xi$, or

\noindent
$(G8)$ $q+\beta \not\in \left\{ \begin{array}{l} -\alpha+{\mathbb Z}_{\geq 0} \mbox{ if $(\xi, \sigma^{(0)})$ satisfies (C$\alpha$) for }\alpha \in \{0, \frac{1}{2},1\}, \\
{\mathbb Z}_{\geq 0} \mbox{ if $(\xi, \sigma^{(0)})$ satisfies (CN)}.
\end{array} \right.$

\end{thr}

\noindent
\begin{proof}
First, suppose $(G7)$ holds. Let
\begin{align*}
  \pi&=L(\delta([\nu^{-q}\xi, \nu^{-q+w}\xi]) \otimes \sigma^{(e0)})\\
  &=L_{sub}(\delta([\omega_{\sigma^{(0)}}' \nu^{q-w+2\beta}\check{\xi}, \omega_{\sigma^{(0)}}' \nu^{q+2\beta}\check{\xi}]) \otimes \omega_{\Sigma}'(c_{\Sigma}\cdot \sigma^{(e0)})),  
\end{align*}
where $\omega_{\Sigma}'$ and $c_{\Sigma}$ are as in (\ref{conventions}). Then,
\begin{equation}\label{irreducibility calculation}
\begin{array}{rl}
\pi  \hookrightarrow & \delta([\omega_{\sigma^{(0)}}' \nu^{q-w+2\beta}\check{\xi}, \omega_{\sigma^{(0)}}' \nu^{q+2\beta}\check{\xi}]) \rtimes \omega_{\Sigma}'(c_{\Sigma}\cdot \sigma^{(e0)}) \\
 \hookrightarrow & \omega_{\sigma^{(0)}}' \nu^{q+2\beta}\check{\xi} \times \omega_{\sigma^{(0)}}' \nu^{q+2\beta-1}\check{\xi} \times \dots \times \omega_{\sigma^{(0)}}' \nu^{q-w+2\beta+1}\check{\xi} \\
 & \times \omega_{\sigma^{(0)}}' \nu^{q-w+2\beta}\check{\xi} \rtimes \omega_{\Sigma}'(c_{\Sigma} \cdot \sigma^{(e0)}) \\
 \cong & \omega_{\sigma^{(0)}}' \nu^{q+2\beta}\check{\xi} \times \omega_{\sigma^{(0)}}' \nu^{q+2\beta-1}\check{\xi} \times \dots \times \omega_{\sigma^{(0)}}' \nu^{q-w+2\beta+1}\check{\xi} \\
 & \times \nu^{-q+w}\xi \rtimes \omega_{\Sigma}'\omega_{\nu^{q-w}\check{\xi}}'(c_{\xi}c_{\Sigma} \cdot \sigma^{(e0)}) \\
 \cong & \nu^{-q+w}\xi \times \omega_{\sigma^{(0)}}' \nu^{q+2\beta}\check{\xi} \times \omega_{\sigma^{(0)}}' \nu^{q+2\beta-1}\check{\xi} \\
 & \times \dots \times \omega_{\sigma^{(0)}}' \nu^{q-w+2\beta+1}\check{\xi} \nu^{-q+w}\xi \rtimes \omega_{\Sigma}'\omega_{\nu^{q-w}\check{\xi}}'(c_{\xi}c_{\Sigma}\cdot \sigma^{(e0)}) \\
 \vdots & \mbox{ (iterating)} \\
 \cong & \nu^{-q+w}\xi \times \nu^{-q+w-1}\xi \times \dots \times \nu^{-q+1}\xi \times \nu^{-q} \xi \rtimes \sigma^{(e0)}.
\end{array}
\end{equation}
By Lemma~\ref{Lemma 5.5}, $\pi \hookrightarrow \lambda \rtimes \sigma^{(e0)}$ for some irreducible $\lambda \leq \nu^{-q+w}\xi \times \nu^{-q+w-1}\xi \times \dots \times \nu^{-q+1}\xi \times \nu^{-q} \xi$. A subquotient other than $\delta([\nu^{-q} \xi, \nu^{-q+w}\xi])$ would produce (by Frobenius reciprocity) a term of the form $\nu^x \xi \otimes \leq \mu^{\ast}(\pi)$ with $x \not=-q+w$, which is not possible. Thus $\lambda=\delta([\nu^{-q} \xi, \nu^{-q+w}\xi])$ and
\[
\pi \hookrightarrow \delta([\nu^{-q} \xi, \nu^{-q+w}\xi]) \rtimes \sigma^{(e0)}.
\]
We now have $\pi$ appearing as both a subrepresentation and the Langlands quotient in $\delta([\nu^{-q} \xi, \nu^{-q+w}\xi]) \rtimes \sigma^{(e0)}$. As $\pi$ must appear with multiplicity one, we see that $\delta([\nu^{-q} \xi, \nu^{-q+w}\xi]) \rtimes \sigma^{(e0)}$ is irreducible.

Now, suppose it is $(G8)$ which holds; by the argument above, we may also assume $(G7)$ does not, i.e., that $\omega_{\sigma^{(0)}}'\check{\xi} \cong \xi$. If $q+\beta \not\equiv 0, \frac{1}{2}\!\! \mod\! 1$, the same argument as above works. The case where $q+\beta \in \frac{1}{2}{\mathbb Z}$ but $q+\beta \not\equiv \alpha\!\! \mod\! 1$ (resp., $q+\beta \not\equiv\!\! \mod\! 1$ in the (CN) case) is covered by Lemmas~\ref{dsirrlem} and \ref{twisting}. If $q+\beta \equiv \alpha\!\! \mod\! 1$ (or $0$ in the (CN) case), the same basic argument as above also works, but the commuting/inverting argument in (\ref{irreducibility calculation}) is less obvious. To see that $\nu^{q+2\beta-k}\xi \rtimes \omega_k'c_k\sigma^{(e0)}$, $0 \leq k \leq w$ is irreducible, note that by Lemma~\ref{twisting},
\[
\nu^{q+2\beta-k}\xi \rtimes \omega_k'c_k\sigma^{(e0)} \cong \nu^{\beta}(\nu^{q+\beta-k}\xi \rtimes \omega_k'c_k\sigma^{(0)}).
\]
Then, by $(G8)$,we have
\[
q+\beta < -\alpha \Rightarrow q+\beta-k < -\alpha
\]
implying the needed irreducibility. To see that $\nu^{q+2\beta-k}\xi \times \nu^{-q+\ell}\xi$, $0 \leq k<\ell \leq w$ is irreducible, observe that since $\ell >k \geq 0$, we have $-\alpha \leq \frac{k+\ell-1}{2}$. Then,
\[
q+\beta<\frac{k+\ell-1}{2} \Rightarrow q+2\beta<-q+\ell-1,
\]
implying the needed irreducibility.

In the other direction, suppose both $(G7)$ and $(G8)$ fail. Except when $-q+\beta=0$ in the (C1) case, the conditions {\bf (EDS1)-(EDS4)} are satisfied. Then $\delta(\Sigma; \sigma^{(e0)})_{\psi_a}$ is an essentially square-integrable subquotient of $\delta(\Sigma) \rtimes \sigma^{(e0)}$ (Proposition~\ref{gdsprop1} and the discussion at the end of \S \ref{gdssection}). On the other hand, the conditions on $q,w$ are those required in the Langlands classification; consequently, $L(\Sigma \otimes \sigma^{(e0)})$ is the unique irreducible quotient of $\delta(\Sigma) \rtimes \sigma^{(e0)}$. Reducibility is then clear. When $q+\beta=0$ in the (C1) case, we note that the generalized Steinberg representation $\delta([\nu^{\beta+1}\xi, \nu^{\beta+w}\xi]; \sigma^{(e0)})$ is generic (Corollary~\ref{Steinberggenericity}), so there is a generic subquotient (necessarily a subrepresentation) of $\nu^{\beta}\xi \rtimes \delta([\nu^{\beta+1}\xi, \nu^{\beta+w}\xi]; \sigma^{(e0)})$, which is essentially tempered but not essentially square-integrable. By genericity and supercuspidal support considerations, this generic essentially tempered representation must also be a subquotient of $\delta([\nu^{-q} \xi, \nu^{-q+w}\xi]) \rtimes \sigma^{(e0)}$. This is cleary distinct from the Langlands quotient and reducibility follows. \end{proof}

\begin{rem}
We used the generic property to keep the proof above relatively short, but based on known examples (e.g., \cite{Tad98b}, \cite{KLM20}, \cite{BJ03}), we expect the same reducibility points when $\sigma^{(e0)}$ is not generic. The above result does not deal with the case $q+\beta=\frac{w}{2}$ (nor can it be obtained from (\ref{w_0actionII})). Again, based on known examples, we expect reducibility if and only if $q \equiv \alpha\!\! \mod \! 1$ (resp., never in the (CN) case).
\end{rem}

\begin{note}\label{interpretation}
Condition (G8) in Theorem~\ref{thm1intro} is restatement of (G8) above. In particular, note that by (\ref{irreducibility equation}),
the condition $(\rho, \sigma^{(0)})$ (C$\alpha$) corresponds to $\nu^x \rho \rtimes \sigma^{(e0)}$ having reducibility at $x=\beta+\alpha$. It then suffices to  show that for (C$\alpha$), $q+\beta \not\in -\alpha+{\mathbb Z}_{\geq 0} \Leftrightarrow \pm \alpha \not\in\{-q-\beta, -q+1-\beta, \dots, -q+w-\beta\}$ (noting $-a=-q$ and $b=-q+w$).

We check this in the most complicated case, namely the (C1) case. We first note that if $q+\beta \not\in {\mathbb Z}$, both clearly hold. Thus we restrict our attention to the case where $q+\beta \in {\mathbb Z}$.

\ 

\noindent
\underline{($\Rightarrow$:)}
We have $q+\beta \not\in -1+{\mathbb Z}_{\geq 0} \Leftrightarrow q+\beta<-1 \Leftrightarrow 1<-q-\beta$. Then $-1<-q-\beta$ as well, so neither $\pm 1 \in [-q-\beta, -q+w-\beta]$.

\ 

\noindent
\underline{($\Leftarrow$):}
Suppose both $\pm 1 \not\in [-q-\beta, -q-\beta+w]$. Observe that since $-q-\beta+w \geq 1$, $1 \not\in[-q-\beta, -q-\beta+w] \Rightarrow -q-\beta>1$ (so $-1 \not\in[-q-\beta, -q-\beta+w]$ as well), or equivalently, $q+\beta \not\in -1+{\mathbb Z}_{\geq 0}$.

%First, observe that $-1 \not\in [-q-\beta, -q-\beta+w] \Rightarrow -q-\beta+w>1$ (as $-q-\beta+w=1$ requires $-q-\beta \geq 0$). Then $1 \not\in [-q-\beta, -q-\beta+w] \Rightarrow 1 \not\in [-q-\beta, 1] \Rightarrow 1 \not\in -q-\beta+{\mathbb Z}_{\geq 0}$, or equivalently, $q+\beta \not\in -1+{\mathbb Z}_{\geq 0}$.

\end{note}

\section{Functoriality for quasi-split classical groups}\label{sect5}

Let $G_n = SO_{2n+1}, Sp_{2n}, SO_{2n}, SO_{2n+2}^*, U_{2n+1}, U_{2n}$, quasi-split classical groups of rank $n$. From now on to the end of this paper, we focus on these groups. We follow \cite{CPSS11} to recall the Langlands functoriality for $G_n$ as follows. 

\begin{enumerate}
\item When $G_n = SO_{2n+1}$, the $L$-group is ${}^LG_n=Sp_{2n}(\mathbb{C}) \times W_F$ with connected component ${}^LG_n^0=Sp_{2n}(\mathbb{C})$, and we have the natural embedding
$$i: {}^LG_n=Sp_{2n}(\mathbb{C}) \times W_F \hookrightarrow GL_{2n}(\mathbb{C}) \times W_F = {}^LGL_{2n}.$$

\item When $G_n = Sp_{2n}$, the $L$-group is ${}^LG_n=SO_{2n+1}(\mathbb{C}) \times W_F$ with connected component ${}^LG_n^0=SO_{2n+1}(\mathbb{C})$, and we have the natural embedding
$$i: {}^LG_n=SO_{2n+1}(\mathbb{C}) \times W_F \hookrightarrow GL_{2n+1}(\mathbb{C}) \times W_F = {}^LGL_{2n+1}.$$

\item When $G_n = SO_{2n}$, the split special even orthogonal group, the $L$-group is ${}^LG_n=SO_{2n}(\mathbb{C}) \times W_F$ with connected component ${}^LG_n^0=SO_{2n}(\mathbb{C})$, and we have the natural embedding
$$i: {}^LG_n=SO_{2n}(\mathbb{C}) \times W_F \hookrightarrow GL_{2n}(\mathbb{C}) \times W_F = {}^LGL_{2n}.$$

\item When $G_n = SO_{2n+2}^*$, the quasi-split non-split special even orthogonal group, the $L$-group is ${}^LG_n=SO_{2n+2}(\mathbb{C}) \rtimes W_F$ with connected component ${}^LG_n^0=SO_{2n+2}(\mathbb{C})$. Here the 
Weil group acts through the quotient $W_F/W_E \cong Gal(E/F)$ which gives the Galois structure of $SO_{2n+2}^*$, write the quadratic extension as  $E=F(\sqrt{\varepsilon})$. More explicitly, let $h \in O_{2n+2}$ be any element of negative determinant, then for the non-trivial element $\sigma \in Gal(E/F)$, it acts on ${}^LG_n^0$ by $\sigma(g)=h^{-1}gh$. 
Choose an $L$-homomorphism 
$$\xi: W_F \rightarrow O_{2n+2}(\mathbb{C}) \times W_F \hookrightarrow GL_{2n+2}(\mathbb{C}) \times W_F,$$
which induces the isomorphism 
$$W_F / W_E \rightarrow Gal(E/F) \cong O_{2n+2}(\mathbb{C}) / SO_{2n+2}(\mathbb{C}),$$
 i.e., $\xi$ factors through $W_F/W_E \cong Gal(E/F)$ and sends the non-trivial element $\sigma$ to $h \times \sigma$. Write $\xi$ as $\xi(w)=\xi'(w) \times w$ with $\xi'(w) \in O_{2n+2}(\mathbb{C})$. Then we have the following embedding
 $$i: SO_{2n+2}(\mathbb{C}) \rtimes W_F \hookrightarrow GL_{2n+2}(\mathbb{C}) \times W_F,$$
 given by 
 $i(g \times 1)=g \times 1, i(1 \times w) = \xi(w) = \xi'(w) \times w$, where $g \in SO_{2n+2}(\mathbb{C})$ and $w \in W_F$. 

\item When $G_n = U_{2n+1}=U(J_{2n+1}')$, the odd quasi-split unitary group, where $J_{2n+1}'=\left( \begin{array}{cccccc}
	 & & & & & 1 \\
	 & & & & -1 &\\
  & & & 1 & & \\
  & & \iddots & & & \\
  & -1 & & & & \\
  1 & & & & & 
\end{array} \right),$
the $L$-group is ${}^LG_n=GL_{2n+1}(\mathbb{C}) \rtimes W_F$ with connected component ${}^LG_n^0=GL_{2n+1}(\mathbb{C})$. Here the 
Weil group acts through the quotient $W_F/W_E \cong Gal(E/F)$ which gives the Galois structure of $U_{2n+1}$, write the quadratic extension as  $E=F(\sqrt{\varepsilon})$. More explicitly, the non-trivial element $\sigma \in Gal(E/F)$ acts on ${}^LG_n^0$ by $\sigma(g)=J_{2n+1}'^{-1}{}^tg^{-1}J_{2n+1}'$. 
The standard representation of ${}^LG_n$ is $\mathbb{C}^{2n+1} \times \mathbb{C}^{2n+1}$, where the action of ${}^LG_n^0$ is by $(g \times 1) (v_1, v_2) = (g v_1, \sigma(g)v_2)$, and the Weil group acts through the quotient $W_F/W_E \cong Gal(E/F)$ with the action of the non-trivial Galois element by 
$(1 \times \sigma) (v_1, v_2)=(v_2, v_1)$. Then we have the following embedding
$$i: {}^LG_n  \hookrightarrow (GL_{2n+1}(\mathbb{C}) \times GL_{2n+1}(\mathbb{C})) \rtimes W_F = {}^L(Res_{E/F} GL_{2n+1}),$$
by $i(g \times w) = (g \times \sigma(g)) \times w$, where on the right hand side, $W_F$ acts on $GL_{2n+1}(\mathbb{C}) \times GL_{2n+1}(\mathbb{C})$ through the quotient $W_F/W_E \cong Gal(E/F)$ with $\sigma(g_1 \times g_2) = g_2 \times g_1$. 

\item When $G_n = U_{2n}=U(J_{2n}')$, the even quasi-split unitary group, where $J_{2n}'=\begin{pmatrix}
&J_n\\
-J_n&
\end{pmatrix}$ and $J_n$ as in the case of $U_{2n+1}$, 
the $L$-group is ${}^LG_n=GL_{2n}(\mathbb{C}) \rtimes W_F$ with connected component ${}^LG_n^0=GL_{2n}(\mathbb{C})$. Here the 
Weil group acts through the quotient $W_F/W_E \cong Gal(E/F)$ which gives the Galois structure of $U_{2n}$, again, write the quadratic extension as  $E=F(\sqrt{\varepsilon})$. More explicitly, the non-trivial element $\sigma \in Gal(E/F)$ acts on ${}^LG_n^0$ by $\sigma(g)=J_{2n}'^{-1}{}^tg^{-1}J_{2n}'$. 
Similar to the case of $G_n = U_{2n+1}$, we have the following embedding
$$i: {}^LG_n  \hookrightarrow (GL_{2n}(\mathbb{C}) \times GL_{2n}(\mathbb{C})) \rtimes W_F = {}^L(Res_{E/F} GL_{2n}),$$
by $i(g \times w) = (g \times \sigma(g)) \times w$, where on the right hand side, $W_F$ acts on $GL_{2n}(\mathbb{C}) \times GL_{2n}(\mathbb{C})$ through the quotient $W_F/W_E \cong Gal(E/F)$ with $\sigma(g_1 \times g_2) = g_2 \times g_1$.
\end{enumerate}

Recall from the introduction that 
$N = 2n$ for $G_n=SO_{2n+1}, U_{2n}, SO_{2n}$, 
$N = 2n+2$ for $G_n=SO_{2n+2}^*$, 
 $N = 2n+1$ for $G_n=Sp_{2n}, U_{2n+1}$.
Also recall that $H_N = GL_N$ when $G_n=SO_{2n+1}, Sp_{2n}, SO_{2n}, SO_{2n+2}^*$, and let 
$H_N=Res_{E/F} GL_N$ when $G_n=U_{2n+1}, U_{2n}$. 
We recall Table \ref{tab:Langlands functoriality} which summarizes the cases of funtoriality we will consider from $G_n$ to $H_N$:

\begin{table}[H]
\begin{tabular}{ |c|c|c| } 
 \hline
$G_n$ & $i: {}^L G_n \hookrightarrow {}^L H_N$ & $H_N$ \\ 
 \hline
$SO_{2n+1}$ & $Sp_{2n}(\mathbb{C}) \times W_F \hookrightarrow GL_{2n}(\mathbb{C}) \times W_F$ & $GL_{2n}$ \\ 
 $Sp_{2n}$ & $SO_{2n+1}(\mathbb{C}) \times W_F \hookrightarrow GL_{2n+1}(\mathbb{C}) \times W_F$ & $GL_{2n+1}$ \\ 
 $SO_{2n}$ & $SO_{2n}(\mathbb{C}) \times W_F \hookrightarrow GL_{2n}(\mathbb{C}) \times W_F$ & $GL_{2n}$\\
  $SO_{2n+2}^*$ & $SO_{2n+2}(\mathbb{C}) \rtimes W_F \hookrightarrow GL_{2n+2}(\mathbb{C}) \times W_F$ & $GL_{2n+2}$\\
  $U_{2n+1}$ & $GL_{2n+1}(\mathbb{C}) \rtimes W_F \hookrightarrow GL_{2n+1}^{\times 2}(\mathbb{C}) \rtimes W_F$ & $Res_{E/F} GL_{2n+1}$\\
    $U_{2n}$ & $GL_{2n}(\mathbb{C}) \rtimes W_F \hookrightarrow GL_{2n}^{\times 2}(\mathbb{C})  \rtimes W_F$ & $Res_{E/F} GL_{2n}$\\
 \hline
\end{tabular}
\end{table}

\section{Surjectivity of local Langlands functorial lifting maps}\label{sect6}

In this section, first we carry out the image of the local Langlands functorial lifting from $\Pi^{(sg)}(G_n)$ (generic supercuspidal representations) to $H_{N}$ following \cite{CKPSS04} and \cite{CPSS11}, then using the descent method as in \cite{JS03}, we prove that the rest of local Langlands functorial lifting given by Cogdell, 
Piatetski-Shapiro, Shahidi \cite{CPSS11} is also surjective. 
In each case, we write down the corresponding local Langlands parameters. 
To the complete, we provide a uniform proof for $G_n$, following \cite{JS03, Liu11, JL14}.

First, we define
\begin{equation}\label{rho}
R:=\begin{cases}
\Lambda^{2} & \text{ if } G_{n}=SO_{2n+1},\\
Sym^{2} & \text{ if } G_{n}=Sp_{2n}, \text{ or } SO_{2n}, \text{ or } SO_{2n+2}^*,\\
Asai & \text{ if } G_{n}= U_{2n+1},\\
Asai\otimes\delta & \text{ if } G_{n}=U_{2n};
\end{cases}
\end{equation}
and,
\begin{equation}\label{rho-}
R^{-}:=\begin{cases}
Sym^{2} & \text{ if } G_{n}=SO_{2n+1},\\
\Lambda^{2} & \text{ if } G_{n}=Sp_{2n}, \text{ or } SO_{2n}, \text{ or } SO_{2n+2}^*,\\
Asai \otimes \delta & \text{ if } G_{n}=U_{2n+1},\\
Asai & \text{ if }  G_{n}=U_{2n}.
\end{cases}
\end{equation}
For symplectic or orthogonal groups, $Sym^2$ and $\Lambda^2$ denote the symmetric and
exterior second powers of the standard representation of $GL_m(\mathbb{C})$, respectively. 
For unitary groups, $Asai$ denotes the Asai representation of the $L$-group of $Res_{E/F}(GL_{m})$ and $\delta$ is the character associated to the quadratic extension $E/F$ via the class field theory. 
Recall that given 
$\tau$ an irreducible unitary supercuspidal representation of $H_k$, $\check{\tau} = \tilde{\tau}$, if $H_k=GL_k$, $\check{\tau}=\tilde{\tau}^{\iota}$,
 if $H_k=Res_{E/F}GL_k$, where the involution $\iota$ is the nontrivial element in the Galois group $Gal_{E/F}$.
Given 
$\tau$ an irreducible unitary supercuspidal representation of $H_k$,
such that $\check{\tau}=\tau$, 
we have the following identities:
 $$
 L(s,\tau\times\tau^*)=L(s,\tau,R)L(s,\tau,R^{-}),
 $$
 where $\tau^* = \tau$, if $H_k=GL_k$; $\tau^* = \tau^{\iota}$,
 if $H_k=Res_{E/F}GL_k$.

\begin{rem} \label{rmk1}
Assume $\sigma^{(0)}$ is an irreducible generic supercuspidal representation of $G_n$, $\tau$ is an irreducible unitary supercuspidal representation of $H_k$. 
If $L(\sigma^{(0)} \times \tau, s)$ has a pole at $s=0$ (case
($C1$)), then $L(\tau, R, s)$ has a pole at $s=0$. 
\end{rem}

Let $\Phi(G_n)$ be the set of local Langlands parameters for $G_n$
(for a definition and discussion of the local Langlands reciprocity conjecture,
see \cite[Introduction]{Liu11} and the references therein). These are
${}^LG_n$-conjugacy classes of admissible homomorphisms 
$$W_F \times SL_2(\mathbb{C}) \rightarrow {}^LG_n,$$
where $W_F \times SL_2(\mathbb{C})$ is the Weil-Deligne group. 

Note that when $G_n=SO_{2n}, SO_{2n+2}^*$, 
we have the embedding 
$${}^LG_n \rightarrow {}^LGL_{N}.$$
Given a local Langlands parameter $\phi \in \Phi(GL_{N})$, $\phi: W_F \times SL_2(\mathbb{C}) \rightarrow {}^LGL_{N}$, assume that it factors through ${}^LG_n$ and $\phi \ncong c \phi$ within ${}^LG_n$, where $c \phi$
is the c-conjugate of $\phi$, here $c$ is the fixed outer automorphism (see \S \ref{sect3}). 
Then $\phi$ gives two elements in 
$\Phi(G_n)$ (see \cite[Chapter 1]{Art13}), which are denoted by $\phi$ and $c \phi$. To identify $\phi$ and $c \phi$ in this situation, let $\widetilde{\Phi}(G_n)$ be the set of $c$-conjugacy classes of $\phi \in \Phi(G_n)$. Write $\widetilde{\phi}=c \phi$.
Note that for any $\phi \in \Phi(G_n)$, if $\phi \ncong c \phi$, then they automatically have the same twisted local factors since they come from the same local Langlands parameter $\phi \in \Phi(GL_{N})$. Define the twisted local factors of $\widetilde{\phi}$ to be those of $\phi$. From now on, we use the notation $\widetilde{\Phi}(G_n)$, and when $G_n=SO_{2n+1}, Sp_{2n}, U_{2n+1}, U_{2n}$, $\widetilde{\Phi}(G_n)={\Phi}(G_n)$, $\widetilde{\phi}=\phi$. When convenience, write $G=G_n$.

\subsection{Supercuspidal generic representations}

Let $\Pi^{(sg)}(G_n)$ be the set of all equivalence classes of irreducible generic
supercuspidal representations of $G_n$. Let $\Pi^{(sg)}_{\varepsilon}(H_{N})$ be the set of all equivalence classes of irreducible tempered
representations of $H_{N}(F)$ of the following form:
\begin{equation}\label{isobaricsum}
\tau_1 \times \tau_2 \times
\cdots \times \tau_r,
\end{equation}
with central character $\chi$ being trivial when restricting to $F^*$ except when $G_n=SO_{2n+2}^*$, in which case it is 
the quadratic character $\eta_{\varepsilon}$  associated to the square class $\varepsilon$ defining $G_n$.  
Here for each $1 \leq i \leq r$,
$\tau_i$ is an irreducible unitary supercuspidal 
representation of $H_{n_i}(F)$ such that $\tau \cong \check{\tau}$, $L(\tau_i, R, s)$ has
a pole at $s = 0$, and for $i \neq j$, $\tau_i \not \cong \tau_j$. 

Cogdell, Kim, 
Piatetski-Shapiro, Shahidi \cite{CKPSS04, CPSS11}, and Kim, Krishnamurthy \cite{KK04, KK05}, constructed the following local Langlands functorial lifting map:

\begin{thr}[Cogdell, Kim, 
Piatetski-Shapiro, Shahidi, Krishnamurthy] \label{thm2}
There is a map $l$ from 
$\Pi^{(sg)}(G_n)$ to $\Pi^{(sg)}_{\varepsilon}(H_{N})$ and it preserves the local factors:
$$L(\sigma \times \tau, s) = L(l(\sigma) \times \tau, s),$$
$$\epsilon(\sigma \times \tau, s, \psi) = \epsilon(l(\sigma)
\times \tau, s, \psi),$$ for any $\sigma \in \Pi^{(sg)}(G_n)$ and any irreducible
generic representation $\tau$ of $H_k(F)$ $(k$ any positive
integer$)$.
\end{thr}

The following theorem, which is one of the main ingredients of the results in this section, shows that the map $l$ in Theorem \ref{thm2} is surjective. 
Arthur \cite{Art13} and Mok \cite{Mok15} proved this result using the trace formula method and the global descent result of Ginzburg, Rallis and Soudry \cite{GRS11}. 
Jiang and Soudry \cite{JS12} (for $G_n=SO_{2n+1}, Sp_{2n}, SO_{2n}$, $SO_{2n+2}^*$), Soudry and Tanay \cite{ST15} (for $G_n=U_{2n}$),
constructed the local descent map from irreducible supercuspidal representations of $H_{N}$ to irreducible supercuspidal representations of $G_n$.
The generalization of the descent map from irreducible supercuspidal representations of $H_N$ to representations of the form in \eqref{isobaricsum} is straightforward for $G_n=SO_{2n+1}, Sp_{2n}$, $SO_{2n}, SO_{2n+2}^*$, but for $G_n=U_{2n}$, further work may be needed. 

\begin{thr}[Arthur, Jiang-Soudry, Mok,  Soudry-Tanay] \label{thm5}
For $G_n=SO_{2n+1}, Sp_{2n}, SO_{2n}, SO_{2n+2}^*$, $U_{2n+1}, U_{2n}$, 
The map $l$ in Theorem \ref{thm2} is surjective.
\end{thr}

% Soudry and Tanay \cite{ST15} (for $G_n=U_{2n}$), constructed the local descent map from supercuspidal representations of $GL_{N}$ to irreducible supercuspidal representations of $G$. However, the proof of the surjectivity of the map $l$ requires further work. We will make the following assumption here:

% \begin{as} \label{thm5}
% For $G_n=U_{2n}, U_{2n+1}$, 
% The map $l$ in Theorem \ref{thm2} is also surjective.
% \end{as}

% {\color{red} Check: Whether the work of Mork, etc, on the Arthur classification, shows this assumption???}

\begin{rem}\label{rmk2}
For $\sigma \in \Pi^{(sg)}(G_n)$, assume $\tau_1 \times \tau_2 \times
\cdots \times \tau_r \in \Pi^{(sg)}_{\alpha}(H_{N})$ is the lifting of $\sigma$. Then it is clear 
that $L(\sigma \times \tau_i, s)$ has a pole at $s=0$, $1 \leq i \leq r$. 
Therefore, by Remark \ref{rmk1}, each pair $(\tau_i, \sigma)$ must be of $(C1)$. 
\end{rem}

We have the following proposition about lifting images of $\sigma$
and $c \sigma$ when $\sigma \ncong c \sigma \in \Pi^{(sg)}(G_n)$.

\begin{prop}\label{prop2}
If $\sigma \ncong c \sigma \in \Pi^{(sg)}(G_n)$, then
$l(\sigma)=l(c \sigma) \in \Pi^{(sg)}_{\alpha}(H_N)$.
In particular, 
\begin{align*}
L(\sigma \times \tau, s) & = L(c \sigma \times \tau, s),\\
\epsilon(\sigma \times \tau, s, \psi)
&= \epsilon(c \sigma \times \tau, s, \psi),
\end{align*}
for any irreducible generic representation $\tau$ of $H_k(F)$, where $k$ is any positive integer.
\end{prop}

\begin{proof}
The proof is similar to that of \cite[Proposition 4.4]{JL14} which is omitted here. 
\end{proof}

Next, we figure out the corresponding parameters of irreducible supercuspidal generic representations of $G_n(F)$.
We need to recall the following result.

\begin{thr}[\cite{Hen10, CST17, Sha20}]\label{thm3} 
The local Langlands reciprocity map $r$ for $H_k(F)$
has the following property: 
\begin{align*}
    \gamma(\phi, R, s, \psi)&=\gamma(r(\phi), R, s, \psi),
\end{align*}
for any local Langlands parameter $\phi \in \Phi(H_k)$.
\end{thr}

As in \cite{JS03} and \cite{Liu11}, using the above result, we have the following proposition.

\begin{prop} \label{prop1}
Assume $\tau$ is an irreducible
 supercuspidal representation of $H_k(F)$ having local Langlands parameter $\phi$ (which is an irreducible
 admissible $k$-dimensional complex representation of $W_F$) such that $\tau \cong \check{\tau}$. Then $L(\tau, R, s)$
 has a pole at $s=0$ if and only if $L(\phi, R, s)$ has a pole at $s=0$. 
\end{prop}

\begin{proof}
By definition
$$\gamma(\check{\tau}, R, s, \psi) = \epsilon(\check{\tau}, R, s, \psi) \cdot \frac{L(\tau, R, 1-s)}{L(\check{\tau}, R, s)}.$$
If the local L-function $L(\tau, R, s)$ has a pole at $s=0$, then the gamma function
$\gamma(\check{\tau}, R, s, \psi)$ has a pole at $s=1$. Hence by
Theorem \ref{thm3}, the gamma function $\gamma(\check{\phi}, R, s, \psi)$
also has a pole at $s=1$. Since we also have
$$ \gamma(\check{\phi}, R, s, \psi) = \epsilon(\check{\phi}, R, s, \psi)\frac{L(\phi, R, 1-s)}{L(\check{\phi}, R, s)},$$
the L-function $L(\phi, R, s)$ has a pole at $s=0$. Similarly we can prove the other direction. This completes the proof.
\end{proof}

Given a local $L$-parameter $\phi$ of $G_n$, 
for convenience, 
if $L(\phi, R, s)$ has a pole at $s=0$, then we say $\phi$ is of type $R$; if $L(\phi, R^-, s)$ has a pole at $s=0$, then we say $\phi$ is of type $R^-$.

Let $\Phi^{(sg)}(G_n)$ be the subset of $\Phi(G_n)$
consisting of all parameters of type
$$\phi = \bigoplus_i \phi_i$$
with the  properties that:

(1) $\phi_i$'s are irreducible 
self-dual (resp. self-conjugate-dual in the case of unitary groups, see \cite[Section 3]{GGP12})
representations of $W_F$, and $\phi_i \not \cong \phi_j$ if $i \neq j$;

(2) for each $i$, $L(\phi_i, R, s)$ has a pole at $s=0$.

Note that for any $\phi \in \Phi(G_n)$, $det(\phi)$ is trivial except when $G_n=SO_{2n+2}^*$, in which case it is  
the quadratic character $\eta_{\varepsilon}$ associated to the square class $\varepsilon$ defining $G_n$.

\noindent
Let $\widetilde{\Phi}^{(sg)}(G_n)$ be the image of $\Phi^{(sg)}(G_n)$ in $\widetilde{\Phi}(G_n)$.
As a consequence of Theorem \ref{thm5} and Proposition \ref{prop1}, we have the
following result for irreducible generic
supercuspidal representations of $G$:

\begin{thr} \label{thm6}
There is a surjective map $\iota$ from
$\Pi^{(sg)}(G_n)$ to the set $\widetilde{\Phi}^{(sg)}(G_n)$ and it 
preserves the local factors:
$$L(\sigma \times \tau, s) = L(\iota(\sigma) \otimes r^{-1}(\tau),
s),$$
$$\epsilon(\sigma \times \tau, s, \psi) =
\epsilon(\iota(\sigma) \otimes r^{-1}(\tau), s, \psi),$$ for
any $\sigma \in \Pi^{(sg)}(G_n)$ and any irreducible
generic representations $\tau$ of $H_{k_{\tau}}(F)$, with all
$k_{\tau} \in \mathbb{Z}_{>0}$. Here, $r^{-1}(\tau)$ is the
irreducible admissible representation of of $W_F \times SL_2(\mathbb{C})$ of dimension
$k_{\tau}$ which corresponds to $\tau$ under the local Langlands
reciprocity map for $H_{k_{\tau}}$.
\end{thr}

\begin{rem}\label{rmk6}
Note that as mentioned at the beginning of this section, for any $\widetilde{\phi} \in \widetilde{\Phi}(G_n)$, its twisted local factors are defined to be those of $\phi$ (a representative of $\widetilde{\phi}$).

Also note that by Proposition \ref{prop2} and Theorem \ref{thm6},
if $\sigma \ncong c \sigma \in \Pi^{(sg)}(G_n)$,
then they have the same lifting image and the same twisted local factors.
\end{rem}

\subsection{Square-integrable generic representations}

First, we recall the classification of square-integrable generic representations of $G_n(F)$ given in \S \ref{gdssection}.

Let $P{'}$ be a finite set of irreducible supercuspidal self-dual
(or self-conjugate-dual in the case of unitary groups) representations $\tau$ of $H_{k_{\tau}}(F)$. Assume that
for each $\tau \in P{'}$, there is a sequence of segments
\begin{equation}
D_i(\tau)=[v^{-a_i(\tau)}\tau, v^{b_i(\tau)}\tau], i=1, 2, \cdots,
e_{\tau},
\end{equation}
satisfying
\begin{equation} \label{equ8}
2a_i(\tau) \in \mathbb{Z} \hbox{ and } 2b_i(\tau) \in
\mathbb{Z}_{\geq 0},
\end{equation}
and
\begin{equation} \label{equ9}
a_1(\tau) < b_1(\tau) < a_2(\tau) < b_2(\tau) < \cdots
< a_{e_{\tau}}(\tau) < b_{e_{\tau}}(\tau).
\end{equation}
Let
$\sigma^{(0)}$ be an irreducible supercuspidal
generic representation of $G_{n{'}}(F)$. Assume that \\
(DS1) if $(\tau, \sigma^{(0)})$ satisfies $(C1)$, then $-1 \leq
a_i(\tau) \in \mathbb{Z}\smallsetminus \{0\}$, for $1
\leq i \leq e_{\tau}$;\\
(DS2) if $(\tau, \sigma^{(0)})$ satisfies $(C0)$, then
$a_i(\tau) \in \mathbb{Z}_{\geq0}$, for $1 \leq i \leq e_{\tau}$;\\
(DS3) if $(\tau, \sigma^{(0)})$ satisfies $(C \frac{1}{2})$, then
$a_i(\tau) \in -\frac{1}{2} + \mathbb{Z}_{\geq0}$, for $1 \leq i
\leq e_{\tau}$;\\
(DS4) if $(\tau, \sigma^{(0)})$ satisfies $(CN)$, then
$a_i(\tau) \in \mathbb{Z}_{\geq0}$, for $1 \leq i
\leq e_{\tau}$.\\
Then the unique generic constituent of
\begin{equation} \label{equ40}
(\times_{\tau \in P{'}} \times_{i=1}^{e_{\tau}} \delta(D_i(\tau)))
\rtimes \sigma^{(0)}
\end{equation}
is square-integrable. Assume that the \eqref{equ40} is a representation of $G_n(F)$.
Then every square-integrable generic representation of $G_n(F)$
can be obtained this way for a unique set consisting of a finite set
$P{'}$, segments $\{D_i(\tau)|1 \leq i \leq e_{\tau}, \tau \in
P{'}\}$ and a unique generic supercuspidal representation
$\sigma^{(0)}$, satisfying conditions
\eqref{equ8}, \eqref{equ9}, (DS1) - (DS4).

Recall that we say $(\tau, \sigma^{(0)})$ satisfies $(C \alpha)$,
where $\alpha \in \{0, \frac{1}{2}, 1\}$, if $v^{\pm \alpha} \tau
\rtimes \sigma^{(0)}$ reduces, and $v^{\pm \beta} \tau \rtimes
\sigma^{(0)}$ is irreducible for all $|\beta| \neq \alpha$. 
We say $(\tau, \sigma^{(0)})$ satisfies 
$(CN)$ if $\tau \cong \check{\tau}$, $k_{\tau}$ is odd, and $c \sigma^{(0)} \ncong \sigma^{(0)}$. 
By 
\cite{Sha90a, Sha92, ACS16}, we know that our $(\tau, \sigma^{(0)})$ must satisfy
one of $(C \alpha)$ or $(CN)$, and the followings hold:

(1) $(\tau, \sigma^{(0)})$ satisfies $(C1)$ if and only if
$L(\sigma^{(0)} \times \tau, s)$ has a pole at $s=0$;

(2) $(\tau, \sigma^{(0)})$ satisfies $(C0)$ or $(CN)$ if and only if $L(\tau,
R, s)$ has a pole at $s=0$, but $L(\sigma^{(0)} \times \tau, s)$
is holomorphic at $s=0$;

(3) $(\tau, \sigma^{(0)})$ satisfies $(C \frac{1}{2})$ if and only
if $L(\tau, R^-, s)$ has a pole at $s=0$.

Therefore, for convenience, we rephrase conditions (DS1) -- (DS4) as follows:\\
(DS1$'$) ($C1$), if $L(\sigma^{(0)} \times \tau, s)$ has a pole at $s=0$,
then $-1 \leq a_i(\tau) \in \mathbb{Z}\smallsetminus \{0\}$, for $1
\leq i \leq e_{\tau}$;\\
(DS2$'$) $(C0)$ or $(CN)$, if $L(\tau, R, s)$ has a pole at $s=0$, but
$L(\sigma^{(0)} \times \tau, s)$ is holomorphic at $s=0$, then
$a_i(\tau) \in \mathbb{Z}_{\geq0}$, for $1 \leq i \leq e_{\tau}$;\\
(DS3$'$) $(C\frac{1}{2})$, if $L(\tau, R^-, s)$ has a pole at
$s=0$, then $a_i(\tau) \in -\frac{1}{2} + \mathbb{Z}_{\geq0}$, for
$1 \leq i \leq e_{\tau}$.

Let $\Pi^{(dg)}_{\varepsilon}(H_N)$ be the set of all equivalence classes of
irreducible tempered representations of $H_N(F)$ of the following form:
\begin{equation} \label{equldg1}
\delta([\nu^{-m_1}\tau_1, \nu^{m_1}\tau_1]) \times \delta([\nu^{-m_2}\tau_2, \nu^{m_2}\tau_2])  \times \cdots
\times \delta([\nu^{-m_r}\tau_r, \nu^{m_r}\tau_r]),
\end{equation}
with central character $\chi$ being trivial when restricting to $F^*$ except when $G_n=SO_{2n+2}^*$, in which case it is 
the quadratic character $\eta_{\varepsilon}$  associated to the square class $\varepsilon$ defining $G_n$.  
Here the 
segments $[v^{-m_i} \tau_i, v^{m_i} \tau_i]$ are pairwise distinct,
$\tau_i \cong \check{\tau}_i$,
and satisfy the following properties
for each $i$:\\
(1) if $L(\tau_i, R^-, s)$ has a pole at $s=0$, then $m_i \in
\frac{1}{2} + \mathbb{Z}_{\geq0}$; \\
(2) if $L(\tau_i, R, s)$ has a pole at $s=0$, then $m_i \in
\mathbb{Z}_{\geq0}.$

% From the definition, we can see that any $\pi \in \Pi^{(dg)}_{\alpha}(G)$ is of orthogonal type.

% By \cite{Sha90a}, conditions (DS1) -- (DS4) in \S \ref{gdssection} can be described as follows, and it is convenient
% for us to use this description:

% \medskip
% \noindent
% (DS1$'$) In the ($C1$) case, if $L(\sigma^{(0)} \times \tau, s)$ has a pole at $s=0$,
% then $-1 \leq a_i(\tau) \in \mathbb{Z}\smallsetminus \{0\}$, 
% for $1 \leq i \leq e_{\tau}$ (see Proposition \ref{gdsprop1} for $e_{\tau}$);\\
% (DS2$'$) In the $(C0)$ or $(CN)$ cases, if $L(\tau, R, s)$ has a pole at $s=0$, but
% $L(\sigma^{(0)} \times \tau, s)$ is holomorphic at $s=0$, then
% $a_i(\tau) \in \mathbb{Z}_{\geq0}$, for $1 \leq i \leq e_{\tau}$;\\
% (DS3$'$) In the $(C\frac{1}{2})$ case, if $L(\tau, R^-, s)$ has a pole at
% $s=0$, then $a_i(\tau) \in -\frac{1}{2} + \mathbb{Z}_{\geq0}$, for
% $1 \leq i \leq e_{\tau}$.

Then we have the following theorem which is analogous to \cite[Theorem 2.1]{JS04}, \cite[Theorem 4.8]{Liu11}, and \cite[Theorem 4.8]{JL14}. 

\begin{thr} \label{thm7}
There is a surjective map $l$ $($which extends the one in Theorem \ref{thm5}$)$ from $\Pi^{(dg)}(G_n)$ to $\Pi^{(dg)}_{\varepsilon}(H_N)$ and it preserves the local factors:
\begin{equation} \label{equldg2}
L(\sigma \times \pi, s)=L(l(\sigma) \times \pi, s), 
\end{equation}
\begin{equation} \label{equldg3}
\epsilon(\sigma \times \pi, s, \psi)=\epsilon(l(\sigma) \times
\pi, s, \psi),
\end{equation}
for any $\sigma \in
\mathrm{\Pi}^{(dg)}(G_n)$ and any irreducible generic
representation $\pi$ of any $H_k(F)$, $k \in
\mathbb{Z}_{>0}$.
\end{thr}

\noindent
\begin{proof}
This map has already been given by Cogdell, Kim, Piatetski-Shapiro, and Shahidi (see \cite{CKPSS04} and \cite{CPSS11}), so it suffices to prove the surjectivity. That is, given a $\rho \in \Pi^{(dg)}(H_N(F))$, to construct a $\sigma_{\rho} \in \Pi^{(dg)}(G_n)$ such that $\rho =
l(\sigma_{\rho})$ and \eqref{equldg2}, \eqref{equldg3} hold.

We follow the idea as in \cite{JS04}. 
Let $\Pi^{(ss)}(H_k)$ be the set of
equivalence classes of irreducible supercuspidal
representations $\tau$ of $H_k(F)$ ($k \in \mathbb{Z}_{>0}$) such that $\tau \cong \check{\tau}$.
Given $\rho \in \Pi^{(dg)}_{\alpha}(H_N)$,
let
$$P(\rho):=\{\tau \in \Pi^{(ss)}(H_k) | L(\rho
\times \tau, s) \, \, has \,\, a \,\, pole \,\, in \,\, \mathbb{R},
k \in \mathbb{Z}_{>0}\}.$$ Then $P(\rho)$ is finite.
For $\tau \in P(\rho)$, we list the real poles of $L(\rho \times
\tau, s)$ as follows
\begin{equation} \label{equldg4}
-m_{d_{\tau}}(\tau) < \dots < -m_2(\tau) < -m_1(\tau) \leq 0. 
\end{equation}
Put $d_{\tau} = 0$ if $L(\rho \times \tau, s)$ is holomorphic for
$\tau$ irreducible supercuspidal ($\tau \cong \check{\tau}$ or not). We consider
the following subset of $P(\rho)$:
\begin{flushleft}
$A(\rho)=\{\tau \in P(\rho)|L(\tau_i, R, s) \hbox{ has a
pole at s=0, and }d_{\tau}\hbox{ is odd} \},$\\
$B(\rho)=\{\tau \in P(\rho)|L(\tau_i, R, s) \hbox{ has a
pole at s=0, and }d_{\tau}\hbox{ is even} \},$\\
$C(\rho)=\{\tau \in P(\rho)|L(\tau_i, R^-, s)\hbox{ has a pole
at s=0}\}.$ 
\end{flushleft} 
Here $R$ and $R^-$ in the $L$-functions are defined in \eqref{rho} and \eqref{rho-}. 
Then
$$P(\rho)=A(\rho) \cup B(\rho) \cup C(\rho).$$
Further, if $\tau \in A(\rho) \cup B(\rho)$, then
$\{m_i(\tau)\}^{d_{\tau}}_{i=1} \subset \mathbb{Z}_{\geq0}$; if 
$\tau \in C(\rho)$, then $\{m_i(\tau)\}^{d_{\tau}}_{i=1} \subset
\frac{1}{2} + \mathbb{Z}_{\geq0}$.

Observe that  for $\tau \in
A(\rho)$, $d_{\tau}$ is odd and the central character $\omega_{\tau}$ is quadratic on $F^*$; for $\tau \in
B(\rho)$, $d_{\tau}$ is even and the central character $\omega_{\tau}$ is quadratic on $F^*$;  for $\tau \in C(\rho)$, $L(\tau, R^-, s)$ has a pole at $s=0$ which implies that the central character $\omega_{\tau}$ is trivial on $F^*$.
Hence, the following character is trivial on $F^*$:
$$\prod_{\tau \in A(\rho)} \omega_{\tau}^{d_{\tau}-1}
\prod_{\tau \in B(\rho)} \omega_{\tau}^{d_{\tau}} \prod_{\tau \in C(\rho)} \omega_{\tau}^{d_{\tau}}.$$
Therefore, the representation $\times_{\tau \in A(\rho)}
\tau$ is a representation of $H_{2k}(F)$ with central character $\chi_0$, which is trivial when restricting to $F^*$ except when $G_n=SO_{2n+2}^*$, in which case it is $\eta_{\alpha}$, $k$ is an integer,
$2k=\sum_{\tau \in A(\rho)} k_{\tau}$, where $k_{\tau}$ is so defined that $\tau$ is a representation of $H_{k_{\tau}}(F)$. Since for $\tau \in A(\rho)$, $L(\tau, R, s)\hbox{ has a pole at s=0}$, by Theorem
\ref{thm5}, there exists an irreducible supercuspidal generic
representation $\sigma^{(0)}$ (not necessarily unique up to
equivalence) of $G_{k}(F)$, such that
\begin{equation} \label{equldg5}
l(\sigma^{(0)})=\times_{\tau \in A(\rho)} \tau 
\end{equation}
on $H_{*}(F)$. 

Let
\begin{align}\label{equldg}
    \begin{split}
        A_0(\rho)&=\{\tau \in A(\rho)|d_{\tau}=1 \hbox{ and }
m_1(\tau)=0\},\\
A_1(\rho)&=\{\tau \in A(\rho)|d_{\tau} \geq 3 \hbox{ and }m_1(\tau)=0\},\\
A_2(\rho)&=\{\tau \in A(\rho)|m_1(\tau) \geq 1\}.
    \end{split}
\end{align}
Then they form a partition of $A(\rho)$. For $\tau \in
A_1(\rho)$, let
\begin{equation} \label{equldg6}
\Delta_i(\tau)=\delta([\nu^{-m_{2i}(\tau)}\tau,
\nu^{m_{2i+1}(\tau)}\tau]), i=1, 2, \dots, \frac{d_{\tau}-1}{2}; 
\end{equation}
for $\tau \in A_2(\rho)$, let
\begin{equation} \label{equldg7}
\Delta_0(\tau)=\delta([\nu\tau,
\nu^{m_1(\tau)}\tau]), \Delta_i(\tau)=\delta([\nu^{-m_{2i}(\tau)}\tau,
\nu^{m_{2i+1}(\tau)}\tau]),
\end{equation}
$$i=1, 2, \dots, \frac{d_{\tau}-1}{2}.$$

For $\tau \in B(\rho)$, let
\begin{equation} \label{equldg8}
\Delta_i(\tau)=\delta([\nu^{-m_{2i-1}(\tau)}\tau,
\nu^{m_{2i}(\tau)}\tau]), i=1, 2, \dots, \frac{d_{\tau}}{2}.
\end{equation}

Similarly, for $\tau \in C(\rho)$, if $d_{\tau}$ is odd, let
\begin{equation} \label{equldg9}
\Delta_0(\tau)=\delta([\nu^{\frac{1}{2}}\tau,
\nu^{m_1(\tau)}\tau]), \Delta_i(\tau)=\delta([\nu^{-m_{2i}(\tau)}\tau,
\nu^{m_{2i+1}(\tau)}\tau]),
\end{equation}
$$i=1, 2, \dots, \frac{d_{\tau}-1}{2}.$$
Finally, for $\tau \in C(\rho)$, if $d_{\tau}$ is even, let
\begin{equation} \label{equldg10}
\Delta_i(\tau)=\delta([\nu^{-m_{2i-1}(\tau)}\tau,
\nu^{m_{2i}(\tau)}\tau]), i=1, 2, \dots, \frac{d_{\tau}}{2}.
\end{equation}

We  now define
$$J_{\tau}=\left\{
             \begin{array}{ll}
               \{1, 2, \dots, \frac{d_{\tau}-1}{2}\}, & \hbox{in case \eqref{equldg6};} \\
               \{0, 1, 2, \dots, \frac{d_{\tau}-1}{2}\}, & \hbox{in cases \eqref{equldg7}, \eqref{equldg9};} \\
               \{1, 2, \dots, \frac{d_{\tau}}{2}\}, & \hbox{in cases \eqref{equldg8}, \eqref{equldg10}.}
             \end{array}
           \right.$$
and let $\sigma_{\rho}$ be the unique irreducible generic
constituent of
\begin{equation} \label{equldg11}
(\times_{\tau \in P(\rho) \smallsetminus A_0(\rho)} \times_{j \in J_{\tau}}
\Delta_j(\tau)) \rtimes \sigma^{(0)},
\end{equation}
%or
%\begin{equation} \label{equldg12}
%c((\times_{\tau \in P(\rho) \smallsetminus A_0(\rho)} \times_{j \in J_{\tau}}
%\Delta_j(\tau)) \rtimes 1), \text{ if } k=0. 
%\end{equation}
where possibly $\sigma^{(0)} = 1 \otimes c$.
Observe that
$\sigma_{\rho}$ is a representation of $G_n(F)$. It is now easy to see
that the sequence of segments in \eqref{equldg6}-\eqref{equldg10}, together with
$\sigma^{(0)}$ satisfy  (DS1$'$)-(DS3$'$), hence
$\sigma_{\rho}$ is square-integrable.

% The proof that the local factors are preserved is similar to that in \cite{JS04} or \cite{CKPSS04}.

Next we will show that formula \eqref{equldg2} and \eqref{equldg3} hold for the pair
$(\sigma_{\rho}, \rho)$. 
First, we show that
\begin{equation} \label{equ14}
\gamma(\sigma_{\rho} \times \pi, s, \psi) = \gamma(\rho \times \pi, s, \psi), 
\end{equation}
for all irreducible generic representations $\pi$ of $H_k(F)$ with
all $k \in \mathbb{Z}_{>0}$. By the multiplicativity of twisted gamma functions (\cite{Sha90b}), it is  enough to show \eqref{equ14} for
supercuspidal representation $\pi$, see also \cite[Lemma 7.2]{CKPSS04}. 

Since $\sigma_{\rho}$ is a
constituent of \eqref{equldg11}, using Theorem \ref{thm5} and the fact that
$l(\sigma^{(0)})=\times _{\tau \in A(\rho)} \tau$, we have
\begin{equation} \label{equ16}
\begin{split}
& \quad \gamma(\sigma_{\rho} \times \pi, s, \psi)\\
& =[\prod_{\tau \in P(\rho)\smallsetminus A_0(\rho)} \prod_{j \in
J_{\tau}} \gamma(\Delta_j(\tau) \times \pi, s, \psi)
\gamma(\Delta_j(\tau)^\vee \times \pi, s, \psi)]\\
& \quad \cdot \gamma(\sigma^0 \times \pi, s, \psi)\\
& =[\prod_{\tau \in P(\rho)\smallsetminus A_0(\rho)} \prod_{j \in
J_{\tau}} \gamma(\Delta_j(\tau) \times \pi, s, \psi)
\gamma(\Delta_j(\tau)^\vee \times \pi, s, \psi)]\\
& \quad \cdot \prod_{\tau \in A(\rho)} \gamma(\tau \times \pi, s,
\psi).
\end{split}
\end{equation}

We split the product \eqref{equ16} into the following five terms\\ \\
(1): $\prod_{\tau \in A_0(\rho)} \gamma(\tau \times \pi, s, \psi)$,\\
$(2)_i$: $\prod_{\tau \in A_i(\rho)} \gamma(\tau \times \pi, s,
\psi) \prod_{j \in J_{\tau}} \gamma(\Delta_j(\tau) \times \pi, s,
\psi)
\gamma(\Delta_j(\tau)^\vee \times \pi, s, \psi), \\
\qquad\qquad i=1, 2,$\\
$(3)_B$: $\prod_{\tau \in B(\rho)} \prod_{j \in J_{\tau}}
\gamma(\Delta_j(\tau) \times \pi, s, \psi)
\gamma(\Delta_j(\tau)^\vee \times \pi, s, \psi),$\\
$(4)_C$: $\prod_{\tau \in C(\rho)} \prod_{j \in J_{\tau}}
\gamma(\Delta_j(\tau) \times \pi, s, \psi)
\gamma(\Delta_j(\tau)^\vee \times \pi, s, \psi)$.\\

Next, we consider $\gamma(\rho \times \pi, s, \psi)$. By the
assumption on $\rho$, \eqref{equldg4} and the multiplicativity of
gamma functions (\cite{Sha90b}), we have
\begin{equation} \label{equ18}
\begin{split}
& \quad \gamma(\rho \times \pi, s, \psi)  \\
& =\prod_{i=1}^r \gamma(\delta([\nu^{-m_i}\tau_i, \nu^{m_i}\tau_i]) \times \pi, s, \psi)\\
& =\prod_{\tau \in P(\rho)} \prod_{i=1}^{d_{\tau}} \gamma(\delta([\nu^{-m_i(\tau)}\tau, \nu^{m_i(\tau)}\tau]) \times \pi, s, \psi).
\end{split}
\end{equation}
It suffice to show that the product in \eqref{equ18} consists of exactly 
the factors which appear in the five products above.

Note that each term in the product (1) appears in \eqref{equ18} since
for $\tau \in A_0(\rho)$, we have $d_{\tau} = 1$ and $m_1(\tau) =
0$.

For the induced representation
$$\Delta_j(\tau) \times \Delta_j(\tau)^\vee,$$
by the multiplicativity of twisted gamma functions, we
have
$$\gamma((\Delta_j(\tau) \times \Delta_j(\tau)^\vee) \times
\pi, s, \psi) = \gamma(\Delta_j(\tau) \times \pi, s, \psi)
\gamma(\Delta_j(\tau)^\vee \times \pi, s, \psi)$$ for $j \geq
1$.

We consider two cases. In the first case the representation
$\Delta_j(\tau)$ appears in \eqref{equldg6}, \eqref{equldg7}, \eqref{equldg9}, where
$$\Delta_j(\tau) \times
\Delta_j(\tau)^\vee=\delta[v^{-m_{2j}(\tau)}\tau,
v^{m_{2j+1}(\tau)}\tau] \times \delta[v^{-m_{2j+1}(\tau)}\tau,
v^{m_{2j}(\tau)}\tau].$$ By \cite{Zel80}, we know that if $D_i(\tau)$ is a
segment, then the unique generic constituent of the representation
$\delta(D_i(\tau)) \times \delta(D_i(\tau)^\vee)$ is
$$\delta(D_i(\tau) \cap D_i(\tau)^\vee) \times \delta(D_i(\tau)
\cup D_i(\tau)^\vee).$$
So the unique generic constituent of the
representation $\Delta_j(\tau) \times \Delta_j(\tau)^\vee$ is
\begin{equation*}
\begin{split}
\delta([\nu^{-m_{2j}(\tau)}\tau, \nu^{m_{2j}(\tau)}\tau]) \times \delta([\nu^{-m_{2j+1}(\tau)}\tau, \nu^{m_{2j+1}(\tau)}\tau]).
\end{split}
\end{equation*}
By multiplicativity of gamma functions, 
\begin{equation} \label{equ19}
\begin{split}
& \quad \gamma((\Delta_j(\tau) \times \Delta_j(\tau)^\vee) \times
\pi, s, \psi) \\
& =\gamma(\delta([\nu^{-m_{2j}(\tau)}\tau, \nu^{m_{2j}(\tau)}\tau]) \times \pi, s, \psi)
\gamma(\delta([\nu^{-m_{2j+1}(\tau)}\tau, \nu^{m_{2j+1}(\tau)}\tau]) \times \pi, s, \psi).
\end{split}
\end{equation}

In the second case the representation $\Delta_j(\tau)$ appears in \eqref{equldg8} and \eqref{equldg10}, 
similarly, we have
\begin{equation} \label{equ20}
\begin{split}
& \quad \gamma((\Delta_j(\tau) \times \Delta_j(\tau)^\vee) \times
\pi, s, \psi) \\
& =\gamma(\delta([\nu^{-m_{2j-1}(\tau)}\tau, \nu^{m_{2j-1}(\tau)}\tau]) \times \pi, s, \psi)
\gamma(\delta([\nu^{-m_{2j}(\tau)}\tau, \nu^{m_{2j}(\tau)}\tau]) \times \pi, s, \psi).
\end{split}
\end{equation}

By \eqref{equ19},
\begin{equation} \label{equ21}
\begin{split}
& \quad \prod_{\tau \in A_1(\rho)} \gamma(\tau \times \pi, s,
\psi) \prod_{j \in J_{\tau}} \gamma(\Delta_j(\tau) \times \pi, s,
\psi) \gamma(\Delta_j(\tau)^\vee \times \pi, s, \psi) \\
& =\prod_{\tau \in A_1(\rho)} \gamma(\tau \times \pi, s,
\psi) \prod_{k=2}^{d_{\tau}} \gamma(
\delta([\nu^{-m_{k}(\tau)}\tau, \nu^{m_{k}(\tau)}\tau])
\times \pi, s, \psi)\\
& =\prod_{\tau \in A_1(\rho)} \prod_{k=1}^{d_{\tau}} \gamma(
\delta([\nu^{-m_{k}(\tau)}\tau, \nu^{m_{k}(\tau)}\tau]) \times \pi, s, \psi).
\end{split}
\end{equation}
So, the product of type $(2)_1$ appears in \eqref{equ18}.

Similarly, by \eqref{equ20}, the product of type
$(3)_B$ appears in \eqref{equ18}, and also the following part of $(4)_C$
appears in \eqref{equ18}
\begin{equation} \label{equ23}
\begin{split}
& \quad \prod_{\substack{\tau \in C(\rho) \\ d_{\tau} \,\, even}} \prod_{j \in J_{\tau}}
\gamma(\Delta_j(\tau) \times \pi, s, \psi)
\gamma(\Delta_j(\tau)^\vee \times \pi, s, \psi) \\
& =\prod_{\substack{\tau \in C(\rho) \\ d_{\tau} \,\, even}} \prod_{i=1}^{d_{\tau}}
\gamma(\delta([\nu^{-m_{i}(\tau)}\tau, \nu^{m_{i}(\tau)}\tau]) \times \pi, s, \psi).
\end{split}
\end{equation}

The product $(2)_2$ can be treated as in \eqref{equ21}, except that we still have to
consider that $j=0$ and $\tau \in A_2(\rho)$ in \eqref{equldg7}, where
$$\Delta_0(\tau) \times \Delta_0(\tau)^\vee \times \tau =
\delta[v \tau, v^{m_1(\tau)} \tau] \times \delta[v^{-m_1(\tau)}
\tau, v^{-1} \tau] \times \tau.$$ By \cite{Zel80}, the unique generic
constituent of this representation is
$$\delta[v^{-m_1(\tau)} \tau, v^{m_1(\tau)} \tau].$$
So,
\begin{equation} \label{equ79}
\begin{split}
& \quad \gamma(\tau \times \pi, s, \psi) \gamma(\Delta_0(\tau) \times \pi, s, \psi)
\gamma(\Delta_0(\tau)^\vee, s, \psi)\\
& = \gamma(\delta[v^{-m_1(\tau)} \tau, v^{m_1(\tau)} \tau] \times \pi, s, \psi).
\end{split}
\end{equation}
By \eqref{equ19} and \eqref{equ79}, for $j \geq 1$
\begin{equation} \label{equ22}
\begin{split}
& \quad \prod_{\tau \in A_2(\rho)} \gamma(\tau \times \pi, s,
\psi) \prod_{j \in J_{\tau}} \gamma(\Delta_j(\tau) \times \pi, s,
\psi) \gamma(\Delta_j(\tau)^\vee \times \pi, s, \psi)\\
& = \prod_{\tau \in A_2(\rho)} \gamma(\delta[v^{-m_1(\tau)} \tau, v^{m_1(\tau)} \tau] \times \pi, s,
\psi) \prod_{i=2}^{d_{\tau}} \gamma(\delta[v^{-m_i(\tau)} \tau, v^{m_i(\tau)} \tau] \times \pi, s,
\psi)\\
& = \prod_{\tau \in A_2(\rho)} \prod_{i=1}^{d_{\tau}} \gamma(\delta[v^{-m_i(\tau)} \tau, v^{m_i(\tau)} \tau] \times \pi, s,
\psi).
\end{split}
\end{equation}
Therefore, the product $(2)_2$ appears in \eqref{equ18}.

The product of type $(4)_C$ with $d_{\tau}$ odd can be treated
similarly to the last case and to \eqref{equ22}. We only have to consider
$j=0$ and $\tau \in C(\rho)$ in \eqref{equldg9}, where
$$\Delta_0(\tau) \times \Delta_0(\tau)^\vee = \delta[v^{\frac{1}{2}} \tau, v^{m_1(\tau)} \tau] \times
\delta[v^{-m_1(\tau)} \tau, v^{-\frac{1}{2}} \tau].$$ By \cite{Zel80}, the
unique generic constituent of this representation is
$$\delta[v^{-m_1(\tau)} \tau, v^{m_1(\tau)} \tau].$$
As in \eqref{equ23} and \eqref{equ22}, 
\begin{equation} \label{equ24}
\begin{split}
& \quad \prod_{\substack{\tau \in C(\rho) \\ d_{\tau} \,\, odd}} \prod_{j \in J_{\tau}}
\gamma(\Delta_j(\tau) \times \pi, s, \psi)
\gamma({\Delta}_j(\tau)^\vee \times \pi, s, \psi)\\
& =\prod_{\substack{\tau \in C(\rho) \\ d_{\tau} \,\, odd}} \prod_{i=1}^{d_{\tau}}
\gamma(\delta[v^{-m_i(\tau)} \tau, v^{m_i(\tau)} \tau] \times \pi, s, \psi).
\end{split}
\end{equation}
Multiplying
\eqref{equ23} and \eqref{equ24}, we see that $(3)_C$ appears in \eqref{equ18}. Since
$$P(\rho) = A(\rho) \cup B(\rho) \cup C(\rho),$$
the identity \eqref{equ14} is now proved.

% Note that the generic constituent $\sigma_{\rho}$ of \eqref{equldg11} is not
% necessarily uniquely determined by \eqref{equ14}, because of Jiang's conjecture \ref{conj1}.

Since both $\sigma_{\rho}$ and $\rho$ are tempered, by \cite{Sha90a}, 
see also \cite{CKPSS04}, \cite{CS98} and \cite{JS04}, we know that \eqref{equldg2} and \eqref{equldg3} follow from \eqref{equ14}. \end{proof}

\begin{rem}\label{rmk3}
Note that if $\rho^{(2)} = l(\sigma^{(2)})$, then
the $A_2(\rho^{(2)})=\{\tau \in P'(\sigma^{(2)})|(\tau, \sigma^{(0)}) \text{ satisfies } (C1)\}$, see
Proposition \ref{gdsprop1} for the definition of $P'(\sigma^{(2)})$. 

% In this case, by Part (2), Part (3b) and Part (3c) of 
% Definition \ref{def1}, if $\{\Sigma_i\}_{i=1}^f$ is an
% $G_n(F)$-generic sequence of segments with respect to $\sigma^{(t)}$,
% then for $1 \leq i \leq f$, $\Sigma_i$ and $\check{\Sigma}_i$ are not linked
% to any segment associated to $\rho^{(2)}$.
% That is, as a representation of $GL_*(F)$, 
% $\delta(\Sigma_1) \times \cdots \times \delta(\Sigma_f) \times \rho^{(2)}$
% is irreducible and generic,
% where $\sigma^{(2)}$ is the irreducible discrete
% series generic representation occurring in $\sigma^{(t)}$,
% and $\rho^{(2)} = l(\sigma^{(2)})$.

% Note that if in Part (3b) of Definition \ref{def1}, 
% $1+w_i > b_j(\tau)$, then $\Sigma_i$ is linked to 
% $[\nu^{-b_j(\tau)}\tau, \nu^{b_j(\tau)} \tau]$,
% which is a segment associated to $\rho^{(2)}$
% (see the proof of Theorem \ref{thm7}).
% Then, $\delta(\Sigma_1) \times \cdots \times \delta(\Sigma_f) \times \rho^{(2)}$
% is no longer irreducible and generic.

% In the case of irreducible generic 
% representations of $SO_{2n+1}(F)$ and $Sp_{2n}(F)$  (see \cite{M2}), Part (2) of
% Definition 5.1 of \cite{JngS2} and Definition 4.14 of \cite{Liu}
% implies that $1+w_i \leq b_j(\tau)$ 
% holds in Part (3b).
\end{rem}

Next we generalize Theorem \ref{thm6} to $\Pi^{(dg)}(G_n)$. Let $\Phi^{(d)}(G_n)$ be the subset of $\Phi(G_n)$
consisting of all the local Langlands parameters of type
$$\phi = \bigoplus_i \phi_i \otimes S_{2m_i + 1},$$
where the $\phi_i$'s are irreducible 
self-dual (resp.  self-conjugate-dual in the case of unitary groups)
representations of $W_F$
of dimension $k_{\phi_i}$, the $S_{2m_i+1}$'s are irreducible
representations of $SL_2(\mathbb{C})$ of dimension $2m_i+1$, and they
satisfy the following:

(1) for each $i$, $\phi_i \otimes S_{2m_i+1}$ is irreducible
and $L(\phi_i \otimes S_{2m_i+1}, R, s)$ has a pole at $s=0$;

(2) $\phi_i \otimes S_{2m_i+1}$ and $\phi_j \otimes S_{2m_j+1}$ are
not equivalent if $i \neq j$;

(3) the image $\phi(W_F \times SL_2(\mathbb{C}))$ is not contained
in any proper Levi subgroup of ${}^LG_n$.

% Note that for any $\phi \in \Phi(G_n)$,  $det(\phi)$ is trivial except when $G_n=SO_{2n+2}^*$, in which case it is  
% the quadratic character $\eta_{\varepsilon}$ associated to the square class $\varepsilon$ defining $G_n$.

The local Langlands parameters in $\Phi^{(d)}(G_n)$ are called
$discrete$. 
Let $\widetilde{\Phi}^{(d)}(G_n)$ be the image of $\Phi^{(d)}(G_n)$ in $\widetilde{\Phi}(G_n)$.
The following theorem is analogous to  \cite[Theorem 2.2]{JS04}, \cite[Theorem 4.9]{Liu11}, and \cite[Theorem 4.10]{JL14}. 

\begin{thr} \label{thm10}
There is a surjective map $\iota$ $($which
extends the one in Theorem \ref{thm6}$)$ from $\Pi^{(dg)}(G_n)$ to the
set $\widetilde{\Phi}^{(d)}(G_n)$ and it  preserves the local factors:
\begin{equation}
L(\sigma \times \tau, s) = L(\iota(\sigma) \otimes r^{-1}(\tau),
s), 
\end{equation}
\begin{equation}
\epsilon(\sigma \times \tau, s, \psi) =
\epsilon(\iota(\sigma) \otimes r^{-1}(\tau), s, \psi),  
\end{equation}
for all $\sigma \in \Pi^{(dg)}(G_n)$ and all irreducible
generic representations $\tau$ of any $H_{k_{\tau}}(F)$,
$k_{\tau} \in \mathbb{Z}_{>0}$. Here, $r^{-1}(\tau)$ is the
irreducible admissible representation of $W_F \times SL_2(\mathbb{C})$ of dimension
$k_{\tau}$ corresponding to $\tau$ by the local Langlands
reciprocity map for $H_{k_{\tau}}$.\end{thr}

\begin{proof}
Given $\sigma \in \Pi^{(dg)}(G_n)$, by Theorem \ref{thm7}, $l(\sigma) \in \Pi^{(dg)}(H_N)$ and has the form \eqref{equldg1}. Let $\phi=
r^{-1}(l(\sigma))$, then $\phi \in \Phi^{(d)}(G_n)$. 
Define $\iota(\sigma)=\widetilde{\phi}$, the image of $\phi$ in $\widetilde{\Phi}(G_n)$.
Therefore, we have constructed a map $\iota$ from $\Pi^{(g)}(G_n)$ to
$\widetilde{\Phi}^{(g)}(G_n)$, which naturally extends the one in Theorem \ref{thm6}. Since $l$ preserves local factors, so is $\iota$.

To show that $\iota$ is surjective, given any $\widetilde{\phi} \in 
\widetilde{\Phi}^{(g)}(G_n)$, let 
$\phi \in \Phi^{(t)}(G_n)$ be a representative, which has the following multiplicity one
decomposition
\begin{equation} \label{equ28}
\phi = \bigoplus_{i=1}^r \phi_i \otimes S_{2m_i+1}, 2m_i \in \mathbb{Z}_{\geq0},
\end{equation}
where each $\phi_i \otimes S_{2m_i+1}$
is irreducible
and $L(\phi_i \otimes S_{2m_i+1}, R, s)$ has a pole at $s=0$ for $1 \leq i \leq r$. So, $\phi_i$ is
self-dual (resp.  self-conjugate-dual in the case of unitary groups),
and it is (conjugate)-orthogonal ((conjugate)-symplectic, respectively) if and
only if $S_{2m_i+1}$ is orthogonal (symplectic, respectively), i.e.
if and only if $m_i \in \mathbb{Z}_{>0}$ ($m_i \in \frac{1}{2} +
\mathbb{Z}_{>0}$, respectively). Hence, for $1 \leq i \leq r$,\\
(D1) if $L(R^-(\phi_i), s)$ has a pole at $s=0$, then $m_i
\in \frac{1}{2} + \mathbb{Z}_{>0}$;\\
(D2) if $L(R(\phi_i), s)$ has a pole at $s=0$, then $m_i \in
\mathbb{Z}_{>0}$.\\
In the case of unitary groups, for the definitions of conjugate-orthogonal, conjugate-symplectic, see \cite[Section 3]{GGP12}.

Let $\tau_i = r(\phi_i)$ be the irreducible self-dual (resp. self-conjugate-dual in the case of unitary groups) supercuspidal
representation of $H_{k_{\phi_i}}(F)$, corresponding to $\phi_i$.
By Theorem \ref{thm3}, we have
\begin{equation} \label{equ25}
L(R^-(\phi_i), s) \text{ has a pole at } s=0 \Leftrightarrow
L(\tau_i, R^-, s) \text{ has a pole at } s=0, 
\end{equation}
and
\begin{equation} \label{equ26}
L(R(\phi_i), s) \text{ has a pole at } s=0 \Leftrightarrow
L(\tau_i, R, s) \text{ has a pole at } s=0. 
\end{equation}
And
\begin{equation} \label{equ1000}
\begin{split}
r(\phi_i \otimes S_{2m_i+1}) & = \delta[v^{-m_i} r(\phi_i), v^{m_i} r(\phi_i)] =
\delta[v^{-m_i} \tau_i, v^{m_i} \tau_i];\\
r(\phi) & = \times_{i=1}^r 
\delta[v^{-m_i} \tau_i, v^{m_i} \tau_i].
\end{split}
\end{equation}
So $r(\phi)$ has the form \eqref{equldg1}, and by \eqref{equ25}, \eqref{equ26} and the conditions (D1),
(D2), the conditions (1) and (2) in the definition of $\Pi^{(dg)}(H_N)$ hold, i.e., $r(\phi) \in \Pi^{(dg)}(H_N(F))$. Therefore, by Theorem \ref{thm7}, there
exists a square-integrable generic representation $\sigma$ of
$G$, such that
$l(\sigma) = r(\phi)$, and
$$L(\sigma \times r(\phi{'}), s) = L(r(\phi) \times r(\phi{'}), s)
= L(\phi \otimes \phi{'}, s)$$
$$\epsilon(\sigma \times r(\phi{'}), s, \psi) = \epsilon(r(\phi) \times r(\phi{'}), s, \psi)
= L(\phi \otimes \phi{'}, s, \psi)$$ for generic parameters $\phi{'}$ for $H_k(F)$ with $k \in
\mathbb{Z}_{>0}$ (i.e., $r(\phi{'})$ is an irreducible generic representations of $H_k(F)$). 
As a result, $\iota(\sigma) = \widetilde{\phi}$.

This completes the proof of the theorem.
\end{proof}

\begin{rem}\label{rmk7}
Suppose
$\sigma^{(2)} \in \Pi^{(dg)}(G_n)$ is the unique generic
constituent of 
$(\times_{\tau \in P{'}} \times_{i=1}^{e_{\tau}} \delta(D_i(\tau)))
\rtimes \sigma^{(0)}$
(possibly $\sigma^{(0)}= 1 \otimes c$) and $\sigma^{(0)} \ncong c \sigma^{(0)}$.
Then, by \S \ref{gdssection}, $\sigma^{(2)} \ncong c \sigma^{(2)}$ are both in $\Pi^{(dg)}(G_n)$,
and $c \sigma^{(2)}$ is the unique generic
constituent of 
$(\times_{\tau \in P{'}} \times_{i=1}^{e_{\tau}} \delta(D_i(\tau)))
\rtimes c \sigma^{(0)}$. 
%(n.b.: if $\sigma^{(2)}$ is $\chi$-generic, $c\sigma^{(2)}$ is $c\chi$-generic). 
Note that if $(\tau, \sigma^{(0)})$
satisfies $(C \beta)$, then so does $(\tau, c \sigma^{(0)})$.
%and when used with parabolic induction, $1 \otimes e$
%and $1 \otimes c$ are not regarded as the same representation.

By Remark \ref{rmk6}, Theorem \ref{thm7} and Theorem \ref{thm10},
and by the multiplicativity of local factors (see \cite{Sha90b}, \cite{JS12} and \cite{CKPSS04}),
in the above situation, $\sigma^{(2)}$ and $c \sigma^{(2)}$
have the same lifting image and the same twisted local factors.
\end{rem}

\subsection{Tempered generic representations}

Let $\Pi^{(tg)}_{\varepsilon}(H_N)$ be 
the set of equivalence classes of
tempered representations of $H_N(F)$ of the following form
\begin{equation} \label{equltg1}
\delta([\nu^{-h_1}\lambda_1, \nu^{h_1} \lambda_1]) \times \delta([\nu^{-h_2}\lambda_2, \nu^{h_2} \lambda_2]) \times \dots
\times \delta([\nu^{-h_f}\lambda_f, \nu^{h_f} \lambda_f]), 
\end{equation}
with central character $\chi$ being trivial when restricting to $F^*$ except when $G_n=SO_{2n+2}^*$, in which case it is 
the quadratic character $\eta_{\varepsilon}$  associated to the square class $\varepsilon$ defining $G_n$.  
Here $\lambda_1,
\lambda_2, \dots \, \lambda_f$ are unitary supercuspidal
representations, and $2h_i \in \mathbb{Z}_{\geq0}$, such that for $1
\leq i \leq f$:

$(1)$ if $\lambda_i \not \cong \check{\lambda}_i$, then
$\delta([\nu^{-h_i}\lambda_i, \nu^{h_i} \lambda_i])$ occurs in \eqref{equltg1} as many times as
$\check\delta([\nu^{-h_i}\lambda_i, \nu^{h_i} \lambda_i])=\delta([\nu^{-h_i}\check\lambda_i, \nu^{h_i} \check\lambda_i])$ does;

$(2)$ if $L(\lambda_i, R^-, s)$ has a pole at $s=0$, and $h_i
\in \mathbb{Z}_{\geq0}$, then $\delta([\nu^{-h_i}\lambda_i, \nu^{h_i} \lambda_i])$ occurs an
even number of times in \eqref{equltg1};

$(3)$ if $L(\lambda_i, R, s)$ has a pole at $s=0$, and $h_i
\in \frac{1}{2} + \mathbb{Z}_{\geq0}$, then $\delta([\nu^{-h_i}\lambda_i, \nu^{h_i} \lambda_i])$ 
occurs an even number of times in \eqref{equltg1}.

The following theorem is analogous to  \cite[Theorem 4.1]{JS04}, \cite[Theorem 4.12]{Liu11}, and \cite[Theorem 4.12]{JL14}.

\begin{thr} \label{thm9}
There is a surjective map $l$ $($which extends the one in Theorem \ref{thm7}$)$
from $\Pi^{(tg)}(G_n)$ to $\Pi^{(tg)}_{\varepsilon}(H_N)$ and it preserves the local factors:
\begin{equation} \label{equltg2}
L(\sigma \times \pi, s) = L(l(\sigma) \times \pi, s), 
\end{equation}
\begin{equation} \label{equltg3}
\epsilon(\sigma \times \pi, s, \psi) = \epsilon(l(\sigma) \times \pi, s, \psi),
\end{equation}
for any $\sigma \in \Pi^{(tg)}(G_n)$ and any
irreducible generic representation $\pi$ of any $H_k(F)$, $k
\in \mathbb{Z}_{>0}$.\end{thr}

\begin{proof}
This map has already been given by Cogdell, Kim, Piatetski-Shapiro, and Shahidi (see \cite{CKPSS04} and \cite{CPSS11}), so it suffices to prove the surjectivity. That is, given a $\rho \in \Pi^{(tg)}(H_N(F))$, to construct a $\sigma_{\rho} \in \Pi^{(tg)}(G_n)$ such that $\rho =
l(\sigma_{\rho})$ and \eqref{equltg2}, \eqref{equltg3} hold.

Define the following sets $N, W$, and $R$ from the factors
of the induced representation in \eqref{equltg1}:

$N$ consists of $\delta([\nu^{-h_i}\lambda_i, \nu^{h_i} \lambda_i])'s$ with $1 \leq i \leq f$
such that $\lambda_i \not \cong \check{\lambda}_i$;

$W$ consists of $\delta([\nu^{-h_i}\lambda_i, \nu^{h_i} \lambda_i]){'}s$ with $1 \leq i \leq f$
such that $L(\lambda_i, R^-, s)$ has a pole at $s=0$, and $h_i
\in \mathbb{Z}_{\geq0}$, or $L(\lambda_i, R, s)$ has a pole at
$s=0$, and $h_i \in \frac{1}{2} + \mathbb{Z}_{\geq0}$;

$S$ consists of $\delta([\nu^{-h_i}\lambda_i, \nu^{h_i} \lambda_i])'s$ with $1 \leq i \leq f$
such that $L(\lambda_i, R^-, s)$ has a pole at $s=0$, and $h_i
\in \frac{1}{2} + \mathbb{Z}_{\geq0}$, or $L(\lambda_i, R, s)$
has a pole at $s=0$, and $h_i \in \mathbb{Z}_{\geq0}$.

Note that these sets are taken with multiplicities. Denote by
$\mu_i{'}$ the multiplicity of $\delta([\nu^{-h_i}\lambda_i, \nu^{h_i} \lambda_i])$ in \eqref{equltg1}.

By the above definition, for $\delta([\nu^{-h_i}\lambda_i, \nu^{h_i} \lambda_i]) \in W$, $\mu_i{'} =
2\mu_i$ is even. Let
$$\{\delta([\nu^{-h_{i_1}}\lambda_{i_1}, \nu^{h_{i_1}} \lambda_{i_1}]), \delta([\nu^{-h_{i_2}}\lambda_{i_2}, \nu^{h_{i_2}} \lambda_{i_2}]), \cdots, \delta([\nu^{-h_{i_u}}\lambda_{i_u}, \nu^{h_{i_u}} \lambda_{i_u}])$$ be the set of all different elements in $W$.
Let
$$J_W(\rho) = \times_{j=1}^u \underbrace{(\delta([\nu^{-h_{i_j}}\lambda_{i_j}, \nu^{h_{i_j}} \lambda_{i_j}]) \times
\cdots \times \delta([\nu^{-h_{i_j}}\lambda_{i_j}, \nu^{h_{i_j}} \lambda_{i_j}])}_{\mu_{i_j} \, copies}.$$

By the above definition, if $\delta([\nu^{-h_{i}}\lambda_{i}, \nu^{h_{i}} \lambda_{i}]) \in N$, then
$\check{\delta}([\nu^{-h_{i}}\lambda_{i}, \nu^{h_{i}} \lambda_{i}]) = \delta([\nu^{-h_{i}}\check{\lambda}_{i}, \nu^{h_{i}} \check{\lambda}_{i}])
\in N$, and the multiplicities of $\delta([\nu^{-h_{i}}\lambda_{i}, \nu^{h_{i}} \lambda_{i}])$ and $\check{\delta}([\nu^{-h_{i}}\lambda_{i}, \nu^{h_{i}} \lambda_{i}])$ are equal. Let
$\{\delta([\nu^{-h_{z_1}}\lambda_{z_1}, \nu^{h_{z_1}} \lambda_{z_1}]), \check{\delta}([\nu^{-h_{z_1}}\lambda_{z_1}, \nu^{h_{z_1}} \lambda_{z_1}])$,
 $\cdots, \delta([\nu^{-h_{z_v}}\lambda_{z_v}, \nu^{h_{z_v}} \lambda_{z_v}]), \check{\delta}([\nu^{-h_{z_v}}\lambda_{z_v}, \nu^{h_{z_v}} \lambda_{z_v}])\}$
be the set of different elements in $N$.

Let
$$J_N(\rho) = \times_{j=1}^v \underbrace{(\delta([\nu^{-h_{z_j}}\lambda_{z_j}, \nu^{h_{z_j}} \lambda_{z_j}]) \times
\cdots \times \delta([\nu^{-h_{z_j}}\lambda_{z_j}, \nu^{h_{z_j}} \lambda_{z_j}])}_{\mu_{z_j}{'} \,
copies}.$$

For $S$, we consider following two cases
\begin{eqnarray*}
& & S_1 = \{\delta([\nu^{-h_i}\lambda_i, \nu^{h_i} \lambda_i]) \in S, | \mu_i{'} = 2\mu_i+1
\text{ is odd} \},\\
& & S_2 = \{\delta([\nu^{-h_i}\lambda_i, \nu^{h_i} \lambda_i]) \in S, | \mu_i{'} = 2\mu_i
\text{ is even} \}.
\end{eqnarray*}

Let
\begin{align*}
     & \{\delta([\nu^{-m_1}\tau_1, \nu^{m_1} \tau_1]), \cdots, \delta([\nu^{-m_r}\tau_r, \nu^{m_r} \tau_r])\} \\
     := & \{\delta([\nu^{-h_{x_1}}\lambda_{x_1}, \nu^{h_{x_1}} \lambda_{x_1}]), \cdots, \delta([\nu^{-h_{x_r}}\lambda_{x_r}, \nu^{h_{x_r}} \lambda_{x_r}])\}   
\end{align*} 
be the set of all different elements of $S_1$, and let
\begin{align*}
\{\delta([\nu^{-h_{t_1}}\lambda_{t_1}, \nu^{h_{t_1}} \lambda_{t_1}]), \cdots, \delta([\nu^{-h_{t_c}}\lambda_{t_c}, \nu^{h_{t_c}} \lambda_{t_c}])\}   
\end{align*}
be
the set of all different elements of $S_2$. In these two cases, we
define
\begin{equation} \label{equ51}
J_{S_1}(\rho) = \times_{j=1}^r \underbrace{(\delta([\nu^{-h_{x_j}}\lambda_{x_j}, \nu^{h_{x_j}} \lambda_{x_j}]) \times
\cdots \times \delta([\nu^{-h_{x_j}}\lambda_{x_j}, \nu^{h_{x_j}} \lambda_{x_j}])}_{\mu_{x_j} \, copies}, 
\end{equation}
\begin{equation} \label{equ52}
J_{S_2}(\rho) = \times_{j=1}^c \underbrace{(\delta([\nu^{-h_{t_j}}\lambda_{t_j}, \nu^{h_{t_j}} \lambda_{t_j}]) \times
\cdots \times \delta([\nu^{-h_{t_j}}\lambda_{t_j}, \nu^{h_{t_j}} \lambda_{t_j}])}_{\mu_{t_j} \,
copies}.
\end{equation}
Note that in \eqref{equ51}, twice of the multiplicity of
$\delta([\nu^{-h_{x_j}}\lambda_{x_j}, \nu^{h_{x_j}} \lambda_{x_j}])$ is decreased by 1. The reason for doing this is that we want to use Theorem
\ref{thm7}.

Then by assumption, the induced representation
\begin{equation}\label{rho2}
    \rho^{(2)} := \delta([\nu^{-m_1}\tau_1, \nu^{m_1} \tau_1]) \times \cdots \times \delta([\nu^{-m_r}\tau_r, \nu^{m_r} \tau_r])
\end{equation}
is in $\Pi^{(dg)}_{\varepsilon}(H_{n{''}})$, for some $n{''}$. Hence by Theorem \ref{thm7}, there exists $\sigma^{(2)} \in \Pi^{(dg)}(G_{n{''}})$, such that $l(\sigma^{(2)}) =
\rho^{(2)}$. 

Let $\delta([\nu^{-p_1}\eta_1, \nu^{p_1} \eta_1])
\cdots \delta([\nu^{-p_d}\eta_d, \nu^{p_d} \eta_d])
$ be the list of all factors (with possible repetitions)
which appear in $J_N(\rho) \times J_W(\rho) \times J_{S_1}(\rho)
\times J_{S_2}(\rho)$. 
Define $\sigma_{\rho}$ to be the
unique generic constituent of
\begin{align}\label{equ53}
\begin{split}
\delta([\nu^{-p_1}\eta_1, \nu^{p_1} \eta_1])
\times \cdots \times  \delta([\nu^{-p_d}\eta_d, \nu^{p_d} \eta_d]) \rtimes \sigma^{(2)}.  
\end{split}
\end{align}
Then by the discussion in \S \ref{gtsection}, $\sigma_{\rho}$ is in $\Pi^{(tg)}(G_{n})$.

Next, we show that \eqref{equltg2} and \eqref{equltg3} hold. First we show that
$$\gamma(\sigma_{\rho} \times \pi, s, \psi) = \gamma(\rho \times
\pi, s, \psi).$$ 
Again, by the multiplicativity of twisted gamma functions (\cite{Sha90b}), it is  enough to show this for
supercuspidal representation $\pi$, see also \cite[Lemma 7.2]{CKPSS04}. 

By the multiplicativity of twisted gamma functions, and by Theorems \ref{thm5} and \ref{thm7}, for an irreducible
supercuspidal representation $\pi$ of $H_k(F)$, we have
\begin{eqnarray*}
& & \quad \gamma(\sigma_{\rho} \times \pi, s, \psi)\\
& & =[\prod_{i=1}^d \gamma(\delta([\nu^{-p_i}\eta_i, \nu^{p_i} \eta_i]) \times \pi, s, \psi)
\gamma(\check{\delta}([\nu^{-p_i}\eta_i, \nu^{p_i} \eta_i]) \times \pi, s,
\psi)]\\
& & \quad \cdot 
[\prod_{i=1}^r \gamma(\delta([\nu^{-m_i}\tau_i, \nu^{m_i} \tau_i]) \times \pi, s, \psi)]\\
& & =[\prod_{i=1}^f \gamma(\delta([\nu^{-h_i}\lambda_i, \nu^{h_i} \lambda_i]) \times \pi, s,
\psi)]\\
& & =\gamma(\rho \times
\pi, s, \psi).\\
\end{eqnarray*}
Since $\sigma_{\rho}$ and $\rho$ are both tempered, we can get
\eqref{equltg2} and \eqref{equltg3} in the same way.
This completes the proof.
\end{proof}

\begin{rem}\label{rmk4}
%{\color{red}delete this remark???}
By Theorem \ref{thm9}, for each $\sigma^{(t)}
\in \Pi^{(tg)}(G_n)$ as in \eqref{tempered generic},
\begin{align}\label{equltg4}
\begin{split}
\rho^{(t)}=l(\sigma^{(t)}) & = \delta([\nu^{-e_1}\beta_1, \nu^{e_1}\beta_1])\times \dots \times \delta([\nu^{-e_c}\beta_c, \nu^{e_c}\beta_c]) \times l(\sigma^{(2)})\\
& \quad \times \delta([\nu^{-e_c}\check{\beta}_c, \nu^{e_c}\check{\beta}_c])
\times \dots \times \delta([\nu^{-e_1}\check{\beta}_1, \nu^{e_1}\check{\beta}_1]),
\end{split}
\end{align}
which is irreducible and generic.
\end{rem}

Next, we write down the parameters for representations in $\Pi^{(tg)}(G_n)$.
From \eqref{equltg4}, we can see that the local Langlands parameter of $\sigma$ is
$$\phi_{\sigma^{(2)}} \oplus \bigoplus_{i=1}^c [\phi_{\beta_i}
\times S_{2e_i+1} \oplus \check{\phi}_{\beta_i} \times
S_{2e_i+1}],$$
where $\check{\phi}_{\beta_i}$ is irreducible representation of $W_F$ corresponding to $(r({\phi}_{\beta_i}))^{\vee}$ under the local Langlands
reciprocity map $r$ for general linear groups.

Let $\Phi^{(t)}(G_n)$ be the subset of
$\Phi(G_n)$ consisting of the local Langlands parameters $\phi$
with the property that $\phi(W_F)$ is bounded. The parameters in
$\Phi^{(t)}(G_n)$ are called $tempered$. Then we have the following 
result that the local Langlands parameters corresponding to representations in 
$\Pi^{(tg)}(G_n)$ are exactly the tempered parameters.
Let $\widetilde{\Phi}^{(t)}(G_n)$ be the image of $\Phi^{(t)}(G_n)$ in $\widetilde{\Phi}(G_n)$.
The following theorem is analogous to  \cite[Theorem 4.2]{JS04}, \cite[Theorem 4.13]{Liu11}, and \cite[Theorem 4.14]{JL14}.

\begin{thr} \label{thm11}
There is a surjective map $\iota$ $($which
extends the one in Theorem \ref{thm10}$)$ from $\Pi^{(tg)}(G_n)$ to the
set $\widetilde{\Phi}^{(t)}(G_n)$ and it preserves the local factors:
$$L(\sigma \times \tau, s) = L(\iota(\sigma) \otimes r^{-1}(\tau),
s),$$
$$\epsilon(\sigma \times \tau, s, \psi) =
\epsilon(\iota(\sigma) \otimes r^{-1}(\tau), s, \psi),$$
for all $\sigma \in \Pi^{(tg)}(G_n)$ and all irreducible
generic representations $\tau$ of any $H_{k_{\tau}}(F)$,
$k_{\tau} \in \mathbb{Z}_{>0}$. Here, $r^{-1}(\tau)$ is the
irreducible admissible representation of $W_F \times SL_2(\mathbb{C})$ of dimension
$k_{\tau}$ corresponding to $\tau$ by the local Langlands
reciprocity map for $H_{k_{\tau}}$.
\end{thr}

\begin{proof}
Let $\sigma$ be an irreducible tempered generic
representation of $G_n(F)$. Consider $l(\sigma) \in \Pi^{(tg)}(H_N)$ given in Theorem \ref{thm9}.
Then $\phi = r^{-1}(l(\sigma)) \in \Phi^{(t)}(G_{n})$. So $l(\sigma) = r(\phi)$.
Define $\iota(\sigma)=\widetilde{\phi}$, the image of $\phi$ in $\widetilde{\Phi}(G_n)$.
Therefore, we have constructed a map $\iota$ from $\Pi^{(g)}(G_n)$ to
$\widetilde{\Phi}^{(g)}(G_n)$, which naturally extends the one in Theorem \ref{thm10}. Since $l$ preserves local factors, so is $\iota$.

To show $\iota$ is surjective, given any $\widetilde{\phi} \in 
\widetilde{\Phi}^{(g)}(G_n)$, let 
$\phi \in \Phi^{(t)}(G_n)$ be a representative. Composing $\phi$ with the embedding 
$$i: {}^L G_n \hookrightarrow {}^L H_N,$$
we obtain a
$N$-dimensional representation of $W_F \times SL_2(\mathbb{C})$.
Since the image of $\phi$ preserves a non-degenerate 
bilinear form, so it can be decomposed into the
following form
$$\phi = J_N{'}(\phi) \oplus J_W{'}(\phi) \oplus J_{S_2}{'}(\phi)
\oplus J_{S_1}{'}(\phi) \oplus J_2(\phi),$$ where each
summand is as follows.

$J_N{'}(\phi)$ is
\begin{equation} \label{equ54}
J_N{'}(\phi) = \bigoplus_{j=1}^v \mu{'}_{z_j} (\phi_{z_j}
\otimes S_{2h_{z_j}+1} \oplus \check{\phi}_{z_j} \otimes
S_{2h_{z_j}+1}), 
\end{equation}
with the properties that

(1) $2h_{z_j} \in \mathbb{Z}_{\geq0}$;

(2) $\mu{'}_{z_j} \in \mathbb{Z}_{>0}$ are the multiplicities;

(3) $\phi_{z_j} \not \cong \check{\phi}_{z_j}$;

(4) $\phi_{z_1}, \phi_{z_2}, \cdots, \phi_{z_v}$ are pairwise
non-equivalent irreducible bounded representations of $W_F$.

$J_W{'}(\phi)$ is
\begin{equation} \label{equ55}
J_W{'}(\phi) = \bigoplus_{j=1}^u 2\mu_{i_j} (\phi_{i_j}
\otimes S_{2h_{i_j}+1}), 
\end{equation}
with the properties that

(1) $2h_{i_j} \in \mathbb{Z}_{\geq0}$;

(2) $\mu_{i_j} \in \mathbb{Z}_{>0}$ are the half of the
multiplicities;

(3) $\phi_{i_1}, \phi_{i_2}, \cdots, \phi_{i_u}$ are pairwise
non-equivalent irreducible bounded self-dual (resp. self-conjugate-dual in the case of unitary groups) representations of
$W_F$, such that $\phi_{i_j} \otimes S_{2h_{i_j}+1}{'}s$ are
of type $R^-$. 
That is for each $j$, either $\phi_{i_j}$ is of type $R^-$
and $h_{i_j} \in \mathbb{Z}_{\geq0}$, or $\phi_{i_j}$ is of type $R$
and $h_{i_j} \in \frac{1}{2} + \mathbb{Z}_{\geq0}$.

$J_{S_2}{'}(\phi)$ is
\begin{equation} \label{equ56}
J_{S_2}{'}(\phi) = \bigoplus_{j=1}^c 2\mu_{t_j} (\phi_{t_j}
\otimes S_{2h_{t_j}+1}),
\end{equation}
with the properties that

(1) $2h_{t_j} \in \mathbb{Z}_{\geq0}$;

(2) $\mu_{t_j} \in \mathbb{Z}_{>0}$ are the half of the
multiplicities;

(3) $\phi_{t_1}, \phi_{t_2}, \cdots, \phi_{t_c}$ are pairwise
non-equivalent irreducible bounded self-dual (resp. self-conjugate-dual in the case of unitary groups) representations of
$W_F$, such that $\phi_{t_j} \otimes S_{2h_{t_j}+1}{'}s$ are
of type $R$. That is for each $j$, either $\phi_{t_j}$ is of type $R^-$
and $h_{t_j} \in \frac{1}{2} + \mathbb{Z}_{\geq0}$, or $\phi_{t_j}$
is of type $R$ and $h_{t_j} \in \mathbb{Z}_{\geq0}$.

$J_{S_1}{'}(\phi)$ and $J_2(\phi)$ are
\begin{equation} \label{equ57}
J_{S_1}{'}(\phi) \oplus J_2(\phi)= \bigoplus_{j=1}^r (2\mu_{x_j}+1) (\phi_{x_j}
\otimes S_{2h_{x_j}+1}),
\end{equation}
\begin{equation} \label{equ1012}
J_2(\phi)= \bigoplus_{j=1}^r (\phi_{x_j}
\otimes S_{2h_{x_j}+1)}, 
\end{equation}
with the properties that

(1) $2h_{x_j} \in \mathbb{Z}_{\geq0}$;

(2) $\mu_{x_j} \in \mathbb{Z}_{>0}$, $2\mu_{x_j}+1$ are the
multiplicities;

(3) $\phi_{x_1}, \phi_{x_2}, \cdots, \phi_{x_r}$ are pairwise
non-equivalent irreducible bounded self-dual (resp. self-conjugate-dual in the case of unitary groups) representations of
$W_F$, such that $\phi_{x_j} \otimes S_{2h_{x_j}+1}{'}s$ are
of type $R$. We note that some of the summations in \eqref{equ54}, \eqref{equ55}, 
\eqref{equ56} and \eqref{equ1012}, 
may be empty.

Let $\sigma$ be the unique irreducible generic constituent of the
following induced representation of $G_{n}(F)$
$$J_N(\sigma) \times J_W(\sigma) \times J_{S_2}(\sigma) \times J_{S_1}(\sigma) 
\rtimes \sigma^{(2)},$$ where
\begin{eqnarray*}
& & J_N(\sigma) = \times_{j=1}^v \delta([\nu^{-h_{z_j}}r(\phi_{z_j}),
\nu^{h_{z_j}}r(\phi_{z_j})])^{\times \mu{'}_{z_j}},\\
& & J_W(\sigma) = \times_{j=1}^u \delta([\nu^{-h_{i_j}}r(\phi_{i_j}),
\nu^{h_{i_j}}r(\phi_{i_j})])^{\times \mu_{i_j}},\\
& & J_{S_2}(\sigma) = \times_{j=1}^c
\delta([\nu^{-h_{t_j}}r(\phi_{t_j}),
\nu^{h_{t_j}}r(\phi_{t_j})])^{\times \mu_{t_j}},\\
& & J_{S_1}(\sigma) = \times_{j=1}^r
\delta([\nu^{-h_{x_j}}r(\phi_{x_j}),
\nu^{h_{x_j}}r(\phi_{x_j})])^{\times \mu_{x_j}},
\end{eqnarray*}
and $\sigma^{(2)} \in \Pi^{(dg)}(
G_{n{''}})$ (for some $n{''}$) is the one given in Theorem \ref{thm10}, such that 
$y(\sigma^{(2)}) = J_2(\phi)$. Then
$\sigma$ is tempered and generic, and by the proof of Theorem \ref{thm9},
we have that $l(\sigma) = r(\phi)$ and 
$$L(\sigma \times \tau, s) = L(r^{-1}(l(\sigma))) \otimes r^{-1}(\tau), s) = L(\iota(\sigma) \otimes r^{-1}(\tau),
s),$$
$$\epsilon(\sigma \times \tau, s, \psi) = \epsilon(r^{-1}(l(\sigma))) \otimes r^{-1}(\tau), s, \psi) = \epsilon(\iota(\sigma) \otimes r^{-1}(\tau),
s, \psi),$$ for all irreducible generic representations $\tau$
of $H_{k_{\tau}}(F)$, with all $k_{\tau} \in \mathbb{Z}_{>0}$. Here
$r^{-1}(\tau)$ is the irreducible admissible representation of $W_F \times SL_2(\mathbb{C})$
of dimension $k_{\tau}$, corresponding to $\tau$ by the local
Langlands reciprocity map for $H_{k_{\tau}}$.
Therefore, $\widetilde{\phi} = \iota(\sigma)$.

This completes the proof of the theorem.
\end{proof}

\begin{rem}\label{rmk8}
By \S 3.3, if $\sigma^{(t)} \in \Pi^{(tg)}(G_n)$ 
is the unique generic constituent of
$\delta([\nu^{-e_1}\beta_1, \nu^{e_1}\beta_1]) \times \dots \times \delta([\nu^{-e_c}\beta_c, \nu^{e_c}\beta_c])
\rtimes \sigma^{(2)}$ (possibly $\sigma^{(2)} = 1 \otimes c$)
and $\sigma^{(2)} \ncong c \sigma^{(2)}$, then $\sigma^{(t)} \ncong c \sigma^{(t)}$,
both are in $\Pi^{(tg)}(G_n)$, and $c \sigma^{(t)}$
is the unique generic constituent of
$\delta([\nu^{-e_1}\beta_1, \nu^{e_1}\beta_1]) \times \dots \times \delta([\nu^{-e_c}\beta_c, \nu^{e_c}\beta_c])
\rtimes c \sigma^{(2)}$.
%(n.b.: if $\sigma^{(t)}$ is $\chi$-generic, $c\sigma^{(t)}$ is $c\chi$-generic). 

By Remark \ref{rmk7}, Theorem \ref{thm9}, Theorem \ref{thm11},
and the multiplicativity of local factors (see \cite{Sha90b}, \cite{JS12} and \cite{CKPSS04}),
in the above situation, $\sigma^{(t)}$ and $c \sigma^{(t)}$
have the same lifting image and the same twisted local factors.
\end{rem}

\subsection{Generic representations}

 First, we make the following definition which is based on Theorem \ref{thm1intro} and the classification of generic representations in \S \ref{sect4}. Note that $\beta=0$ for the groups considered in this section.

\begin{defn} \label{def1}
Let $G_n = SO_{2n+1}, Sp_{2n}, SO_{2n}, SO_{2n+2}^*, U_{2n+1}, U_{2n}$, quasi-split classical groups of rank $n$.
Let $\{\Sigma_i\}_{i=1}^f$ and $\sigma^{(t)}$ be as in \S \ref{generic}.
Then $\{\Sigma_i\}_{i=1}^f$ is called a
$G_n$-generic sequence of segments with respect to $\sigma^{(t)}$
if it satisfies the following conditions,

$(1)$ the segment $\Sigma_i$ is not linked to either $\Sigma_j$ or
$\check{\Sigma}_j$ for $1 \leq i \neq j \leq f$.

$(2)$ for $1 \leq i \leq f$, $\Sigma_i$ and $\check{\Sigma}_i$ are not linked to any
segment, which corresponds to a representation in any of the
families
$$\delta([\nu^{-a_i(\tau)}\tau, \nu^{b_i(\tau)}\tau]), i=1, 2, \cdots,
e_{\tau}, \tau \in P',$$
$$ \{\delta([\nu^{-e_j}\beta_j, \nu^{e_j}\beta_j])\}_{j=1}^c, \{\delta([\nu^{-e_j}\check\beta_j, \nu^{e_j}\check\beta_j]) |
\beta_j \not \cong \check{\beta}_j, 1 \leq j \leq c \};$$

$(3)$ one of the following three conditions holds,

$(3a)$ $\xi_i \not \cong \check{\xi_i}$; or

$(3b)$ there exists $\tau \in
X'$, such that $\tau \cong \xi_i$, $q_i = -1$,
and there is some $1 \leq j \leq e_{\tau}$, with 
$a_j(\tau) = -1$ and $1+w_i \leq b_j(\tau)$; or

$(3c)$ $(\xi_i, \sigma^{(0)})$ is $(C\alpha)$ $(\alpha = 0,
\frac{1}{2}, 1)$, but $\pm\alpha \not \in \{-q_i, -q_i + 1, \cdots,
-q_i + w_i\}$; $(\xi_i, \sigma^{(0)})$ is $(CN)$, but $q_i \notin \mathbb{Z}_{\geq 0}$.

See Definition \ref{def2} and Proposition \ref{gdsprop1} for the definitions of $P'$ and $X'$. 
\end{defn}

Let $\Pi^{(g)}_{\varepsilon}(H_N)$ be 
the set of equivalence classes of
irreducible representations $\pi$ of $H_N(F)$  
which are
Langlands quotients of representations
\begin{equation} \label{equlg1}
\delta(\Sigma_1) \times \cdots \times
\delta(\Sigma_f) \times \rho^{(t)} \times \delta(\check{\Sigma}_f)
\times \cdots \times \delta(\check{\Sigma}_1), 
\end{equation}
with central character $\chi$ being trivial when restricting to $F^*$ except when $G_n=SO_{2n+2}^*$, in which case it is 
the quadratic character $\eta_{\varepsilon}$  associated to the square class $\varepsilon$ defining $G_n$.  
Here
$\{\Sigma_j\}_{j=1}^f$ are of the form \eqref{equg1}, $\xi_1, \xi_2,
\cdots, \xi_f$ are irreducible unitary and supercuspidal, with
possible repetitions, $q_i \in \mathbb{R}$, $w_i \in
\mathbb{Z}_{\geq0}$, and $\rho^{(t)}
\in \Pi^{(tg)}_{\alpha}(H_{N^*})$, such that the following hold:

$(1)$ $\frac{w_1}{2} - q_1 \geq \frac{w_2}{2} - q_2 \geq \cdots \geq
\frac{w_f}{2} - q_f > 0$;

$(2)$ The segment $\Sigma_i$ is not linked to either $\Sigma_j$ or
$\check{\Sigma}_j$ for $1 \leq i \neq j \leq f$;

$(3)$ The representations $\delta(\Sigma_i) \times \rho^{(t)}$ and
$\delta(\check{\Sigma}_i) \times \rho^{(t)}$ are irreducible for
all $1 \leq i \leq f$;

$(4)$ Assume $\xi_i \cong \check{\xi}_i$ and $2q_i \in \mathbb{Z}$, such
that if $L(\xi_i, R^-, s)$ has a pole at $s = 0$, then $q_i \in
\frac{1}{2} + \mathbb{Z}$, and if $L(\xi_i, R, s)$ has a pole at
$s = 0$, then $q_i \in \mathbb{Z}$. Then $\Sigma_i$ is
not linked to $\check{\Sigma}_i$. Moreover, if $L(\rho^{(0)}
\times \xi_i, s)$ has a pole at $s=0$
and $q_i \in \mathbb{Z}$, then either (a) $-q_i \geq 2$ or (b) $q_i =
-1$, $\xi_i = \tau
\in A_2(\rho^{(2)}$) and there is some $1 \leq j \leq e_{\tau}$
such that $a_j(\tau) = -1$ and $1 + w_i \leq b_j(\tau)$. See \eqref{rho2} for $\rho^{(2)}$ and see 
\eqref{equldg} for the definition of $A_2(\rho^{(2)})$.

The following theorem is analogous to  \cite[Theorem 5.1]{JS04}, \cite[Theorem 4.15]{Liu11}, and \cite[Theorem 4.16]{JL14}.

\begin{thr}\label{thmgen}
There is a surjective map $l$ $($which extends the one in Theorem \ref{thm9}$)$
from $\Pi^{(g)}(G_n)$ to $\Pi^{(g)}_{\varepsilon}(H_N)$ and it preserves the local factors:
\begin{equation} \label{equlg2}
L(\sigma \times \pi, s) = L(l(\sigma) \times \pi, s), 
\end{equation}
\begin{equation} \label{equlg3}
\epsilon(\sigma \times \pi, s, \psi) = \epsilon(l(\sigma) \times \pi, s, \psi), 
\end{equation}
for any $\sigma \in \Pi^{(g)}(G_n)$ and any
irreducible generic representation $\pi$ of any $H_k(F)$, $k
\in \mathbb{Z}_{>0}$.
\end{thr}

\begin{proof}
This map has already been given by Cogdell, Kim, Piatetski-Shapiro, and Shahidi (see \cite{CKPSS04} and \cite{CPSS11}), so it suffices to prove the surjectivity. That is, given a $\rho \in \Pi^{(g)}_{\varepsilon}(H_N)$, to construct a $\sigma_{\rho} \in \Pi^{(g)}(G_n)$ such that $\rho =
l(\sigma_{\rho})$ and \eqref{equlg2}, \eqref{equlg3} hold.

Define a representation
 $\sigma_{\rho}$ of  $G$ as
 $$\sigma_{\rho} := \pi_1 \times \pi_2 \times \cdots \times \pi_f \rtimes
\sigma^{(t)}$$ 
where $\pi_i=\delta(\Sigma_i)$ and 
$\sigma^{(t)}$ is the irreducible tempered
generic representation of $G_{n^{*}}(F)$ attached to $\rho^{(t)}$.
Now, first we want to prove that $\sigma_{\rho}$ is a generic
representation of $G_n(F)$, so we have to verify the conditions in
the Definition \ref{def1}.

(1) From the condition (2) in the definition of $\rho \in \Pi^{(g)}_{\varepsilon}(H_N)$, 
we can see that the
segment $\Sigma_i$ is not linked to either $\Sigma_j$ or
$\check{\Sigma}_j$ for $1 \leq i \neq j \leq f$, that is condition
(1) in Definition \ref{def1}.

(2) For $1 \leq i \leq f$, if $\Sigma_i$ is linked to some segment
corresponding to a representation in any of the families
$$\{\delta([\nu^{-m_j}\tau_j, \nu^{m_j}\tau_j])\}_{j=1}^r, \{\delta([\nu^{-e_j}\beta_j, \nu^{e_j}\beta_j])\}_{j=1}^c, \{\delta([\nu^{-e_j}\check{\beta}_j, \nu^{e_j}\check{\beta}_j])\}_{j=1}^c$$ 
which completely determine $\rho^{(t)}$ as in \eqref{equltg4} and \S \ref{gtsection}, then by the
classification theory in \cite{Zel80}, we know that the representation
$\delta(\Sigma_i) \times \rho^{(t)}$ and $\delta(\check{\Sigma}_i)
\times \rho^{(t)}$ must be reducible, this contradicts to condition
(3) in the definition of $\rho \in \Pi^{(g)}_{\varepsilon}(H_N)$. So condition (2) in Definition \ref{def1} holds.

(3) If $\xi_i$ is not self-dual (resp. self-conjugate-dual in the case of unitary groups), then condition (3a) in Definition
\ref{def1} holds. Otherwise, the condition (4) in the definition 
of $\rho \in \Pi^{(g)}_{\varepsilon}(H_N)$ holds. So if
$L(\rho^{(0)} \times \xi_i, s)$ has a pole at $s=0$, then either
$-q_i \geq 2$, i. e. $\pm1 \not \in \{-q_i, -q_i + 1, \cdots, -q_i +
w_i\}$, that is condition (3C1), or $q_i = -1$ and $\xi_i \in
A_2(\rho^{(2)}$), that is condition (3b). If $L(\xi_i, R, s)$
has a pole at $s = 0$, but $L(\sigma^{(0)} \times \xi_i, s)$ has no
pole at $s = 0$, then $q_i \in \mathbb{Z}$, and $\Sigma_i$ is not
linked to $\check{\Sigma}_i$. The linkage condition means $\pm0
\not \in \{-q_i, -q_i + 1, \cdots, -q_i + w_i\}$ or $q_i \notin \mathbb{Z}_{\geq 0}$, otherwise,
$\Sigma_i$ must link to $\check{\Sigma}_i$. So, condition (3C0) or (3CN) 
holds. If $L(\xi_i, R^-, s)$ has a pole at $s = 0$, then $q_i
\in \frac{1}{2} + \mathbb{Z}$, and $\Sigma_i$ is not linked to
$\check{\Sigma}_i$, which means $\pm\frac{1}{2} \not \in \{-q_i,
-q_i + 1, \cdots, -q_i + w_i\}$. Otherwise, $\Sigma_i$ also must
link to $\check{\Sigma}_i$. So, condition (3C$\frac{1}{2}$) also
holds.

So, all conditions in Definition \ref{def1} are satisfied, which means
$\sigma_{\rho}$ is indeed generic. Next, we have to prove \eqref{equlg2} and \eqref{equlg3}. By \cite[Lemma 7.2]{CKPSS04}, for equalities of local factors, we only have to 
consider twisting irreducible supercuspidal representations of $H_k$, with all $k
\in \mathbb{Z}_{>0}$. 

First, we prove the compatibility of local gamma functions using
the multiplicity of local gamma functions (\cite{Sha90b}, \cite{JPSS83}). First we know that
\begin{equation} \label{equ61}
\gamma(\sigma_{\rho} \times \pi, s, \psi) = [\prod_{i=1}^f
\gamma(\pi_i \times \pi, s, \psi) \gamma(\check{\pi}_i \times \pi,
s, \psi)] \gamma(\sigma^{(t)} \times \pi, s, \psi). 
\end{equation}
On the other hand, by Theorem \ref{thm9}, we know that
$$\gamma(\sigma^{(t)} \times \pi, s, \psi) = \gamma(\rho^{(t)} \times \pi, s, \psi).$$
So from \eqref{equ61} we get that
\begin{align*}
& \quad \gamma(\sigma_{\rho} \times \pi, s, \psi)\\
& = [\prod_{i=1}^f
\gamma(\pi_i \times \pi, s, \psi) \gamma(\check{\pi}_i \times \pi,
s, \psi)] \gamma(\rho^{(t)} \times \pi, s, \psi)\\
& = \gamma((\delta(\Sigma_1) \times \cdots \times \delta(\Sigma_f)
\times \rho^{(t)} \times \delta(\check{\Sigma}_f) \times \cdots
\times \delta(\check{\Sigma}_1)) \times \pi, s, \psi)\\
& = \gamma(\rho \times \pi, s, \psi)
\end{align*}
for $\rho$ is the Langlands
quotient of representation
$$\delta(\Sigma_1) \times \cdots \times
\delta(\Sigma_f) \times \rho^{(t)} \times \delta(\check{\Sigma}_f)
\times \cdots \times \delta(\check{\Sigma}_1).$$ Hence,
\begin{equation} \label{equ63}
\gamma(\sigma_{\rho} \times \pi, s, \psi) = \gamma(\rho \times
\pi, s, \psi). 
\end{equation}

Then, we want to prove the equality of local $L$-factors
$$L(\sigma_{\rho} \times \pi, s) = L(\rho \times \pi, s),$$
using the multiplicity of local $L$-factors (\cite{Sha90b}, \cite{JPSS83}). First we have that
\begin{equation} \label{equ62}
L(\sigma_{\rho} \times \pi, s) = [\prod_{i=1}^f
L(\pi_i \times \pi, s) L(\check{\pi}_i \times \pi, s)]
L(\sigma^{(t)} \times \pi, s). 
\end{equation}
On the other hand, by
Theorem \ref{thm9}, we know that
$$L(\sigma^{(t)} \times \pi, s) = L(\rho^{(t)} \times \pi, s).$$
So from \eqref{equ62} we get that
\begin{align*}
& \quad L(\sigma_{\rho} \times \pi, s)\\
& = [\prod_{i=1}^f
L(\pi_i \times \pi, s) L(\check{\pi}_i \times \pi,
s)] L(\rho^{(t)} \times \pi, s)\\
& = L((\delta(\Sigma_1) \times \cdots \times \delta(\Sigma_f)
\times \rho^{(t)} \times \delta(\check{\Sigma}_f) \times \cdots
\times \delta(\check{\Sigma}_1)) \times \pi, s)\\
& = L(\rho \times \pi, s)
\end{align*}
for $\rho$ is the Langlands
quotient of representation
$$\delta(\Sigma_1) \times \cdots \times
\delta(\Sigma_f) \times \rho^{(t)} \times \delta(\check{\Sigma}_f)
\times \cdots \times \delta(\check{\Sigma}_1).$$ 
Hence,
\begin{equation} \label{equ64}
L(\sigma_{\rho} \times \pi, s) = L(\rho \times \pi, s), 
\end{equation}
this proves \eqref{equlg2}.

We can rewrite \eqref{equ63} as
$$\epsilon(\sigma_{\rho} \times \pi, s, \psi) \frac{L(\sigma_{\rho} \times
\check{\pi}, 1-s)}{L(\sigma_{\rho} \times \pi, s)} = \epsilon(\rho
\times \pi, s, \psi) \frac{L(\rho \times \check{\pi}, 1-s)}{L(\rho
\times \pi, s)}.$$ Then combining with \eqref{equ64}, we obtain that
$$\epsilon(\sigma_{\rho} \times \pi, s, \psi) = \epsilon(\rho
\times \pi, s, \psi).$$ This proves \eqref{equlg3}, hence completes the proof of the theorem.
\end{proof}

At last, we assign the corresponding parameters for representations in $\Pi^{(g)}(G_n)$. 
Let $\Phi^{(g)}(G_n)$ be the subset of $\Phi(G_n)$ consisting of elements of the following form:
$$\phi = \phi^{(t)} \oplus \bigoplus_{i=1}^f \left(\lvert \cdot \rvert^{-q_i+\frac{w_i}{2}}r^{-1}(\xi_i) \otimes
S_{w_i+1} \oplus \lvert \cdot \rvert^{q_i-\frac{w_i}{2}}r^{-1}(\check{\xi}_i)
\otimes S_{w_i+1} \right),$$ 
where $\phi^{(t)}$ is a representative of $\iota(\sigma^{(t)})$ for an irreducible tempered representation $\sigma^{(t)}$ of $G_{n^*}(F)$ ($n^*\leq n$), and the sequence 
$$\{\Sigma_j=[v^{-q_j}\xi_j, v^{-q_j + w_j}\xi_j]\}_{j=1}^f$$ is a
$G_n$-generic sequence of segments with respect to $\sigma^{(t)}$ (see Definition~\ref{def1}). Here,
$\iota$ is the reciprocity map
given in Theorem \ref{thm11} for irreducible tempered generic
representations in $\Pi^{(tg)}(G_n)$, $r$ is the reciprocity map
for $H_{*}(F)$, and $\lvert \cdot \rvert^s$ is the character of $W_F$
normalized as in \cite{Tat79} via local class field theory.
Let $\widetilde{\Phi}^{(g)}(G_n)$ be the image of $\Phi^{(g)}(G_n)$ in $\widetilde{\Phi}(G_n)$.

The following theorem is analogous to the result in the last paragraph of 
Section 5 of \cite{JS04}, \cite[Theorem 4.17]{Liu11}, and \cite[Theorem 4.17]{JL14}. 

\begin{thr} \label{thm19}
There is a surjective map $\iota$ $($which
extends the one in Theorem \ref{thm11}$)$ from $\Pi^{(g)}(G_n)$ to
$\widetilde{\Phi}^{(g)}(G_n)$ and it preserves the local factors:
$$L(\sigma \times \tau, s) = L(\iota(\sigma) \otimes r^{-1}(\tau),
s),$$
$$\epsilon(\sigma \times \tau, s, \psi) =
\epsilon(\iota(\sigma) \otimes r^{-1}(\tau), s, \psi),$$
for all $\sigma \in \Pi^{(g)}(G_n)$ and all irreducible
generic representations $\tau$ of any $H_{k_{\tau}}(F)$,
$k_{\tau} \in \mathbb{Z}_{>0}$. Here, $r^{-1}(\tau)$ is the
irreducible admissible representation of $W_F \times SL_2(\mathbb{C})$ of dimension
$k_{\tau}$  corresponding to $\tau$ by the local Langlands
reciprocity map for $H_{k_{\tau}}$.
\end{thr}

\begin{proof}
Given any $\sigma \in \Pi^{(g)}(G_n)$, by the classification of generic representations of $G_n(F)$ in \S \ref{generic}, there exists an irreducible tampered generic representation $\sigma^{(t)}$ of $G_{n^*}(F)$ and 
a sequence of segments $\{\Sigma_j=[v^{-q_j}\xi_j, v^{-q_j + w_j}\xi_j]\}_{j=1}^f$ which is a
$G_n$-generic sequence of segments with respect to $\sigma^{(t)}$ (see Definition~\ref{def1}), such that 
$$
\sigma = \delta(\Sigma_1) \times \delta(\Sigma_2) \times \cdots \times \delta(\Sigma_f) \rtimes
\sigma^{(t)}.
$$

Let $\phi^{(t)}$ be a representative of $\iota(\sigma^{(t)})$
and let 
$$\phi = \phi^{(t)} \oplus \bigoplus_{i=1}^f \left(\lvert \cdot \rvert^{-q_i+\frac{w_i}{2}}r^{-1}(\xi_i) \otimes
S_{w_i+1} \oplus \lvert \cdot \rvert^{q_i-\frac{w_i}{2}}r^{-1}(\check{\xi}_i)
\otimes S_{w_i+1} \right),$$
which is exactly $r^{-1}(l(\sigma))$. 
It is easy to see that $\phi \in \Phi^{(g)}(G_n)$.
Define $\iota(\sigma)=\widetilde{\phi}$, the image of $\phi$ in $\widetilde{\Phi}(G_n)$.
Therefore, we have constructed a map $\iota$ from $\Pi^{(g)}(G_n)$ to
$\widetilde{\Phi}^{(g)}(G_n)$, which naturally extends the one in Theorem \ref{thm11}. Since $l$ preserves local factors, so is $\iota$.

To show that this map is surjective, take any $\widetilde{\phi} \in 
\widetilde{\Phi}^{(g)}(G_n)$ and let 
$$\phi = \phi^{(t)} \oplus \bigoplus_{i=1}^f \left(\lvert \cdot \rvert^{-q_i+\frac{w_i}{2}}r^{-1}(\xi_i) \otimes
S_{w_i+1} \oplus \lvert \cdot \rvert^{q_i-\frac{w_i}{2}}r^{-1}(\check{\xi}_i)
\otimes S_{w_i+1} \right)$$
be a representative,
where $\phi^{(t)}$ is a tempered parameter of  $G_{n^*}$. 
Let 
$\sigma^{(t)}$ be an irreducible tampered generic representation of $G_{n^*}$ lifting to $\widetilde{\phi}^{(t)}$. Then 
the sequence of segments 
$$\{\Sigma_j=[v^{-q_j}\xi_j, v^{-q_j + w_j}\xi_j]\}_{j=1}^f$$ is a
$G_n$-generic sequence of segments with respect to $\sigma^{(t)}$.

Let 
$$
\sigma = \delta(\Sigma_1) \times \delta(\Sigma_2) \times \cdots \times \delta(\Sigma_f) \rtimes
\sigma^{(t)}.
$$
By the classification of irreducible generic representations in \S \ref{generic}, we can see that $\sigma$ is irreducible and generic. Hence $\sigma \in \Pi^{(g)}(G_n)$ and we can also easily see that $\iota(\sigma)$ is actually equal to $\widetilde{\phi}$. Therefore, $\iota$ is indeed surjective.

This completes the proof of the theorem.
\end{proof}

\begin{rem}\label{rmknew1}
Given any $\sigma \in \Pi^{(g)}(G_n)$, let $\iota(\sigma)=\widetilde{\phi} \in \widetilde{\Phi}^{(g)}(G_n)$ and let $\phi$ is a representative of $\widetilde{\phi}$. 
Recall the embedding 
$i: {}^LG_n  \hookrightarrow {}^LH_N$ in \S \ref{sect5} (see also Table \ref{tab:Langlands functoriality}). 
Form Theorem \ref{thmgen} and Theorem \ref{thm19}, we can see that the composition $i \circ \phi$ is actually the local Langlands parameter corresponding to the lifing $l(\sigma)$ of $\sigma$.

A local Langlands parameter $\phi \in {\Phi}(G_n)$ is called generic if there is a generic representation in the corresponding local $L$-packet. From Theorem \ref{thm19},
we can see that
${\Phi}^{(g)}(G_n)$ is actually the set of all generic local Langlands parameters of $G_n(F)$.
\end{rem}

\begin{rem}\label{rmk9}
By \S \ref{generic}, if $\sigma^{(g)} \in \Pi^{(g)}(G_n)$ is
the irreducible generic representation 
$\pi_1 \times \pi_2 \times \cdots \times \pi_f \rtimes
\sigma^{(t)}$
(possibly $\sigma^{(t)} = 1 \otimes c$) and $\sigma^{(t)} \ncong c \sigma^{(t)}$,
then $\sigma^{(g)} \ncong c \sigma^{(g)}$, both are in $\Pi^{(g)}(G_n)$,
and $c \sigma^{(g)}$ is
the irreducible generic representation 
$\pi_1 \times \pi_2 \times \cdots \times \pi_f \rtimes
c \sigma^{(t)}$.
%(n.b.: if $\sigma^{(g)}$ is $\chi$-generic, $c\sigma^{(g)}$ is $c\chi$-generic).

By Remark \ref{rmk8}, Theorem \ref{thmgen} and Theorem \ref{thm19},
and by the multiplicativity of local factors (see \cite{Sha90b}, \cite{JS12} and \cite{CKPSS04}),
in the above situation, $\sigma^{(g)}$ and $c \sigma^{(g)}$
have the same lifting image and the same twisted local factors.
We record this result as the following theorem.
\end{rem}

\begin{thr}\label{thm20}
For any $\sigma \in \Pi^{(g)}(G_n)$,
if $\sigma \ncong c \sigma$, then 
$l(\sigma)=l(c \sigma)$,
and $\iota(\sigma)=\iota(c \sigma)$.
That is, they have the same lifting image and the same twisted local factors.
\end{thr}

\section{Representations attached to parameters}\label{sect7}

In this section, as in \cite{JS04}, \cite{Liu11},
and \cite{JL14}, we associate
an irreducible representation of $G_n(F)$ to each local
Langlands parameter $\widetilde{\phi} \in \widetilde{\Phi}(G_n)$. 
The key idea is to analyze the structures of local
Langlands parameters. The following proposition is analogous to 
\cite[Proposition 6.1]{JS04}, \cite[Proposition 5.1]{Liu11}, and 
\cite[Proposition 5.1]{JL14}. 

\begin{prop} \label{prop3}
Let ${\phi} \in {\Phi}(G_n)$ be a local Langlands parameter. Then either $\phi \in \Phi^{(t)}(G_n)$, 
or 
\begin{equation} \label{equrp1}
\phi = \phi^{(t)} \oplus \phi^{(n)}, 
\end{equation}
where $\phi^{(t)} \in \Phi^{(t)}(G_{n^{*}})$ $(n^{*} < n)$ and 
$\phi^{(n)}$ is of the following form
\begin{equation} \label{equrp2}
\phi^{(n)} = \bigoplus_{i=1}^f \left(\lvert \cdot \rvert^{-q_i+\frac{w_i}{2}} \phi_i \otimes
S_{w_i+1} \oplus \lvert \cdot \rvert^{q_i-\frac{w_i}{2}} \check{\phi}_i \otimes
S_{w_i+1}\right).
\end{equation}
Here, $f \in \mathbb{Z}_{>0}$,
$w_1, w_2, \dots, w_f \in \mathbb{Z}_{\geq0}$, $q_1, q_2, \dots, q_f \in \mathbb{R}$, 
such that $\phi_i$ is an
irreducible bounded representation of $W_F$ for $1 \leq i \leq f$, 
$$\frac{w_1}{2} - q_1 \geq \frac{w_2}{2} - q_2 \geq \dots \geq  \frac{w_f}{2} - q_f > 0.$$
$\lvert \cdot \rvert$ is the character of $W_F$
normalized as in \cite{Tat79} via local class field theory, and  $\check{\phi}_{i}$ is irreducible representation of $W_F$ corresponding to $r({\phi}_{i})^{\vee}$ under the local Langlands
reciprocity map $r$ for general linear groups.
\end{prop}

\begin{proof}
Given a parameter $\phi \in \Phi(G_n)$, assume $V = \mathbb{C}^{N}$ be the corresponding non-degenerate space of dimension $N$, with a form $<,>$ which is (conjugate)-orthogonal or (conjugate)-symplectic (see \cite[Section 3]{GGP12}). 

Let $V_1$ be the direct sum of all irreducible subspaces, which are stable under
the action of $W_F \times SL_2(\mathbb{C})$ and in which $\phi(W_F)$ is bounded. Let $V_2$
be the direct sum of all irreducible subspaces, which are stable under the
action of $W_F \times SL_2(\mathbb{C})$ and in which $\phi(W_F)$ is unbounded.
Then $$V = V_1 \oplus V_2.$$

First, let us show that both subspaces $V_1$ and $V_2$ are non-degenerate with respect to
the restriction of the non-degenerate form $<,>$. Let $rad(V_i)$ be the radical of
$(V_i, <,>|_{V_i})$, that is $rad(V_i)=\{v \in V_i|<v,w>=0,\forall w \in V_i\}$. Then $rad(V_i)$ is stable
under the action of $W_F \times SL_2(\mathbb{C})$, since
$\phi(g)$ preserves the form.

For any $v_1 \in rad(V_1)$, assume that $\phi_1 \otimes
S_{w_1+1}$ corresponds to an arbitrary irreducible summand $V_2{'}$
of $V_2$, where $\phi_1$ is an irreducible unbounded representation
of $W_F$. Write
$$\phi_1 = |\cdot|^t \phi_1{'},$$
where $\phi_1{'}(W_F)$ is bounded and $0 \neq t \in \mathbb{R}$.
Then, for any $v_2 \in V_2{'}$, $w \in W_F$,we have
$$<v_2, \phi_1(w^{-1})(v_1)> = <\phi_1(w)(v_2), v_1> = |w|^t <(\phi_1{'}(w)
\otimes id)(v_2), v_1>.$$ Since $<v_2, \phi_1(v_1)>$ is
bounded, but $|w|^t <(\phi_1{'}(w) \otimes id)(v_2), v_1>$ is
unbounded, so
$$<\phi_2(w)(v_2), v_1> = 0, \forall v_2 \in V_2{'}.$$
Since $v_2$ is arbitrary, we have
$$<v_2, v_1> = 0, \forall v_2 \in V_2{'}.$$
Since $V_2{'}$ is arbitrary, we get
$$<v_2, v_1> = 0, \forall v_2 \in V_2.$$
Since $v_1 \in rad(V_1)$, we have
$$<v, v_1> = 0, \forall v \in V.$$
Since $V$ is non-degenerate, $v_1$ must be zero. So, $rad(V_1) = 0$,
that is $V_1$ is non-degenerate. Similarly, we can show that is $V_2$ is also non-degenerate.

Denote by $\phi^{(t)}$ the sub-representation of $W_F \times
SL_2(\mathbb{C})$ on $V_1$, and by $\phi^{(n)}$ the
sub-representation of $W_F \times SL_2(\mathbb{C})$ on $V_2$. Then
following similar arguments as in 
\cite[Proposition 5.1]{Liu11},
$\phi^{(t)} \in \Phi^{(t)}(G_{n-n^{*}})$ and
$\phi^{(n)}$ of the form as in \eqref{equrp2}. This completes the proof of the proposition.
\end{proof}

Let $\Pi{'}(G_n)$ be 
the set of equivalence classes of
irreducible representations $\sigma$ of $G_n(F)$  which are Langlands quotients
$$L(\nu^{x_1}\delta_1 \otimes \dots \otimes \nu^{x_k}\delta_k \otimes \sigma^{(t)}),$$
with central character $\chi_{\alpha}$ being trivial when restricting to $F^*$ except when $G_n=SO_{2n+2}^*$, in which case it is 
the quadratic character $\eta_{\varepsilon}$  associated to the square class $\varepsilon$ defining $G_n$.  
Here $\sigma^{(t)}$ is an
irreducible tempered generic representation of $G_{n^{*}}(F)$ (possibly $\sigma^{(t)}=1 \otimes c$--for the definition, see Remark \ref{rmk c}), $x_1 \geq x_2 \geq \cdots \geq x_k>0$,
and $\delta_i$ is a square-integrable representation of $H_{n_i}(F)$, for $i=1,2,\ldots, k$. 
Then, we have the following result which is analogous to 
\cite[Theorem 6.1]{JS04}, \cite[Theorem 5.2]{Liu11}, and 
\cite[Theorem 5.2]{JL14}:

\begin{thr} \label{thm13}
There is a surjective map $\iota$ $($which extends
the one in Theorem \ref{thm11}$)$ from $\Pi{'}(G_n)$ to the set
$\widetilde{\Phi}(G_n)$ and it preserves the local factors:
$$L(\sigma \times \tau, s) = L(\iota(\sigma) \otimes r^{-1}(\tau),
s),$$
$$\epsilon(\sigma \times \tau, s, \psi) =
\epsilon(\iota(\sigma) \otimes r^{-1}(\tau), s, \psi),$$ for
all $\sigma \in \Pi{'}(G_n)$ and all irreducible admissible
representations $\tau$ of any $H_{k_{\tau}}(F)$, $k_{\tau} \in
\mathbb{Z}_{>0}$. Here, $r^{-1}(\tau)$ is the irreducible admissible
representation of $W_F \times SL_2(\mathbb{C})$ of dimension $k_{\tau}$ corresponding to
$\tau$ by the local Langlands reciprocity map for $H_{k_{\tau}}$.
\end{thr}

\noindent
\begin{proof}
Given any $\sigma \in \Pi'(G_n)$ which is the Langlands quotient $L(\nu^{x_1}\delta_1 \otimes \dots \otimes \nu^{x_k}\delta_k \otimes \sigma^{(t)})$,
where $\sigma^{(t)}$ is an
irreducible tempered generic representation of $G_{n^{*}}(F)$ (possibly $\sigma^{(t)}=1 \otimes c$), $x_1 \geq x_2 \geq \cdots \geq x_k>0$,
and $\delta_i$ is a square-integrable representation of $H_{n_i}(F)$, for $i=1,2,\ldots, k$.

Using the surjective map $\iota$ in Theorem \ref{thm11}, let $\widetilde{\phi}^{(t)} = \iota(\sigma^{(t)}) \in \widetilde{\Phi}^t(G_n)$ and let ${\phi}^{(t)}$ be a representative. Assume that $\phi_i$ is the corresponding Langlands parameter for $\delta_i$ under the local Langlands raciprocity map for $H_{n_i}(F)$, for $i=1,2,\ldots,k$.
Then let 
$$\phi = \bigoplus_{i=1}^k (\lvert \cdot \rvert^{x_i} \phi_i \oplus \lvert \cdot \rvert^{-x_i} \check{\phi}_i) \bigoplus {\phi}^{(t)}.$$
% where $\check{\phi}_i$ is the contragredient of $\phi_i$.
Define $\iota(\sigma)=\widetilde{\phi}$, the image of $\phi$ in $\widetilde{\Phi}(G_n)$. Then using multiplicativity of local factors, it is easy to see that the local factors are preserved. In this way, we construct a map $\iota$ from $\Pi{'}(G_n)$ to the set $\widetilde{\Phi}(G_n)$ which preserves local factors and naturally extends the one in Theorem \ref{thm11}.

To prove that this map $\iota$ is surjective, given any 
$\widetilde{\phi} \in \widetilde{\Phi}(G_n)$, let $\phi \in \Phi(G_n)$ be a representative. By Proposition \ref{prop3}, it can be written as
$$\phi = \phi^{(t)} \oplus \phi^{(n)},$$
where $\phi^{(t)} \in \Phi^{(t)}(G_{n^*})$ $(n^{*} < n)$ and 
$\phi^{(n)}$ is of the following form
\begin{equation*}
\phi^{(n)} = \bigoplus_{i=1}^f \left(\lvert \cdot \rvert^{-q_i+\frac{w_i}{2}} \phi_i \otimes
S_{w_i+1} \oplus \lvert \cdot \rvert^{q_i-\frac{w_i}{2}} \check{\phi}_i \otimes
S_{w_i+1}\right).
\end{equation*}
Here, $f \in \mathbb{Z}_{>0}$,
$w_1, w_2, \dots, w_f \in \mathbb{Z}_{\geq0}$, $q_1, q_2, \dots, q_f \in \mathbb{R}$, 
such that $\phi_i$ is an
irreducible bounded representation of $W_F$ for $1 \leq i \leq f$, and 
$$\frac{w_1}{2} - q_1 \geq \frac{w_2}{2} - q_2 \geq \dots \geq  \frac{w_f}{2} - q_f > 0.$$

By Theorem \ref{thm11}, there exists
$\sigma^{(t)} \in \Pi^{(tg)}(G_{n^*})$ such that
\begin{equation} \label{equrp4}
\iota(\sigma^{(t)}) = \widetilde{\phi}^{(t)} \in \widetilde{\Phi}^{(t)}(G_{n^*}). 
\end{equation}
Using the local Langlands reciprocity map $r$ for $H_k(F)$, define
\begin{equation} \label{equrp5}
\Sigma_i = [v^{-q_i}r(\phi_i), v^{-q_i+w_i}r(\phi_i)], 1 \leq i
\leq f.
\end{equation}
Let $\sigma$ be the Langlands quotient
of the induced representation
$$\delta(\Sigma_1) \times \delta(\Sigma_2) \times \dots
\delta(\Sigma_f) \rtimes \sigma^{(t)},$$
(possibly $\sigma^{(t)} = 1 \otimes c$).
Then, it is easy to see that $\sigma \in \Pi{'}(G_n)$ and $\iota(\sigma)$ is equal to $\widetilde{\phi}$. Therefore $\iota$ is surjective.

This completes the proof of the theorem.
\end{proof}

\begin{rem}\label{rmknew2}
When $\widetilde{\phi} \in \widetilde{\Phi}^{(g)}(G_n)$, let $\phi \in {\Phi}^{(g)}(G_n)$ be a representative, which is a generic local Langlands parameter. 
Then, by definition, $\phi$ is of the form
$$\phi^{(t)} \oplus \bigoplus_{i=1}^f \left(\lvert \cdot \rvert^{-q_i+\frac{w_i}{2}}r^{-1}(\xi_i) \otimes
S_{w_i+1} \oplus \lvert \cdot \rvert^{q_i-\frac{w_i}{2}}r^{-1}(\check{\xi}_i)
\otimes S_{w_i+1} \right),$$
where $\phi^{(t)}$ is a representative of $\iota(\sigma^{(t)})$ for an irreducible tempered representation $\sigma^{(t)}$ of $G_{n^*}(F)$ ($n^* \leq n$) and 
the sequence of segments 
$$\{\Sigma_j=[v^{-q_j}\xi_j, v^{-q_j + w_j}\xi_j]\}_{j=1}^f$$ is a
$G_n$-generic sequence of segments with respect to $\sigma^{(t)}$.

Then by the classification of generic representations in \S \ref{generic},
$$\delta(\Sigma_1) \times \delta(\Sigma_2) \times \dots
\delta(\Sigma_f) \rtimes \sigma^{(t)}$$
is irreducible and generic. The $\sigma$ constructed in Theorem \ref{thm13} is actually
equal to $\delta(\Sigma_1) \times \delta(\Sigma_2) \times \dots
\delta(\Sigma_f) \rtimes \sigma^{(t)}$, hence generic.
From the construction in Theorem \ref{thm19},
we can see that this $\sigma$ indeed matches the one constructed in Theorem \ref{thm19} for this generic local Langlands parameter $\phi$. Hence, the map $\iota$ constructed in Theorem \ref{thm13} is a natural extension of the one constructed in Theorem \ref{thm19}.

Therefore, we can conclude that $\phi \in {\Phi}(G_n)$ is a 
generic local Langlands parameter if and only if the representation $\sigma$ attached to $\phi$ in Theorem \ref{thm13} is generic.
\end{rem}

\begin{rem}\label{rmk10}
When $G_n=SO_{2n}, SO_{2n+2}^*$,
if $\sigma \in \Pi{'}(G_n)$ is the
Langlands quotient of the induced representation
$$\delta(\Sigma_1) \times \delta(\Sigma_2) \times \dots \times
\delta(\Sigma_f) \rtimes \sigma^{(t)},$$
possibly $\sigma^{(t)}=1 \otimes c$
and $\sigma^{(t)} \ncong c \sigma^{(t)}$, then
$\sigma \ncong c \sigma$ and $c \sigma$
is the
Langlands quotient of the induced representation
$$\delta(\Sigma_1) \times \delta(\Sigma_2) \times \dots \times
\delta(\Sigma_f) \rtimes c\sigma^{(t)}.$$
This matches the local Langlands classification for $G_n(F)$ -- see
\cite[Proposition 6.3 and Section 2]{BJ03}.

By Remark \ref{rmk8}, Theorem \ref{thm13},
and the multiplicativity of local factors (see \cite{Sha90b}, \cite{JS12} and \cite{CKPSS04}),
in the above situation, $\sigma$ and $c \sigma$
have the same twisted local factors.
\end{rem}

\appendix

\begin{section}{\texorpdfstring{$F$}{}-roots}\label{F-roots}

In this appendix, we identify the simple $F$-roots and coroots for the similitude groups. In all cases, we take $X=\{e_1, e_2, \dots, e_n,e_0\}$ and $\check{X}=\{\check{e}_1, \check{e}_2, \dots, \check{e}_n, \check{e}_0\}$ and indicate the simple roots and coroots. The data for the corresponding classical groups are obtained by removing $e_0$ and $\check{e}_0$.

\begin{itemize}

\item $G_n=GSp_{2n}$
\[
\begin{array}{l}
\Pi=\{e_1-e_2, e_2-e_3, \dots, e_{n-1}-e_n, 2e_n-e_0\} \\
\check{\Pi}=\{ \check{e}_1-\check{e}_2, \check{e}_2-\check{e}_3, \dots, \check{e}_{n-1}-\check{e}_n, \check{e}_n\}
\end{array}
\]

\item $G_n=GSO_{2n}$
\[
\begin{array}{l}
\Pi=\{e_1-e_2, e_2-e_3, \dots, e_{n-1}-e_n, e_{n-1}+e_n-e_0\} \\
\check{\Pi}=\{ \check{e}_1-\check{e}_2, \check{e}_2-\check{e}_3, \dots, \check{e}_{n-1}-\check{e}_n, \check{e}_{n-1}+\check{e}_n\}
\end{array}
\]

\item $G_n=GSO_{2n+2}^{\ast}$
\[
\begin{array}{l}
\Pi=\{e_1-e_2, e_2-e_3, \dots, e_{n-1}-e_n, e_n-e_0\} \\
\check{\Pi}=\{ \check{e}_1-\check{e}_2, \check{e}_2-\check{e}_3, \dots, \check{e}_{n-1}-\check{e}_n, 2\check{e}_n\}
\end{array}
\]

\item $G_n=GSpin_{2n+1}$
\[
\begin{array}{l}
\Pi=\{e_1-e_2, e_2-e_3, \dots, e_{n-1}-e_n, e_n\} \\
\check{\Pi}=\{ \check{e}_1-\check{e}_2, \check{e}_2-\check{e}_3, \dots, \check{e}_{n-1}-\check{e}_n, 2\check{e}_n-\check{e}_0\}
\end{array}
\]

\item $G_n=GSpin_{2n}$
\[
\begin{array}{l}
\Pi=\{e_1-e_2, e_2-e_3, \dots, e_{n-1}-e_n, e_{n-1}+e_n\} \\
\check{\Pi}=\{ \check{e}_1-\check{e}_2, \check{e}_2-\check{e}_3, \dots, \check{e}_{n-1}-\check{e}_n, \check{e}_{n-1}+\check{e}_n-\check{e}_0\}
\end{array}
\]

\item $G_n=GSpin_{2n+2}^{\ast}$
\[
\begin{array}{l}
\Pi=\{e_1-e_2, e_2-e_3, \dots, e_{n-1}-e_n, e_n\} \\
\check{\Pi}=\{ \check{e}_1-\check{e}_2, \check{e}_2-\check{e}_3, \dots, \check{e}_{n-1}-\check{e}_n, 2\check{e}_n-\check{e}_0\}
\end{array}
\]

\item $G_n=GU_{2n+1}$
\[
\begin{array}{l}
\Pi=\{e_1-e_2, e_2-e_3, \dots, e_{n-1}-e_n, e_n-e_0\} \\
\check{\Pi}=\{ \check{e}_1-\check{e}_2, \check{e}_2-\check{e}_3, \dots, \check{e}_{n-1}-\check{e}_n, 2\check{e}_n\}
\end{array}
\]

\item $G_n=GU_{2n}$
\[
\begin{array}{l}
\Pi=\{e_1-e_2, e_2-e_3, \dots, e_{n-1}-e_n, 2e_n-e_0\} \\
\check{\Pi}=\{ \check{e}_1-\check{e}_2, \check{e}_2-\check{e}_3, \dots, \check{e}_{n-1}-\check{e}_n, \check{e}_n\}
\end{array}
\]

\end{itemize}

\end{section}

% \bibliography{refs}{}

\begin{thebibliography}{wwww}

\bibitem[ACS16]{ACS16}
M. Asgari, J. Cogdell, and F. Shahidi. 
Local transfer and reducibility of induced representations of p-adic
groups of classical type. In 
{\it Advances in the theory of automorphic forms and their L-functions}, 1-22, Contemp. Math., \textbf{664}, {\it American Mathematical Society, Providence, RI}, 2016.

\bibitem[Arc]{Arc} S. Archava, personal communication.

\bibitem[Art13]{Art13}
J. Arthur. 
The endoscopic classification of representations. Orthogonal and
symplectic groups. { American Mathematical Society Colloquium Publications,} \textbf{61}. {\it American Mathematical Society, Providence, RI}, 2013. 

\bibitem[AS06]{AS06}
{M. Asgari and F. Shahidi},
{Generic transfer for general spin groups}, {\it Duke Math. J.} {\bf 132}(2006),
{137-190}.

\bibitem[Asg02]{Asg02}
{M. Asgari},
{Local L-functions for split spinor groups}, {\it Canad. J. Math.}, {\bf 54}(2002), no. 4,
{673-693}.

% \bibitem[ACS]{ACS}
% {M. Asgari, J. Cogdell, and F. Shahidi},
% {Local transfer and reducibility of induced representations of $p$-adic groups of classical type}, in {\it Advances in the Theory of Automorphic Forms and Their $L$-functions}, 1-22, 
% {Contemp. Math.}, {\bf 664}(2016), American Mathematical Society, Providence, RI.




\bibitem[Aub95]{Aub95}
{A.-M. Aubert},
{Dualit\'e dans le groupe de Grothendieck de la cat\'egorie des
repr\'esentations lisses de longueur finie d'un groupe r\'eductif
$p$-adique}, {\it Trans. Amer. Math. Soc.},
{\bf 347}(1995), no. 6,  {2179-2189} and
{Erratum},
{\it Trans. Amer. Math. Soc.}, {\bf 348}(1996), no. 11, 4687-4690.

\bibitem[BG15]{BG15}
{D. Ban and D. Goldberg},
{$R$-groups, elliptic representations, and parameters for GSpin groups},
{\it Forum Math.}, {\bf 27}(2015), no. 6, 3301-3334.

\bibitem[BJ01]{BJ01}
{D. Ban and C. Jantzen},
{The Langlands classification for non-connected $p$-adic groups},
{\it Israel J. Math.}, {\bf 126}(2001), 239--261.

\bibitem[BJ03]{BJ03}
{D. Ban and C. Jantzen},
{Degenerate principal series for even orthogonal groups},
{\it Represent. Theory}, {\bf 7}(2003), 440-480.


\bibitem[BJ08]{BJ08}
{D. Ban and C. Jantzen},
{Jacquet modules and the Langlands classification},
{\it Michigan Math. J.}, {\bf 56}(2008), no. 3,  637-653.

% \bibitem[Bl]{Bl}
% {C. Blondel},
% {$Sp(2N)$-covers for self-contragredient supercuspidal representations of $GL(N)$},
% {\it Ann. Sci. \'Ec. Norm. Sup.} {\bf 37}(2004), no. 4,  533-558.

\bibitem[Bru63]{Bru63}
{F. Bruhat},
{ Lectures on Some Aspects of $p$-adic Analysis},
 Notes by Sunder Lal. {\it Tata Institute of Fundamental Research, Bombay}, 1963.


\bibitem[BZ77] {BZ77} 
I. N. Bernstein and A. V. Zelevinsky,
{Induced representations of reductive $p$-adic groups. \makeatletter \@Roman 1.}
{\it Ann. Sci. \'Ecole Norm. Sup.}, {\bf 10} (1977), no. 4, 441-472.

\bibitem[CFK20]{CFK20}
{Y. Cai, S. Friedberg, and E. Kaplan},
{Doubling constructions: local and global theory, with an application to global functoriality for non-generic cuspidal representations}, preprint. 2020. 

\bibitem[CS98]{CS98}
W. Casselman and F. Shahidi, On the irreducibility of standard modules
for generic representations, 
{\it Ann. Sci. Ec. Norm.}
{\bf 31} (1998), 561-589.

\bibitem[CKPSS04]{CKPSS04}
J. W. Cogdell, H. H. Kim, I. I. Piatetski-Shapiro, and F. Shahidi. Functoriality for the classical groups.
{\it Publ. Math. Inst. Hautes Etudes Sci.}, (99):163-233, 2004.

\bibitem[CPSS11]{CPSS11}
J. W. Cogdell, I. I. Piatetski-Shapiro, and F. Shahidi. Functoriality for the quasi-split classical groups. In
{\it On certain L-functions}, 117–140,
Clay Math. Proc., \textbf{13}. {\it American Mathematical Society, Providence,
RI}, 2011.

\bibitem[CST17]{CST17}
J. W. Cogdell, F. Shahidi, and T.-L. Tsai.
Local Langlands correspondence for $\mathrm{GL}_n$ and the exterior and symmetric square $\epsilon$-factors. {\it Duke Math. J.} 166 (2017), no. 11, 2053-2132. 

\bibitem[GGP12]{GGP12}
W. T. Gan, B. H. Gross and D. Prasad, 
Symplectic local root numbers, central critical L-values
and restriction problems in the representation theory of classical groups. Sur les conjectures
de Gross et Prasad. I, {\it Ast\'erisque} 346 (2012), 1-109.

\bibitem[GI16]{GI16}
W. T. Gan, and A. Ichino,
{The Gross–Prasad conjecture and local theta correspondence}.
{\it Invent. math.} (2016) 206:705--799.

\bibitem[GL18]{GL18}
{W. T. Gan and L. Lomel\'i},
{Globalization of supercuspidal representations over function fields and applications},
{\it J. Eur. Math. Soc.}, {\bf 20}(2018), no. 11, {2813-2858}.

\bibitem[GRS11]{GRS11}
D. Ginzburg, S. Rallis, and D. Soudry. 
{ The descent map from automorphic representations
of GL(n) to classical groups.}
{\it World Scientific Publishing Co. Pte. Ltd., Hackensack, NJ}, 2011.

\bibitem[Gol94]{Gol94}
{D. Goldberg},
{Reducibility of induced representations for $Sp(2n)$ and $SO(n)$},
{\it Amer. J. Math.}, {\bf 116}(1994), no. 5, {1101-1151}.

\bibitem[Gol95]{Gol95}
{D. Goldberg},
{$R$-groups and elliptic representations for unitary groups},
{\it Proc. Amer. Math. Soc.}, {\bf 448}(1995), no. 4, {1267-1276}.

\bibitem[Gol97]{Gol97}
{D. Goldberg},
{$R$-groups and elliptic representations for similitude groups},
{\it Math. Ann.}, {\bf 307}(1997), no. 4, {569-588}.

\bibitem[GP92]{GP92}
B.H. Gross, D. Prasad, 
{On the decomposition of a representation of $SO_n$ when restricted
to $SO_{n-1}$.} 
{\it Can. J. Math.} 44, 974--1002 (1992)

\bibitem[Han10]{Han10}
{M. Hanzer},
{The generalized injectivity conjecture for classical $p$-adic groups},
{\it Int. Math. Res. Not.}, {\bf 2010}, no. 2, {195-237}.

% \bibitem[H-C]{H-C}
% {Harish-Chandra},
% {Harmonic analysis on reductive $p$-adic groups}, in {\it Harmonic Analysis on Homogeneous Spaces},  {167-192}, {Proc. Sympos. Pure Math.}, {\bf 23}(1973), American Mathematical Society, Providence, RI.

\bibitem[HK17]{HK17}
{V. Heiermann and Y, Kim},
{On the generic local Langlands correpondence for $GSpin$ groups}.
{\it Trans. Amer. Math. Soc.}, vol 369, 4275--4291 (2017).

\bibitem[HO13]{HO13}
{V. Heiermann and E. Opdam},
{On the tempered $L$-functions conjecture},
{\it Amer. J. Math.}, {\bf 135}(2013), no. 3,  {777-799}.

\bibitem[Hen10] {Hen10} G. Henniart, Correspondance de Langlands et fonctions $L$ des 
carres exterieur et symetrique, {\it Int.
Math. Res. Not.} 2010, no. 4, 633--673.

\bibitem[HS16]{HS16}
{J. Hundley and E. Sayag},
{Descent construction for $GSpin$ groups},
{\it Mem. Amer. Math. Soc.},  {\bf 243} (2016), no. 1148, v+124 pp.

\bibitem[Jac77]{Jac77} 
H. Jacquet,
{Generic representations}. In
{\it Non-commutative harmonic analysis (Actes Colloq., Marseille-Luminy}, 1976), pp. 91--101. 
Lecture Notes in Math., Vol. {\bf 587}, {\it Springer, Berlin}, 1977.

\bibitem[JPSS83]{JPSS83} 
H. Jacquet, I. I. Piatetskii-Shapiro, and J. A. Shalika. Rankin-Selberg convolutions. {\it Amer. J. Math.},
105(2):367-464, 1983.

\bibitem[Jan93]{Jan93}
 C. Jantzen,
{Degenerate principal series for orthogonal groups},
 {\it J. reine angew. Math.}, {\bf 441}(1993), 61-98.


\bibitem[Jan96a]{Jan96a}
 C. Jantzen,
{Degenerate principal series for symplectic and odd-orthogonal groups},
{\it  Mem. Amer. Math. Soc.}, {\bf 124} (1996), no. 590, viii+100 pp.

\bibitem[Jan96b]{Jan96b}
C. Jantzen,
{Reducibility of certain representations for symplectic and odd-orthogonal groups},
{\it Compositio Math.}, {\bf 104}(1996), no. 1, 55--63.

\bibitem[Jan97]{Jan97}
C. Jantzen,
{On supports of induced representations for symplectic and odd-orthogonal groups},
{\it Amer. J. Math.}, {\bf 119}(1997), no. 6, 1213--1262.

\bibitem[Jan98]{Jan98}
{C. Jantzen}, {Some remarks on degenerate principal series},
{\it Pacific J. Math.}, {\bf 186}(1998), no. 1, 67-87.

\bibitem[Jan00a]{Jan00a}
C. Jantzen,
{Square-integrable representations for symplectic and
odd-orthogonal groups}, {\it Canad. J. Math.}, {\bf 52}(2000), no. 3, 539-581.

\bibitem[Jan00b]{Jan00b}
{C. Jantzen},
{Square-integrable representations for symplectic and
odd-orthogonal groups II}, {\it Represent. Theory}, {\bf 4}(2000),
127-180.


\bibitem[Jan11]{Jan11}
C. Jantzen,
{Discrete series for p-adic $SO(2n)$ and restrictions of representations of $O(2n)$},
{\it Canad. J. Math.}, {\bf 63}(2011), no. 2, 327--380.

\bibitem[Jan18]{Jan18}
C. Jantzen,
{Induced representations and Jacquet modules for classical $p$-adic groups}, {\it Manuscripta Math.}, {\bf 156}(2018), 23-55.

\bibitem[JL14]{JL14}
{C. Jantzen and B. Liu},
{The generic dual of $p$-adic split $SO_{2n}$ and local Langlands parameters}, {\it Israel J. Math.}, {\bf 204}(2014),
{199-260}.

\bibitem[Jia06]{Jia06}
D. Jiang. On local $\gamma$-factors. In Arithmetic geometry and number theory, volume 1 of Ser. 
Number Theory Appl., pages 1-28. {\it World Sci. Publ., Hackensack, NJ}, 2006.

\bibitem[Jia14]{Jia14}
D. Jiang. Automorphic integral transforms for classical groups I: Endoscopy correspondences. In
{\it Automorphic forms and related geometry: assessing the legacy of I. I. Piatetski-Shapiro}, volume 614 of
Contemp. Math., pages 179-242. {\it American Mathematical Society, Providence, RI}, 2014.

\bibitem[JS03]{JS03}
D. Jiang and D. Soudry. The local converse theorem for $SO(2n+1)$ and applications. {\it Ann. of
Math.} (2), 157(3):743-806, 2003.

\bibitem[JS04] {JS04} 
D. Jiang and D. Soudry,
{Generic representations and local Langlands reciprocity law for p-adic $SO_{2n+1}$}. In {\it  Contributions to Automorphic Forms, Geometry, and Number Theory}, 457--519, {\it Johns Hopkins University Press, Baltimore, MD}, 2004.

\bibitem[JS12]{JS12}
D. Jiang and D. Soudry. Appendix: On the local descent from GL(n) to classical groups [appendix to MR2931222]. {\it Amer. J. Math.}, \textbf{134}(3):767-772, 2012.

\bibitem[Kap17]{Kap17}
{E. Kaplan},
{The double cover of odd general spin groups, small representations, and applications}, {\it J. Inst. Math. Jussieu}, {\bf 16}(2017), %vol 3
{609-671}.

\bibitem[Kim09]{Kim09}
{W. Kim},
{Square integrable representations and the standard module conjecture for general spin groups}, {\it Canad. J. Math.}, {\bf 61}(2009), no. 3, 
{617-640}.


\bibitem[Kim15]{Kim15}
{Y. Kim},
{Strongly positive representations of $GSpin_{2n+1}$ and the Jacquet module method}, {\it Math. Z.}, {\bf 279}(2015), no. 1-2, 
{271-296}.

\bibitem[Kim16]{Kim16}
{Y. Kim},
{Strongly positive representations of even spin groups}, {\it Pacific J. Math.}, {\bf 280}(2016), no. 1, 
{69-88}.

\bibitem[KK04]{KK04}
H. H. Kim and M. Krishnamurthy. Base change lift for odd unitary groups. In {\it Functional analysis VIII},
volume 47 of Various Publ. Ser. (Aarhus), pages 116-125. {\it Aarhus University, Aarhus}, 2004.

\bibitem[KK05]{KK05}
H. H. Kim and M. Krishnamurthy. Stable base change lift from unitary groups to GLn. {\it Int. Math.
Res. Pap.}, (1):1-52, 2005.

\bibitem[KLM20]{KLM20}
{Y. Kim, B. Liu, and I. Mati\'c},
{Degenerate principal series for classical and odd GSpin groups in the general case}, {\it Represent. Theory} \textbf{24}(2020), 403-434.

\bibitem[KM19]{KM19}
{Y. Kim and I. Mati\'c}.
{Classification of strongly positive representations of even general unitary groups}. In
Representations of reductive p-adic groups, 161-174, Progr. Math., 328, {\it Birkh\"{a}user/Springer, Singapore}, 2019.

\bibitem[Kon03]{Kon03}
{T. Konno},
{A note on the Langlands classification and irreducibility of induced representations of $p$-adic groups},
{\it Kyushu J. Math.}, \textbf{57}(2003),
383-409.

\bibitem[Kud94]{Kud94}
S. Kudla, 
The local Langlands correspondence: the non-Archimedean case.
Motives (Seattle, WA, 1991), 365--391, Proc. Sympos. Pure Math. 55, Part 2,
Amer. Math. Soc., Providence, RI, 1994.

\bibitem[LMT04]{LMT04}
E. Lapid, G. Mui'c and M. Tadi'c. 
On the generic unitary dual of quasi-split classical groups. 
{\it Int. Math. Res. Not.}, (26):1335-1354, 2004.

\bibitem[Liu11] {Liu11} 
B. Liu,
{Genericity of representations of p-adic $Sp_{2n}$ and local Langlands parameters},
{\it Canad. J. Math.}, {\bf 63}(2011), no. 5, 1107-1136. 

\bibitem[LS23]{LS22}
B. Liu and F. Shahidi. Jiang's conjecture on the wave front set of local Arthur packets. preprint. 2023.

\bibitem[M{\oe}14]{Moe14}
{C. M{\oe}glin},
{Paquets stables des s\'eries discr\`etes accessibles par endoscopie tordue; leur param\`etre des Langlands}. In {\it Automorphic Forms and Related Geometry: Assessing the Legacy of I. I. Piatetski-Shapiro}, 295-336, {Contemp. Math.}, {\bf 614}(2014), {\it American Mathematical Society, Providence, RI}.

% \bibitem[M{\oe}2]{Moe}
% {C. M{\oe}glin},
% {Paquetes stable de s\'eries discr\`ete des groupes classiques p-adiques: param\`etres de Langlands et exhaustivit\'e},
% (preprint).

\bibitem[Mok15]{Mok15}
Chung Pang Mok. Endoscopic classification of representations of quasi-split unitary groups. {\it Mem. Amer. Math. Soc.}, \textbf{235}(2015), no. 1108, vi+248 pp.

\bibitem[MT02]{MT02}
{C. M{\oe}glin and M. Tadi\'c},
{Construction of discrete series for classical $p$-adic
groups}, {\it J. Amer. Math. Soc.} {\bf 15}(2002), no. 3, 
{715-786}.

\bibitem[MVW87]{MVW87}
{C. M{\oe}glin, M.-F. Vigneras, and J.-L. Waldspurger},
{Correspondances de Howe sur un corps $p$-adiques}. {Lecture Notes in Mathematics}, {\bf 1291}. {\it Springer-Verlag, Berlin} 1987.



\bibitem[Mui98a] {Mui98a} 
G. Mui\'c,
{Some results on square integrable representations; irreducibility of standard representations},
{\it Internat. Math. Res. Notices}, {\bf 1998}, no. 14, 705-726.

\bibitem[Mui98b] {Mui98b} 
G. Mui\'c,
{On irreducible generic representations of Sp(n,F) and SO(2n+1,F)},
{\it Glas. Mat. Ser. III}, {\bf 33(53)}(1998), no. 1, 19-31.


%\bibitem[Mu3]{Mu2001}
%{G. Mui\'c},%
%{A proof of Casselman-Shahidi's conjecure for quasi-split classical groups}, {\it Canad. Math. Bull.} {\bf 44}(2001),
%{298-312}.


\bibitem[Mui04]{Mui04}
{G Mui\'c},
{Composition series of generalized principal series; the case of strongly positive discrete series}, {\it Israel J. Math.}, {\bf 140}(2004), 157-202.

\bibitem[Mui06] {Mui06} 
G. Mui\'c,
{Construction of Steinberg type representations for reductive $p$-adic groups},
{\it Math. Z.}, {\bf 253}(2006), no. 3, 635--652.


\bibitem[Rod73] {Rod73} 
F. Rodier, 
{Whittaker models for admissible representations}. In
{\it Harmonic Analysis on Homogeneous Spaces}, 425-430, Proc. Sympos. Pure Math. {\bf 26}(1973), {\it American Mathematical Society, Providence, RI}.

\bibitem[ST93]{S-T}
{P. Sally and M. Tadi\'c},
{Induced representations and classification for $GSp(2,F)$ and $Sp(2,F)$}, {\it M\'em. Soc. Math. France}, {\bf 52}(1993),
{75-133}.

\bibitem[SS97]{SS97}
P. Schneider and U. Stuhler,
Representation theory and sheaves on the Bruhat-Tits building.
{\it Publ. Math. Inst. Hautes \'Etudes Sci.} No. {\bf 85} (1997), 97--191.

\bibitem[Ser97]{Ser97}
J.-P. Serre,
Galois Cohomology. 
Springer Monographs in Mathematics. {\it Springer-Verlag Berlin}, 1997.

\bibitem[Sha74]{Sha74}
J. A. Shalika. The multiplicity one theorem for $GL_n$. {\it Ann. of Math.} (2), 100:171-193, 1974.

\bibitem[Sha88]{Sha88}
F. Shahidi. On the Ramanujan conjecture and finiteness of poles for certain L-functions. {\it Ann. of Math.} (2), 127(3):547-584, 1988.

\bibitem[Sha90a]{Sha90a}
{F. Shahidi},
{A proof of Langlands conjecture on Plancherel measure; complementary
series for $p$-adic groups}
{\it Ann. of Math.}, {\bf 132}(1990), no. 2, {273-330}.

\bibitem[Sha90b]{Sha90b}
F. Shahidi. On multiplicativity of local factors. In {\it Festschrift in honor of I. I. Piatetski-Shapiro on the occasion of his sixtieth birthday}, Part II (Ramat Aviv, 1989), volume 3 of Israel Math. Conf. Proc., pages 279-289. {\it Weizmann, Jerusalem}, 1990.

\bibitem[Sha91]{Sha91}
F. Shahidi. Langlands' conjecture on Plancherel measures for $p$-adic groups. In {\it Harmonic analysis on reductive groups}, (Brunswick, ME, 1989), Progress in Math., {\bf 101}, pages 277-295. {\it Birkh\"{a}user, Boston}, 1991.

\bibitem[Sha92]{Sha92}
{F. Shahidi},
{Twisted endoscopy and reducibility of induced representations
for $p$-adic groups},
{\it Duke Math. J.}, {\bf 66}(1992), no. 1, {1-41}.

\bibitem[Sha10]{Sha10}
F. Shahidi. Eisenstein series and automorphic L-functions, American Mathematical Society Colloquium Publications, \textbf{58}. {\it American Mathematical Society, Providence, RI}, 2010.

\bibitem[Sha11]{Sha11}
F. Shahidi. 
Arthur packets and the Ramanujan conjecture, 
{\it Kyoto J. Math.} 51 (2011), no. 1,
1--23.

\bibitem[Sha20]{Sha20}
D. Shankman.
Local Langlands correspondence for Asai $L$ and epsilon factors.
Thesis. 2020. 

\bibitem[Sil79]{Sil79}
A. Silberger, 
{Introduction to Harmonic Analysis on Reductive p-adic Groups.}
%Based on lectures by Harish-Chandra at the Institute for Advanced Study, 1971--1973.
Mathematical Notes, {\bf 23}. {\it Princeton University Press, Princeton, N.J.; University of Tokyo Press, Tokyo}, 1979.

\bibitem[Sil80]{Sil80}
{A. Silberger},
{Special representations of reductive $p$-adic groups are not integrable},
{\it Ann. of Math.}, {\bf 111}(1980), no. 3, 571-587.


\bibitem[Spr98]{Spr98}
{T. Springer},
{Linear Algebraic Groups (second edition)}. Progress in Mathematics. {\it Birkh\"{a}user Boston, Inc., Boston, MA}, {\bf 9}, 1998.


\bibitem[ST15]{ST15}
D. Soudry and Y. Tanay. 
On local descent for unitary groups. {\it J. Number Theory}, \textbf{146} (2015), 557-626.

\bibitem[Tad92]{Tad92}
M. Tadic.
Notes on representations of non-Archimedean $SL(n)$. 
{\it Pacific J. Math.}, \textbf{152}(2): 375-396, 1992.

\bibitem[Tad94]{Tad94}
{M. Tadi\'c},
{Representations of $p$-adic symplectic groups},
{\it Compositio Math.}, {\bf 90}(1994), no. 2, {123-181}.



\bibitem[Tad95]{Tad95}
{M. Tadi\'{c}}, 
{Structure arising from induction and Jacquet modules of representations of
classical $p$-adic groups},
{\it J. Algebra}, {\bf 177}(1995), no. 1, 1-33.


\bibitem[Tad98a]{Tad98a}
M. Tadi\'c, 
{On reducibility of parabolic induction},
{\it Israel J. Math.}, {\bf 107}(1998), 29--91.


\bibitem[Tad98b]{Tad98b}
{M. Tadi\'{c}}, 
{On regular square integrable representations of $p$-adic groups},
{\it Amer. J. Math.}, {\bf 120}(1998), no. 1,  159-210.

\bibitem[Tad02] {Tad02} 
M. Tadi\'c, 
{A family of square-integrable representations of
classical p-adic groups in the case of general half-integral reducibilities},
{\it Glas. Mat. Ser. III}, {\bf 37(57)}(2002), no. 1, 21--57.

\bibitem[Tat67]{Tat67}
J. T. Tate. Fourier analysis in number fields, and Hecke's zeta-functions. In {\it Algebraic Number Theory (Proc. Instructional Conf., Brighton, 1965)}, pages 305-347. {\it Thompson, Washington, D.C.}, 1967.

\bibitem[Tat79]{Tat79}
J. Tate. Number theoretic background. In {\it Automorphic forms, representations and L-functions (Oregon State Univ., Corvallis, Ore., 1977), Part 2}, Proc. Sympos. Pure Math., XXXIII, pages 3-26. {\it American Mathematical Society, Providence, R.I.}, 1979.


\bibitem[Wal03]{Wal03}
{J.-L. Waldspurger}, 
{La formule de Plancherel pour les groupes $p$-adiques d'apr\`es Harish-Chandra},
{\it J. Inst. Math. Jussieu}, {\bf 2}(2003), no. 2, 235-333.


\bibitem[Xu18]{Xu18}
{B. Xu},
{L-packets of quasi-split $GSp(2n)$ and $GO(2n)$},
{\it Math. Ann.}, \textbf{370}(1-2):710-189, 2018.


\bibitem[Zel80] {Zel80} 
A. V. Zelevinsky,
{Induced representations of reductive p-adic groups. II. On irreducible representations of $GL(n)$.}
{\it Ann. Sci. \'Ecole Norm. Sup.} {\bf (4)13} (1980), no. 2, 165-210. 

















\end{thebibliography}
% \bibliographystyle{alpha}

\end{document}